\documentclass[a4paper,12pt]{amsart}

\usepackage{amsmath,amssymb,amstext}
\usepackage{mathrsfs}
\usepackage{oldgerm}
\usepackage{amsthm}
\usepackage{amsmath}
\usepackage{mathrsfs}
\usepackage[utf8]{inputenc}
\usepackage[mathcal]{euscript}
\usepackage{float}
\usepackage{mathdots}
\usepackage{fullpage}
\usepackage{enumerate}
\usepackage{diagbox}
\usepackage{mathtools}

\usepackage{tikzpeople}
\usepackage[arrow, matrix, curve]{xy}

\usepackage[linesnumbered,ruled,vlined]{algorithm2e}

\usepackage{hyperref}

\newtheorem{theorem}{Theorem}[section]
\newtheorem*{theorem*}{Theorem}
\newtheorem{lem}[theorem]{Lemma}

\numberwithin{equation}{section}

\newtheorem{cor}[theorem]{Corollary}

\newtheorem*{cor*}{Corollary}
\theoremstyle{remark}
\newtheorem*{claim}{Claim}
\newtheorem*{rem}{Remark}

\newtheorem*{ex}{Example}
\theoremstyle{definition}
\newtheorem{defi}[theorem]{Definition}

\makeatletter
\newcommand*{\rom}[1]{\expandafter\@slowromancap\romannumeral #1@}
\makeatother

\DeclareMathOperator{\SL}{SL}
\DeclareMathOperator{\GL}{GL}
\DeclareMathOperator{\PGL}{PGL}
\DeclareMathOperator{\Coeff}{Coeff}
\DeclareMathOperator{\id}{id}
\DeclareMathOperator{\Image}{Im}
\DeclareMathOperator{\Ker}{Ker}
\DeclareMathOperator{\Quid}{Quid}
\DeclareMathOperator{\charac}{char}
\DeclareMathOperator{\lcm}{lcm}

\newcommand{\Z}{\mathbb{Z}}
\newcommand{\N}{\mathbb{N}}

\newcommand{\R}{\mathbb{R}}
\newcommand{\C}{\mathbb{C}}

\newcommand{\M}{\mathcal{M}}

\newcommand{\pr}{\mathbb{P}}

\newcommand{\F}{\mathbb{F}}

\newcommand{\G}{\mathscr{G}}

\allowdisplaybreaks

\title{Counting tame $\SL_3$- and $\SL_4$-frieze patterns over finite fields}

\address{Institute of Algebra, Number Theory and discrete Mathematics, Leibniz Universität Hannover, Welfengarten 1, Hannover, DE}

\email{surmann@math.uni-hannover.de}

\author{Lucas Surmann}

\begin{document}

\begin{abstract}
In this article we count tame $ \SL_3 $- and $ \SL_4 $-frieze patterns with width $ w $ over a finite field $ K $, as well as some tame $ \SL_k $-frieze patterns for higher $ k $. Let $ n = w + k + 1 $. We consider the sets $ C_k(n) $ of tuples of $ n $ points in the projective space $ \pr^{k-1}(K) $, such that $ k $ consecutive points are always independent (the first and last point in the tuple are considered to be consecutive).\\
First assume $ \gcd(k,n) = 1 $. In this case, we prove that the problem of counting tame $ \SL_k $-frieze patterns can be reduced to counting $ C_k(n) $. We also show that $ \lvert C_k(n) \rvert $ is essentially already known as long as $ k $ and $ n $ are coprime, and we derive the number of tame $ \SL_k $-frieze-patterns in that case.\\
In the case $ \gcd(k,n) \neq 1 $, we define certain subsets $ C_k^*(n) $ and show that it is sufficient to count these sets. Afterwards, we count $ C_k^*(n)$ in the cases $ k = 3 $ and $ k = 4 $ and thus the number of tame $ \SL_3 $- and $ \SL_4 $-frieze patterns for any width $ w $.
\end{abstract}

\maketitle

\section*{Introduction}

$ \SL_2 $-frieze-patterns are infinite arrays of numbers with a finite number of rows where all $ 2 \times 2 $-determinants of entries in a so called $ 2 \times 2 $-diamond in the array are required to equal $ 1 $. They were introduced by Conway and Coxeter \cite{Conway_Coxeter_1973_A}, \cite{Conway_Coxeter_1973_B} and are also known as Conway-Coxeter frieze patterns if the entries are positive integers. They have been studied a lot in the decades since. As a result, they are quite well understood. $ \SL_k $-frieze patterns are a generalization of Conway-Coxeter frieze patterns where $ k \times k $-determinants of $ k \times k $-diamonds are considered instead for some $ k \geq 2 $. The idea of using higher determinants in frieze-patterns is originally due to Cordes and Roselle \cite{cordes1972generalized}, who introduced so called $ m $-friezes. Here, certain $ k \times k $-determinants are required to be equal to $ 1 $ for all $ 2 \leq k \leq m $, whereas an $ \SL_k $-frieze pattern only requires $ k \times k $-determinants to be equal to $ 1 $. $ \SL_k $-frieze patterns arise as a special case of $ \SL_k $-tilings, where both the number of rows and columns is infinite. $ \SL_k $-tilings were first introduced by Bergeron and Reutenauer \cite{bergeron2010}.\\
$ \SL_k $-frieze patterns are defined to be array of numbers (or more generally elements of a ring $ R $) with a finite number of infinite rows such that every block of $ k \times k $ adjacent entries -- a so called $ k \times k $-diamond -- has determinant $ 1 $. Additionally, we require that the highest and lowest $ k - 1 $ rows only contain zeros, while the rows immediately below respectively above that only contain ones.
Let $ w $ be the number of nontrivial rows in a given $ \SL_k $-frieze pattern. Then we say that the frieze pattern has width $ w $. Often it is also required that every $ (k+1) \times (k+1) $-diamond has determinant zero, in which case the $ \SL_k $-frieze pattern is called tame.
One big advantage of studying tame $ \SL_k $-frieze patterns instead of $ \SL_k $-frieze patterns in general is that their rows are periodic with period $ n = w + k + 1 $ \cite[Corollary 7.1.1]{morier2014linear}, so there are only finitely many different elements to deal with.\\
Not much is known about $ \SL_k $-frieze patterns for $ k \geq 3 $. We know that there is a connection to the Grassmannian and cluster algebras, see \cite{morier2014linear} and \cite{baur2021friezes}. This link is further explored in \cite{cuntz2023}. \cite{surmann} uses the cluster structure to calculate entries of tame $ \SL_3 $-frieze patterns from a cluster. For wild $ \SL_k $-frieze patterns even less is known, but there is some work by Cuntz about the existence of wild $ \SL_k $-frieze patterns with specific properties \cite{cuntz2017wild}.\\
In this paper we specifically consider tame $ \SL_k $-frieze patterns with entries in a finite field $ K = \F_q $. Our goal is to count such frieze patterns. The case $ k = 2 $ was already done by Morier-Genoud in 2021 \cite{morier2021counting}. Later, tame $ \SL_2 $-frieze patterns were also counted more generally over finite, commutative, local rings by Böhmler and Cuntz \cite{bohmler2024frieze}. In this paper we count tame $ \SL_3 $-frieze patterns as well as tame $ \SL_4 $-frieze patterns of any width $ w $ over finite fields. We will also count $ \SL_k $-frieze of width $ w $ for all $ k $ such that $ k $ and $ n = w + k + 1 $ are coprime. We closely follow the approach in \cite{morier2021counting}, generalizing the method used there to higher $ k $.\\
In the case $ \gcd(k,n) = 1 $ we show that tame $ \SL_k $-frieze patterns of width $ w = n-k-1 $ over a field $ K $ are closely related to the spaces
\[
C_k(n) := \{ (v_1,\dots,v_n) \in (\pr^{k-1}(K))^n \mid \{v_i,v_{i+1},\dots,v_{i+k-1}\} \text{ is independent for } i \in [n] \},
\]
where the indices are considered modulo $ n $. In other words, $ C_k(n) $ contains tuples of $ n $ points from the projective space $ \pr^{k-1}(K) $ with the requirement that $ k $ consecutive points have to be independent. Here we consider $ v_n $ and $ v_1 $ to be consecutive. The group $ \PGL(k,K) $ acts on $ C_k(n) $ and this fact is also fundamental to our method. We denote the set of orbits under this action as $ \M_k(n) $.\\
In the case $ g := \gcd(k,n) > 1 $ one has to instead consider certain subspaces of $ C_k(n) $: Choose some lift $ (V_1,\dots,V_n) \in (K^k)^n $ and define $ d_i := \det(V_i,V_{i+1},\dots,V_{i+k-1}) $. Consider the equations
\begin{equation}\label{eq:c*equation}
d_i d_{i+g} \dots d_{i+n-g} = (-1)^{(k-1)\frac{k}{g}} d_{i+1} d_{i+g+1} \dots d_{i+n-g+1} \forall i \in [g-1].
\end{equation}
Then we write
\[
C_k^*(n) := \{ (v_1,\dots,v_n) \in C_k(n) \mid \eqref{eq:c*equation} \text{ holds for some lift } (V_1,\dots,V_n)\text{ of } (v_1,\dots,v_n)\}.
\]
Note that this definition is independent of the choice of lift. We will prove that the problem of counting tame $ \SL_k $-frieze patterns can be reduced to counting these spaces. More precisely, we have
\begin{theorem*}[\ref{thm:numberOfFriezes}]
Let $ k \geq 2 $ and $ w \geq 1 $. Let $ n := w + k + 1 $. Let $ g := \gcd(k,n) $. Let $ K = \F_q $ be a finite field with $ q $ elements.
\begin{enumerate}
    \item If $ g = 1 $, then the number of tame $ \SL_k $-frieze patterns of width $ w $ is
    \[
    f_q(k,n) = \frac{\lvert C_k(n) \rvert}{\lvert \PGL(k,K) \rvert} = \lvert \M_k(n) \rvert.
    \]
    \item If $ g > 1 $, then the number of tame $ \SL_k $-frieze patterns of width $ w $ is
    \[
    f_q(k,n) = \frac{\lvert C_k^*(n) \rvert (q-1)^{g-1}}{\lvert \PGL(k,K) \rvert}.
    \]
\end{enumerate}
\end{theorem*}
Galashin and Lam \cite{galashin2024positroids} gave the point count of the subspace $ \Pi_{k,n}^\circ \subset \G(k,n) $ of the Grassmannian, where none of the cyclically consecutive Plücker coordinates vanish, in the case $  \gcd(k,n) = 1 $:
\begin{theorem*}[Galashin, Lam, 2024]
Let $ q $ be some prime power and $ \gcd(k,n) = 1 $. Then the point count of $ \Pi_{k,n}^\circ $ over the finite field $ \F_q $ is
\[
\lvert \Pi_{k,n}^\circ \rvert = (q-1)^{n-1} \frac{1}{[n]_q} \left[ \begin{matrix} n \\ k \end{matrix} \right]_q.
\]
\end{theorem*}
Here,
\begin{align*}
[n]_q &:= q^{n-1} + \dots + q + 1,\\
[n]_q! &:= [n]_q \dots [1]_q,
\intertext{and}
\left[ \begin{matrix} n \\ k \end{matrix} \right]_q &:= \frac{[n]_q!}{[n-k]_q! [k]_q!}.
\end{align*}
From this we can derive the following result:
\begin{cor*}[\ref{cor:numberOfFriezegcd1}]
Let $ \gcd(k,n) = 1 $. Then the number of tame $ \SL_k $-frieze patterns over $ K = \F_q $ with width $ w = n-k-1 $ is
\[
f_q(k,n) = \frac{1}{[n]_q} \left[ \begin{matrix} n \\ k \end{matrix} \right]_q.
\]
\end{cor*}

Unfortunately, counting the spaces $ C_k^*(n) $ turns out to be quite hard in general and we only managed to count $ C^*_k(n) $ for $ k = 3 $ and $ k = 4 $. It may be possible to adapt our methods for counting these sets to higher $ k $, but doing so involves a system of recursions that grows exponentially with $ k $, meaning it gets quite unwieldy for higher $ k $.\smallskip\\
In section 1 we explore the relationship between tame $ \SL_k $-frieze patterns and their quiddity sequence, an $ n $-periodic sequence of vectors $ q_i \in K^{k-1} $ that uniquely determines the frieze pattern in question. We use the fact that the variety of tame $ \SL_k $-frieze patterns can be embedded into the Grassmannian $ \G(k,n) $ in such a way that all entries correspond to certain Plücker coordinates \cite{morier2014linear}, \cite{baur2021friezes}. The entries of the quiddity sequence vectors then also correspond to certain Plücker coordinates. We then show, using Plücker relations, how to calculate arbitrary entries from the quiddity sequence and that the set of tame $ \SL_k $-frieze patterns of a given width are in bijection with the set of quiddity sequences of a specific period. We also give a criterion for when a sequence of vectors is the quiddity sequence of a tame $ \SL_k $-frieze pattern.\smallskip\\
Section 2 then establishes the connection between $ C_k^{(*)}(n) $ and $ \Quid(k,n) $. We show that a certain set $ \Coeff(v) \subset \Quid(k,n) $ can be assigned to each $ v \in C_k^{(*)}(n) $ and that $ \Quid(k,n) $ is the union of these sets. Moreover we prove that $ \Coeff(v) $ only depends on the class of $ v $ under the action of $ \PGL(k,K) $ and characterizes this class. We also show how to count $ \Coeff(v) $ and the stabilizer $ \PGL(k,K)_v $. Putting everything together we then prove Theorem \ref{thm:numberOfFriezes}. Afterwards, we show how to use the result from \cite{galashin2024positroids} to solve the case $ \gcd(k,n) = 1 $.\smallskip\\
Section 3 is dedicated to counting $ C_3(n) $ and thus tame $ \SL_3 $-frieze patterns with $ n \not\equiv 0 \mod 3 $. This involves defining some other, similar sets and solving a system of recursions. The main result from this section is a special case of Corollary \ref{cor:numberOfFriezegcd1}, but instead of relying on \cite{galashin2024positroids} we give a different, constructive method for counting $ C_3(n) $ that does not use any advanced theory. Only some basic Linear Algebra and Projective Geometry is required. Additionally, we also count $ C_3(n) $ in the case where $ n $ is divisible by $ 3 $, something which does not follow from \cite{galashin2024positroids}. More importantly, we derive some intermediate results that will be used in the next section.\smallskip\\
In section 4 we take care of the remaining tame $ \SL_3 $-frieze patterns by counting $ C_3^*(n) $. The method here is very similar to the previous section, we define some sets, derive a system of recursions for the number of elements of these sets, and then solve it. Our result here is
\begin{cor*}[\ref{cor:k3divisibleBy3}]
Let $ \lvert K \rvert = q < \infty $. Let $ w \geq 1 $ and $ n = w + 4 $. If $ \gcd(3,n) \neq 1 $, then the number of tame $ \SL_3 $-frieze patterns of width $ w $ is
\[
f_q(3,n) = \frac{q^{2n} + 3f(q)q^{\frac{4}{3}n} + (q^5-3q^4-2q^3+5q^2+9q+6)q^n + 3f(q)q^{\frac{2}{3}n+1} + q^3}{(q^2+q+1)(q+1)q^3(q-1)^2}
\]
with
\[
f(q) = q^3-q^2-q-2.
\]
\end{cor*}
We suspect that there is a way to formulate this with $ q $-expressions like in the case $ \gcd(k,n) = 1 $, but were unable to find it. \smallskip\\
Section 5 counts $ C_4(n) $ and thus $ \SL_4 $-frieze patterns where $ n $ is odd, or equivalently where the width $ w = n - 5 $ is even. Like in Section $ 3 $, this is a special case of Corollary \ref{cor:numberOfFriezegcd1}, but again we give a different, constructive proof and obtain intermediate results needed for the next section. We also count $ C_4(n) $ in the case $ n $ even. We use the same method as in Section $ 3 $, solving a system of recursions.\smallskip\\
Finally, Section 6 completes the case $ k = 4 $ by counting $ C_4^*(n) $ in the cases $ \gcd(4,n) = 2 $ and $ \gcd(4,n) = 4 $. We start with the easier case $ \gcd(4,n) = 2 $ and get the following result:
\begin{cor*}[\ref{cor:k4gcd2}]
Let $ n \geq 6 $ with $ \gcd(4,n) = 2 $ and $ \lvert K \rvert = q $. Then the number of tame $ \SL_4 $-frieze patterns over $ K $ with width $ n - 5 $ is
\begin{align*}
f_q(4,n) = \frac{q^{3n} + f(q)q^{2n} - g(q)q^{\frac{3}{2}n+1} + f(q)q^{n+2} + q^6}{(q^3+q^2+q+1)(q^2+q+1)(q+1)q^6(q-1)^3},
\end{align*}
where
\begin{align*}
f(q) &:= q^5-q^4-q^3-4q^2-2q-2,\\
g(q) &:= 2q^5-2q^4-2q^3-6q^2-4q-4.
\end{align*}
\end{cor*}
In the case $ \gcd(4,n) = 4 $ we obtain
\begin{cor*}[\ref{cor:k4gcd4}]
Let $ n \in \N $ be divisible by $ 4 $. Let $ \lvert K \rvert = q $.
Set
\[
s := \begin{cases}
    + &\text{if }\charac(K) \neq 2 \text{ and } 8 \mid n,\\
    - &\text{otherwise}.
\end{cases}
\]
Then
\begin{align*}
f_q(4,n) = \frac{f(q)}{(q^3+q^2+q+1)(q^2+q+1)(q+1)q^6(q-1)^3}
\end{align*}
is the number of tame $ \SL_4 $-frieze patterns of width $ n - k - 1 $ over $ K $, where
\begin{align*}
f(q) = q^{3n} &+ 4f_1^s(q)q^{\frac{9}{4}n} + f_2(q)q^{2n} + 6f_3^s(q)q^{\frac{7}{4}n} + f_4^s(q)q^{\frac{3}{2}n}\\
&+ 6f_3^s(q)q^{\frac{5}{4}n+1} + f_2(q)q^{n+2} + 4f_1^s(q)q^{\frac{3}{4}n+3} + q^6
\end{align*}
and
\begin{align*}
f_1^+(q) &:= -q^3-q^2-q-1,\\
f_2(q) &:= 3q^5-3q^4+3q^3-6q^2-6,\\
f_3^+(q) &:= -q^6+q^5+3q^4+6q^3+7q^2+5q+3,\\
f_4^+(q) &:= 3q^7+6q^6-26q^5-41q^4-73q^3-56q^2-44q-9
\intertext{and}
f_1^-(q) &:= q^4-q^3-q^2-q-2,\\
f_3^-(q) &:= q^7-3q^6-q^5+5q^3+9q^2+7q+6,\\
f_4^-(q) &:= q^9-6q^8+4q^7+29q^6+3q^5-20q^4-74q^3-79q^2-74q-24.
\end{align*}
\end{cor*}
Like in the case $ k =3 $ there may very well be a closed formula in terms of $ q $-expressions, but finding it remains an open problem.

\section{$ SL_k $-frieze patterns and quiddity sequences}
\begin{defi}[\cite{bergeron2010}]
Let $ w \in \N_{>0}, k \geq 2 $ and let $ K $ be a field. Consider the following infinite array of elements of $ K $
\[
\xymatrix@C=0.06em@R=.09em{
\dots && 0 && 0 && 0 && \dots \\
&&&&& \ddots \\
&& \dots && 0 && 0 && 0 && \dots \\
&&& \dots && 1 && 1 && 1 && \dots \\
&&&& \dots && m_{0,w} && m_{1,w} && m_{2,w} && \dots \\
&&&&&&&&& \ddots \\
&&&&&& \dots && m_{0,1} && m_{1,1} && m_{2,1} && \dots \\
&&&&&&& \dots && 1 && 1 && 1 && \dots \\
&&&&&&&& \dots && 0 && 0 && 0 && \dots \\
&&&&&&&&&&&&& \ddots \\
&&&&&&&&&& \dots && 0 && 0 && 0 && \dots
}
\]
where the $ k-1 $ highest and lowest rows only contain $ 0 $, the next rows after that contain only $ 1 $ and all other entries $ m_{i,j} $ for $ i \in \Z $ and $ j \in \{1,\dots,w\} $ are not known.
\begin{itemize}
    \item[(i)] The array is called a $ \SL_k $-frieze pattern if all $ k \times k $-diamonds of entries have determinant $ 1 $. $ w $ is called the width of the frieze pattern.
    \item[(ii)] If all $ (k+1) \times (k+1) $-diamonds have determinant $ 0 $ the frieze pattern is called tame.
\end{itemize}
\end{defi}

One can more generally define $ \SL_k $-frieze patterns over rings, this is used for instance in \cite{baur2021friezes}. We will only consider tame $ \SL_k $-frieze patterns over fields in this paper. Note that tame $ \SL_k $-frieze patterns are periodic with horizontal period $ n := w + k + 1 $ (see for instance \cite{morier2014linear} Corollary 7.1.1).
\begin{defi}
We write $ \mathcal{F}_K(k,n) $ for the set of all tame $ \SL_k $-frieze patterns with period $ n $ over the field $ K $. If $ K = \F_q $ is a finite field we write $ f_q(k,n) := \lvert \mathcal{F}_{\F_q}(k,n) \rvert $. 
\end{defi}
Recall that the Grassmannian $ \G(k,n) $ over $ K $ can be embedded into projective space using the Plücker embedding
\begin{align*}
\G(k,n) &\to \pr(\bigwedge\nolimits^k K^n) \\
\langle v_1,\dots,v_k\rangle &\mapsto [v_1 \wedge \dots \wedge v_k].
\end{align*}
Let $ \{e_1,\dots,e_n\} $ be the standard basis of $ K^n $. Then a basis of $ \bigwedge\nolimits^k K^n $ is given by the vectors $ e_{i_1} \wedge \dots \wedge e_{i_k} $ with $ 1 \leq i_1 < i_2 < \dots < i_k \leq n $. We denote the corresponding coordinates in the homogeneous coordinate ring $ K[\pr(\bigwedge^k K^n)] $ by $ x_{\{i_1,\dots,i_n\}} $. Now let $ \pi: K[\pr(\bigwedge^k K^n)] \to K[\G(k,n)] $ be the canonical projection.

\begin{defi}
Let $ I \subset [n] $ be a subset with $ k $ elements. Then the Plücker coordinate indexed by $ I $ is defined as
\[
\Delta^I := \pi(x_I).
\]
If $ |I| < k $ we will use the convention $ \Delta^I = 0 $.
\end{defi}

Plücker coordinates fulfill Plücker relations:

\begin{lem}[$ 3 $-term Plücker relations]
    Let $ I \subset [n] $ be a subset with $ k-2 $ elements. Let $ a,b,c,d \in [n] $ with $ a \neq b $ and $ c \neq d $. If $ a,b,c,d $ are all pairwise distinct, we also assume that $ (a,b) $ and $ (c,d) $ are crossing pairs, i.e. $ \lvert (a,b) \cap \{c,d\} \rvert = 1 $, where we assume w.l.o.g. that $ a < b $. Then
    \[
    \Delta^{I,a,b} \Delta^{I,c,d} = \Delta^{I,a,c} \Delta^{I,b,d} + \Delta^{I,a,d} \Delta^{I,b,c}.
    \]
Here $ \Delta^{I,x,y} $ is short for $ \Delta^{I\cup \{x,y\}} $.
\end{lem}
\begin{proof}
If $ a,b,c,d $ are pairwise distinct and $ a,b,c,d \notin I $ this is $ (1) $ from \cite{scott2006grassmannians}, page 5. If $ \{a,b,c,d\} \cap I \neq \emptyset $ then all products are zero by our convention and the equation holds. If $ a = c $, then the equation simplifies to $ \Delta^{I,a,b} \Delta^{I,a,d} = 0 + \Delta^{I,a,d} \Delta^{I,b,a} $ which is trivial. The cases $ a = b, b = c $ and $ b = d $ are analogous.
\end{proof}

There are also Plücker relations with more than $ 3 $ terms, but $ 3 $-term Plücker relations imply all others.

In \cite{morier2014linear} it was proved that the variety of tame $ \SL_k $-frieze patterns of width $ w = n - k - 1 $ is a subvariety of the Grassmannian $ \G(k,n) $. More precisely, in \cite{baur2021friezes} it was shown that each tame $ \SL_k $-frieze pattern of this width arises from specializing an arbitrary cluster in the cluster structure of the coordinate ring of the Grassmannian. Moreover, each entry of the $ \SL_k $-frieze pattern is then the value of a specific Plücker coordinate at the point of the Grassmannian corresponding to the tame $ \SL_k $-frieze pattern. To make this more precise we introduce the following notation:

\begin{defi} \label{def:plücker_coord}
Let $ i, j \in \Z $ and $ s \in [k-1] $. We write
\[
\Delta_j(i) := \Delta^{\{i,i+1,\dots,i+k-2\} \cup \{i+k-1+j\}}
\]
and
\[
\Delta_j^s(i) := \Delta^{\{i,i+1,\dots,i+k-1\} \cup \{i+k-1+j\} \backslash \{i+s\}}.
\]
The indices on the right hand side are considered modulo $ n $.
\end{defi}
Plücker coordinates of the form $ \Delta^{\{i,i+1,\dots,i+k-1\}} $ are called consecutive Plücker coordinates and the construction from \cite{baur2021friezes} requires us to specialize all the consecutive Plücker coordinates to $ 1 $. The Plücker coordinates of the form $ \Delta_j(i) $ are called "almost consecutive Plücker coordinates" and these are exactly the Plücker coordinates whose values appear in our $ \SL_k $-frieze pattern. Specifically, the fundamental region of our frieze pattern can be written as

\[
\xymatrix@C=0.06em@R=.05em{
\Delta_w(1) && \Delta_w(2) & \dots & \Delta_w(n)\\
& \Delta_{w-1}(2) && \Delta_{w-1}(3) & \dots & \Delta_{w-1}(1)\\
& \ddots && \ddots && \ddots \\
&& \Delta_2(w-1) && \Delta_2(w) & \dots & \Delta_2(w-2)\\
&&& \Delta_1(w) && \Delta_1(w+1) & \dots & \Delta_1(w-1).\\
}
\]

Here we assume that it is clear from context that $ \Delta_j(i) $ and $ \Delta_j^s(i) $ refer to the value of these Plücker coordinates at the point in the Grassmannian corresponding to a specific $ \SL_k $-frieze pattern. If we want to specify the frieze pattern we can write $ v_F(\Delta_j(i)) $ or $ v_F(\Delta_j^s(i)) $ for some tame $ \SL_k $-frieze pattern $ F $. \\
The relationship also holds true for the trivial rows of the frieze pattern: $ \Delta_0(i) $ and $ \Delta_{w+1}(i) $ are both consecutive Plücker coordinates and can therefore be identified with the rows of ones. And we have $ \Delta_j(i) = 0 $ for $ j \in \{-1,\dots,-k+1\} $ and for $ j \in \{w+2,\dots,w+k = n-1\} $, so the relationship applies to the rows of zeroes as well. Also, note that the entries on the same south-east diagonal as a given entry $ \Delta_j(i) $ are of the form $ \Delta_{j-t}(i+t) $ for some integer $ t $.

\begin{defi}
Let $ F $ be a tame $ \SL_k $-frieze pattern with width $ w $. Consider the vectors
\[
q_i := \begin{pmatrix}
    \Delta_1^{k-1}(i)\\
    - \Delta_1^{k-2}(i+1)\\
    \vdots\\
    (-1)^{k-2} \Delta_1^{1}(i+k-2)
\end{pmatrix}
\]
for $ i \in \Z $. The sequence $ (q_i)_{i \in \Z} $ of these vectors is called the quiddity sequence of $ F $. If we want to specify the frieze pattern, we write $ q_i(F) $. Set $ n := w + k + 1 $. The set of all such sequences for a fixed $ w $ and $ k $ is denoted by
\[
\Quid(k,n) = \{(q_i(F))_{i \in \Z} \mid F \in \mathcal{F}_K(k,n)\}.
\]
If we want to specify the base field, we write $ \Quid_K(k,n) $.
\end{defi}

It is unclear who invented the notion of quiddity sequences for tame $ \SL_k $-frieze patterns. Similar definitions can for instance be found in \cite{cuntz2017wild} and \cite{peterson2025slk}.

\begin{rem}
The quiddity sequence, like the tame $ \SL_k$-frieze pattern it is based on, has period $ n = w + k + 1 $. In particular, there are only $ n $ quiddity vectors that need to be considered. This is why we choose $ n $ as an index in $ \Quid(k,n) $ instead of $ w $.
\end{rem}

In the special case $ k = 2 $ this definition coincides with the usual definition of the quiddity sequence of a Conway-Coxeter frieze pattern. For $ k = 3 $ the quiddity vectors each contain one entry from the highest and lowest nontrivial row in the frieze pattern. For higher $ k $ some entries of the quiddity vectors do not appear directly as entries of the corresponding frieze pattern, but the entries of the highest and lowest nontrivial rows are still the lowest and highest entries of the quiddity vectors. The quiddity sequence is of interest because of the next result:

\begin{lem} \label{lem:recursion}
Let $ F $ be a tame $ \SL_k $-frieze pattern with entries in some field $ K $ and let $ (q_i)_{i \in \Z} $  with
\[
q_i = \begin{pmatrix}
    q_{i,1}\\
    q_{i,2}\\
    \vdots\\
    q_{i,k-1}
\end{pmatrix}
\]
be the corresponding quiddity sequence. Let $ i,j \in \Z $. Then the entry $ \Delta_j(i) $ of $ F $ fulfills the following equation:
\[
\Delta_j(i) = \sum_{t=1}^{k-1} q_{i,t} \Delta_{j-t}(i+t) + (-1)^{k-1} \Delta_{j-k}(i+k).
\]
In other words $ \Delta_j(i) $ is a linear expression in $ k $ other entries of the frieze pattern, specifically the next $ k $ entries on the same south-east diagonal. And the coefficients are the entries of the quiddity vector $ q_i $, except for the last coefficient which is simply $ (-1)^{k-1} $.
\end{lem}
\cite[Theorem 3.4.1]{morier2014linear} establishes this relationship between difference equations and tame $ \SL_k $-frieze patterns over $ \R $. \cite[Theorem 2.11]{surmann} contains the special case $ k = 3 $ over $ \C $.
\begin{proof}[Proof of Lemma \ref{lem:recursion}]
Apply the $ 3 $-term Plücker relations with $ a = i, b = i + k, c \in \{i+1,\dots,i+k-1\}, d = i + k - 1 + j $ and $ I = \{i+1,\dots,i+k-1\} \backslash \{c\} $. We get
\[
\Delta^{I,i,i+k} \Delta^{I,c,i+k-1+j} = \Delta^{I,i,c} \Delta^{I,i+k,i+k-1+j} + \Delta^{I,i,i+k-1+j} \Delta^{I,i+k,c}.
\]
Using the notation from Definition \ref{def:plücker_coord} this becomes
\[
\Delta_1^{c-i}(i) \Delta_{j-1}(i+1) = \Delta_0(i) \Delta_{j-1}^{c-i-1}(i+1) + \Delta_j^{c-i}(i) \Delta_0(i+1).
\]
Now let $ s \in \{0,1,\dots,k-1\} $ and set $ c = i + s $. Using $ \Delta_0(i) = 1 $ for all $ i $ and rearranging yields
\[
\Delta_j^s(i) = \Delta_1^s(i) \Delta_{j-1}(i+1) - \Delta_{j-1}^{s-1}(i+1).
\]
We now recursively substitute this equation into itself and get
\[
\Delta_j^s(i) = \sum_{t=0}^{s-1} (-1)^t \Delta_1^{s-t}(i+t) \Delta_{j-t-1}(i+t+1) + (-1)^s \Delta_{j-s}^0(i+s).
\]
Note that $ \Delta_{j-s}^0(i+s) = \Delta_{j-s-1}(i+s+1) $. Finally we choose $ s = k-1 $. Then $ \Delta_j^{k-1}(i) = \Delta_j(i) $ and our equation becomes
\begin{align*}
\Delta_j(i) &= \sum_{t=0}^{k-2} (-1)^t \Delta_1^{k-1-t}(i+t) \Delta_{j-t-1}(i+t+1) + (-1)^{k-1} \Delta_{j-k}(i+k)\\
&= \sum_{t=1}^{k-1} q_{i,t} \Delta_{j-t}(i+t) + (-1)^{k-1} \Delta_{j-k}(i+k)
\end{align*}
as claimed.
\end{proof}

\begin{lem} \label{lem:friezeQuidBijection}
All entries of a tame $ \SL_k $-frieze pattern are uniquely determined by its quiddity sequence. Likewise, the quiddity sequence is uniquely determined by the entries in $ F $. Hence there is a bijection between the set of tame $ \SL_k $-frieze patterns of period $ n $ and $ \Quid(k,n) $ and accordingly,
\[
f_q(k,n) = \lvert \Quid_{\F_q}(k,n) \rvert.
\]
\end{lem}
\begin{proof}
The lowest $ k $ rows are trivial. All other entries $ \Delta_j(i) $ can be calculated recursively using Lemma \ref{lem:recursion}, since the formula there only requires entries from lower rows and the quiddity vectors. \smallskip\\
For the other direction let $ F $ be some tame $ \SL_k $-frieze pattern. We want to show that its quiddity sequence $ (q_i)_{i \in \Z} $ can be calculated from the entries in $ F $. We use (downwards) induction by $ s $. \\
$ s = k-1 $: Apply Lemma \ref{lem:recursion} with $ j = 1 $. Then $ \Delta_{j-1}(i+1) = 1 $ and $ \Delta_{j-t}(i+t) = 0 $ for all $ t > 1 $. Hence the equation simplifies to
\[
\Delta_1(i) = q_{i,1}.
\]
$ s+1 \mapsto s $: We can assume $ 1 \leq s < k-1 $. Now apply Lemma \ref{lem:recursion} with $ j = k - s $. Then $ \Delta_{j-t}(i+t) = 1 $ for $ t = k - s $ and $ \Delta_{j-t}(i+t) = 0 $ for $ t > k - s $. Substituting this into the equation yields
\[
\Delta_j(i) = \sum_{t=1}^{k-s-1} q_{i,t} \Delta_{j-t}(i+t) + q_{i,k-s}.
\]
After rearranging for $ q_{i,k-s} $ it becomes clear that this entry of $ q_i $ can be written as an expression in the entries of $ F $ and the already known entries of $ q_i $.
\end{proof}

The last Lemma shows that it is sufficient to count $ \Quid(k,n) $ instead of counting tame $ \SL_k $-frieze patterns of a prescribed width directly. We now derive a criterion for when a given sequence $ (q_i)_{i \in \Z} $ is in $ \Quid(k,n) $.

\begin{lem} \label{lem:quiddityCriterion}
Let $ K $ be a field, let $ k \geq 2 $ and let $ (q_i)_{i \in \Z} \in K^{(k-1) \times \Z} $ be an $ n $-periodic sequence of vectors. Consider the matrices
\[
A_i = 
    \begin{pmatrix}
        q_{i-1,1} & q_{i-1,2} & \dots  & q_{i-1,k-1} & (-1)^{k-1} \\
        1       & 0       & \dots  & 0         & 0 \\
        0       & 1       & \dots  & 0         & 0 \\
        \vdots  & \vdots  & \ddots & \vdots    & \vdots \\
        0       & 0       & \dots  & 1         & 0
    \end{pmatrix} \in K^{k\times k}.
\]
Then $ (q_i)_i \in \Quid(k,n) $ if and only if
\begin{equation}
    A_1 \dots A_n = (-1)^{k-1}I_k.
\end{equation}
\end{lem}
\begin{proof}
First assume that $ (q_i)_i $ is the quiddity sequence of a tame $ \SL_k $-frieze pattern. For a given entry $ \Delta_j(i) $ of this frieze pattern with $ 0 \leq j \leq n - 1 $ let $ v_j(i) $ be the vector of the next $ k $ entries (including $ \Delta_j(i) $) on the same south-east diagonal, i.e.
\[
v_j(i) =
\begin{pmatrix}
    \Delta_j(i) \\
    \Delta_{j-1}(i+1) \\
    \vdots \\
    \Delta_{j-k+1}(i+k-1)
\end{pmatrix}.
\]
If $ j < n - 1 $, then by Lemma \ref{lem:recursion} and construction of the matrices $ A_i $ these vectors have the property
\begin{equation} \label{eq:recursionMatrix}
A_i v_j(i) = v_{j+1}(i-1).
\end{equation}
For $ j = n - 1 $ we have
\begin{equation} \label{eq:recursionMatrix2}
v_{n-1}(i) =
\begin{pmatrix}
    0 \\
    \vdots \\
    0 \\
    1
\end{pmatrix}
\text{ and }
A_i v_{n-1}(i) =
\begin{pmatrix}
    (-1)^{k-1} \\
    0 \\
    \vdots \\
    0
\end{pmatrix}
= (-1)^{k-1} v_0(i-1).
\end{equation}

\eqref{eq:recursionMatrix} and \eqref{eq:recursionMatrix2} put together imply
\[
A_1 \dots A_n v_j(n) = (-1)^{k-1} v_j(n)
\]
for all $ 0 \leq j \leq n - 1 $. Moreover, $ \{v_0(n),\dots,v_{k-1}(n)\} $ is a basis of $ K^k $. Therefore
\[
A_1 \dots A_n = (-1)^{k-1} I_k.
\]
For the other direction, assume that $ A_1 \dots A_n = (-1)^{k-1} I_k $. We construct an $ SL_k $-frieze pattern $ F $. Let $ e_j(i) $ be the entries with $ i \in \Z $ and $ j \in \{-k+1, \dots, w+k = n-1 \} $. We set $ e_j(i) := 0 $ for $ j < 0 $ and $ e_0(i) := 1 $ for all $ i \in \Z $. Then we recursively define all other entries as
\[
e_j(i) := \sum_{t=1}^{k-1} q_{i,t} e_{j-t}(i+t) + (-1)^{k-1} e_{j-k}(i+k),
\]
which is the same formula as in Lemma \ref{lem:recursion} with $ e_j(i) $ instead of $ \Delta_j(i) $. We now need to verify that $ F $ is indeed an $ \SL_k $-frieze pattern. Let
\[
v_j(i) :=
\begin{pmatrix}
    e_j(i) \\
    e_{j-1}(i+1) \\
    \vdots \\
    e_{j-k+1}(i+k-1)
\end{pmatrix}.
\]
First we will prove that the highest $ k-1 $ rows are filled with zeroes and the row below that with ones, or equivalently that $ v_{n-1}(i) = e_k $ for all $ i $ where $ e_k $ is the last vector of the standard basis $ \{e_1,\dots,e_k\} $ of $ K^k $. \eqref{eq:recursionMatrix} still holds for our new vectors $ v_j(i) $ by their definition and therefore we have
\begin{equation} \label{eq:almostFullCircle} 
v_{n-1}(i) = A_{i+1} \dots A_n A_1 \dots A_{i-1} v_0(i-1).
\end{equation}
We know that $ A_1 \dots A_n = (-1)^{k-1} I_k $ and it follows that $ A_{i} \dots A_n A_1 \dots A_{i-1} = (-1)^{k-1} I_k $ as well. Combined with \eqref{eq:almostFullCircle} this implies
\[
A_i v_{n-1}(i) = (-1)^{k-1} v_0(i-1) = (-1)^{k-1} e_1.
\]
Moreover, $ A_i $ is invertible and thus the map $ v \mapsto A_i v $ is bijective. And since we also have $ A_i e_k = (-1)^{k-1} e_1 $, it follows that $ v_{n-1}(i) = e_k $. \\
Next we need to check that all $ k \times k $-diamonds in $ F $ have determinant $ 1 $. For any entry $ e_j(i) $ with $ 0 \leq j \leq w+1 $ the diamond $ M_j(i) $ starting in $ e_j(i) $ can be written as
\[
M_j(i) = (v_j(i),v_{j+1}(i),\dots,v_{j+k-1}(i)).
\]
Clearly \eqref{eq:recursionMatrix} implies
\begin{align*}
    A_i M_j(i) &= (A_i v_j(i), A_i v_{j+1}(i), \dots, A_i v_{j+k-1}(i)) \\
               &= (v_{j+1}(i-1), v_{j+2}(i-1), \dots, v_{j+k}(i-1)) = M_{j+1}(i-1).
\end{align*}
In particular, we have $ M_j(i) = A_{i+1} \dots A_{i+j} M_0(i+j) $. Hence it is sufficient to calculate the determinants of $ A_i $ and $ M_0(i) $ for all $ i $. Laplace expansion of $ A_i $ along  the last column yields
\[
\det A_i = (-1)^{k+1} (-1)^{k-1} \det I_{k-1} = 1.
\]
Additionally, $ M_0(i) $ is an upper triangular matrix with only $ 1 $ on the main diagonal, therefore $ \det M_0(i) = 1 $. It follows that $ \det M_j(i) = 1 $ for all relevant $ i $ and $ j $. \\
We have shown that $ F $ is an $ \SL_k $-frieze pattern, the only thing that remains is proving that $ F $ is tame. Let $ N_j(i) $ for $ i \in \Z $ and $ j \in \{1,\dots,w\} $ be the $ (k+1) \times (k+1) $-diamond starting in $ e_j(i) $. Then we can write
\[
N_j(i) = (w_j(i),w_{j+1}(i),\dots,w_{j+k}(i)),
\]
where
\[
w_j(i) :=
\begin{pmatrix}
    e_j(i) \\
    e_{j-1}(i+1) \\
    \vdots \\
    e_{j-k}(i+k)
\end{pmatrix}.
\]
Consider the $ (k+1) \times (k+1) $ matrices
\[
B_i :=
\begin{pmatrix}
    q_{i-1,1} & q_{i-1,2} & \dots  & q_{i-1,k-1} & (-1)^{k-1} & 0 \\
    1       & 0       & \dots  & 0           & 0          & 0 \\
    0       & 1       & \dots  & 0           & 0          & 0 \\
    \vdots  & \vdots  & \ddots & \vdots      & \vdots     & \vdots \\
    0       & 0       & \dots  & 1           & 0          & 0 \\
    0       & 0       & \dots  & 0           & 1          & 0
\end{pmatrix}.
\]
We have $ B_i w_j(i) = w_{j+1}(i-1) $ and thus $ B_{i+1} N_{j-1}(i+1) = N_j(i) $ for all $ j \geq 2 $. Since $ B_i $ has a zero column its determinant vanishes, and thus $ \det N_j(i) = 0 $ for all $ j \geq 2 $. For the case $ j = 1 $ we use
\[
N_1(i) = B_{i+1}
\begin{pmatrix}
    v_0(i+1) & v_1(i+1) & \dots & v_k(i+1) \\
    *        & *        & \dots & *
\end{pmatrix},
\]
where the entries of the last row of the matrix on the right do not matter. Clearly, we have $ \det N_1(i) = 0 $ as well and $ F $ is a tame $ \SL_k $-frieze pattern with the desired quiddity sequence.

\end{proof}

\section{$ C_k(n), C_k^*(n) $ and $ \SL_k $-frieze patterns}

\begin{defi}
Let $ K $ be some field. Define
\[
C_k(n) := \{ (v_1,\dots,v_n) \in (\pr^{k-1}(K))^n \; | \; \{v_i,v_{i+1}\dots,v_{i+k-1}\} \text{ independent for } i \in [n] \},
\]
where the indices of the $ v_i $ are modulo $ n $ and
\[
\M_k(n) := C_k(n) / \PGL(k,K).
\]
\end{defi}
If $ n $ and $ k $ are relatively prime, we will be able to show that there is a correspondence between $ \Quid(k,n) $ and the set $ C_k(n) $. On the over hand, if $ g := \gcd(k,n) > 1 $, then we instead need to consider a subset of $ C_k(n) $. Let $ (v_1,\dots,v_n) \in C_k(n) $ and let $ (V_1,\dots,V_n) \in (K^k)^n $ be an arbitrary lift. Then we can look at the determinants $ d_i := \det(V_i,V_{i+1},\dots,V_{i+k-1}) $ for $ i \in [n] $, where the indices are modulo $ n $. Now consider the equations
\begin{equation} \label{eq:subsetDetCondition}
d_{i} d_{i+g} \dots d_{i+n-g} = (-1)^{(k-1)\frac{k}{g}} d_{i+1} d_{i+g+1} \dots d_{i+n-g+1}.
\end{equation}
for $ i \in [g-1] $. The reason why we choose the sign this way will become apparent during the proof of the next lemma. Note that all determinants in this equation are nonzero because of $ (v_1,\dots,v_n) \in C_k(n) $. Also note that the equation does not depend on the specific choice of lift, because replacing $ V_i $ with $ \lambda V_i $ for some nonzero scalar $ \lambda \in K^* $ will lead to both sides of the equation being multiplied with the same power of $ \lambda $. Therefore \eqref{eq:subsetDetCondition} is a condition not on the lift but rather on the underlying tuple $ (v_1,\dots,v_n) \in C_k(n) $ of projective points. \\
Additionally, note that the transformation $ (V_1,\dots,V_n) \mapsto (A V_1,\dots, A V_n) $ for some $ A \in \GL(K,k) $ does not affect \eqref{eq:subsetDetCondition} either since this is equivalent to multiplying all determinants with $ \det(A) $. Therefore \eqref{eq:subsetDetCondition} only depends on the class of $ (v_1,\dots,v_n) $ under the action of $ \PGL(k,K) $. This means that the following definitions are well defined:
\begin{defi}
Let $ k, n \in \N $ with $ k \geq 2 $ and $ n \geq k $. Let $ g := gcd(k,n) $. If $ g > 1 $ we define
\begin{align*}
C_k^*(n) := \{ (v_1,\dots,v_n) \in C_k(n) \mid &\eqref{eq:subsetDetCondition} \text{ holds for some lift } (V_1,\dots,V_n)\text{ of } (v_1,\dots,v_n)\\
&\text{for all } i \in [g-1]\}
\intertext{and}
\M_k^*(n) &:= C_k^*(n) / \PGL(k,K).
\end{align*}
\end{defi}

Let $ (v_1,\dots,v_n) \in C_k(n) $ and let $ (V_1,\dots,V_n) $ again be some lift. We continue the sequence of vectors $ V_1,\dots,V_n $ (anti-)periodically:
\[
V_i^\prime := (-1)^{k-1} V_i
\]
for $ i \in [n] $. The next step in our construction requires a specific choice of lift:
\begin{defi}
Let $ (v_1,\dots,v_n) \in C_k(n) $. We say that a lift $ (V_1,\dots,V_n) \in (K^k)^n $ has constant consecutive determinants if
\begin{align*}
det(V_1,\dots,V_k) &= det(V_2,\dots,V_{k+1}) = \dots = det(V_{n-k+1},\dots,V_n) \\
                   &= det(V_{n-k+2},\dots,V_n,V^\prime_1) = \dots = det(V_n,V^\prime_1,\dots,V^\prime_{k-1}).
\end{align*}
\end{defi}

We now prove that such lifts do in fact exist:

\begin{lem} \label{lem:suitableLift}
Let $ (v_1,\dots,v_n) \in C_k(n) $.
\begin{enumerate}[(i)]
\item If $ gcd(k,n) = 1 $, then a lift $ (V_1, \dots, V_n) \in (K^k)^n $ with constant consecutive determinants exists.
\item If $ gcd(k,n) = g > 1 $, then a lift with constant consecutive determinants exists if and only if $ (v_1,\dots,v_n) \in C^*_k(n) $.
\end{enumerate}
Moreover, if $ (V_1,\dots,V_n) $ is a lift with constant consecutive determinants, then
\[
\{ (\lambda_1 V_1, \dots, \lambda_g V_g, \lambda_1 V_{g+1}, \dots, \lambda_g V_{2g}, \dots, \lambda_1 V_{n-g+1}, \dots \lambda_g V_n) \mid \lambda_i \in K^* \}
\]
is the set of all such lifts.
\end{lem}

\begin{proof}
Choose any lift $ (V_1,\dots,V_n) $. We are searching for coefficients $ \lambda_i, c \in K^* $ with
\[
det(\lambda_i V_i,\dots,\lambda_{i+k-1} V_{i+k-1}) = c
\]
for all $ i \in [n] $, where $ V_{n+j} := V^\prime_j $. Equivalently, we can write
\begin{equation} \label{eq:lambdas}
\lambda_i \dots \lambda_{i+k-1} = b_i c,
\end{equation}
where $ b_i := 1 / det(V_i,\dots,V_{i+k-1}) $ and the indices on the lambdas are understood modulo $ n $. If we rearrange \eqref{eq:lambdas} for $ \lambda_{i+1} \dots \lambda_{i+k-1} $ and substitute this into the equation for $ i + 1 $ we get
\begin{equation} \label{eq:moduloK}
\lambda_{i+k} = \frac{b_{i+1}}{b_i} \lambda_i
\end{equation}
after rearranging again. Now let $ g := \gcd(k,n) $ be the greatest common divisor and $ m := \lcm(k,n) $ be the least common multiple. Let $ j $ be the factor with $ jk = m $ and $ l $ the factor with $ lg = n $. By applying \eqref{eq:moduloK} repeatedly we get
\begin{equation} \label{eq:periodic}
\lambda_i = \lambda_{i+m} = \frac{b_{i+(j-1)k+1} b_{i+(j-2)k+1} \dots b_{i+1}}{b_{i+(j-1)k} b_{i+(j-2)k} \dots b_i} \lambda_i
\end{equation}
Here $ \{ i + sk \; | \; 0 \leq s \leq j-1\} = \{ i + sg \; | \; 0 \leq s \leq l-1 \} $, since the indices are modulo $ n $. Combined with $ \lambda_i \neq 0 $, \eqref{eq:periodic} implies
\begin{equation} \label{eq:detCondition}
b_i b_{i+g} \dots b_{i+n-g} = b_{i+1} b_{i+g+1} \dots b_{i+n-g+1}.
\end{equation}
If $ g = 1 $, this condition is trivial. Otherwise, consider once again the numbers $ d_j := \det(V_j,\dots,V_{j+k-1}) $ used in the definition of \eqref{eq:subsetDetCondition}. Note that the definition there assumes the sequence $ (V_n)_n $ to be $ n $-periodic while the definition in this Lemma uses $ V_{n+j} := V_j^\prime = (-1)^{k-1}V_j $. Thus we have
\begin{equation}\label{eq:signChange}
b_j = \begin{cases}
    \frac{1}{d_j}, &\text{if } 0 < j \leq n-k+1 \text{ or } k \text{ odd},\\
    (-1)^{j-n+k-1}\frac{1}{d_j}, &\text{if } k \text{ even and } n-k+1 < j \leq n. 
\end{cases}
\end{equation}
Also note that \eqref{eq:detCondition} does not change under $ i \mapsto i + g $, therefore we may assume $ n-k+1 \leq i \leq n-k+g $. Then
\[
\frac{b_{i+jg}}{b_{i+jg+1}} = \begin{cases}
    (-1)^{k-1}\frac{d_{i+jg+1}}{d_{i+jg}}, &\text{if } 0 \leq j < \frac{k}{g},\\
    \frac{d_{i+jg+1}}{d_{i+jg}}, &\text{otherwise}.
\end{cases}
\]
It follows that substituting \eqref{eq:signChange} into \eqref{eq:detCondition} for all $ b_j $ changes the sign if and only if $ k $ is even and $ \frac{k}{g} $ is odd. Therefore, \eqref{eq:detCondition} is equivalent to \eqref{eq:subsetDetCondition}. Hence $ (v_1,\dots,v_n) $ can only have a lift with constant consecutive determinants if $ (v_1,\dots,v_n) \in C_k^*(n) $ and it remains to show that the system is solvable in that case (or always if $ g = 1 $). \\
Let $ e,f \in \Z $ be Bézout coefficients with $ en + fk = g $. We can choose them with $ f > 0 $. Then we have $ fk \equiv g \; mod \; n $ and thus
\[
\lambda_i = \frac{b_{i-k+1} b_{i-2k+1} \dots b_{i-fk+1}}{b_{i-k} b_{i-2k} \dots b_{i-fk}} \lambda_{i-fk} = \frac{b_{i-k+1} b_{i-2k+1} \dots b_{i-fk+1}}{b_{i-k} b_{i-2k} \dots b_{i-fk}} \lambda_{i-g},
\]
where we again repeatedly apply \eqref{eq:moduloK}. By writing
\[
c_i := \prod_{q=1}^f \frac{b_{i-qk+1}}{b_{i-qk}}
\]
we can simplify this to $ \lambda_i = c_i \lambda_{i-g} $. If $ i \equiv r \; mod \; g $ with $ 1 \leq r \leq g $, then
\[
\lambda_i = c_i c_{i-g} \dots c_{r+g} \lambda_r.
\]
Now we can rewrite \eqref{eq:lambdas} with $ i = 1 $ as
\[
\lambda_1 \dots \lambda_g c_{g+1} \lambda_1 \dots c_{2g} \lambda_g \dots (c_{k-g+1} \dots c_{g+1}) \lambda_1 \dots (c_k \dots c_{2g}) \lambda_g = b_1 c.
\]
By abbreviating the product of all constants on the left as $ c^\prime $ and rearranging we get
\[
\lambda_g^t = \frac{b_1 c}{c^\prime \lambda_1^t \dots \lambda_{g-1}^t},
\]
where $ t = k/g $. We can choose $ c $ such that the right hand side permits a $ t $-th root. Also note that every $ \lambda_g \in K^* $ can be a solution with the right choice of $ c $. Thus we have our solution candidate with free variables $ \lambda_1, \dots, \lambda_g $:
\begin{equation} \label{eq:solutionLambdas}
    \lambda_{sg + r} := c_{sg+r} \dots c_{g+r} \lambda_r, \\
\end{equation}
for $ 1 \leq r \leq g $. We now need to verify that this does indeed solve \eqref{eq:lambdas} for all $ i $. The first equation ($ i = 1 $) is clearly satisfied by our choice of $ c $. For all other equations it is enough to check that \eqref{eq:moduloK} holds for $ 1 \leq i < n $. For $ i \leq n-k $ this is equivalent to
\begin{equation} \label{eq:cProduct}
c_{i+k} c_{i+k-g} \dots c_{i+g} = \frac{b_{i+1}}{b_i},
\end{equation}
which we can rewrite as
\[
\prod_{p=0}^{t-1}c_{i+k-pg} = \prod_{p=0}^{t-1} \prod_{q=1}^f \frac{b_{i+k-pg-qk+1}}{b_{i+k-pg-qk}} = \prod_{p = 0}^{ft-1} \frac{b_{i-pg+1}}{b_{i-pg}} = \frac{b_{i+1}}{b_i},
\]
where we use $ k = tg $. Also note that $ (ft-1)g = ftg - g = fk - g \equiv g - g = 0 \mod n $, i.e. $ (ft-1)g $ is multiple of $ n $, lets say $ (ft-1)g = ln $. Taking the first term out of the product and using the $ n $-periodicity of the indices, we then get the equivalent equation
\[
\frac{b_{i+1}}{b_i} \left( \prod_{p=1}^{n/g} \frac{b_{i-pg+1}}{b_{i-pg}} \right)^l = \frac{b_{i+1}}{b_i}.
\]
Here the product in the brackets is $ 1 $ because of \eqref{eq:detCondition} and the equation holds. \\
Now consider $ n-k < i < n $. Let $ i \equiv r \; mod \; g $ with $ 0 < r < g $. \eqref{eq:moduloK} is then equivalent to
\begin{equation} \label{eq:lastCheck}
1 = \frac{b_{i+1}}{b_i} c_i c_{i-g} \dots c_{r+g}.
\end{equation}
To prove this, consider
\[
c_r c_{r+g} \dots c_{r-g} = \prod_{p=0}^{\frac{n}{g}-1} \prod_{q=1}^f \frac{b_{r-qk+pg+1}}{b_{r-qk+pg}} = \prod_{q=1}^f \prod_{p=0}^{\frac{n}{g}-1} \frac{b_{r-qk+pg+1}}{b_{r-qk+pg}} = \prod_{q=1}^f 1 = 1,
\]
where we used \eqref{eq:detCondition} on the inner product for the second to last equation. Consequently we have $ c_i c_{i-g} \dots c_{r+g} = \frac{1}{c_r c_{r-g} \dots c_{i+g}} $. Inserting this in \eqref{eq:lastCheck} and rearranging yields
\[
c_r c_{r-g} \dots c_{i+g} = \frac{b_{i+1}}{b_i}.
\]
This can be proven exactly like \eqref{eq:cProduct}. We have shown that a lift with constant consecutive determinants exists if and only if $ (v_1,\dots,v_n) \in C_4^*(n) $. \smallskip\\
Now let $ (V_1,\dots,V_n) $ be a lift with constant consecutive determinants. Then clearly $ \lambda_i = 1 $ for all $ i \in [n] $ and $ c = \det(V_1,\dots,V_k) $ is a solution of \eqref{eq:lambdas}. This implies that the product of constants on the right hand side of \eqref{eq:solutionLambdas} is equal to $ 1 $. Hence all solutions have the form $ \lambda_i = \lambda_r $ where $ 1 \leq r \leq g $ and $ i \equiv r \mod g $. On the other hand we already proved that $ \lambda_r $ for $ 1 \leq r \leq g $ can be chosen arbitrarily in $ K^* $ (possibly requiring a different choice of $ c $). This proves the second part of the Lemma.
\end{proof}

Now let $ (v_1,\dots,v_n) \in C_k(n) $ (or $ \in C_k^*(n) $ if $ gcd(k,n) > 1 $). Choose a lift $ (V_1,\dots,V_n) $ as in Lemma \ref{lem:suitableLift}. We define $ V_{i+n} = (-1)^{k-1} V_i $ anti-periodically for all $ i \in \Z $. For every $ i \in \Z $ expand $ V_i $ in the basis $ (V_{i+1},\dots,V_{i+k}) $.  Note that $ \det(V_i,\dots,V_{i+k-1}) = \det(V_{i+1},\dots,V_{i+k}) $ and therefore
\begin{align*}
0 &= \det(V_i,\dots,V_{i+k-1}) - (-1)^{k-1} \det(V_{i+k},V_{i+1},\dots,V_{i+k-1})\\
  &= \det(V_i - (-1)^{k-1} V_{i+k},V_{i+1},\dots,V_{i+k-1}).
\end{align*}
This implies that we can write
\begin{equation} \label{eq:basisExpansion}
V_i = a_{i,1} V_{i+1} + a_{i,2} V_{i+2} + \dots + a_{i,k-1} V_{i+k-1} + (-1)^{k-1} V_{i+k}
\end{equation}
with suitable coefficients $ a_{i,j} \in K $, where $ j \in [k-1] $ and $ i \in \Z $. Let
\[
a_i = \begin{pmatrix}
a_{i,1}\\
\vdots\\
a_{i,k-1}
\end{pmatrix}
\]
be the vector containing the coefficients of this basis expansion (except for the last coefficient $ (-1)^{k-1} $, which is not included).

\begin{defi}
Let $ v \in C_k^{(*)}(n) $. We write
\begin{align*}
\Coeff(v) = \{(a_i)_{i \in \Z} \mid &a_i \in K^{k-1}, (a_i)_i \text{ fulfills \eqref{eq:basisExpansion} for some lift } (V_1,\dots,V_n)\text{ of $ v $ with}\\
&\text{constant consecutive determinants}\}
\end{align*}
for the set of sequences of coefficient vectors that can arise from $ v $ for different choices of lifts with constant consecutive determinants.
\end{defi}

\begin{lem} \label{lem:coeffQuid}
We have
\[
\Quid(k,n) = \bigcup_{v\in C_k^{(*)}(n)} \Coeff(v).
\]
\end{lem}
\begin{proof}
Let $ (a_i)_i \in \Coeff(v) $ for $ v \in C_k^{(*)}(n) $ and let $ (V_1,\dots,V_n) $ be the corresponding lift with constant consecutive coefficients. Like in Lemma \ref{lem:quiddityCriterion}, consider the matrices
\[
A_i =
\begin{pmatrix}
    a_{i-1,1} & a_{i-1,2} & \dots  & a_{i-1,k-1} & (-1)^{k-1} \\
    1       & 0       & \dots  & 0         & 0 \\
    0       & 1       & \dots  & 0         & 0 \\
    \vdots  & \vdots  & \ddots & \vdots    & \vdots \\
    0       & 0       & \dots  & 1         & 0
\end{pmatrix}.
\]
Then we have
\[
A_i
\begin{pmatrix}
    V_i \\
    \vdots \\
    V_{i+k-1}
\end{pmatrix}
=
\begin{pmatrix}
    V_{i-1} \\
    \vdots \\
    V_{i+k-2}
\end{pmatrix}
\]
for all $ i \in \Z $, where we use the $ V_i $ as row vectors. In particular, it follows that
\[
\begin{pmatrix}
    V_0 \\
    \vdots \\
    V_{k-1}
\end{pmatrix}
= A_1 \dots A_n
\begin{pmatrix}
    V_n \\
    \vdots \\
    V_{n+k-1}
\end{pmatrix}
= A_1 \dots A_n (-1)^{k-1}
\begin{pmatrix}
    V_0 \\
    \vdots \\
    V_{k-1}
\end{pmatrix}.
\]
Since $ (V_0,\dots,V_{k-1}) $ is a basis of $ K^k $ we get $ A_1 \dots A_n = (-1)^{k-1} I_k $ and by Lemma \ref{lem:quiddityCriterion} $ (a_i)_i \in \Quid(k,n) $. \smallskip\\
For the other direction let $ (q_i)_i \in \Quid(k,n) $ and $ F $ be the $ \SL_k $-frieze pattern of width $ n-k-1 $ generated by this quiddity sequence. Now we consider the vectors $ V_i := (\Delta_{n-i}(i),\dots,\Delta_{n-i+k+1}(i)) $. Then the matrices
\[
M_{n-i}(i) =
\begin{pmatrix}
    V_i \\
    \vdots \\
    V_{i+k-1}
\end{pmatrix}
\]
are all $ k \times k $-diamonds of $ F $ and have determinant $ 1 $. In particular,
\[
v := (\pr(\langle V_1 \rangle),\dots,\pr(\langle V_n \rangle)) \in C_k(n).
\]
Moreover, $ (V_1,\dots,V_n) $ is a lift with constant consecutive determinants and thus $ v \in C_k^*(n) $ by Lemma \ref{lem:suitableLift} if $ g > 1 $. Also recall that in the proof of Lemma \ref{lem:quiddityCriterion} we showed
\[
A_i M_{n-i}(i) = M_{n-i+1}(i-1)
\]
for all $ i \in \Z $, where $ V_{n+i} = (-1)^{k-1} V_i $ and the matrices $ A_i $ use the entries of the vectors in the quiddity sequence. Consequently, \eqref{eq:basisExpansion} holds for the $ V_i $ and $ (q_i)_i \in \Coeff(v) $ follows.
\end{proof}
Next we will show that different choices of lift can indeed result in different sequences of coefficient vectors $ (a_i)_i $. Moreover, we will show how these different sequences are related to each other for a fixed $ v \in C_k^{(*)}(n) $.
\begin{lem} \label{lem:coeffStructure}
Let $ g := \gcd(k,n) $ and $ v \in C_k^{(*)}(n) $. Let $ (V_1,\dots,V_n) $ be some lift with constant consecutive determinants and $ (a_i)_i \in \Coeff(v) $ the corresponding coefficient vector sequence.
\begin{enumerate}[(i)]
    \item If $ g = 1 $, then $ \Coeff(v) = \{ (a_i)_i \} $.
    \item If $ g > 1 $, then $ \Coeff(v) = \{ (\begin{pmatrix}
        \frac{\lambda_i}{\lambda_{i+1}}a_{i,1}\\
        \vdots\\
        \frac{\lambda_i}{\lambda_{i+k-1}}a_{i,k-1}
    \end{pmatrix})_i \mid \lambda_1,\dots,\lambda_{g-1} \in K^*\} $, where $ \lambda_g = 1 $ and the indices of the lambdas are understood modulo $ g $.
\end{enumerate}
\end{lem}
\begin{proof}
According to Lemma \ref{lem:suitableLift} the set of suitable lifts is exactly
\[
\{ (\lambda_1 V_1, \dots, \lambda_g V_g, \lambda_1 V_{g+1}, \dots, \lambda_g V_{2g}, \dots, \lambda_1 V_{n-g+1}, \dots \lambda_g V_n) \mid \lambda_i \in K^* \}.
\]
Note that multiplying \eqref{eq:basisExpansion} with some nonzero scalar $ \mu \in K^* $ shows that the lifts $ (V_1,\dots,V_n) $ and $ (\mu V_1,\dots,\mu V_n) $ lead to the same coefficients. Therefore we can divide any lift by $ \lambda_g $ and assume $ \lambda_g = 1$. If $ g = 1 $ this leaves only one lift, namely $ (V_1,\dots,V_n) $ itself and we are done. \smallskip\\
Now let $ g > 1 $. Consider some lift $ (V^\prime_1,\dots,V^\prime_n) = (\lambda_1 V_1,\dots,\lambda_nV_n) $ with parameters $ \lambda_1, \dots, \lambda_{g-1} $ (we still understand the indices of the lambdas modulo $ g $ and assume $ \lambda_g = 1 $). Then define a sequence of vectors $ (b_i)_i $ with
\[
b_{i,j} := \frac{\lambda_i}{\lambda_{i+j}}a_{i,j}.
\]
Multiplying \eqref{eq:basisExpansion} by $ \lambda_i $ yields
\begin{align*}
V_i &= \sum_{j=1}^{k-1} a_{i,j} V_{i+j} + (-1)^{k-1} V_{i+k} \\
\Leftrightarrow \lambda_i V_i &= \sum_{j=1}^{k-1} \frac{\lambda_i}{\lambda_{i+j}} a_{i,j} \lambda_{i+j} V_{i+j} + (-1)^{k-1} \frac{\lambda_i}{\lambda_{i+k}} \lambda_{i+k} V_{i+k} \\
\Leftrightarrow V^\prime_i &= \sum_{j=1}^{k-1} b_{i,j} V^\prime_{i+j} + (-1)^{k-1} V^\prime_{i+k},
\end{align*}
where we use that $ k $ is a multiple of $ g $. Clearly $ (b_i)_i $ is the sequence that arises from the lift $ (V^\prime_1,\dots,V^\prime_n) $.
\end{proof}

\begin{lem}\label{lem:coeffPGL}
Let $ v,w \in C_k^{(*)}(n) $ with classes $ [v], [w] \in \M_k^{(*)}(n) $.
\begin{enumerate}[(i)]
    \item If $ [v] = [w] $, then $ \Coeff(v) = \Coeff(w) $.
    \item If $ [v] \neq [w] $, then $ \Coeff(v) \cap \Coeff(w) = \emptyset $.
\end{enumerate}
\end{lem}
\begin{proof}
Let $ (V_1,\dots,V_n) $ and $ (W_1,\dots,W_n) $ be lifts with constant consecutive determinants of $ v $ and $ w $ and let $ (a_i)_i \in \Coeff(v) $ be the coefficient vector sequence yielded by $ (V_1,\dots,V_n) $.\\
First assume $ [v] = [w] $.  Then there are $ M \in \GL(k,K) $ and $ \lambda_i \in K^* $ with $ W_i = \lambda_i M V_i $. Since
\begin{align*}
\det(W_i,\dots,W_{i+k-1}) &= \det(\lambda_i M V_i, \dots, \lambda_{i+k-1} M V_{i+k-1}) \\
                          &= \lambda_i \dots \lambda_{i+k-1} \det(M) \det(V_i,\dots,V_{i+k-1})
\end{align*}
and both lifts have constant consecutive determinants, it follows that $ \lambda_i \dots \lambda_{i+k-1} = \lambda_{i+1} \dots \lambda_{i+k} $ for all $ i \in [n] $ (the indices are considered modulo $ n $). This implies $ \lambda_i = \lambda_{i+k} $ and combined with Bézouts identity we get $ \lambda_i = \lambda_{i+g} $ where $ g = \gcd(k,n) $. By Lemma \ref{lem:suitableLift} $ (M V_1,\dots,M V_n) $ is then another lift with constant consecutive determinants for $ w $. Now applying $ M $ to both sides of \eqref{eq:basisExpansion} and using linearity we get
\[
    M V_i = a_{i,1} M V_{i+1} + a_{i,2} M V_{i+2} + \dots + a_{i,k-1} M V_{i+k-1} + (-1)^{k-1} M V_{i+k}.
\]
Consequently, $ (a_i)_i $ is also the coefficient vector sequence yielded by $ (M V_1, \dots M V_n) $ and $ (a_i)_i \in \Coeff(w) $. Lemma \ref{lem:coeffStructure} now implies that $ \Coeff(v) = \Coeff(w) $. \smallskip\\
For the other direction assume that the lifts $ (V_1,\dots,V_n) $ and $ (W_1,\dots,W_n) $ both yield the same coefficient vector sequence $ (a_i)_i \in \Coeff(v) \cap \Coeff(w) $. $ (V_1,\dots,V_k) $ and $ (W_1,\dots,W_k) $ are both ordered bases of $ K^k $, thus there exists a unique matrix $ M \in \GL(k,K) $ with $ M V_i = W_i $ for $ i \in [k] $. We claim $ M V_i = W_i $ for all $ i \in [n] $. Induction by $ i $: \smallskip\\
$ i \leq k $: Trivial. \smallskip\\
$ i-1 \mapsto i, i > k $: Consider \eqref{eq:basisExpansion} for $ i-k $ and rearrange for $ W_i $. This yields
\begin{align*}
W_{i} &= \sum_{j=1}^{k-1} a_{i-k,j} W_{i-k+j} + (-1)^{k-1} W_{i-k} \\
      &= \sum_{j=1}^{k-1} a_{i-k,j} M V_{i-k+j} + (-1)^{k-1} M V_{i-k} \\
      &= M\bigg (\sum_{j=1}^{k-1} a_{i-k,j} V_{i-k+j} + (-1)^{k-1} V_{i-k}\bigg ) \\
      &= M V_i.
\end{align*}
The second equality uses the induction hypothesis and the last equality uses that the coefficients for both lifts are the same. \smallskip\\
We proved that $ W_i = M V_i $ for all $ i \in [n] $. Hence $ [v] = [w] $. By contraposition, it follows that $ \Coeff(v) \cap \Coeff(w) = \emptyset $ if $ [v] \neq [w] $.
\end{proof}

Lemma \ref{lem:coeffStructure} implies that $ \lvert \Coeff(v) \rvert \leq (q-1)^{g-1} $ if $ \lvert K \rvert = q $, but we do not always have equality. To better understand $ \Coeff(v) $ we need the next definition:

\begin{defi}
Let $ v \in C_k(n) $ and $ (V_1,\dots,V_n) $ be a lift. Let $ r \in \N $ and let $ K^k = \bigoplus_{s=1}^r U_s $ be a decomposition of $ K^k $ into a direct sum of subspaces, where $ \dim(U_i) \geq 1 $ for all $ i \in [r] $. W.l.o.g we assume $ \dim(U_i) \leq \dim(U_j) $ if $ i < j $.
\begin{enumerate}[(i)]
    \item We say that $ v $ has the decomposition $ (U_1,\dots,U_r) $ if $ V_i \in \bigcup_{j=1}^r U_j $ for all $ i \in [n] $.
    \item If none of the subspaces $ U_j $ can be further decomposed into a direct sum $ U_i = W_1 \oplus W_2 $ such that $ (U_1,\dots,U_{j-1},W_1,W_2,U_{j+1},\dots,U_r) $ is again a decomposition of $ v $ (up to ordering), then we call this the maximal decomposition and say that $ v $ has type $ (\dim(U_1),\dots,\dim(U_r)) $.
\end{enumerate}
\end{defi}
This definition does not depend on the choice of lift. \\
Obviously, each $ V_i $ can only be contained in one of the subspaces $ U_i $. We can show that the distribution of the $ V_i $ in these subspaces is $ g $-periodic:

\begin{lem} \label{lem:decompositionMap}
Let $ v \in C_k(n) $ and let $ (U_1,\dots,U_r) $ be a decomposition of $ v $. For every $ i \in [n] $ there is exactly one $ j \in [r] $ with $ V_i \in U_j $. Let $ \varphi:\Z \to [r] $ be the map that maps $ i $ to that $ j $. Here the argument for $ \varphi $ is considered modulo $ n $. Let $ g := \gcd(k,n) $. Then $ \varphi $ is $ g $-periodic. In particular, $ r \leq g $ for any decomposition $ (U_1,\dots,U_r) $.
\end{lem}
\begin{proof}
The fact that $ V_i $ is only contained in one $ U_i $ follows immediately from the sum being direct and $ V_i \neq 0 $. To prove that $ \varphi $ is $ g $-periodic we will first show that $ \varphi $ is $ k $-periodic. For any $ i \in [n] $ the vectors $ V_i,\dots,V_{i+k-1} $ are a basis of $ K^k $. In particular, they are linearly independent and thus each $ U_j $ can contain at most $ \dim(U_j) $ of theses $ k $ vectors. But since $ \sum_{j=1}^r \dim(U_j) = k $ and since all $ V_i $ are contained in one of these subspaces, each of them contains exactly $ \dim(U_i) $ of the $ k $ vectors. The same argument can be applied to the vectors $ V_{i+1},\dots,V_{i+k} $. But since $ V_i \in U_{\varphi(i)} $, only $ \dim(U_{\varphi(i)}) - 1 $ of the vectors $ V_{i+1},\dots,V_{i+k-1} $ can be in $ U_{\varphi(i)} $. Hence $ V_{i+k} \in U_{\varphi(i)} $, which implies $ \varphi(i) = \varphi(i+k) $.\smallskip\\
Next choose Bézout coefficients $ a,b \in \Z $ with $ an + bk = g $. Then
\[
\varphi(i) = \varphi(i+an) = \varphi(i+an+bk) = \varphi(i+g)
\]
for all $ i \in \Z $. \smallskip\\
Finally, note that $ \varphi $ must be surjective, since every $ U_j $ contains at least one $ V_i $. Since $ \varphi $ is $ g $-periodic, this is only possible if $ r \leq g $.
\end{proof}

Next we will show that the maximal decomposition deserves the name.

\begin{lem} \label{lem:maximalDecomposition}
Let $ v \in C_k(n) $ and let $ (U_1,\dots,U_r) $ be a maximal decomposition of $ v $. Let $ (W_1,\dots,W_s) $ also be a decomposition of $ v $. Then for every $ j \in [r] $ there is an $ l \in [s] $ with $ U_j \subseteq W_l $. In particular, the maximal decomposition is unique up to order.
\end{lem}
\begin{proof}
Let $ j \in [r] $. Let $ M := \{ V_i \mid i \in [n], V_i \in U_j \} $. Then the argument from the proof of Lemma \ref{lem:decompositionMap} can be used to show that $ M $ generates $ U_j $. Moreover, for every $ V_i \in M $ there is a $ l \in [s] $ such that $ V_i \in W_l $. Let $ N = \{ l\in[s] \mid M \cap W_l \neq \emptyset \}.$ If $ \lvert N \rvert = 1 $, i.e. $ N = \{l\} $, then $ U_j = \langle M \rangle \subseteq W_l $ and we are done. Otherwise let $ U_{j,l} := U_j \cap W_l $ for $ l \in N $. Then replacing $ U_j $ with the nonzero subspaces $ U_{j,i} $ yields a finer decomposition of $ v $ in contradiction to $ (U_1,\dots,U_r) $ being maximal. The first part of the Lemma is proved. \smallskip\\
Now let $ (U_1,\dots,U_r) $ and $ (W_1,\dots,W_s) $ be two maximal decompositions of $ v $. What we just proved implies that for every $ j \in [r] $ there is an $ l \in [s] $ with $ U_j \subseteq W_l $. A second application shows that there is also an $ m \in [r] $ with $ W_l \subseteq U_m $. Here we have $ j = m $ because otherwise the sum of the $ U_j $ would not be direct. Hence $ U_j = W_l $ and both decompositions have the same subspaces.
\end{proof}

\begin{ex}
Let $ k = 3 $. Then there are three possible types for $ v \in C_3(n) $ if $ n $ is divisible by $ 3 $: $ (3), (1,2) $ and $ (1,1,1) $ (if $ \gcd(3,n) =1 $ only type $ (3) $ is possible by Lemma \ref{lem:decompositionMap}). If the type is $ (1,1,1) $, then $ v $ only contains three distinct projective points and is of the form
\[
(v_1,v_2,v_3,v_1,v_2,v_3,\dots,v_1,v_2,v_3).
\]
If $ v $ has type $ (1,2) $, then the decomposition has the form $ (L,P) $ where $ L $ is a projective line in $ \pr^2(K) $, while $ P $ is a single projective point. W.l.o.g., assume $ \varphi(1) = 2 $ for the map from Lemma \ref{lem:decompositionMap}. Then $ v $ has the form
\[
(P,v_1,v_2,P,v_3,v_4,\dots,P,v_{\frac{2n}{3}-1},v_{\frac{2n}{3}})
\]
with $ v_i \in L $ for all $ i $. Moreover, $ \{v_i,v_{i+1}\} $ is independent for all $ i $ (where $ i $ is understood modulo $ \frac{2n}{3} $). If we choose a homography $ L \to \pr^1(K) $ we can identify $ (v_1,\dots,v_\frac{2n}{3}) $ with an element from $ C_2(\frac{2n}{3}) $. Elements of type $ (3) $ are those that cannot be decomposed in such a fashion.
\end{ex}

Let $ v \in C_k^{(*)}(n) $ and $ (a_i)_i \in \Coeff(v) $. We will show that the possible decompositions of $ v $ are related to the question of which of the entries of the coefficient vectors are zero.
\begin{lem} \label{lem:decompositionZeros}
Let $ k \geq 2 $ and $ n \geq k $. Let $ g := \gcd(k,n) $. Let $ v \in C_k^{(*)}(n) $ with decomposition $ (U_1,\dots,U_r) $ and the corresponding map $ \varphi: \Z \to [r] $. Let $ (a_i)_i \in \Coeff(v) $.
\begin{enumerate}[(i)]
    \item If $ \varphi(i) \neq \varphi(i+j) $ for some $ i \in [n] $ and $ j \in [k-1] $, then $ a_{i,j} = 0 $.
    \item Let $ t \in [r] $ and let $ M := \varphi^{-1}(t) \cap [g] $. If the decomposition is maximal and $ M_1,M_2 \subset M $ are two nonempty subsets with $ M_1 \cap M_2 = \emptyset $ and $ M_1 \cup M_2 = M $, then there are $ i \in [n] $ and $ j \in [k-1] $ with $ i \in M_1 $ and $ i+j \in M_2 $ (modulo $ g $), such that $ a_{i,j} \neq 0 $.
\end{enumerate}
\end{lem}
\begin{proof}
(i): Let $ (V_1,\dots,V_n) $ be a lift of $ v $ with constant consecutive determinants that yields $ (a_i)_i \in \Coeff(v) $, let $ U := U_{\varphi(i)} $ and let $ M := \{ j \in [k] \mid \varphi(i) = \varphi(i+j) \} $. Then $ \{V_{i+1},\dots V_{i+k}\} $ is a basis of $ K^k $ and $ \{V_{i+j} \mid j \in M \} $ is a basis of $ U $. We can expand $ V_i \in U $ in the latter basis:
\[
V_i = \sum_{j \in M} a^\prime_{i,j} V_{i+j}
\]
with coefficients $ a^\prime_{i,j} \in K $. At the same time we have the basis expansion \eqref{eq:basisExpansion}. By uniqueness of basis expansions, these two expansion must be identical. In particular, $ a_{i,j} = 0 $ for all $ j \notin M $. \smallskip\\
(ii): Assume that $ a_{i,j} = 0 $ for all $ i \in [n] $ and $ j \in [k-1] $ with $ i \in M_1 $ and $ i+j \in M_2 $. Then let
\[
W_l := \big \langle V_i \mid i \in [k], i \in M_l \big \rangle
\]
for $ l = 1,2 $. Clearly $ U_t = W_1 \oplus W_2 $ with the argument from the proof of Lemma \ref{lem:decompositionMap}. We claim that $ V_i \in W_l $ for all $ i \in [n] $ with $ i \in M_l $. Induction: \smallskip\\
$ i \leq k $: Follows immediately from the definition of $ W_l $. \smallskip\\
$ i > k, i \in M_l $: We can assume $ V_j \in W_l $ for all $ j < i $ with $ j \in M_l $. Then consider the basis expansion \eqref{eq:basisExpansion} for $ V_{i-k} $:
\begin{align*}
V_{i-k} &= \sum_{j = 1}^{k-1} a_{i-k,j} V_{i-k+j} + (-1)^{k-1} V_{i} \\
    &= \sum_{\substack{j \in [k-1]\\i+j \in M_l}} a_{i-k,j} V_{i-k+j} + (-1)^{k-1} V_i,
\end{align*}
where we used (i) and our assumption. $ V_{i-k} $ and the sum on the right hand side except for $ (-1)^{k-1} V_i $ are in $ W_l $, therefore $ V_i $ must also be in $ W_l $. \smallskip\\
Hence all $ V_i \in U_t $ are also contained in either $ W_1 $ or $ W_2 $ and the decomposition is clearly not maximal.
\end{proof}

\begin{lem} \label{lem:countingCoeff}
Let $ k \geq 2 $ and $ n \geq k $. Let $ g := \gcd(k,n) $. Let $ K = \F_q $ be a finite field. Let $ v \in C_k^{(*)}(n) $ and let $ (U_1,\dots,U_r) $ be the maximal decomposition of $ v $. Then $ \lvert \Coeff(v) \rvert = (q-1)^{g-r} $.
\end{lem}
\begin{proof}
Let $ (a_i)_i \in \Coeff(v) $ and $ (V_1,\dots,V_n) $ be a corresponding lift with constant consecutive determinants. Let $ \varphi: \Z \to [r] $ be the map associated to the decomposition. For $ l \in [r] $ let
\[
M_l := \varphi^{-1}(l).
\]
Now let $ (\lambda_1,\dots,\lambda_{g-1}) $ and $ (\mu_1,\dots,\mu_{g-1}) $ be two sets of parameters ($ \lambda_g = \mu_g = 1 $ is assumed). We want to study when $ (\lambda_1 V_1,\dots,\lambda_n V_n) $ and $ (\mu_1 V_1,\dots,\mu_n V_n) $ yield the same coefficients and when they do not. This is equivalent to asking when
\[
\frac{\lambda_{i}}{\lambda_{i+j}}a_{i,j} = \frac{\mu_{i}}{\mu_{i+j}}a_{i,j}
\]
for all $ i \in [n] $ and $ j \in [k-1] $. If we consider the $ \mu_i $ to be fixed, we can divide the $ \lambda_i $ by $ \mu_i $ and see that it is sufficient to consider
\begin{equation}\label{eq:coefficientEquality}
\frac{\lambda_{i}}{\lambda_{i+j}}a_{i,j} = a_{i,j}
\end{equation}
for all $ i \in [n] $ and $ j \in [k-1] $.\\
First assume that $ \lambda_s = \lambda_t $ whenever $ s, t \in M_l $ for some $ l \in [r] $. Let $ i \in [n] $ and $ j \in [k-1] $. There are two cases: \smallskip\\
Case 1. $ \varphi(i) = \varphi(i+j) $: \\
Then $ \lambda_i = \lambda_{i+j} $ by our assumption and \eqref{eq:coefficientEquality} is trivial. \smallskip\\
Case 2. $ \varphi(i) \neq \varphi(i+j) $ \\
Then $ a_{i,j} = 0 $ by Lemma \ref{lem:decompositionZeros} (i) and \eqref{eq:coefficientEquality} holds. \smallskip\\
We have shown that \eqref{eq:coefficientEquality} holds if the $ \lambda_i $ are constant on the sets $ M_l $. In particular, the coefficients will not change if we multiply all $ \lambda_i $ with a fixed $ i \in M_l $ by the same value. \smallskip\\
Now let $ x \in M_l $ and let $ y = \lambda_x $. Let $ N = \{i\in M_l \mid \lambda_i = y \} $. If $ N \neq M_l $ we can apply Lemma \ref{lem:decompositionZeros} (ii) to $ N $ and $ M_l \backslash N $. This yields that there are $ i \in [n] $ and $ j \in [k-1] $ with $ i \in N $ and $ i+j \in M_l \backslash N $ and $ a_{i,j} \neq 0 $. Then
\[
\frac{\lambda_i}{\lambda_{i+j}}a_{i,j} \neq a_{i,j},
\]
since $ \lambda_i = y \neq \lambda_{i+j} $. Therefore \eqref{eq:coefficientEquality} can only hold if the $ \lambda_i $ are constant on the sets $ M_l $. \\
Now write
\[
M_l = \{z_{l,1}, \dots, z_{l,r_l}\}
\]
with $ z_{i,j} \in [g] $, where we assume that the $ z_{i,j} $ are in ascending order within $ M_l $. Then set $ \lambda_{z_{l,r_l}} = 1 $ for all $ l \in [r] $ while choosing all other $ \lambda_i $ arbitrarily. Since we can multiply on any of the sets $ M_l $ by some nonzero scalar without changing the coefficients, we can get all elements of $ \Coeff(v) $ this way. On the other hand, if $ (\lambda_i) \neq (\mu_i) $ are two different choices of parameters with $ \lambda_{z_{l,r_l}} = \mu_{z_{l,r_l}} = 1 $ for all $ l $, they will yield different coefficients, since $ \frac{\lambda_i}{\mu_i} $ is then not constant on at least one of the sets $ M_l $. \\
Hence our construction yields a bijective map $ (K^*)^{g-r} \to \Coeff(v) $ and $ \lvert \Coeff(v) \rvert = (q-1)^{g-r} $ follows.
\end{proof}

Now that we understand $ \Coeff(v) $, we still need to study how $ \PGL(k,K) $ acts on $ C_k^{(*)}(n) $. As it turns out, this is also related to the decompositions of $ v $:
\begin{lem} \label{lem:countingStabilizers}
Let $ k \geq 2 $ and $ n \geq k $. Let $ \lvert K \rvert = q $ be finite. Let $ v \in C_k^{(*)}(n) $. Let $ (U_1,\dots,U_r) $ be the maximal decomposition of $ v $. Then
\[
\rvert \PGL(k,K)_v \rvert = (q-1)^{r-1},
\]
where $ \PGL(k,K)_v $ is the stabilizer of $ v $.
\begin{proof}
\begin{claim}
\[
\PGL(k,K)_v = \{[\phi] \mid \phi \in \GL(K^k), \phi\mid_{U_j} = \lambda_j \id_{U_j} \text{ for some } \lambda_i \in K^* \}.
\]
\end{claim}
$ \subseteq $: Let $ (V_1,\dots,V_n) $ be a lift of $ v = (v_1,\dots,v_n) $. Let $ [\phi] \in \PGL(k,K)_v $ with $ \phi \in \GL(K^k) $. Then $ \phi(V_i) = \lambda_i V_i $ for some $ \lambda_i \in K^* $. In other words, the vectors $ V_i $ are all eigenvectors of $ \phi $. Therefore each $ V_i $ is contained in some eigenspace of $ \phi $ and the eigenspaces $ E_1,\dots,E_s $ of $ \phi $ are a decomposition of $ v $. Since $ (U_1,\dots,U_r) $ is a maximal decomposition, Lemma \ref{lem:maximalDecomposition} implies that for each $ j \in [r] $ there is an $ l \in [s] $ with $ U_j \subseteq E_l $. Since $ \phi $ is a nonzero multiple of $ id $ on $ E_l $, the same is true on $ U_j $ and $ \phi $ is in the set on the right hand side of the claim. \smallskip\\
$\supseteq $: Let $ \phi \in \GL(K^k) $ with $ \phi\mid_{U_j} = \lambda_j \id_{U_j} $ with $ \lambda_1,\dots,\lambda_r \in K^* $ and for all $ j \in [r] $. Then each $ V_i $ is an eigenvector of $ \phi $ with nonzero eigenvalue and therefore $ \phi(\langle V_i \rangle) = \langle V_i \rangle $. This implies $ [\phi](v_i) = v_i $ for all $ i \in [n] $ and thus $ [\phi] \in \PGL(k,K)_v $ and we have proved the claim. \smallskip\\
There are exactly $ (q-1)^r $ linear maps $ \phi \in \GL(K^k) $ such as in the set on the right hand side of the claim, since linear maps of this form are uniquely determined by the values of the $ \lambda_j $. Taking the quotient group by the centre $ Z(\GL(K^k)) = \{\lambda id_{K^k} \mid \lambda \in K^* \} $ leaves us with exactly $ (q-1)^{r-1} $ elements in $ \PGL(k,K)_v $.
 \end{proof}
\end{lem}

Now we bring everything together:

\begin{theorem} \label{thm:numberOfFriezes}
Let $ k \geq 2 $ and $ w \geq 1 $. Let $ n := w + k + 1 $. Let $ g := \gcd(k,n) $. Let $ K = \F_q $ be a finite field with $ q $ elements.
\begin{enumerate}
    \item If $ g = 1 $, then the number of tame $ \SL_k $-frieze patterns of width $ w $ is
    \[
    f_q(k,n) = \frac{\lvert C_k(n) \rvert}{\lvert \PGL(k,K) \rvert} = \lvert \M_k(n) \rvert.
    \]
    \item If $ g > 1 $, then the number of tame $ \SL_k $-frieze patterns of width $ w $ is
    \[
    f_q(k,n) = \frac{\lvert C_k^*(n) \rvert (q-1)^{g-1}}{\lvert \PGL(k,K) \rvert}.
    \]
\end{enumerate}
\end{theorem}
\begin{proof}
By Lemmas \ref{lem:coeffQuid} and \ref{lem:coeffPGL} (i) we have
\[
\Quid(k,n) = \bigcup_{v \in C_k^{(*)}(n)} \Coeff(v) = \bigcup_{[v] \in \M_k^{(*)}(n)} \Coeff(v).
\]
According to Lemma \ref{lem:coeffPGL} (ii) the second union is disjoint. Hence
\[
\lvert \Quid(k,n) \rvert = \sum_{[v] \in \M_k^{(*)}} \lvert \Coeff(v) \rvert = \sum_{v \in C_k^{(*)}} \frac{\lvert \Coeff(v) \rvert}{\lvert \PGL(k,K) \cdot v \rvert},
\]
where $ \PGL(k,K) \cdot v $ denotes the orbit of $ v $ under the group action. Using the Orbit-Stabilizer-Theorem
\[
\lvert \PGL(k,K) \rvert = \lvert \PGL(k,K) \cdot v \rvert \lvert \PGL(k,K)_v \rvert
\]
we get
\[
\lvert \Quid(k,n) \rvert = \sum_{v \in C_k^{(*)}(n)} \frac{\lvert \Coeff(v) \rvert \lvert \PGL(k,K)_v \rvert}{\lvert \PGL(k,K) \rvert}.
\]
Now we can apply Lemma \ref{lem:countingCoeff} and Lemma \ref{lem:countingStabilizers}:
\[
\lvert \Quid(k,n) \rvert = \sum_{v \in C_k^{(*)}(n)} \frac{(q-1)^{g-r(v)}(q-1)^{r(v)-1}}{\lvert \PGL(k,K) \rvert} = \frac{\lvert C_k^{(*)}(n) \rvert (q-1)^{g-1}}{\lvert \PGL(k,K) \rvert},
\]
where $ r(v) $ is the number of subspaces in the maximal decomposition of $ v $. By Lemma \ref{lem:friezeQuidBijection}, $ f_q(k,n) = \lvert \Quid_{\F_q}(k,n) \rvert $.
If $ g = 1 $, then all elements of $ C_k(n) $ have trivial stabilizers under the action of the $ \PGL(k,K) $ by Lemma \ref{lem:countingStabilizers}. Hence
\[
\lvert \M_k(n) \rvert = \frac{\lvert C_k(n) \rvert}{\lvert \PGL(k,K) \rvert}.
\]
\end{proof}

The cardinality of $ \M_k(n) $ is effectively already known in the case $ \gcd(k,n) = 1 $. To see this, we need to consider a specific subset of the Grassmannian:

\begin{defi}
Let
\[
    \Pi_{k,n}^\circ := \{ V \in \G(k,n) \mid \Delta^{1,\dots,k}(V), \Delta^{2,\dots,k+1}(V), \dots, \Delta^{n,\dots,k-1}(V) \neq 0 \}
\]
be the subset of the Grassmannian $ \G(k,n) $ over the finite field $ K = \F_q $, where none of the consecutive Plücker coordinates vanishes.
\end{defi}

We now show how $ \M_k(n) $ and $ \Pi_{k,n}^\circ $ are related:

\begin{lem} \label{lem: mk(n) and pi}
Let $ n > k $ and $ \gcd(k,n) = 1 $. Then
\[
\lvert \M_k(n) \rvert = \frac{\lvert \Pi_{k,n}^\circ \rvert}{(q-1)^{n-1}}.
\]
\end{lem}
\begin{proof}
Let $ M(n,k) \subset K^{n \times k} $ be the set of max-rank $ n \times k $-matrices with entries in the field $ K $. Consider the action of $ \GL(k,n) $ on $ M(n,k) $ by right-multiplication. Then the Grassmannian $ \G(n,k) $ can be equivalently defined as the set of orbits
\[
\G(k,n) = M(n,k) / \GL(k,K)
\]
under this action (the columns of such a matrix generate a $ k $-dimensional subspace and the action of the $ \GL(k,K) $ corresponds to sequences of elementary column operations, hence the $ k $-dimensional subspace generated by the columns does not change under this action). In this definition of the Grassmannian, all Plücker coordinates are determinants of $ k \times  k $-minors of the given matrix and the consecutive Plücker coordinates in particular correspond to $ k \times k $-minors formed by consecutive rows (where the first and last row are considered to be consecutive). Let $ M^\circ(n,k) $ be the set of all $ n \times k $-matrices such that each consecutive $ k \times k $-minor is linearly independent. Then
\[
\Pi_{k,n}^\circ = M^\circ(n,k) / \GL(k,K).
\]
Now consider the map
\begin{align*}
    M^\circ(n,k) &\to C_k(n) \\
    \begin{pmatrix}
        v_1 \\
        \vdots \\
        v_n
    \end{pmatrix} &\mapsto ([v_1],\dots,[v_n]).
\end{align*}
Clearly, this map is surjective and each element of $ C_k(n) $ has exactly $ (q-1)^n $ preimages. Hence
\[
\lvert C_k(n) \rvert = \frac{\lvert M^\circ(n,k) \rvert}{(q-1)^n}.
\]
The actions of $ \GL(k,K) $ on $ M^\circ(n,k) $ and $ \PGL(k,K) $ on $ C_k(n) $ are both free, as long as $ \gcd(k,n) = 1 $. This follows from the fact that $ \GL(k,K) $ acts freely on any basis of $ K^k $ and from Lemma \ref{lem:countingStabilizers}. Moreover, we have $ \PGL(k,K) = \GL(k,K) / Z(\GL(k,K)) $ and $ \lvert Z(\GL(k,K)) \rvert = q-1 $. It follows that
\[
\lvert \M_k(n) \rvert = \frac{\lvert C_k(n) \rvert}{\lvert \PGL(k,K) \rvert} = \frac{q-1}{(q-1)^n} \frac{\lvert M^\circ(n,k) \rvert}{\lvert \GL(k,K) \rvert} = \frac{\lvert \Pi_{k,n}^\circ \rvert}{(q-1)^{n-1}}.
\]
\end{proof}

To give the point count of $ \Pi_{k,n}^\circ $ it is convenient to use the following notation:

\begin{defi}[Gaussian polynomials]
Let $ n \in \N_{>0} $ and $ q $ be some prime power. Then set
\begin{enumerate}
    \item $ [n]_q := 1 + q + \dots + q^{n-1} $,
    \item $ [n]_q! := [1]_q[2]_q \dots [n]_q $,
    \item $ \left[ \begin{matrix} n \\ k \end{matrix} \right]_q := \frac{[n]_q!}{[k]_q![n-k]_q!} $.
\end{enumerate}
\end{defi}

Galashin and Lam \cite{galashin2024positroids} calculated the point count of $ \Pi_{k,n}^\circ $ over finite fields in the case $ \gcd(k,n) = 1 $:

\begin{theorem}[Galashin, Lam, 2024] \label{thm:galashinLam}
Let $ q $ be some prime power and $ \gcd(k,n) = 1 $. Then the point count of $ \Pi_{k,n}^\circ $ over the finite field $ \F_q $ is
\[
\lvert \Pi_{k,n}^\circ \rvert = (q-1)^{n-1} \frac{1}{[n]_q} \left[ \begin{matrix} n \\ k \end{matrix} \right]_q.
\]
\end{theorem}
\begin{proof}
This is \cite{galashin2024positroids}[Corollary 1.2]
\end{proof}

\begin{cor}\label{cor:numberOfFriezegcd1}
Let $ \gcd(k,n) = 1 $. Then the number of tame $ \SL_k $-frieze patterns over $ K = \F_q $ with width $ w = n-k-1 $ is
\[
f_q(k,n) = \frac{1}{[n]_q} \left[ \begin{matrix} n \\ k \end{matrix} \right]_q.
\]
\end{cor}
\begin{proof}
This follows immediately from Theorem \ref{thm:numberOfFriezes}, Lemma \ref{lem: mk(n) and pi} and Theorem \ref{thm:galashinLam}.
\end{proof}

\section{The case $ k = 3 $}

In addition to $ C_3(n) $ we will consider some other sets with similar definitions:
\begin{defi}
Let $ s_1,s_2 \in \{+,-\} $ and $ I_{s_1,s_2} := [n-2] \cup \{ n-2 + j \mid j \in \{1,2\}, s_j = + \} $. Define
\begin{align*}
C_3^{s_1,s_2}(n) := \{ (v_1,\dots,v_n) \in (\pr^2(K))^n \mid &\{v_i,v_{i+1},v_{i+2}\} \text{ independent } \Leftrightarrow i \in I_{s_1,s_2} \\
&\text{ and } v_n \neq v_1 \}
\end{align*}
and
\[
c_3^{s_1,s_2}(n) := \lvert C_3^{s_1,s_2}(n) \rvert.
\]
\end{defi}
Note that $ C_3^{++}(n) = C_3(n) $ and that $ C_3^{+-}(n) $ and $ C_3^{-+}(n) $ are in bijection via the map $ (v_1,\dots,v_n) \mapsto (v_2,\dots,v_n,v_1) $, hence $ c_3^{+-}(n) = c_3^{-+}(n) $. Also, the condition $ v_n \neq v_1 $ is redundant unless $ s_1 = s_2 = - $, because $ \{v_{n-1},v_n,v_1 \} $ independent or $ \{v_n,v_1,v_2\} $ independent immediately implies $ v_n \neq v_1 $.

\begin{lem} \label{lem:recursionsk3}
Let $ K $ be finite with $ \lvert K \rvert = q $. The following recursions hold for $ n \geq 4 $:
\begin{align}
c_3(n) &= (q-1)^2c_3(n-1) + 2q(q-1)c_3^{+-}(n-1) + q^2c_3^{--}(n-1), \label{eq:recursion++}\\
c_3^{+-}(n) &= (q-1)c_3(n-1) + qc_3^{+-}(n-1), \label{eq:recursion+-}\\
c_3^{--}(n) &= (q-1)c_3^{+-}(n-1) + qc_3(n-2). \label{eq:recursion--}
\end{align}
\end{lem}
\begin{proof}
\eqref{eq:recursion++}: Let $ v = (v_1,\dots,v_n) \in C_3(n) $. Then $ \{v_{n-1},v_n,v_1\} $ is independent, which implies $ v_{n-1} \neq v_1 $. Also $ \{v_i,v_{i+1},v_{i+2}\} $ is independent for $ i \leq n-3 $. Hence $ (v_1,\dots,v_{n-1}) \in C_3^{s_1s_2}(n-1) $ for some $ s_1,s_2 \in \{+,-\} $. \\
Conversely, let $ (v_1,\dots,v_{n-1}) \in C_3^{s_1s_2}(n-1) $ and $ v_n \in \pr^2(K) $. Let
\[
L_1 := v_{n-2} \vee v_{n-1}, L_2 := v_{n-1} \vee v_1, L_3 := v_1 \vee v_2.
\]
Then $ (v_1,\dots,v_n) \in C_3(n) $ if and only if $ v_n \notin L_1 \cup L_2 \cup L_3 $. \smallskip\\
Case 1: $ s_1 = s_2 = + $. \\
Then $ \{v_{n-2},v_{n-1},v_1\} $ is independent and therefore $ L_1 \neq L_2 $. Likewise, $ \{v_{n-1},v_1,v_2\} $ independent implies $ L_2 \neq L_3 $. And $ L_1 = L_3 $ would imply that $ v_{n-2},v_{n-1},v_1 $ and $ v_2 $ all lie on one projective line which contradicts $ \{v_{n-2},v_{n-1},v_1\} $ independent. Accordingly, $ L_1, L_2, L_3 $ are pairwise different. Moreover, two different projective lines in $ \pr^2(K) $ always share exactly one point, thus $ \lvert L_i \cap L_j \rvert = 1 $ for all $ i,j \in [3] $ with $ i \neq j $. Furthermore, we have $ L_1 \cap L_2 = \{v_{n-1}\} $ and $ L_2 \cap L_3 = \{v_1\} $, hence $ L_1 \cap L_2 \cap L_3 = \emptyset $. And we have $ \lvert L \rvert = q+1 $ for all projective lines $ L \subset \pr^2(K) $. By the exclusion-inclusion principle, it follows that
\begin{align*}
\lvert L_1 \cup L_2 \cup L_3 \lvert &= \lvert L_1 \rvert + \lvert L_2 \rvert + \lvert L_3 \rvert - \lvert L_1 \cap L_2 \rvert - \lvert L_1 \cap L_3 \rvert - \lvert L_2 \cap L_3 \rvert + \lvert L_1 \cap L_2 \cap L_3 \rvert \\
&= 3(q+1) - 3 + 0 = 3q.
\end{align*}
Hence we have
\[
\lvert \pr^2(K) \backslash (L_1 \cup L_2 \cup L_3) \rvert = q^2+q+1-3q = q^2-2q+1=(q-1)^2
\]
suitable choices of $ v_n $ for every $ (v_1,\dots,v_{n-1}) \in C_3(n-1) $. Consequently, there are exactly $ (q-1)^2 c_3(n-1) $ elements of $ C_3(n) $ with $ (v_1,\dots,v_{n-1}) \in C_3(n-1) $. \smallskip\\
Case 2: $ s_1 = +, s_2 = - $. \\
Then $ L_2 = L_3 $ since $ \{v_{n-1},v_1,v_2\} $ is dependent. But we still have $ L_1 \neq L_2 $ because $ \{v_{n-2},v_{n-1},v_1\} $ is independent. Therefore $ \lvert L_1 \cup L_2 \cup L_3 \rvert = 2q + 1 $. Hence
\[
\lvert \pr^2(K) \backslash (L_1 \cup L_2 \cup L_3) \rvert = q^2+q+1-2q-1 = q^2-q = q(q-1)
\]
and there are $ q(q-1)c_3^{+-}(n-1) $ elements of $ C_3(n) $ with $ (v_1,\dots,v_{n-1}) \in C_3^{+-}(n-1) $. \smallskip\\
Case 3: $ s_1 = -, s_2 = + $. \\
This case is symmetrical to Case 2 and thus there are
\[
q(q-1)c_3^{-+}(n-1) = q(q-1)c_3^{+-}(n-1)
\]
elements of $ C_3(n) $ with $ (v_1,\dots,v_{n-1}) \in C_3^{-+}(n-1) $. \smallskip\\
Case 4: $ s_1 = s_2 = - $. \\
Then we have $ L_1 = L_2 = L_3 $ and
\[
\lvert \pr^2(K) \backslash (L_1 \cup L_2 \cup L_3) \rvert = q^2+q+1-q-1 = q^2
\]
possible choices of $ v_n $. Hence there are $ q^2 c_3^{--}(n-1) $ elements of $ C_3(n) $ with $ (v_1,\dots,v_{n-1}) \in C_3^{--}(n-1) $. \smallskip\\
All the cases put together show \eqref{eq:recursion++}. \smallskip\\
\eqref{eq:recursion+-}: Let $ (v_1,\dots,v_n) \in C_3^{+-}(n) $. Then $ \{v_n,v_1,v_2\} $ is dependent, which implies $ L_3 := v_1 \vee v_2 = v_n \vee v_1 $. At the same time $ \{v_{n-1},v_n,v_1\} $ is independent. In particular, $ v_{n-1} \notin v_n \vee v_1 = L_3 $. Hence $ \{v_{n-1},v_1,v_2\} $ is independent and thus $ (v_1,\dots,v_{n-1}) \in C_3^{s+}(n-1) $ with $ s \in \{+,-\} $. \smallskip\\
Conversely, let $ (v_1,\dots,v_{n-1}) \in C_3^{s+}(n-1) $ and $ v_n \in \pr^2(K) $. Let $ L_1 = v_{n-2} \vee v_{n-1} $ and $ L_2 = v_{n-1} \vee v_1 $. Then $ (v_1,\dots,v_n) \in C_3^{+-}(n) $ is equivalent to $ v_n \in L_3 \backslash (L_1 \cup L_2) $. \smallskip\\
If $ s = + $, then just like in the proof of \eqref{eq:recursion++} case 1 the lines $ L_1, L_2 $ and $ L_3 $ are distinct and two of them intersect in exactly one point, but $ L_1 \cap L_2 \cap L_3 = \emptyset $. Hence there are exactly two points on $ L_3 $ that also lie in $ L_1 \cup L_2 $ and thus $ \lvert L_3 \backslash (L_1 \cup L_2) \rvert = q-1 $. \smallskip\\
If $ s = - $, then $ L_1 = L_2 \neq L_3 $. In this case we have $ \lvert L_3 \backslash (L_1 \cup L_2) \rvert = q $. \\
It follows that
\[
c_3^{+-}(n) = (q-1)c_3(n-1) + qc_3^{-+}(n-1) = (q-1)c_3(n-1) + qc_3^{+-}(n-1)
\]
as claimed. \smallskip\\
\eqref{eq:recursion--}: Let $ (v_1,\dots,v_n) \in C_3^{--}(n) $. Then $ v_{n-1} \vee v_n = v_n \vee v_1 = v_1 \vee v_2 =: L_2 $. We separate two cases: \\
Case 1: $ v_{n-1} \neq v_1 $. \\
Then $ v_{n-1} \vee v_1 = L_2 $. We have $ \{v_{n-2},v_{n-1},v_n\} $ independent, which implies
\[
v_{n-2} \notin v_{n-1} \vee v_n = L_2 = v_{n-1} \vee v_1.
\]
Hence $ \{v_{n-2},v_{n-1},v_1\} $ is independent. At the same time $ \{v_{n-1},v_1,v_2\} $ is clearly dependent, since these three points all lie on $ L_2 $. Thus $ (v_1,\dots,v_{n-1}) \in C_3^{+-}(n-1) $. \\
For the other direction, let $ (v_1,\dots,v_{n-1}) \in C_3^{+-}(n-1) $ and $ v_n \in \pr^2(K) $. Let $ L_1 = v_{n-2} \vee v_{n-1} $. Then we have $ (v_1,\dots,v_n) \in C_3^{--}(n) $ if and only if $ v_n \in L_2 \backslash (L_1 \cup \{v_1\}) $. Since $ L_1 \cap L_2 = \{v_{n-1}\} $, there are $ q-1 $ suitable choices for $ v_n $. \smallskip\\
Case 2: $ v_{n-1} = v_1 $.\\
Then $ \{v_{n-3},v_{n-2},v_1\} = \{v_{n-3},v_{n-2},v_{n-1}\} $ is independent. Also, $ \{v_{n-2},v_{n-1},v_n\} $ independent implies
\[
v_{n-2} \notin v_{n-1} \vee v_n = L_2 = v_1 \vee v_2,
\]
and thus $ \{v_{n-2},v_1,v_2\} $ is independent. Hence $ (v_1,\dots,v_{n-2}) \in C_3(n-2). $ \\
On the other hand, let $ (v_1,\dots,v_{n-2}) \in C_3(n-2), v_{n-1} = v_1 $ and $ v_n \in \pr^2(K) $. Then $ \{v_{n-3},v_{n-2},v_{n-1}\} = \{v_{n-3},v_{n-2},v_1\} $ is clearly independent. Let $ L_1 := v_{n-2} \vee v_{n-1} = v_{n-2} \vee v_1 $ and $ L_2 = v_1 \vee v_2 $. Then $ (v_1,\dots,v_n) \in C_3^{--}(n) $ is equivalent to $ v_n \in L_2 \backslash L_1 $ ($ v_n \neq v_1 $ follows automatically, since $ v_1 = v_{n-1} \in L_1 $). Therefore, $ v_{n-1} $ is uniquely determined and we have $ q $ possible choices for $ v_n $. \smallskip\\
Both cases put together yield
\[
c_3^{--}(n) = (q-1)c_3^{+-}(n-1) + qc_3(n-2)
\]
as desired.
\end{proof}

With these recursions we can calculate $ c_3(n) $:

\begin{theorem} \label{thm:ck(n)relativelyPrime}
Let $ n \geq 3 $ and $ \lvert K \rvert = q < \infty $. Then we have
\[
c_3(n) =
\begin{cases}
    q^{2n} - q^{n+2} - q^{n+1} + q^3,   &n \not\equiv 0 \mod 3, \\
    q^{2n} + 2q^{n+2} + 2q^{n+1} + q^3, &n \equiv 0 \mod 3.
\end{cases}
\]
\end{theorem}
\begin{proof}
We have
\begin{align*}
c_3^{+-}(n) &=
\begin{cases}
    q^{2n-1} - q^{n+2} + q^n - q^3, \quad\quad\quad\quad\quad\quad\, &n \equiv 0 \mod 3, \\
    q^{2n-1} + q^{n+2} - q^n - q^3, &n \equiv 1 \mod 3, \\
    q^{2n-1} - q^3,                 &n \equiv 2 \mod 3
\end{cases}
\intertext{and}
c_3^{--}(n) &=
\begin{cases}
    q^{2n-2} - q^{n+1} - q^n + q^3, &n \equiv 0 \mod 3, \\
    q^{2n-2} - q^{n+2} - q^{n-1} + q^3, &n \equiv 1 \mod 3, \\
    q^{2n-2} + q^{n+2} + q^{n+1} + q^n + q^{n-1} + q^3, &n \equiv 2 \mod 3.
\end{cases}
\end{align*}
It can be easily verified that these solutions fulfill the recursions in Lemma \ref{lem:recursionsk3} for $ n \geq 4 $. Thus it only remains to check the starting values $ c_3(3), c_3^{+-}(3), c_3^{--}(3) $ and $ c_3(2) $.\\
For $ (v_1,v_2,v_3) \in C_3(3) $ we can choose $ v_1 \in \pr^2(K) $ arbitrarily, so there are $ q^2+q+1 $ options. $ v_2 $ can be picked from $ \pr^2(K) \backslash \{v_1\} $, hence we have $ q^2+q $ possibilities. For $ v_3 $ we can pick any point in $ \pr^2(K) \backslash v_1 \vee v_2 $, which leaves $ q^2 $ choices. Hence
\[
c_3(3) = (q^2+q+1)(q^2+q)q^2 = q^6 + 2q^5 + 2q^4 + q^3.
\]
We also have $ c_3^{+-}(3) = c_3^{--}(3) = 0 $, because $ \{v_1,v_2,v_3\} $ independent implies that the sets $ \{v_2,v_3,v_1\} $ and $ \{v_3,v_1,v_2\} $ are also independent. Finally, $ c_3(2) = 0 $ as the condition $ \{v_1,v_2,v_1\} $ independent is impossible to fulfill. This also fits with the formulas we claim.
\end{proof}

\begin{cor} \label{cor:k3relativelyPrime}
Let $ \lvert K \rvert = q < \infty $. Let $ w \geq 1 $ and $ n = w + 4 $. If $ \gcd(3,n) = 1 $, then the number of tame $ \SL_3 $-frieze patterns of width $ w $ is
\[
f_q(3,n) = \frac{q^{2n} - q^{n+2} - q^{n+1} + q^3}{(q^2+q+1)(q+1)q^3(q-1)^2} = \frac{(q^{n-1}-1)(q^{n-2}-1)}{(q^2+q+1)(q+1)(q-1)^2}.
\]
\end{cor}
\begin{proof}
This follows immediately from Theorem \ref{thm:numberOfFriezes} and Theorem \ref{thm:ck(n)relativelyPrime} and
\[
\lvert \PGL(3,K) \rvert = (q^2+q+1)(q+1)q^3(q-1)^2.
\]
\end{proof}

\section{The case $ k = 3 $ and $ 3 \mid n $}

Let $ n \in 3\N $ and let $ (v_1,\dots,v_n) \in (\pr^2(K))^n $. Let $ (V_1,\dots,V_n) \in (K^3)^n $ be any lift and let $ d_i := \det(V_i,V_{i+1},V_{i+2}) $. Recall the equation
\begin{equation} \label{eq:subsetDetConditionk3}
d_i d_{i+3} \dots d_{i+n-3} = (-1)^{(k-1)\frac{k}{g}} d_{i+1} d_{i+4} \dots d_{i+n-2}
\end{equation}
and that $ C_3^*(n) $ was defined as the set of all $ (v_1,\dots,v_n) \in C_3(n) $ such that \eqref{eq:subsetDetConditionk3} holds for $ i = 1 $ and $ i = 2 $. Since $ k = 3 $ is odd we can omit the sign. To count this set, we again define several other sets:

\begin{defi}
Let $ n \equiv 0 \mod 3 $. We write
\begin{align*}
C_3^{**}(n) &:= \{(v_1,\dots,v_n) \in C_3(n) \mid \eqref{eq:subsetDetConditionk3} \text{ holds for } i = 1 \}\\
\intertext{and}
C_3^{+-*}(n) &:= \{(v_1,\dots,v_n) \in C^{+-}_3(n) \mid \eqref{eq:subsetDetConditionk3} \text{ holds for } i = 1 \}.\\
\intertext{Now let $ n \equiv 2 \mod 3 $. Let $ v = (v_1,\dots,v_n) \in (\pr^2(K))^n $ and let $ (V_1,\dots,V_n) $ be any lift. Let $ d_i := \det(V_i,V_{i+1},V_{i+2}) $. Then we write}
W_i &:= d_i d_{i+3} \dots d_{i+n-5} V_{i+n-2} - d_{i+1} d_{i+4} \dots d_{i+n-4} V_{i}\\
\intertext{for $ i = 1,2 $. Using this we define}
C_3^{--*}(n) &:= \{ v \in C_3^{--}(n) \mid \{W_1,V_n\},\{W_2,V_1\} \text{ linearly dependent}\}\\
\intertext{and}
C_3^{--**}(n) &:= \{ v \in C_3^{--}(n) \mid \{W_1,V_n\} \text{ linearly dependent}\},\\
\intertext{as well as}
C_3^{+-*}(n) &:= \{ v \in C_3^{+-}(n) \mid \{W_2,V_1\} \text{ linearly dependent}\}\\
\intertext{and}
C_3^{-+*}(n) &:= \{ v \in C_3^{-+}(n) \mid \{W_1,V_n\} \text{ linearly dependent} \}.\\
\intertext{Finally, let $ n \equiv 1 \mod 3 $. Then we set}
C_3^{+-*}(n) &:= \{(v_1,\dots,v_n) \in C_3^{+-}(n) \mid d_1 d_4 \dots d_{n-3} = d_3 d_6 \dots d_{n-4} \det(V_{n-1},V_1,V_2) \}\\
\intertext{and}
C_3^{-+*}(n) &:= \{(v_1,\dots,v_n) \in C_3^{+-}(n) \mid d_2 d_5 \dots d_{n-2} = d_3 d_6 \dots d_{n-4} \det(V_{n-1},V_n,V_2) \}.
\end{align*}
For all these sets, we use a lower-case $ c $ to denote their cardinality.
\end{defi}
Note that all these definitions are independent of the chosen lift $ (V_1,\dots,V_n) $. Also note $ c_3^{+-*}(n) = c_3^{-+*}(n) $ for $ n \equiv 2 \mod 3 $ by symmetry.

\begin{lem} \label{lem:recursionsk3*}
The following recursions hold for $ n \geq 4 $:
\begin{enumerate}[(i)]
    \item If $ n \equiv 0 \mod 3 $:
    \begin{align}
        c_3^*(n) &= c_3(n-1) + 2q c_3^{+-*}(n-1) + q^2 c_3^{--*}(n-1), \label{eq:recursion++*}\\
        c_3^{**}(n) &= (q-1)c_3(n-1) + qc_3^{+-}(n-1) + q(q-1)c_3^{+-*}(n-1) + q^2c_3^{--**}(n-1), \label{eq:recursion++**}\\
        c_3^{+-*}(n) &= c_3(n-1) + qc_3^{+-*}(n-1). \label{eq:recursion+-*0}
    \intertext{\item If $ n \equiv 2 \mod 3 $:}
        c_3^{+-*}(n) &= c_3(n-1) + qc_3^{+-*}(n-1), \label{eq:recursion+-*2}\\
        c_3^{--*}(n) &= c_3^{+-*}(n-1) + q c_3^*(n-2), \label{eq:recursion--*}\\
        c_3^{--**}(n) &= c_3^{+-}(n-1) + q c_3^{**}(n-2). \label{eq:recursion--**}\\
    \intertext{\item If $ n \equiv 1 \mod 3 $:}
        c_3^{+-*}(n) &= c_3^{-+*}(n) = (q-1)c_3^{**}(n-1) + qc_3^{+-*}(n-1) \label{eq:recursion+-*1}
    \end{align}
\end{enumerate}
\end{lem}
\begin{proof}
First consider $ n \equiv 0 \mod 3 $. For $ (v_1,\dots,v_n) \in C_3^*(n) $ or $ C_3^{**}(n) $ or $ C_3^{+-*}(n) $ let $ (V_1
,\dots,V_n) $ be any lift. Let $ L_1 := v_{n-2} \vee v_{n-1} $, $ L_2 := v_{n_1} \vee v_1 $ and $ L_3 := v_1 \vee v_2 $. Let $ d_i := \det(V_i,V_{i+1},V_{i+2}) $ and
\[
W_1 := d_1 d_4 \dots d_{n-5} V_{n-2} - d_2 d_5 \dots d_{n-4} V_1.
\]
as well as
\[
W_2 := d_2 d_5 \dots d_{n-4} V_{n-1} - d_3 d_6 \dots d_{n-3} V_2.
\]
We also write $ E_1 := \langle W_1, V_{n-1} \rangle $ and $ E_2 := \langle W_2, V_1 \rangle $. Note that the definitions of $ W_1, W_2, E_1 $ and $ E_2 $ are independent of the choice of lift and also do not depend on $ v_n $. \smallskip\\
\eqref{eq:recursion++*}: Let $ (v_1,\dots,v_n) \in C_3^*(n) $. Like in the proof of \eqref{eq:recursion++}, we have $ (v_1,\dots,v_{n-1}) \in C_3^{s_1s_2}(n-1) $ with $ s_1,s_2 \in \{+,-\} $. In addition, we have
\begin{align*}
    d_1 d_4 \dots d_{n-5} \det(V_{n-2},V_{n-1},V_n) &= d_2 d_5 \dots d_{n-4} \det(V_{n-1},V_n,V_1) \\
    \Leftrightarrow d_1 d_4 \dots d_{n-5} \det(V_{n-2},V_{n-1},V_n) &= d_2 d_5 \dots d_{n-4} \det(V_1,V_{n-1},V_n) \\
    \Leftrightarrow \det(W_1,V_{n-1},V_n) &= 0, \\
\intertext{and}
    d_2 d_5 \dots d_{n-4} \det(V_{n-1},V_n,V_1) &= d_3 d_6 \dots d_{n-3} \det(V_n,V_1,V_2) \\
    \Leftrightarrow d_2 d_5 \dots d_{n-4} \det(V_{n-1},V_n,V_1) &= d_3 d_6 \dots d_{n-3} \det(V_2,V_n,V_1) \\
    \Leftrightarrow \det(W_2,V_n,V_1) &= 0.
\end{align*}
This is equivalent to requiring that $ \{W_1,V_{n-1},V_n\} $ and $ \{W_2, V_n, V_1\} $ be linearly dependent. We separate cases: \smallskip\\
Case 1: $ s_1 = s_2 = + $. \\
Then $ E_1 $ and $ E_2 $ have dimension $ 2 $, since $ \{V_{n-2},V_{n-1},V_1\} $ and $ \{V_{n-1},V_1,V_2\} $ are linearly independent. The determinant conditions are then equivalent to $ V_n \in E_1 \cap E_2 $. We have $ \dim(E_1 \cap E_2) = 1 $, because $ E_1 = E_2 $ would imply that $ V_{n-2},V_{n-1},V_1 $ and $ V_2 $ all lie in this plane, contradicting $ \{v_{n-2},v_{n-1},v_1\} $ independent. Thus $ v_n $ is uniquely determined by this condition. We have $ \pr(E_1) \cap L_1 = \{ v_{n-1}\} $, because if $ \pr(E_1) = L_1 $ we would have $ V_{n-2} \in E_1 $, which implies $ V_1 \in E_1 $, contradicting $ \{v_{n-2},v_{n-1},v_1\} $ independent. Similarly we also have $ \pr(E_1) \cap L_2 = \{ v_{n-1} \} $ and $ \pr(E_2) \cap L_3 = \{ v_1 \} $. Since $ V_{n-1}, V_1 \notin E_1 \cap E_2 $, this implies $ \pr(E_1 \cap E_2) \cap (L_1 \cup L_2 \cup L_3) = \emptyset $. Hence for any $ (v_1,\dots,v_{n-1}) \in C_3(n-1) $ and the point $ v_n $ determined by $ V_n \in E_1 \cap E_2 $ we have $ (v_1,\dots,v_n) \in C_3^*(n) $. \smallskip\\
Case 2: $ s_1 = +, s_2 = - $. \\
Then we still have $ \dim(E_1) = 2 $ with the same argument as in Case 1. However, $ \{v_{n-1},v_1,v_2\} $ dependent implies $ \pr(E_2) \subseteq L_2 = L_3 $. Equality of these sets would imply that $ \{W_2,V_n,V_1\} $ is linearly independent since $ v_n \notin L_2 $. Therefore $ \{W_2,V_1\} $ must be linearly dependent. Hence $ (v_1,\dots,v_{n-1}) \in C_3^{+-*}(n-1) $.\\
Conversely, let $ (v_1,\dots,v_{n-1}) \in C_3^{+-*}(n-1) $ and $ v_n \in \pr^2(K) $. Then $ (v_1,\dots,v_n) \in C_3^*(n) $ is equivalent to $ v_n \notin L_1 \cup L_2 $ and $ v_n \in \pr(E_1) $. We still have $ \pr(E_1) \cap L_1 = \pr(E_1) \cap L_2 = \{v_{n-1}\} $. Thus all but one point on $ \pr(E_1) $ is a suitable choice for $ v_n $. Hence there are $ q c_3^{+-*}(n-1) $ points in $ C_3^*(n) $ with $ s_1 = + $ and $ s_2 = - $. \smallskip\\
Case 3: $ s_1 = -, s_2 = + $. \\
Symmetrical to case 2, thus we also have
\[
q c_3^{-+*}(n-1) = q c_3^{+-*}(n-1)
\]
points that fall in this case. \smallskip\\
Case 4: $ s_1 = s_2 = -$:
Then $ \pr(E_1), \pr(E_2) \subseteq L_1 = L_2 = L_3 $ and with the same argument as in case 2 $ \{W_1,V_{n-1}\} $ and $ \{W_2,V_1\} $ have to be linearly dependent. Thus $ (v_1,\dots,v_{n-1}) \in C_3^{--*}(n-1) $. The determinant conditions are then automatically fulfilled regardless of our choice of $ v_n $ and just like in the proof of \eqref{eq:recursion++} case 4 there are $ q^2 $ viable choices for $ v_n $. \smallskip\\
Hence we have
\[
c_3^*(n) = c_3(n-1) + 2qc_3^{+-*}(n-1) + q^2c_3^{--*}(n-1).
\]
\eqref{eq:recursion++**}: This is similar to \eqref{eq:recursion++*} except we only have the first determinant condition, which is equivalent to $ \{W_1,V_{n-1},V_1\} $ being linearly dependent. We separate the same cases: \smallskip\\
Case 1: $ s_1 = s_2 = + $.\\
Here we require $ v_n \in \pr(E_1) $ and $ v_n \notin L_1 \cup L_2 \cup L_3 $. We have $ \pr(E_1) \cap L_1 = \pr(E_1) \cap L_2 = \{v_{n-1}\} $. But $ L_3 $ does not contain $ v_{n-1} $ and therefore intersects $ \pr(E_1) $ in a different point. This leaves us with $ q-1 $ possibilities for $ v_n $. \smallskip\\
Case 2: $ s_1 = +, s_2 = - $. \\
Almost the same as case 2 for \eqref{eq:recursion++*}, except that now we do not have a requirement for $ \{W_2,V_1\} $. Thus any $ (v_1,\dots,v_{n-1}) \in C_3^{+-}(n-1) $ can be extended, and since $ L_2 = L_3 $, $ v_n \in \pr(E_1) \backslash (L_1\cup L_2) $ again leaves us with $ q $ choices. \smallskip\\
Case 3: $ s_1 = -, s_2 = + $. \\
In this case we must have $ \{W_1,V_{n-1}\} $ linearly dependent. Hence $ (v_1,\dots,v_{n-1}) \in C_3^{-+*}(n-1) $. Conversely we have $ q(q-1) $ options for $ v_n \in \pr^2(K) $ which do not lie on $ L_1 = L_2 $ or $ L_3 $. Thus there are
\[
q(q-1)c_3^{-+*}(n-1) = q(q-1)c_3^{+-*}(n-1)
\]
points in $ C_3^{**}(n) $ that fall in this case.
\smallskip\\
Case 4: $ s_1 = s_2 = - $. \\
We still require $ \{W_1,V_{n-1}\} $ linearly dependent, hence $ (v_1,\dots,v_{n-1}) \in C_3^{--**}(n-1) $. Conversely, any of the $ q^2 $ choices for $ v_n $ that do not line in $ L_1 = L_2 = L_3 $ work. \smallskip\\
\eqref{eq:recursion+-*0}: Just like in the proof of \eqref{eq:recursion+-} we have $ (v_1,\dots,v_{n-1}) \in C_3^{s+}(n-1) $ with $ s \in \{+,-\} $. \smallskip\\
Case 1: $ s = + $. Then $ \dim(E_1) = 2 $ and $ v_n $ must lie in $ \pr(E_1) $ as well as in $ L_3 $. Since $ \pr(E_1) \neq L_3 $ there is only one possibility. On the other hand, $ L_2 \cap L_3 = \{v_1\} $ and $ L_1 \cap \pr(E_1) = \{v_{n-1}\} $ and $ v_{n-1},v_1 \notin L_3 \cap \pr(E_1) $, hence this one possibility does indeed work. \smallskip\\
Case 2: $ s = - $. Then $ \{W_1,V_{n-1}\} $ is linearly dependent and thus $ (v_1,\dots,v_{n-1}) \in C_3^{-+*}(n-1) $. Conversely we have $ q $ choices for $ v_n $, every point in $ L_3 \backslash L_1 $ since $ L_1 = L_2 $. \smallskip\\
Put together we get
\[
c_3^{+-*}(n) = c_3(n-1) + (q-1)c_3^{-+*}(n-1) = c_3(n-1) + (q-1)c_3^{+-*}(n-1).\medskip
\]
Now we consider $ n \equiv 2 \mod 3 $. Let $ (v_1,\dots,v_n) \in C_3^{+-*}(n) $ or $ C_3^{--*}(n) $ or $ C_3^{--**}(n) $. Let $ (V_1,\dots,V_n) $ be some lift and let $ d_i := \det(V_i,V_{i+1},V_{i+2}) $. We now define
\[
W_1 := d_1 d_4 \dots d_{n-4} V_{n-1} - d_2 d_5 \dots d_{n-3} V_1
\]
and
\[
W_2 := d_2 d_5 \dots d_{n-3} V_n - d_3 d_6 \dots d_{n-2} V_2.
\]
We write $ E := \langle V_n, V_1 \rangle = \langle V_1, V_2 \rangle $ and consider the linear map
\begin{align*}
    \phi: E &\to E \\
    V &\mapsto d_2 d_5 \dots d_{n-3} V - d_3 d_6 \dots d_{n-5} \det(V_{n-2},V_{n-1},V) V_2.
\end{align*}
Note that the determinants in these definitions are all nonzero with the possible exception of $ \det(V_{n-2},V_{n-1},V) $. Moreover, we clearly have $ \phi(V_n) = W_2 $. Also let $ L_1 = v_{n-2} \vee v_{n-1} $ and $ L_2 = v_{n-1} \vee v_1 $ and $ L_3 = v_1 \vee v_2 $. \smallskip\\
\eqref{eq:recursion+-*2}: Let $ (v_1,\dots,v_n) \in C_3^{+-*}(n) $. Again, we have $ (v_1\dots,v_{n-1}) \in C_3^{s+}(n-1) $ with $ s \in \{+,-\} $. We also have $ \{W_2,V_1\} = \{\phi(V_n),V_1\} $ linearly dependent. Note that we can analyze the behavior of $ \phi $ from inputting $ V_1 $ and $ V_2 $. For $ V_2 $ we get
\[
\phi(V_2) = (d_2 d_5 \dots d_{n-3} - d_3 d_6 \dots d_{n-5} \det(V_{n-2},V_{n-1},V_2)) V_2 =: D V_2.
\]
By contrast, $ \phi(V_1) $ depends on $ s $. \smallskip\\
Case 1: $ s = + $. \\
Let $ (v_1,\dots,v_{n-1}) \in C_3(n-1) $ and $ v_n \in \pr^2(K) $. Let $ (V_1,\dots,V_n) $ be a lift. We need $ v_n \in L_3 \backslash (L_1 \cup L_2) $ with $ \{\phi(V_n),V_1\} $ linearly dependent. Note that $ v_n \in L_3 $ is equivalent to $ V_n \in E \backslash \{ 0 \} $. If $ D = 0 $, then $ \phi $ is not bijective and $ \Image(\phi) = \langle \phi(V_1) \rangle $ and $ \Ker(\phi) = \langle V_2 \rangle $. Since $ s = + $ we have $ \det(V_{n-2},V_{n-1},V_1) \neq 0 $, which implies that $ \phi(V_1) $ is not a multiple of $ V_1 $. In this case, $ \{\phi(V_n),V_1\} $ is only linearly dependent for $ \phi(V_n) = 0 $, i.e. for $ v_n = v_2 $. Hence $ v_n $ is uniquely determined. Moreover, $ \phi(V_2) = 0 $ implies $ \det(V_{n-2},V_{n-1},V_2) \neq 0 $, therefore $ v_2 \notin L_1 $. And we also have $ L_2 \cap L_3 = \{v_1\} $, hence $ v_2 \notin L_2 $. Consequently, $ (v_1,\dots,v_{n-1},v_2) \in C_3^{+-*}(n) $. \\
If $ D \neq 0 $, then $ \phi $ is bijective and there is exactly one subspace $ \langle V_n \rangle \subset E = \langle V_1, V_2 \rangle $ of dimension $ 1 $ with $ \{\phi(V_n), V_1\} $ linearly dependent. Thus $ v_n $ is again uniquely determined. $ v_n \in L_1 $ would imply $ \det(V_{n-2},V_{n-1},V_n) = 0 $, in which case $ \phi(V_n) $ would be a nonzero multiple of $ V_n $. But we know that $ \phi(V_n) $ is a multiple of $ V_1 $, so $ v_n = v_1 $ would follow. And since we already showed that $ \phi(V_1) $ is not a multiple of $ V_1 $, this cannot be the case and we have $ v_n \notin L_1 $. Since $ L_2 \cap L_3 = \{v_1\} $, we also have $ v_n \notin L_2 $. \\
Thus in either case there is exactly one choice of $ v_n $ that works. \smallskip\\
Case 2: $ s = - $. \\
We now have $ \det(V_{n-2},V_{n-1},V_1) = 0 $, thus $ \phi(V_1) $ is a nonzero multiple of $ V_1 $. If $ D \neq 0 $, then $ \phi $ is again bijective and $ v_n $ thus uniquely determined, but this time $ \{\phi(V_1), V_1\} $ is linearly dependent and thus $ v_n = v_1 $ is the unique solution. But then $ (v_1,\dots,v_n) \notin C_3^{+-*}(n) $, contradicting our assumption. Hence we must have $ D = 0 $, which implies $ (v_1,\dots,v_{n-1}) \in C_3^{-+*}(n-1) $. \\
Conversely, let $ (v_1,\dots,v_{n-1}) \in C_3^{-+*}(n-1) $ and $ v_n \in \pr^2(K) $. Let $ (V_1,\dots,V_n) $ be a lift. Then $ D = 0 $ and thus $ \phi(V_2) = 0 $ and $ \Image(\phi) = \langle \phi(V_1) \rangle = \langle V_1 \rangle $. Hence any $ v_n \in L_3 = \pr(E) $ will fulfill $ \{\phi(V_n),V_1\} $ linearly dependent. Consequently, any of the $ q $ elements of $ L_3 \backslash (L_1=L_2) $ is a viable choice for $ v_n $. \smallskip\\
In summary, we get
\[
c_3^{+-*}(n) = c_3(n-1) + q c_3^{-+*}(n-1) = c_3(n-1) + q c_3^{+-*}(n-1)
\]
using \eqref{eq:recursion+-*1}. \smallskip\\
\eqref{eq:recursion--*}: Like we showed in the proof of \eqref{eq:recursion--}, $ (v_1,\dots,v_{n-1}) \in C_3^{+-}(n-1) $ if $ v_{n-1} \neq v_1 $ and $ (v_1,\dots,v_{n-2}) \in C_3(n-2) $ if $ v_{n-1} = v_1 $. \\
Case 1: $ v_{n-1} \neq v_1 $. \\
Then $ W_1 \neq 0 $, hence $ \{W_1,V_n\} $ is linearly dependent only if $ \langle W_1 \rangle = \langle V_n \rangle $. $ v_n $ is uniquely determined by this condition. Then $ \{\phi(V_n),V_1\} $ linearly dependent implies $ \{\phi(W_1),V_1\} $ linearly dependent, thus
\begin{align*}  
\phi(W_1) = &d_2 d_5 \dots d_{n-3} W_1 - d_3 d_6 \dots d_{n-5} \det(V_{n-2},V_{n-1},W_1) V_2 \\
          = &d_2 d_5 \dots d_{n-3} d_1 d_4 \dots d_{n-4} V_{n-1} - (d_2 d_5 \dots d_{n-3})^2 V_1 \\
            &- d_3 d_6 \dots d_{n-5} d_1 d_4 \dots d_{n-4} \det(V_{n-2},V_{n-1},V_{n-1}) V_2 \\
            &+ d_3 d_6 \dots d_{n-5} d_2 d_5 \dots d_{n-3} \det(V_{n-2},V_{n-1},V_1) V_2
\end{align*}
is a multiple of $ V_1 $. Here $ \det(V_{n-2},V_{n-1},V_{n-1}) = 0 $ and we can divide by $ d_2 d_5 \dots d_{n-3} $. Then the term simplifies to
\begin{equation} \label{eq:phi(W1)}
\phi(W_1) = d_1 d_4 \dots d_{n-4} V_{n-1} - d_2 d_5 \dots d_{n-3} V_1 + d_3 d_6 \dots d_{n-5} \det(V_{n-2},V_{n-1},V_1) V_2.
\end{equation}
Clearly,
\[
d_1 d_4 \dots d_{n-4} V_{n-1} + d_3 d_6 \dots d_{n-5} \det(V_{n-2},V_{n-1},V_1) V_2
\]
is still a multiple of $ V_1 $. Define 
\begin{align*}
    \phi^\prime: E &\to E, \\
    V &\mapsto d_1 d_4 \dots d_{n-4} V + d_3 d_6 \dots d_{n-5} \det(V_{n-2},V,V_1) V_2.
\end{align*}
Then $ \phi^\prime(V_{n-1}) $ is a multiple of $ V_1 $ because of what we just showed. At the same time $ \phi^\prime(V_1) $ is a nonzero multiple of $ V_1 $ and
\begin{align*}
\phi^\prime(V_2) &= (d_1 d_4 \dots d_{n-4} + d_3 d_6 \dots d_{n-5} \det(V_{n-2},V_2,V_1)) V_2 \\
          &= (d_1 d_4 \dots d_{n-4} - d_3 d_6 \dots d_{n-5} \det(V_{n-2},V_1,V_2)) V_2 \\
          &=: D V_2
\end{align*}
If $ D \neq 0 $, then $ \phi^\prime $ is bijective and $ (\phi^\prime)^{-1}(\langle V_1 \rangle) = \langle V_1 \rangle $, which implies $ v_{n-1} = v_1 $, contradicting our assumption. Hence $ D = 0 $ and $ (v_1,\dots,v_{n-1}) \in C_3^{+-*}(n-1) $. \\
Conversely, let $ (v_1,\dots,v_{n-1}) \in C_3^{+-*}(n-1) $ and $ v_n = \pr(\langle W_1 \rangle) $, where $ W_1 $ is defined using $ (v_1,\dots,v_{n-1}) $. Then $ \{W_1,V_n\} $ is linearly dependent. We also have $ D = 0 $ and thus $ \phi^\prime $ is not bijective. Since $ \phi^\prime(V_1) $ is a nonzero multiple of $ V_1 $, we have $ \Image(\phi^\prime) = \langle V_1 \rangle $. In particular, it follows that $ \phi^\prime(V_{n-1}) $ is a multiple of $ V_1 $ and with \eqref{eq:phi(W1)} $ \phi(W_1) $ is also a multiple of $ V_1 $. Since $ V_n $ is a multiple of $ W_1 $ and $ \phi $ is linear, we then have $ \{W_2 = \phi(V_n), V_1\} $ linearly dependent. Finally, $ v_n \in L_2 = L_3 = \pr(E) $ and $ v_n \notin L_1 $ and $ v_n \neq v_1 $, since $ L_1 \cap L_2 = \{v_{n-1}\} $ and $ W_1 $ is neither a multiple of $ V_{n-1} $ nor a multiple of $ V_1 $. In conclusion, there are exactly $ c_3^{+-*}(n-1) $ elements in $ C_3^{--*}(n) $ that fall in the first case. \smallskip\\
Case 2: $ v_{n-1} = v_1 $. \\
W.l.o.g. we can assume $ V_{n-1} = V_1 $. Then $ W_1 $ is a multiple of $ V_1 $. It follows that $ W_1 = 0 $, because otherwise $ \{W_1,V_n\} $ linearly dependent would contradict $ v_n \neq v_1 $. Hence
\begin{equation} \label{eq:det1--*}
d_1 d_4 \dots d_{n-4} = d_2 d_5 \dots d_{n-6} \det(V_{n-3},V_{n-2},V_{n-1}) = d_2 d_5 \dots d_{n-6} \det(V_{n-3},V_{n-2},V_1).
\end{equation}
Now consider $ \phi $. We have $ \det(V_{n-2},V_{n-1},V_1) = 0 $ and thus $ \phi(V_1) $ is a nonzero multiple of $ V_1 $. We also know that $ W_2=\phi(V_n) $ is a multiple of $ V_1 $, therefore $ \phi $ bijective would imply $ v_n = v_1 $. Clearly, $ \phi $ is not bijective and $ \Image(\phi) = \langle V_1 \rangle $. Write
\[
\phi(V_2) = (d_2 d_5 \dots d_{n-3} - d_3 d_6 \dots d_{n-5} \det(V_{n-2},V_{n-1},V_2)) V_2 =: D V_2.
\]
$ \Image(\phi) = \langle V_1 \rangle $ now implies $ \phi(V_2) = 0 $, hence $ D = 0 $. Thus
\begin{align*}
&d_2 d_5 \dots d_{n-3} = d_3 d_6 \dots d_{n-5} \det(V_{n-2},V_{n-1},V_2) \\
\Leftrightarrow &d_2 d_5 \dots d_{n-6} \det(V_{n-3},V_{n-2},V_1) = d_3 d_6 \dots d_{n-5} \det(V_{n-2},V_1,V_2)
\end{align*}
Combined with \eqref{eq:det1--*} this proves $ (v_1,\dots,v_{n-2}) \in C_3^*(n-2) $. \\
Conversely, let $ (v_1,\dots,v_{n-2}) \in C_3^*(n-2) $, $ v_{n-1} = v_1 $ and $ v_n \in (L_2 = L_3) \backslash L_1 $. Since $ L_1 \cap L_2 = \{v_{n-1}\} = \{v_1\} $ there are $ q $ such choices and we automatically get $ v_n \neq v_1 $. $ \{v_{n-3},v_{n-2},v_{n-1}\} = \{v_{n-3},v_{n-2},v_1\} $ is also clearly independent. The determinant conditions that $ (v_1,\dots,v_{n-2}) $ fulfills together with $ v_{n-1} = v_1 $ imply $ W_1 = 0 $ and $ \phi(V_2) = 0 $ and thus $ \Image(\phi) = \langle \phi(V_1) \rangle = \langle V_1 \rangle $. Hence $ W_2 = \phi(V_n) \in \Image(\phi) $ is a multiple of $ V_1 $ and we see that $ \{W_1,V_n\} $ and $ \{W_2,V_1\} $ are linearly dependent. Then, $ (v_1,\dots,v_n) \in C_3^{--*}(n) $. \smallskip\\
In conclusion we have
\[
c_3^{--*}(n) = c_3^{+-*}(n-1) + q c_3^*(n-2). 
\]
\eqref{eq:recursion--**}: Let $ (v_1,\dots,v_n) \in C_3^{--**}(n) $. We once again have the cases $ v_{n-1} \neq v_1 $ with $ (v_1,\dots,v_{n-1}) \in C_3^{+-}(n-1) $ and $ v_{n-1} = v_1 $ with $ (v_1,\dots,v_{n-2}) \in C_3(n-2) $. \smallskip\\
Case 1: $ v_{n-1} \neq v_1 $. \\
Like in the proof of \eqref{eq:recursion--*}, Case $ 1 $, $ W_1 \neq 0 $ and $ v_n $ is again uniquely determined by $ \{W_1,V_n\} $ linearly dependent. This time however, we do not have another condition. \\
For the other direction, let $ (v_1,\dots,v_{n-1}) \in C_3^{+-}(n-1) $ and $ v_n = \pr(\langle W_1 \rangle) $. Then we clearly have $ \{ W_1, V_n\} $ linearly dependent and $ v_n \in (L_2 = L_3) \backslash L_1 $ and $ v_n \neq v_1 $, since $ W_1 $ is not a multiple of $ V_{n-1} $ or $ V_1 $. \smallskip\\
Case 2: $ v_{n-1} = v_1 $. \\
The same argument as for \eqref{eq:recursion--*}, Case 2 shows
\[
d_1 d_4 \dots d_{n-4} = d_2 d_5 \dots d_{n-6} \det(V_{n-3},V_{n-2},V_{n-1}) = d_2 d_5 \dots d_{n-6} \det(V_{n-3},V_{n-2},V_1).
\]
This implies $ (v_1,\dots,v_{n-2}) \in C_3^{**}(n-2) $. \\
Conversely, let $ (v_1,\dots,v_{n-2}) \in C_3^{**}(n-2), v_{n-1} = v_1 $ and $ v_n \in (L_2=L_3) \backslash L_1 $. There are $ q $ such choices for $ v_n $ and since we determinant condition in $ C_3^{**}(n-2) $ combined with $ v_{n-1} = v_1 $ still implies $ W_1 = 0 $, we have $ (v_1,\dots,v_n) \in C_3^{--**}(n) $ for each of these choices. \\
Putting both cases together, we get
\[
c_3^{--**}(n) = c_3^{+-}(n-1) + qc_3^{**}(n-1).
\]
Finally, consider $ n \equiv 1 \mod 3 $. \smallskip\\
\eqref{eq:recursion+-*1}:  First let $ (v_1,\dots,v_n) \in C_3^{+-*}(n) $. Like we showed in the proof of \eqref{eq:recursion+-}, we have $ (v_1,\dots,v_{n-1}) \in C_3(n-1) $ or $ C_3^{-+}(n-1) $. Let $ (V_1,\dots,V_n) $ be any lift and $ d_i := \det(V_i,V_{i+1},V_{i+2}) $. Then
\[
d_1 d_4 \dots d_{n-3} = d_3 d_6 \dots d_{n-4} \det(V_{n-1},V_1,V_2).
\]
Hence
\[
(v_1,\dots,v_{n-1}) \in C_3^{***}(n-1) := \{(v_1,\dots,v_{n-1}) \in C_3(n-1) \mid d^\prime_1 \dots d^\prime_{n-3} = d^\prime_3 \dots d^\prime_{n-1}\},
\]
respectively
\[
(v_1,\dots,v_{n-1}) \in C_3^{-+*}(n-1) := \{(v_1,\dots,v_{n-1}) \in C_3^{-+}(n-1) \mid d^\prime_1 \dots d^\prime_{n-3} = d^\prime_3 \dots d^\prime_{n-1}\},
\]
where $ d^\prime_i := \det(V_i,V_{i+1},V_{i+2}) $, but unlike in $ d_i $, the indices are considered modulo $ n-1 $. Conversely, the valid choices for $ v_n $ are the same as in the proof of \eqref{eq:recursion+-}, because we do not have any additional conditions for $ v_n $. Hence we have the recursion
\[
c_3^{+-*}(n) = (q-1)c_3^{***}(n-1) + qc_3^{-+*}(n-1).
\]
The map
\[
(v_1,\dots,v_{n-1}) \mapsto (v_{n-1},v_1,\dots,v_{n-2})
\]
yields bijections $ C_3^{***}(n-1) \to C_3^{**}(n-1) $ and $ C_3^{-+*}(n-1) \to C_3^{+-*}(n-1) $. Hence $ c_3^{***}(n-1) = c_3^{**}(n-1) $ and $ c_3^{-+*}(n-1) = c_3^{+-*}(n-1) $ and we get the recursion
\[
c_3^{+-*}(n) = (q-1)c_3^{**}(n-1) + q c_3^{+-*}(n-1).
\]
Now let $ (v_1,\dots,v_n) \in C_3^{-+*}(n) $. Note that
\[
(v_1,\dots,v_n) \mapsto (v_n,\dots,v_1)
\]
is a bijection $ C_3^{-+}(n) \to C_3^{+-}(n) $. By a symmetrical version of the argument in \eqref{eq:recursion+-}, we have $ (v_2,\dots,v_n) \in C_3(n-1) $ or $ C_3^{+-}(n-1) $. Let $ (V_1,\dots,V_n) $ be any lift and $ d_i := \det(V_i,V_{i+1},V_{i+2}) $. Then
\[
d_2 d_4 \dots d_{n-2} = d_3 d_6 \dots d_{n-4} \det(V_{n-1},V_n,V_2).
\]
This is equivalent to $ (v_2,\dots,v_n) \in C_3^{**}(n-1) $ respectively $ (v_2,\dots,v_{n-1}) \in C_3^{+-*}(n-1) $. Conversely, we need $ v_1 \in L_1 \backslash(L_2 \cup L_3) $ with $ L_1 := v_{n-1} \vee v_n $, $ L_2 := v_n \vee v_2 $ and $ L_3 := v_2 \vee v_3 $. Again by symmetry, this leaves $ (q-1) $ respectively $ q $ choices for $ v_1 $. Hence
\[
c_3^{-+*}(n) = (q-1)c_3^{**}(n-1) + qc_3^{+-*}(n-1),
\]
which is the same recursion as for $ c_3^{+-*}(n) $. Therefore $ c_3^{-+*}(n) = c_3^{+-*}(n) $.
\end{proof}

\begin{theorem} \label{thm:ck(n)divisibleBy3}
Let $ \lvert K \rvert = q < \infty $ and $ n \in 3\N_{>0} $. Then
\[
c_3^*(n) = \frac{q^{2n} + 3f(q)q^{\frac{4}{3}n} + (q^5-3q^4-2q^3+5q^2+9q+6)q^n + 3f(q)q^{\frac{2}{3}n+1} + q^3}{(q-1)^2}
\]
with
\[
f(q) = q^3-q^2-q-2.
\]
\end{theorem}
\begin{proof}
The following formulas and the formula in the statement of the Theorem count all the sets in the recursions in Lemma \ref{lem:recursionsk3*}. For $ n \equiv 0 \mod 3 $:
\begin{align*}
c_3^{**}(n) &= \frac{q^{2n} + f(q)q^{\frac{4}{3}n} + (2q^2+2q)q^n + f(q)q^{\frac{2}{3}n + 1} + q^3}{q-1}\\
\intertext{and}
c_3^{+-*}(n) &= \frac{q^{2n-1} + f(q)q^{\frac{4}{3}n-1} + (-q^2+1)q^n - f(q)q^{\frac{2}{3}n+1} - q^3}{q-1}.\\
\intertext{For $ n \equiv 2 \mod n $:}
c_3^{+-*}(n) &= \frac{q^{2n-1} + f(q)q^{\frac{4}{3}(n-2)+2} - f(q)q^{\frac{2}{3}(n-2)+2} - q^3}{q-1}\\
\intertext{and}
c_3^{--*}(n) &= \frac{q^{2n-2} + g(q)q^{\frac{4}{3}(n-2)+1} (q^5-3q^4-q^3+4q^2+8q+7)q^{n-1} + h(q)q^{\frac{2}{3}(n-2)+1} + q^3}{(q-1)^2}\\
\intertext{with}
g(q) &= q^4+q^3-3q^2-4q-4 \text{ and } h(q) = 2q^4-q^3-3q^2-5q-2,\\
\intertext{as well as}
c_3^{--**}(n) &= \frac{q^{2n-2} + f(q)q^{\frac{4}{3}(n-2)+1} + (q^3+q^2+q+1)q^{n-1} + f(q)q^{\frac{2}{3}(n-2)+2} + q^3}{q-1}.\\
\intertext{For $ n \equiv 1 \mod 3 $:}
c_3^{+-*}(n) &= \frac{q^{2n-1} + f(q)q^{\frac{4}{3}(n-1)+1} + (q^2-1)q^n - f(q)q^{\frac{2}{3}(n-1)+1} - q^3}{q-1}.\\
\intertext{For $ n = 3 $ the conditions}
\det(V_1,V_2,V_3) &= \det(V_2,V_3,V_1) = \det(V_3,V_1,V_2)\\
\intertext{are trivial, since the matrices here only differ by column permutations with even sign. Hence}
c_3^*(3) &= c_3^{**}(3) = c_3(3) = q^6 + 2q^5 + 2q^4 + q^3.\\
\end{align*}
We also have $ c_3^{+-}(3) = 0 $, because $ \{v_1,v_2,v_3\} $ independent is equivalent to $ \{v_3,v_1,v_2\} $ independent. This implies $ c_3^{+-*}(3) = 0 $. One can check that the formulas for $ c_3^*(n), c_3^{**}(n) $ and $ c_3^{+-*}(n) $ agree with this. For $ n \geq 4 $ one can verify that the given formulas obey the recursion in Lemma \ref{lem:recursionsk3*}.
\end{proof}

\begin{cor} \label{cor:k3divisibleBy3}
Let $ \lvert K \rvert = q < \infty $. Let $ w \geq 1 $ and $ n = w + 4 $. If $ \gcd(3,n) \neq 1 $, then the number of tame $ \SL_3 $-frieze patterns of width $ w $ is
\[
f_q(3,n) = \frac{q^{2n} + 3f(q)q^{\frac{4}{3}n} + (q^5-3q^4-2q^3+5q^2+9q+6)q^n + 3f(q)q^{\frac{2}{3}n+1} + q^3}{(q^2+q+1)(q+1)q^3(q-1)^2}
\]
with
\[
f(q) = q^3-q^2-q-2.
\]
\end{cor}
\begin{proof}
We have
\[
\lvert \PGL(3,K) \rvert = (q^2+q+1)(q+1)q^3(q-1)^2
\]
and $ \gcd(3,n) \neq 1 $ implies $ \gcd(3,n) = 3 $. Then Theorem \ref{thm:numberOfFriezes} and Theorem \ref{thm:ck(n)divisibleBy3} prove the claim.
\end{proof}

\section{The case $ k = 4 $ and $ n $ odd}

Similar to the previous cases, we define certain sets that we will count alongside $ C_4(n) $:

\begin{defi}\label{def:k4}
Let $ s_1,s_2,s_3 \in \{+,-\} $ and
\[
I_{s_1,s_2,s_3} := [n-3] \cup \{n-3+j \mid j \in \{1,2,3\}, s_j = + \}.
\]
Then define
\begin{align*}
C_4^{s_1s_2s_3}(n) := \{ (v_1,\dots,v_n) \in (\pr^3(K))^n \mid &\{v_i,v_{i+1},v_{i+2},v_{i+3}\} \text{ independent } \Leftrightarrow i \in I_{s_1,s_2,s_3} \\
&\text{ and } \{v_i,v_{i+1},v_{i+2}\} \text{ independent for } i =n-1,n\}.
\intertext{In addition, we need to consider some sets where three consecutive points are not always independent. We denote this with a minus in brackets. Define}
C_4^{s_1(-)[s_2]}(n) := \{ (v_1,\dots,v_n) \in (\pr^3(K))^n \mid
&\{v_i,v_{i+1},v_{i+2},v_{i+3}\} \text{ independent for } i \in [n-3],\\
&\{v_{n-2},v_{n-1},v_n,v_1\} \text{ independent} \Leftrightarrow s_1 = +,\\
&\{v_n,v_1,v_2\} \text{ dependent,}\\
&\{v_{n-1},v_1,v_2,v_3\} \text{ independent} \Leftrightarrow s_2 = + \}
\intertext{and}
C_4^{(-)s_1[s_2]}(n) := \{ (v_1,\dots,v_n) \in (\pr^3(K))^n \mid
&\{v_i,v_{i+1},v_{i+2},v_{i+3}\} \text{ independent for } i \in [n-3],\\
&\{v_n,v_1,v_2,v_3\} \text{ independent} \Leftrightarrow s_1 = +,\\
&\{v_{n-1},v_n,v_1\} \text{ dependent,}\\
&\{v_{n-2},v_{n-1},v_n,v_2\} \text{ independent} \Leftrightarrow s_2 = +\}
\intertext{ as well as}
C_4^{(--)[+]}(n) := \{ (v_1,\dots,v_n) \in (\pr^3(K))^n \mid
&\{v_i,v_{i+1},v_{i+2},v_{i+3}\} \text{ independent for } i \in [n-3],\\
&\{v_{n-1},v_n,v_1\},\{v_n,v_1,v_2\} \text{ dependent and } v_n \neq v_1,\\
&\{v_{n-2},v_1,v_2,v_3\} \text{ independent}\},
\end{align*}
where $ s \in \{+,-\} $. Like before, we use a lower-case $ c $ to denote the cardinality of these sets.
\end{defi}
\begin{rem}
Clearly, $ C_4^{+++}(n) = C_4(n) $. We have $ c_4^{++-}(n) = c_4^{+-+}(n) = c_4^{-++}(n) $ and $ c_4^{+--}(n) = c_4^{--+}(n) $ via the bijection
\[
(v_1,\dots,v_n) \mapsto (v_2,\dots,v_n,v_1).
\]
We also have $ c_4^{s_1(-)[s_2]}(n) = c_4^{(-)s_1[s_2]}(n) $ via the bijection
\[
(v_1,\dots,v_n) \mapsto (v_n,\dots,v_1).
\]
\end{rem}
\begin{lem}\label{lem:recursionsk4}
The following recursions hold for $ n \geq 5 $:
\begin{align}
c_4(n) = &(q-1)^3 c_4(n-1) + 3 q (q-1)^2 c_4^{++-}(n-1) + 2 q^2 (q-1) c_4^{+--}(n-1)\label{eq:k4+++}\\
         &+ q^2 (q-1) c_4^{-+-}(n-1) + q^3 c_4^{---}(n-1),\nonumber\\
c_4^{++-}(n) = &(q-1)^2 c_4(n-1) + 2 q (q-1) c_4^{++-}(n-1) + q^2 c_4^{+--},\label{eq:k4++-}\\
c_4^{+--}(n) = &(q-1)^2 c_4^{++-}(n-1) + q(q-1) c_4^{-+-}(n-1)\label{eq:k4+--}\\
               &+ q(q-1) c_4^{+(-)[+]}(n-1) + q^2 c_4^{-(-)[+]}(n-1),\nonumber\\
c_4^{-+-}(n) = &(q-1)c_4(n-1) + q(q-1)c_4^{+(-)[+]}(n-1) + q^2c_4^{+(-)[-]}(n-1),\label{eq:k4-+-}\\
c_4^{---}(n) = &(q-1)^2c_4^{+--}(n-1) + q(q-1)c_4^{+(-)[-]}(n-1)\label{eq:k4---}\\
               &+ q(q-1)c_4^{-(-)[+]} + q^2c_4^{(--)[+]}(n-1),\nonumber\\
c_4^{+(-)[+]}(n) = &(q-1)c_4(n-1) + qc_4^{++-}(n-1),\label{eq:k4+(-)[+]}\\
c_4^{+(-)[-]}(n) = &(q-1)c^{++-}_4(n-1) + qc_4^{-+-}(n-1),\label{eq:k4+(-)[-]}\\
c_4^{-(-)[+]}(n) = &(q-1)c_4^{++-}(n-1) + q c_4^{+(-)[+]}(n-1),\label{eq:k4-(-)}\\
c_4^{(--)[+]}(n) = &(q-1) c_4^{+(-)[+]}(n-1) + q c_4(n-2).\label{eq:k4(--)}
\end{align}
\end{lem}
\begin{proof}
Let $ E_1 := v_{n-3} \vee v_{n-2} \vee v_{n-1} $, $ E_2 := v_{n-2} \vee v_{n-1} \vee v_1 $, $ E_3 := v_{n-1} \vee v_1 \vee v_2 $ and $ E_4 := v_1 \vee v_2 \vee v_3 $. \smallskip\\
\eqref{eq:k4+++}: Let $ (v_1,\dots,v_n) \in C_4(n) $. Then $ \{v_{n-2},v_{n-1},v_n,v_1\} $ and $ \{v_{n-1},v_n,v_1,v_2\} $ are independent and thus $ \{v_{n-2},v_{n-1},v_1\} $ and $ \{v_{n-1},v_1,v_2\} $ are also independent. Hence
\[
(v_1,\dots,v_{n-1}) \in C_4^{s_1 s_2 s_3}(n-1)
\]
with $ s_1, s_2, s_3 \in \{+,-\} $. \\
For the other direction, let $ (v_1,\dots,v_{n-1}) \in C_4^{s_1s_2s_3}(n-1) $ and $ v_n \in \pr^3(K) $. Then $ (v_1,\dots,v_n) \in C_4(n) $ if and only if $ v_n \notin E_1 \cup E_2 \cup E_3 \cup E_4 $. Note that $ E_i \neq E_{i+1} $ is equivalent to $ s_i = + $ for $ i = 1,2,3 $. Moreover, if $ E_1 = E_3 $, then $ v_{n-3},v_{n-2},v_{n-1},v_1 $ and $ v_2 $ all lie in the same projective plane and hence $ E_1 = E_2 = E_3 $. Likewise, $ E_2 = E_4 $ implies $ E_2 = E_3 = E_4 $ and $ E_1 = E_4 $ implies that all four planes are the same. Thus there are $ m := \lvert \{ i \in [3] \mid s_i = + \} \rvert + 1 $ pairwise different projective planes in which $ v_n $ cannot be. The intersection of two different projective planes in $ \pr^3(K) $ is always a projective line. More specifically, we have $ E_1 \cap E_2 = v_{n-2} \vee v_{n-1} $, $ E_2 \cap E_3 = v_{n-1} \vee v_1 $ and $ E_3 \cap E_4 = v_1 \vee v_2 $ if the corresponding planes are not equal. Hence $ E_1 \cap E_2 \cap E_3 = \{v_{n-1}\} $ and $ E_2 \cap E_3 \cap E_4 = \{v_1\} $ if none of the intersecting planes are equal. The intersections $ E_1 \cap E_3 \cap E_4 $ and $ E_1 \cap E_2 \cap E_4 $ also contain only one point each (assuming the planes involved are not equal), because otherwise we would have $ v_1,v_2 \in E_1 $ respectively $ v_{n-2}, v_{n-1} \in E_4 $, which in turn would imply $ E_1 = E_3 $ respectively $ E_2 = E_4 $.  Finally, $ E_1 \cap E_2 \cap E_3 \cap E_4 = \emptyset $ if $ s_1 = s_2 = s_3 = + $. Thus we can calculate the number of choices for $ v_n $ depending on $ m $ using the inclusion-exclusion principle. For $ m = 4 $:
\begin{align*}
\lvert \pr^3(K) \backslash E_1 \cup E_2 \cup E_3 \cup E_4 \rvert &= q^3 + q^2 + q + 1 - 4 (q^2+q+1) + 6 (q+1) - 4 \cdot 1 + 0 \\
&= q^3 - 3 q^2 + 3 q - 1 = (q-1)^3.
\end{align*}
For $ m = 3 $:
\begin{align*}
\lvert \pr^3(K) \backslash E_1 \cup E_2 \cup E_3 \cup E_4 \rvert &= q^3 + q^2 + q + 1 - 3 (q^2+q+1) + 3 (q+1) - 1 \cdot 1 \\
&= q^3 - 2 q^2 + q = q(q-1)^2.
\end{align*}
For $ m = 2 $:
\begin{align*}
\lvert \pr^3(K) \backslash E_1 \cup E_2 \cup E_3 \cup E_4 \rvert &= q^3 + q^2 + q + 1 - 2 (q^2+q+1) + 1 (q+1) \\
&= q^3 - q^2 = q^2(q-1).
\end{align*}
For $ m = 1 $:
\begin{align*}
\lvert \pr^3(K) \backslash E_1 \cup E_2 \cup E_3 \cup E_4 \rvert = q^3 + q^2 + q + 1 - (q^2+q+1) = q^3.
\end{align*}
This and symmetry imply
\begin{align*}
c_4(n) = &(q-1)^3 c_4(n-1) + 3 q (q-1)^2 c_4^{++-}(n-1) + 2 q^2 (q-1) c_4^{+--}(n-1) \\
         &+ q^2 (q-1) c_4^{-+-}(n-1) + q^3 c_4^{---}(n-1).
\end{align*}
\eqref{eq:k4++-}: Now let $ (v_1,\dots,v_n) \in C_4^{++-}(n) $. Then $ v_n \vee v_1 \vee v_2 = v_1 \vee v_2 \vee v_3 $ and since $ \{v_{n-1},v_n,v_1,v_2\} $ is independent, it follows that $ \{v_{n-1},v_1,v_2,v_3\} $ is also independent. Then $ \{v_{n-1},v_1,v_2\} $ is independent as well. Moreover, $ \{v_{n-2},v_{n-1},v_n,v_1\} $ independent implies $ \{v_{n-2},v_{n-1},v_1\} $ independent. Hence $ (v_1,\dots,v_{n-1}) \in C_4^{s_1s_2+}(n-1) $ with $ s_1,s_2 \in \{+,-\} $. \\
Conversely, let $ (v_1,\dots,v_{n-1}) \in C_4^{s_1s_2+}(n-1) $. Then $ (v_1,\dots,v_n) \in C_4^{++-}(n) $ is equivalent to $ v_n \in E_4 \backslash (E_1 \cup E_2 \cup E_3) $. Like in the last case, two distinct of these four planes intersect in a line, while any three distinct intersect in a point and four distinct have an empty intersection. Unlike before, $ E_4 $ is always distinct from the other planes since $ \{v_{n-1},v_1,v_2,v_3\} $ is independent. We can now again count the choices for $ v_n $ with the inclusion-exclusion principle. If $ s_1 = s_2 = + $:
\begin{align*}
\lvert E_4 \backslash (E_1 \cup E_2 \cup E_3) \rvert &= q^2 + q + 1 - 3 (q+1) + 3 \cdot 1 - 0 \\
&= q^2 - 2 q + 1 = (q-1)^2.
\end{align*}
If $ s_1 \neq s_2 $ we have
\begin{align*}
\lvert E_4 \backslash (E_1 \cup E_2 \cup E_3) \rvert &= q^2 + q + 1 - 2 (q+1) + 1 \cdot 1 \\
&= q^2 - q = q(q-1).
\end{align*}
In the case $ s_1 = s_2 = - $ we get
\begin{align*}
\lvert E_4 \backslash (E_1 \cup E_2 \cup E_3) \rvert &= q^2 + q + 1 - (q + 1) = q^2.
\end{align*}
With $ c_4^{-++}(n-1) = c_4^{+-+}(n-1) = c_4^{++-}(n-1) $ and $ c_4^{--+}(n-1) = c_4^{+--}(n-1) $ this yields
\[
c_4^{++-}(n) = (q-1)^2 c_4(n-1) + 2 q (q-1) c_4^{++-}(n-1) + q^2 c_4^{+--}(n-1).
\]
\eqref{eq:k4+--}: Let $ (v_1,\dots,v_n) \in C_4^{+--}(n) $. Then $ v_{n-1} \vee v_n \vee v_1 = v_n \vee v_1 \vee v_2 = v_1 \vee v_2 \vee v_3 = E_4 $. In particular, $ \{v_{n-1},v_1,v_2,v_3\} $ is dependent. We also have $ \{v_{n-2},v_{n-1},v_n,v_1\} $ independent which implies $ \{v_{n-2},v_{n-1},v_1\} $ independent. We separate cases: \smallskip\\
Case 1: $ \{v_{n-1},v_1,v_2\} $ is independent. \\
$ v_{n-1} $, $ v_1 $ and $ v_2 $ all lie in $ E_4 $. $ \{v_{n-2},v_{n-1},v_n,v_1\} $ independent implies $ v_{n-2} \notin E_4 $, thus $ \{v_{n-2},v_{n-1},v_1,v_2\} $ is also independent. Hence $ (v_1,\dots,v_{n-1}) \in C_4^{s+-}(n-1) $ with $ s \in \{+,-\} $. \smallskip\\
Case 2: $ \{v_{n-1},v_1,v_2\} $ is dependent. \\
Then we still have $ v_{n-2} \notin v_{n-1} \vee v_n \vee v_1 = E_4 $. It follows that $ \{v_{n-2},v_1,v_2,v_3\} $ is independent and thus $ (v_1,\dots,v_{n-1}) \in C_4^{s(-)[+]} $ with $ s \in \{+,-\} $.\smallskip\\
Conversely, let $ (v_1,\dots,v_{n-1}) \in C_4^{s+-}(n-1) $. Let $ L = v_1 \vee v_2 $. Then $ E_2 \neq E_3 = E_4 $ and $ E_1 \neq E_2 $ is equivalent to $ s = + $. If $ E_1 \neq E_2 $, then also $ E_1 \neq E_3 $. We have $ (v_1,\dots,v_n) \in C_4^{+--}(n) $ if and only if $ v_n \in E_3 \backslash (E_1 \cup E_2 \cup L) $ (we need $ v_n \notin L $ since $ (v_1,\dots,v_n) \in C_4^{+--}(n) $ requires $ \{v_n,v_1,v_2\} $ independent). We have $ E_2 \cap E_3 = v_{n-1} \vee v_1 \neq L $, thus $ E_2 \cap E_3 \cap L = \{v_1\} $. We also have $ E_1 \cap E_2 = v_{n-2} \vee v_{n-1} $ if $ s = + $. Since $ \{v_{n-2},v_{n-1},v_1\} $ is independent the lines $ E_1 \cap E_3 $, $ E_2 \cap E_3 $ and $ L $ do not have a common intersection point if $ s = + $. Like we calculated in the proof of \eqref{eq:k4++-}, we then have $ (q-1)^2 $ possibilities for $ v_n $ if $ s = + $ and $ q(q-1) $ if $ s = - $. \\
Now let $ (v_1,\dots,v_{n-1}) \in C_4^{s(-)[+]}(n-1) $. Then $ \{v_{n-1},v_1,v_2\} $ is dependent and thus $ \{v_{n-1},v_n,v_1,v_2\} $ dependent is trivial. Accordingly, $ (v_1,\dots,v_n) \in C_4^{+--}(n) $ if and only if $ v_n \in E_4 \backslash (E_1 \cup E_2 \cup L) $. The condition $ \{v_{n-2},v_1,v_2,v_3\} $ independent guaranties $ E_2 \neq E_4 $. We also have $ E_2 \cap E_4 = v_{n-1} \vee v_1 = v_1 \vee v_2 = L $. $ E_1 \neq E_2 $ is equivalent to $ s = + $. Hence we have one or two forbidden lines on $ E_4 $, depending on $ s $, and get $ q^2 $ viable choices if $ s = - $ and $ q(q-1) $ viable choices if $ s = + $. In conclusion, this leads to
\[
c_4^{+--}(n) = (q-1)^2 c_4^{++-}(n-1) + q(q-1) c_4^{-+-}(n-1) + q(q-1) c_4^{+(-)[+]}(n-1) + q^2 c_4^{-(-)[+]}(n-1).
\]
\eqref{eq:k4-+-}: Let $ (v_1,\dots,v_n) \in C_4^{-+-}(n) $. Then $ v_{n-1} \notin v_n \vee v_1 \vee v_2 = v_1 \vee v_2 \vee v_3 $. Thus $ \{v_{n-1},v_1,v_2,v_3\} $ is independent. If $ \{v_{n-2},v_{n-1},v_1\} $ is dependent, then $ (v_1,\dots,v_{n-1}) \in C_4^{(-)+[s]}(n-1) $ with $ s \in \{+,-\} $ follows. Hence only the case $ \{v_{n-2},v_{n-1},v_1\} $ independent remains. Then  $ E := v_{n-2} \vee v_{n-1} \vee v_1 = v_{n-2} \vee v_{n-1} \vee v_n = v_{n-1} \vee v_n \vee v_1 $. The fact that $ \{v_{n-3},v_{n-2},v_{n-1},v_n\} $ and $ \{v_{n-1},v_n,v_1,v_2\} $ are independent implies $ v_{n-3},v_2 \notin E $. $ \{v_{n-3},v_{n-2},v_{n-1},v_1\} $ and $ \{v_{n-2},v_{n-1},v_1,v_2\} $ are thus independent and we have $ (v_1,\dots,v_{n-1}) \in C_4(n-1) $. \smallskip\\
Conversely, let $ (v_1,\dots,v_{n-1}) \in C_4(n-1) $. Then we have $ (v_1,\dots,v_n) \in C_4^{-+-}(n) $ if and only if $ v_n \in E_2 \cap E_4 \backslash (E_1 \cup E_3) $. Like we showed earlier, these planes are pairwise distinct and any three of them intersect in just one point while all four do not intersect in any point. Hence $ E_2 \cap E_4 $ is a projective line and $ E_1 $ and $ E_3 $ intersect this line in two different points. This leaves exactly $ q-1 $ valid choices $ v_n. $\\
Now let $ (v_1,\dots,v_{n-1}) \in C_4^{(-)+[s]}(n-1) $ with $ s \in \{+,-\} $. Then $ \{v_{n-2},v_{n-1},v_1 \} $ is dependent and thus also $ \{v_{n-2},v_{n-1},v_n,v_1\} $, regardless of how we choose $ v_n $. Therefore we only need $ v_n \in E_4 \backslash (E_1 \cup E_3) $. $ s = + $ implies $ \{v_{n-3},v_{n-2},v_{n-1}, v_2\} $ independent. Then $ v_2 \notin E_1 $ and in particular $ E_1 \neq E_3 $. On the other hand, we have $ v_1 \in v_{n-2} \vee v_{n-1} \subset E_1 $. If $ s = - $, then also $ v_2 \in E_1 $ and thus $ E_1 = E_3 $. At the same time $ \{v_{n-1},v_1,v_2,v_3\} $ independent ensures $ E_3 \neq E_4 $ and also $ E_1 \neq E_4 $. Thus we have $ q(q-1) $ possibilities for $ v_n $ if $ s = + $ and $ q^2 $ possibilities otherwise. In conclusion,
\[
c_4^{-+-}(n) = (q-1)c_4(n-1) + q(q-1)c_4^{+(-)[+]}(n-1) + q^2c_4^{+(-)[-]}(n-1).
\]
Here we use $ c_4^{(-)+[s]}(n-1) = c_4^{+(-)[s]}(n-1) $. \smallskip\\
\eqref{eq:k4---}: Let $ (v_1,\dots,v_n) \in C_4^{---}(n) $. It follows that $ v_{n-2} \vee v_{n-1} \vee v_n = v_{n-1} \vee v_n \vee v_1 = v_n \vee v_1 \vee v_2 = v_1 \vee v_2 \vee v_3 = E_4 $ are all the same plane. But $ v_{n-3} \notin E_4 $ since $ \{v_{n-3},v_{n-2},v_{n-1},v_n\} $ is independent. We separate cases based on whether $ M_1 = \{v_{n-2},v_{n-1},v_1\} $ and $ M_2 = \{v_{n-1},v_1,v_2\} $ are independent. \smallskip\\
Case 1: $ M_1 $ and $ M_2 $ are both independent. \\
Then $ \{v_{n-3}, v_{n-2}, v_{n-1}, v_1\} = \{v_{n-3}\} \cup M_1 $ is independent, whereas $ \{v_{n-2},v_{n-1},v_1,v_2\} $ and $ \{v_{n-1},v_1,v_2,v_3\} $ are clearly dependent since all their points lie in the plane $ E_2 = E_3 = E_4 $. Hence $ (v_1,\dots,v_{n-1}) \in C_4^{+--}(n-1) $. \smallskip\\
Case 2: $ M_1 $ is independent, $ M_2 $ is dependent. \\
$ \{v_{n-3},v_{n-2},v_{n-1},v_1\} $ is again independent. We also have $ \{v_{n-2},v_1,v_2,v_3\} \subset E_4 $ dependent. Hence $ (v_1,\dots,v_{n-1}) \in C_4^{+(-)[-]}(n-1) $. \smallskip\\
Case 3: $ M_1 $ is dependent, $ M_2 $ is independent. \\
Then $ \{v_{n-1},v_1,v_2,v_3\} $ is dependent and $ \{v_{n-3},v_{n-1},v_1,v_2\} = \{v_{n-3}\} \cup M_2 $ is independent. We know $ v_{n-2} \vee v_{n-1} = v_{n-1} \vee v_1 $ and it follows that $ \{v_{n-3},v_{n-2},v_{n-1},v_2\} $ is also independent. Thus $ (v_1,\dots,v_{n-1}) \in C_4^{(-)-[+]}(n-1) $. \smallskip\\
Case 4: $ M_1 $ and $ M_2 $ are both dependent. \\
Then $ \{v_{n-3},v_1,v_2,v_3\} $ is independent and thus $ (v_1,\dots,v_{n-1}) \in C_4^{(--)[+]}(n-1) $. \smallskip\\
Conversely, we now count the possible choices for $ v_n $ in all cases. \\
Case 1. \\
Let $ (v_1,\dots,v_{n-1}) \in C_4^{+--}(n-1) $. For $ (v_1,\dots,v_n) \in C_4^{---}(n) $ we require $ v_n \in E_2 = E_3 = E_4 $ and $ v_n \notin E_1 $. In addition, $ \{v_{n-1},v_n,v_1\} $ and $ \{v_n,v_1,v_2\} $ need to be independent, hence $ v_n $ must not be in $ v_{n-1} \vee v_1 $ or in $ v_1 \vee v_2 $. Together with $ E_1 \cap E_4 = v_{n-2} \vee v_{n-1} $ we have three forbidden lines on $ E_4 $. Since $ M_1 $ and $ M_2 $ are independent these lines are pairwise distinct and do not have a common intersection. Thus there are $ (q-1)^2 $ possibilities for $ v_n $ (see the proof of \eqref{eq:k4++-} for the calculation). \smallskip\\
Case 2. \\
Let $ (v_1,\dots,v_{n-1}) \in C_4^{+(-)[-]}(n-1) $. Again, we need $ v_n \in E_2, E_4 $ and $ v_n \notin E_1 $. But this time the condition $ \{v_{n-1},v_n,v_1,v_2\} $ dependent is trivial since $ M_2 $ is dependent. Also note that $ \{v_{n-2},v_1,v_2,v_3\} $ dependent ensures $ E_2 = E_4 $. We still require $ v_n \notin v_{n-1} \vee v_1, v_1 \vee v_2 $, but these two lines are now the same while $ E_1 \cap E_4 = v_{n-2} \vee v_{n-1} $ is still a different line. Hence there are two forbidden lines on $ E_4 $ and we have $ q(q-1) $ viable choices. \smallskip\\
Case 3. \\
Let $ (v_1,\dots,v_{n-1}) \in C_4^{(-)-[+]}(n-1) $. Then we need $ v_n \in E_3 = E_4 $ and $ v_n \notin E_1 $ whereas $ \{v_{n-2},v_{n-1},v_n,v_1\} $ dependent is trivial thanks to $ M_1 $ dependent. We also have the condition $ \{v_{n-3},v_{n-2},v_{n-1},v_2\} $ independent which guaranties $ E_1 \neq E_4 $. Additionally, we require $ v_n \notin v_{n-1} \vee v_1, v_1 \vee v_2 $ with the two lines being distinct, but we also have $ E_1 \cap E_4 = v_{n-2} \vee v_{n-1} = v_{n-1} \vee v_1 $. Hence there are two forbidden lines like in the last case and we have $ q(q-1) $ options for $ v_n $. \smallskip\\
Case 4. \\
Let $ (v_1,\dots,v_{n-1}) \in C_4^{(--)[+]}(n-1) $. The conditions $ \{v_{n-2},v_{n-1},v_n,v_1\} $ and $ \{v_{n-1},v_n,v_1,v_2\} $ dependent are both trivial in this case. We need $ v_n \in E_4 \backslash (E_1 \cup v_{n-1} \vee v_1 \cup v_1 \vee v_2) $. This is possible because our condition $ \{v_{n-3},v_1,v_2,v_3\} $ independent implies $ E_1 \neq E_4 $. Furthermore, we have $ E_1 \cap E_4 = v_{n-2} \vee v_{n-1} = v_{n-1} \vee v_1 = v_1 \vee v_2 $ so there only is one forbidden line on $ E_4 $. Hence we have $ q^2 $ suitable choices for $ v_n $. \\
In conclusion,
\begin{align*}
c_4^{---}(n) = &(q-1)^2c_4^{+--}(n-1) + q(q-1)c_4^{+(-)[-]}(n-1)\\
               &+ q(q-1)c_4^{-(-)[+]}(n-1) + q^2c_4^{(--)[+]}(n-1)
\end{align*}
where we used $ c_4^{(-)-[+]}(n-1) = c_4^{-(-)[+]}(n-1) $.\smallskip\\
\eqref{eq:k4+(-)[+]}, \eqref{eq:k4+(-)[-]}: Let $ (v_1,\dots,v_n) \in C_4^{+(-)[s]}(n) $ with $ s \in \{+,-\} $. Then $ \{v_{n-1},v_1,v_2,v_3\} $ is independent if and only if $ s = + $. Moreover, we have $ v_n \vee v_1 = v_1 \vee v_2 $ and $ \{v_{n-2},v_{n-1},v_n,v_1\} $ independent. If follows that $ \{v_{n-2},v_{n-1},v_1,v_2\} $ is independent. Hence $ (v_1,\dots,v_{n-1}) \in C_4^{s^\prime+s}(n-1) $ with $ s^\prime \in \{+,-\} $. \\
For the other direction, let $ (v_1,\dots,v_{n-1}) \in C_4^{s^\prime+s}(n-1) $. Then we automatically get that $ \{v_{n-1},v_1,v_2,v_3\} $ independent is equivalent to $ s = + $. We need $ \{v_n,v_1,v_2\} $ dependent, therefore $ v_n \in v_1 \vee v_2 $. Additionally, we require $ \{v_{n-3},v_{n-2},v_{n-1},v_n\} $ and $ \{v_{n-2},v_{n-1},v_n,v_1\} $ independent, hence $ v_n \notin E_1, E_2 $. We have $ \{v_{n-2},v_{n-1},v_1,v_2\} $ independent which implies $ v_2 \notin E_2 $ and that $ v_{n-2} \vee v_{n-1} $ and $ v_1  \vee v_2 $ do not lie in a common plane, thus both $ E_1 $ and $ E_2 $ only intersect $ v_1 \vee v_2 $ in one point each. Clearly, $ v_1 $ is the intersection point with $ E_2 $. We have $ E_1 \cap (v_1 \vee v_2) = \{v_1\} $ if and only if $ s^\prime = - $. Accordingly, there are $ q-1 $ choices for $ v_n $ if $ s^\prime = + $ and $ q $ choices otherwise. Using $ c_4^{-++}(n-1) = c_4^{++-}(n-1) $ we get the equation
\[
c_4^{+(-)[s]}(n) = (q-1)c^{++s}_4(n-1) + qc_4^{s+-}(n-1).
\]
\eqref{eq:k4-(-)}: Let $ (v_1,\dots,v_n) \in C_4^{-(-)[+]}(n) $. Then $ \{v_{n-1},v_1,v_2,v_3\} $ is independent. If $ \{v_{n-2},v_{n-1},v_1\} $ is independent, then $ v_{n-2} \vee v_{n-1} \vee v_1 = v_{n-2} \vee v_{n-1} \vee v_n $. Since $ \{v_{n-3},v_{n-2},v_{n-1},v_n\} $ is independent, $ \{v_{n-3},v_{n-2},v_{n-1},v_1\} $ is then also independent. At the same time we have $ v_{n-2}, v_2 \in v_{n-1} \vee v_n \vee v_1 $, hence $ \{v_{n-2},v_{n-1},v_1,v_2\} $ is dependent. Thus, $ (v_1,\dots,v_{n-1}) \in C_4^{+-+}(n-1) $. \\
Now let $ \{v_{n-2},v_{n-1},v_1\} $ be dependent. Note that $ v_{n-1} \vee v_n \vee v_1 = v_{n-2} \vee v_{n-1} \vee v_n $ and $ v_{n-3} \notin v_{n-2} \vee v_{n-1} \vee v_n $. Hence $ \{v_{n-3},v_{n-1},v_n,v_1 \} $ is independent. We have $ v_n \vee v_1 = v_1 \vee v_2 $, therefore $ \{v_{n-3},v_{n-1},v_1,v_2\} $ is also independent. Additionally, we have $ v_{n-2} \vee v_{n-1} = v_{n-1} \vee v_1 $, thus $ \{v_{n-3},v_{n-2},v_{n-1},v_2\} $ is independent as well. This implies $ (v_1,\dots,v_{n-1}) \in C_4^{(-)+[+]}(n-1) $. \\
Conversely, let $ (v_1,\dots,v_{n-1}) \in C_4^{+-+}(n-1) $ or $ C_4^{(-)+[+]}(n-1) $. If we want $ (v_1,\dots,v_n) \in C_4^{-(-)[+]} $, we require $ v_n \in v_1 \vee v_2 \backslash \{v_1\} $ so that $ \{v_n,v_1,v_2\} $ is dependent and $ v_n \neq v_1 $. Additionally, we require $ \{v_{n-3},v_{n-2},v_{n-1},v_n\} $ independent, $ \{v_{n-2},v_{n-1},v_n,v_1\} $ dependent, $ \{v_{n-1},v_n,v_1\} $ independent and $ \{v_{n-1},v_1,v_2,v_3\} $ independent. The last three conditions follow from the definition of $ C_4^{+-+}(n-1) $ respectively $ C_4^{(-)+[+]}(n-1) $, using $ v_n \vee v_1 = v_1 \vee v_2 $ for the second condition and $ v_{n-1} \notin v_1 \vee v_2 $ for the third condition. It only remains to check which elements of $ v_1 \vee v_2 \backslash \{v_1\} $ fulfill $ v_n \notin E_1 = v_{n-3} \vee v_{n-2} \vee v_{n-1} $. If $ (v_1,\dots,v_{n-1}) \in C_4^{+-+}(n-1) $, then $ \{v_{n-3},v_{n-2},v_{n-1},v_1\} $ is independent. Hence $ v_1 \vee v_2 \not \subset E_1 $ and the intersection between this line and plane is exactly one point. This point is clearly not $ v_1 $, hence we are left with $ q-1 $ viable choices for $ v_n $. \\
If $ (v_1,\dots,v_{n-1}) \in C_4^{(-)+[+]}(n-1) $, then $ \{v_{n-3},v_{n-2},v_{n-1},v_2\} $ is independent and thus again $ v_1 \vee v_2 \not\subset E_1 $. But $ \{v_{n-3},v_{n-2},v_{n-1},v_1\} $ is dependent and therefore the intersection point is $ v_1 $. Accordingly, we have $ q $ choices for $ v_n $. It follows that
\[
c_4^{-(-)[+]}(n) = (q-1)c_4^{++-}(n-1) + q c_4^{+(-)[+]}(n-1),
\]
where we use $ c_4^{+-+}(n-1) = c_4^{++-}(n-1) $ and $ c_4^{(-)+[+]}(n-1) = c_4^{+(-)[+]}(n-1) $. \smallskip\\
\eqref{eq:k4(--)}: Let $ (v_1,\dots,v_n) \in C_4^{(--)[+]}(n) $. Then $ v_{n-1} \vee v_n = v_n \vee v_1 = v_1 \vee v_2 $. If $ v_{n-1} \neq v_1 $, then $ v_{n-1} \vee v_1 = v_{n-1} \vee v_n $ and $ \{v_{n-3},v_{n-2},v_{n-1},v_1\} $ is independent. At the same time $ \{v_{n-1},v_1,v_2\} $ is dependent and $ \{v_{n-2},v_1,v_2,v_3\} $ is independent. Hence $ (v_1,\dots,v_{n-1}) \in C_4^{+(-)[+]}(n-1) $. \\
If $ v_{n-1} = v_1 $, then $ \{v_{n-4},v_{n-3},v_{n-2},v_1\} = \{v_{n-4},v_{n-3},v_{n-2},v_{n-1}\} $ is independent. Also $ \{v_{n-3},v_{n-2},v_1,v_2\} $ is independent since $ \{v_{n-3},v_{n-2},v_{n-1},v_n\} $ is independent and $ v_{n-1} \vee v_n = v_1 \vee v_2 $. Finally, $ \{v_{n-2},v_1,v_2,v_3\} $ is independent by definition of $ C_4^{(--)[+]}(n) $. It follows that $ (v_1,\dots,v_{n-2}) \in C_4(n-2) $.\\
Conversely, let $ (v_1,\dots,v_{n-1}) \in C_4^{+(-)[+]} $. Then $ (v_1,\dots,v_n) \in C_4^{(--)[+]}(n) $ requires that $ \{v_n,v_1,v_2\} $ is dependent and $ v_n \neq v_1 $, which is only true for $ v_n \in v_1 \vee v_2 \backslash \{v_1\} $. We also need $ \{v_{n-1},v_n,v_1\} $ dependent, but that follows from $ \{v_{n-1},v_1,v_2\} $ dependent and $ v_n \vee v_1 = v_1 \vee v_2 $. Another requirement is $ \{v_{n-3},v_{n-2},v_{n-1},v_n\} $ independent. We have $ \{v_{n-3},v_{n-2},v_{n-1},v_1\} $ independent, hence $ v_1 \vee v_2 \not\subset v_{n-3} \vee v_{n-2} \vee v_{n-1} $ and the sole intersection point is not $ v_1 $. The only requirement left is $ \{v_{n-2},v_1,v_2,v_3\} $, but that follows immediately from $ (v_1,\dots,v_n) \in C_4^{(--)[+]}(n) $. Thus there are exactly $ q-1 $ viable choices for $ v_n $.\\
Now let $ (v_1,\dots,v_{n-2}) \in C_4(n-2) $ and $ v_{n-1} := v_1 $. Then $ \{v_{n-1},v_n,v_1\} $ is automatically dependent while $ \{v_n,v_1,v_2\} $ dependent and $ v_n \neq v_1 $ require $ v_n \in v_1 \vee v_2 \backslash \{v_1\} $. We have $ \{v_{n-3},v_{n-2},v_1,v_2\} $ independent, and since $ v_1 \vee v_2 = v_{n-1} \vee v_n $ it follows that $ \{v_{n-3},v_{n-2},v_{n-1},v_n\} $ is independent. Finally, the requirement $ \{v_{n-2},v_1,v_2,v_3\} $ independent is a direct consequence of $ (v_1,\dots,v_{n-2}) \in C_4(n-2) $. Hence there are $ q $ options for $ v_n $ (and only one option for $ v_{n-1} $). This yields
\[
c_4^{(--)[+]}(n) = (q-1) c_4^{+(-)[+]}(n-1) + q c_4(n-2).
\]
\end{proof}
\begin{theorem} \label{thm:c4(n)}
Let $ K $ be a finite field with $ q $ elements and $ n \geq 4 $. Then
\[
c_4(n) = \begin{cases}
q^{3n} - (q^3+q^2+q)q^{2n} + (q^5+q^4+q^3)q^n - q^6, &\text{ if }\gcd(n,4) = 1,\\
q^{3n} - (q^3+q^2+q)q^{2n} + 2(q^4+q^2)q^{\frac{3}{2}n} \\
- (q^5+q^4+q^3)q^n + q^6, &\text{ if }\gcd(n,4) = 2,\\
q^{3n} + 3(q^3+q^2+q)q^{2n} + 2(q^4+q^2)q^{\frac{3}{2}n}\\
+ 3(q^5+q^4+q^3)q^n + q^6, &\text{ if }\gcd(n,4) = 4.
\end{cases}
\]
\end{theorem}
\begin{proof}
We have the following formulas for the other sets in the recursions in Lemma \ref{lem:recursionsk4}:
\begin{align*}
c_4^{++-}(n) &= \begin{cases}
    q^{3n-1} - (q^3-q^2-q-2)q^{2n} - (q^4-q^3+q^2-q)q^{\frac{3}{2}n}\\
    - (2q^5+q^4+q^3-q^2)q^n - q^6, &n \equiv 0 \mod 4,\\
    q^{3n-1} + (q^3-1)q^{2n} - (q^2+1)q^{\frac{3(n+1)}{2}} -(q^5-q^2)q^n + q^6, &n \equiv 1 \mod 4,\\
    q^{3n-1} - (q^2+q+1)q^{2n} - (q^4-q^3+q^2-q)q^{\frac{3}{2}n}\\
    + (q^5+q^4+q^3)q^n - q^6, &n \equiv 2 \mod 4,\\
    q^{3n-1} - (q^2+1)q^{\frac{3(n+1)}{2}} + q^6, &n \equiv 3 \mod 4.
\end{cases}\\
c_4^{+--}(n) &= \begin{cases}
    q^{3n-2} - (q^3-1)q^{2n-1} - (q^3+q)q^{\frac{3}{2}n} + (q^5-q^2)q^{n} + q^6, \: &n \equiv 0 \mod 4,\\
    q^{3n-2} - (q^4+q^2+1)q^{2n-1} + (q^5-q^4+q^3-q^2)q^{\frac{3(n-1)}{2}}\\
    + (q^5+q^3+q)q^{n} - q^6, &n \equiv 1 \mod 4,\\
    q^{3n-2} + (q^4+q^3-q-1)q^{2n-1} - (q^3+q)q^{\frac{3}{2}n}\\
    - (q^5+q^4-q^2-q)q^{n} + q^6, &n \equiv 2 \mod 4,\\
    q^{3n-2} + (q^2+q+1)q^{2n-1} + (q^5-q^4+q^3-q^2)q^{\frac{3(n-1)}{2}}\\
    - (q^5+q^4+q^3)q^{n} - q^6, &n \equiv 3 \mod 4.
\end{cases}\\
c_4^{-+-}(n) &= \begin{cases}
    q^{3n-2} - (2q^3+q^2+q-1)q^{2n-1} + (q^2+1)^2q^{\frac{3}{2}n} \\
    + (q^5-q^4-q^3-2q^2)q^n + q^6, &n \equiv 0 \mod 4,\\
    q^{3n-2} + (2q^3+q^2+q-1)q^{2n-1} + (q^5-q^4\\
    +q^3-q^2)q^{\frac{3(n-1)}{2}} + (q^5-q^4-q^3-2q^2)q^{n} - q^6, &n \equiv 1 \mod 4,\\
    q^{3n-2} - (q^2+q+1)q^{2n-1} + (q^2+1)^2q^{\frac{3}{2}n}\\
    - (q^5+q^4+q^3)q^{n} + q^6, &n \equiv 2 \mod 4,\\
    q^{3n-2} + (q^2+q+1)q^{2n-1} + (q^5-q^4+q^3-q^2)q^{\frac{3(n-1)}{2}} \;\;\;\;\\
    - (q^5+q^4+q^3)q^{n} - q^6, &n \equiv 3 \mod 4.
\end{cases}\\
c_4^{---}(n) &= \begin{cases}
    q^{3n-3} - (q^3+q^2+q)q^{2n-2} + (q^4+q^3+q^2)q^n - q^6, &n \equiv 0 \mod 4,\\
    q^{3n-3} - (q^4+q^2+1)q^{2n-2} + 2(q^4+q^2)q^{\frac{3(n-1)}{2}}\\
    - (q^5+q^3+q)q^{n} + q^6, &n \equiv 1 \mod 4,\\
    q^{3n-3} - (q^5+q+1)q^{2n-2} + (q^5+q^4+1)q^{n} - q^6, &n \equiv 2 \mod 4,\\
    q^{3n-3} + (q^5+q^4+q^3+2q^2+2q+2)q^{2n-2} + 2(q^4\\
    +q^2)q^{\frac{3(n-1)}{2}} + (2q^5+2q^4+2q^3+q^2+q+1)q^{n} + q^6, \quad\,\,\, &n \equiv 3 \mod 4.
\end{cases}\\
c_4^{+(-)[+]}(n) &= \begin{cases}
    q^{3n-2} - (q^3-1)q^{2n-1} - (q^3+q)q^{\frac{3}{2}n} + (q^5-q^2)q^{n} + q^6,\; &n \equiv 0 \mod 4,\\
    q^{3n-2} + (2q^3+q^2+q-1)q^{2n-1} + (q^5-q^4\\
    +q^3-q^2)q^{\frac{3(n-1)}{2}}
    + (q^5-q^4-q^3-2q^2)q^{n} - q^6, &n \equiv 1 \mod 4,\\
    q^{3n-2} - (q^3+q)q^{\frac{3}{2}n} + q^6, &n \equiv 2 \mod 4,\\
    q^{3n-2} - (q^3+q^2+q)q^{2n-1} + (q^5-q^4+q^3-q^2)q^{\frac{3(n-1)}{2}}\\
    + (q^4+q^3+q^2)q^n - q^6, &n \equiv 3 \mod 4.
\end{cases}\\
c_4^{+(-)[-]}(n) &= \begin{cases}
    q^{3n-3} + (q^2+q+1)q^{2n-2} - (q^5+q^4+q^3)q^n - q^6, &n \equiv 0 \mod 4,\\
    q^{3n-3} - (q^4+q^2+1)q^{2n-2} + 2(q^4+q^2)q^{\frac{3(n-1)}{2}}\\
    - (q^5+q^3+q)q^n + q^6, &n \equiv 1 \mod 4,\\
    q^{3n-3} + (q^4+q^3+q^2)q^{2n-2} - (q^3+q^2+q)q^n - q^6,\quad\quad\: &n \equiv 2 \mod 4,\\
    q^{3n-3} - (q^3+q^2+q)q^{2n-2} + 2(q^4+q^2)q^{\frac{3(n-1)}{2}}\\
    - (q^4+q^3+q^2)q^n + q^6, &n \equiv 3 \mod 4.
\end{cases}\\
c_4^{-(-)[+]}(n) &= \begin{cases}
    q^{3n-3} - (q^3+q^2+q)q^{2n-2} + (q^4+q^3+q^2)q^n - q^6, &n \equiv 0 \mod 4,\\
    q^{3n-3} - (q^4-q^3-q+1)q^{2n-2} - (q^5-q^4+2q^3\\
    -q^2+q)q^{\frac{3(n-1)}{2}} - (q^5-q^4-q^2+q)q^n + q^6, &n \equiv 1 \mod 4,\\
    q^{3n-3} + (q^4+q^3+q^2)q^{2n-2} - (q^3+q^2+q)q^n - q^6, &n \equiv 2 \mod 4,\\
    q^{3n-3} - (q^3-1)q^{2n-2} - (q^5-q^4+2q^3-q^2+q)q^{\frac{3(n-1)}{2}} \;\;\\
    + (q^5-q^2)q^n + q^6, &n \equiv 3 \mod 4.
\end{cases}\\
c_4^{(--)[+]}(n) &= \begin{cases}
    q^{3n-4} - (q^4+q^3+q^2)q^{2n-3} + (q^4+2q^2+1)q^{\frac{3}{2}n-1} \\
    - (q^3+q^2+q)q^n + q^6, &n \equiv 0 \mod 4,\\
    q^{3n-4} - (q^4+q^2+1)q^{2n-3} - (q^4-q^3+q^2-q)q^{\frac{3(n-1)}{2}}\;\;\;\\
    + (q^5+q^3+q)q^n - q^6, &n \equiv 1 \mod 4,\\
    q^{3n-4} + (2q^4+2q^3+3q^2+q+1)q^{2n-3} + (q^4+2q^2\\
    +1)q^{\frac{3}{2}n-1} + (q^5+q^4+3q^3+2q^2+2q)q^n + q^6, &n \equiv 2 \mod 4,\\
    q^{3n-4} - (q^3+q^2+q)q^{2n-3} - (q^4-q^3+q^2-q)q^{\frac{3(n-1)}{2}} \\
    + (q^4+q^3+q^2)q^{n} - q^6, &n \equiv 3 \mod 4.
\end{cases}\\
\end{align*}
These formulas fulfill the recursions in Lemma \ref{lem:recursionsk4}. Now lets count $ C_4(4) $. For the first point $ v_1 $, we have $ \lvert \pr^3(K) \rvert $ choices. For the next point $ v_2 $ we can choose every point but $ v_1 $. For the third point $ v_3 $ every point outside of $ v_1 \vee v_2 $ is valid. And for the last point we can choose every point except for the ones in $ v_1 \vee v_2 \vee v_3 $. In conclusion, we have
\[
c_4(4) = (q^3+q^2+q+1)(q^3+q^2+q)(q^3+q^2)q^3.
\]
This agrees with the formula for $ c_4(4) $. We also have
\[
c_4^{s_1s_2s_3}(4) = c_4^{+(-)[s_1]}(4) = c_4^{-(-)[+]}(4) = c_4^{(--)[+]}(4) = 0
\]
for all $ s_1,s_2,s_3 \in \{+,-\} $ except if $ s_1 = s_2 = s_3 = + $. This is because if $ (v_1,v_2,v_3,v_4) $ belongs to any of the corresponding sets, we have $ \{v_1,v_2,v_3,v_4\} $ independent and then no collection of three or four of these points can be dependent. These starting values also agree with our formulas. Finally, we have $ c_4(3) = 0 $, because the condition $ \{v_1,v_2,v_3,v_1 \} $ independent is impossible. Again, this agrees with our formula.
\end{proof}

\begin{cor}\label{cor:k4}
Let $ \lvert K \rvert = q < \infty $. Let $ w \geq 1 $ and $ n = w + 5 $. If $ n $ is odd, then the number of tame $ \SL_4 $-frieze patterns of width $ w $ is
\begin{align*}
f_q(4,n) &= \frac{q^{3n} - (q^3+q^2+q)q^{2n} + (q^5+q^4+q^3)q^n - q^6}{(q^3+q^2+q+1)(q^2+q+1)(q+1)q^6(q-1)^3} \\
&= \frac{(q^{n-1}-1)(q^{n-2}-1)(q^{n-3}-1)}{(q^3+q^2+q+1)(q^2+q+1)(q+1)(q-1)^3}
\end{align*}
\end{cor}
\begin{proof}
This follows from Theorem \ref{thm:numberOfFriezes} and Theorem \ref{thm:c4(n)} and
\[
\lvert \PGL(4,K) \rvert = (q^3+q^2+q+1)(q^2+q+1)(q+1)q^6(q-1)^3.
\]
\end{proof}

\section{The case $ k = 4 $ and $ n $ even}

\subsection{$ \gcd(k,n) = 2 $}

To count tame $ \SL_4 $-frieze patterns with $ \gcd(4,n) = 2 $, we need to count the set $ C_4^*(n) $ for all $ n \geq 6 $ with $ n \equiv 2 \mod 4 $. Let $ (v_1,\dots,v_n) \in C_4(n) $. Choose a lift $ (V_1,\dots,V_n) $ to the underlying vector space $ K^4 $. Define $ d_i := \det(V_i,V_{i+1},V_{i+2},V_{i+3}) $ for $ i = 1, \dots n $ (with the indices being considered modulo $ n $). Recall that $ (v_1,\dots,v_n) \in C_4^*(n) $ if and only if
\[
d_1 d_3 \dots d_{n-1} = d_2 d_4 \dots d_n,
\]
as long as $ n \equiv 2 \mod 4 $ (there is no sign in the equation since $ \frac{k}{\gcd(k,n)} = 2 $ is even). Like in the previous sections, we will derive a system of recursion to solve for $ c_4^*(n) $. For this purpose, we once again need to define some new sets, variations of the sets in Definition \ref{def:k4}:

\begin{defi}
Let $ n \in \N $ be even. For all of the following sets, let $ (V_1,\dots,V_n) $ resp. $ (V_1,\dots,V_{n-1}) $ be an arbitrary lift to $ K^4 $ and let $ d_i := \det(V_i,V_{i+1},V_{i+2},V_{i+3}) $ for $ i \leq n-4 $. Let $ D_1 := d_1 d_3 \dots d_{n-5} $, $ D_2 := d_2 d_4 \dots d_{n-4} $ and $ D_2^\prime := d_2 d_4 \dots d_{n-6} $. Then define
\begin{alignat*}{3}
C_4^\#(n) := \{ &(v_1,\dots,v_n) \in C_4(n)\mid\\
&D_1 \det(V_{n-3},V_{n-2},V_{n-1},V_n) \det(V_{n-1},V_n,V_1,V_2)\\
&= D_2 \det(V_{n-2},V_{n-1},V_n,V_1) \det(V_n,V_1,V_2,V_3) \},\\
C_4^{+(-)[-]\#}(n) := \{ &(v_1,\dots,v_n) \in C_4^{+(-)[-]}(n) \mid\\
&D_1 \det(V_{n-3},V_{n-2},V_{n-1},V_n) \det(V_{n-2},V_{n-1},V_1,V_2)\\
                                                &= -D_2 \det(V_{n-2},V_{n-1},V_n,V_1) \det(V_{n-2},V_1,V_2,V_3)\},\\
C_4^{(--)[+]\#}(n) := \{ &(v_1,\dots,v_n) \in C_4^{(--)[+]}(n) \mid\\
&D_1 \det(V_{n-3},V_{n-2},V_{n-1},V_n) \det(V_{n-3},V_{n-2},V_1,V_2)\\
                                                &= -D_2 \det(V_{n-3},V_{n-2},V_n,V_1) \det(V_{n-2},V_1,V_2,V_3)\},\\
C_4^{-+-\#}(n-1) := \{ &(v_1,\dots,v_{n-1}) \in C_4^{-+-}(n-1) \mid\\
&D_1 \det(V_{n-3},V_{n-2},V_{n-1},V_2)\\
                                                &= D_2 \det(V_{n-2},V_1,V_2,V_3)\},\\
C_4^{---\#}(n-1) := \{ &(v_1,\dots,v_{n-1}) \in C_4^{---}(n-1) \mid\\
&D_1 \det(V_{n-4},V_{n-1},V_1,V_2)\\
                                                &= D_2^\prime \det(V_{n-4},V_{n-2},V_{n-1},V_1) \det(V_{n-4},V_1,V_2,V_3)\}.
\end{alignat*}
\end{defi}
All of these definitions are independent of the choice of lift. Note that $ C_4^\#(n) = C_4^*(n) $ if $ n \equiv 2 \mod 4 $. We can now develop our system of recursions:

\begin{lem}\label{lem:recursionsk42}
Let $ n \geq 6 $ be even. Then
\begin{align}
c_4^\#(n) = &(q-1)^2c_4(n-1) + 3q(q-1)c_4^{++-}(n-1) + 2q^2c_4^{+--}(n-1) \label{eq:+++2}\\
                     &+ (q-1)q^2c_4^{-+-\#}(n-1) +q^3c_4^{---\#}(n-1), \nonumber\\
c_4^{+(-)[-]\#}(n) = &c_4^{++-}(n-1) + qc_4^{-+-\#}(n-1)\label{eq:+(-)[-]2},\\
c_4^{(--)[+]\#}(n) = &c_4^{+(-)[+]}(n-1) + qc_4^\#(n-2)\label{eq:(--)[+]2},\\
c_4^{-+-\#}(n-1) = &(q-1)c_4^\#(n-2) + qc_4^{+(-)[+]}(n-2) + q^2c_4^{+(-)[-]\#}(n-2) \label{eq:-+-2},\\
c_4^{---\#}(n-1) = &(q-1)c_4^{+--}(n-2) + qc_4^{-(-)[+]}(n-2)\label{eq:---2}\\
                          &+ q(q-1)c_4^{+(-)[-]\#}(n-2) + q^2c_4^{(--)[+]\#}(n-2).\nonumber
\end{align}
\end{lem}
\begin{proof}
Whenever we consider a tuple $ (v_1,\dots,v_m) \in (\pr^3(K))^m $ in this proof for $ m = n $ or $ m = n-1 $, let $ (V_1,\dots,V_m) \in (K^4)^m $ be an arbitrary lift. Let $ d_i := \det(V_i,V_{i+1},V_{i+2},V_{i+3}) $, where $ i $ is considered modulo $ n $ (if $ m = n-1 $ we will only consider values of $ i $ up to $ n-4 $). Let $ D_1 := d_1 d_3 \dots d_{n-5} $, $ D_2 := d_2 d_4 \dots d_{n-4} $ and $ D_2^\prime := d_2 d_4 \dots d_{n-6} $. Also let $ E_1 := v_{m-3} \vee v_{m-2} \vee v_{m-1} $, $ E_2 := v_{m-2} \vee v_{m-1} \vee v_1 $, $ E_3 := v_{m-1} \vee v_1 \vee v_2 $ and $ E_4 := v_1 \vee v_2 \vee v_3 $.\smallskip\\
\eqref{eq:+++2}: Let $ (v_1,\dots,v_n) \in C_4^\# $. Then we have
\[
(v_1,\dots,v_{n-1}) \in C_4^{s_1s_2s_3}(n-1)
\]
with $ s_i \in \{+,-\} $ as in the proof of \eqref{eq:k4+++}. Additionally, we have
\begin{align}
  &D_1 \det(V_{n-3},V_{n-2},V_{n-1},V_n) \det(V_{n-1},V_n,V_1,V_2) \label{eq:double_determinant}\\
= &D_2 \det(V_{n-2},V_{n-1},V_n,V_1) \det(V_n,V_1,V_2,V_3).\nonumber
\end{align}
Let $ \lambda, \mu \in K^* $ and consider the maps
\[
\phi_\lambda : K^4 \to K, V \mapsto \det(V_{n-1},V,V_1,V_2) - \lambda \det(V,V_1,V_2,V_3)
\]
and
\[
\psi_\mu:K^4\to K, V \mapsto D_1 \det(V_{n-3},V_{n-2},V_{n-1},V) - \mu D_2 \det(V_{n-2},V_{n-1},V,V_1).
\]
Note that this is not independent of the choice of the lift $ (V_1,\dots,V_n) $, so we need to fix the lift for this part of the proof. Also note that both maps are linear and that all determinants in their definitions are nonzero, except possibly those containing $ V $. If $ V_n^\prime \in K^4 $ is a vector such that $ V_n^\prime \in \Ker(\phi_\lambda) \cap \Ker(\psi_{\lambda^{-1}}) $ for some $ \lambda \in K^* $, then we will have
\[
\det(V_{n-1},V_n^\prime,V_1,V_2) = \lambda \det(V_n^\prime,V_1,V_2,V_3)
\]
and
\[
D_1 \det(V_{n-3},V_{n-2},V_{n-1},V_n^\prime) = \lambda^{-1} D_2 \det(V_{n-2},V_{n-1},V_n^\prime,V_1).
\]
Multiplying these equations yields \eqref{eq:double_determinant} with $ V_n^\prime $ instead of $ V_n $. On the other hand, if $ V_n^\prime \in K^4 $ is a vector that solves \eqref{eq:double_determinant} in this sense and if both sides of that equation are nonzero, we can set
\[
\lambda := \frac{D_2 \det(V_{n-2},V_{n-1},V_n^\prime,V_1)}{D_1 \det(V_{n-3},V_{n-2},V_{n-1},V_n^\prime)}
\]
and
\[
\mu := \frac{\det(V_n^\prime,V_1,V_2,V_3)}{\det(V_{n-1},V_n^\prime,V_1,V_2)}.
\]
Then we have $ V_n^\prime \in \Ker(\phi_\lambda) \cap \Ker(\psi_\mu) $ and $ \lambda \mu = 1 $. Now let $ V_n^\prime $ be a solution of \eqref{eq:double_determinant} with both sides equal to zero. Then we have $ V_n^\prime \in \langle V_{n-3},V_{n-2},V_{n-1}\rangle $ or $ V_n^\prime \in \langle V_{n-1},V_1,V_2\rangle $ and we also have $ V_n^\prime \in \langle V_{n-2},V_{n-1},V_{1}\rangle $ or $ V_n^\prime \in \langle V_1,V_2,V_3\rangle $. Hence the solutions of \eqref{eq:double_determinant} with both sides nonzero are exactly the points in
\[
\bigcup_{\lambda \in K^*} \Ker(\phi_\lambda) \cap \Ker(\psi_{\lambda^{-1}}) \backslash S,
\]
where
\[S := \langle V_{n-3},V_{n-2},V_{n-1}\rangle \cup \langle V_{n-1},V_1,V_2\rangle \cup \langle V_{n-1},V_1,V_2\rangle \cup \langle V_1,V_2,V_3\rangle.
\]
Now let $ M_\lambda := \pr(\Ker(\phi_\lambda)) $, $ N_\mu := \pr(\Ker(\psi_\mu)) $. Then we have $ (v_1,\dots,v_{n-1},v_n^\prime) \in C_4^\#(n) $ if and only if $ v_n^\prime \in M_\lambda \cap N_{\lambda^{-1}} $ for some $ \lambda \in K^* $ and $ v_n^\prime \notin E_i $ for $ i = 1,2,3,4 $. We separate cases: \smallskip\\
Case 1: $ s_1 = s_3 = + $. \\
The image of $ \phi_\lambda $ is at most $ 1 $-dimensional, it then follows from the dimension formula that the kernel is at least $ 3 $-dimensional. Hence the kernel is either $ 3 $-dimensional or $ K^4 $. Since $ s_3 = + $, we have $ \det(V_{n-1},V_1,V_2,V_3) \neq 0 $, but $ \det(V_{n-1},V_{n-1},V_1,V_2) = 0 $. Thus $ \phi_\lambda(V_{n-1}) \neq 0 $ and $ \Ker(\phi_\lambda) \neq K^4 $. Then the kernel must be $ 3 $-dimensional and $ \dim(M_\lambda) = 2 $ (for all $ \lambda \in K^* $). Similarly, it follows from $ s_1 = + $ that $ \psi_\mu(V_1) \neq 0 $ and thus $ \dim(N_\mu) = 2 $ as well. Moreover, $ v_{n-1} \in N_\mu $, but $ v_{n-1} \notin M_\lambda $, hence $ M_\lambda \neq N_\mu $ and $ G_\lambda := M_\lambda \cap N_{\lambda^{-1}} $ is a projective line for all $ \lambda \in K^* $. \\
If $ v \in E_3 $ and $ V $ is a lift of $ v $, the first determinant product in $ \phi_\lambda(V) $ vanishes. Then, $ \phi_\lambda(V) $ will vanish if and only if the second determinant product also vanishes, which happens if and only if $ v \in E_4 $. The same argument also works the other way around. Hence
\[
E_3 \cap M_\lambda = E_4 \cap M_\lambda = E_3 \cap E_4 = v_1 \vee v_2.
\]
Similarly, we have
\[
E_1 \cap N_\mu = E_2 \cap N_\mu = E_1 \cap E_2 = v_{n-2} \vee v_{n-1}.
\]
It follows that
\[
E_i \cap G_\lambda = E_3 \cap E_4 \cap N_{\lambda^{-1}} = (v_1 \vee v_2) \cap N_{\lambda^{-1}}
\]
for $ i = 3,4 $ and
\[
E_i \cap G_\lambda = E_1 \cap E_2 \cap M_\lambda = (v_{n-2} \vee v_{n-1}) \cap M_\lambda
\]
for $ i = 1,2 $. Since $ v_1 \notin N_{\lambda^{-1}} $ and $ v_{n-1} \notin M_\lambda $, these intersections have only one element each. Therefore, there are at most two points on $ G_\lambda $ that also lie on one of the planes $ E_i $. If $ s_2 = - $, then $ E_2 = E_3 $ and $ E_1,E_2,E_3,E_4 $ all intersect $ G_\lambda $ in the same point. If $ s_2 = +, $ then
\[
E_2 \cap E_3 \cap G_\lambda = E_2 \cap N_{\lambda^{-1}} \cap E_3 \cap M_\lambda = E_1 \cap E_2 \cap E_3 \cap E_4 = \emptyset
\]
(the last equality was shown in the proof of \eqref{eq:k4+++}) and the two intersection points are different. Hence there are $ q-1 $ valid choices for $ v_n $ on $ G_\lambda $ if $ s_2 = + $ and $ q $ valid choices if $ s_2 = - $. \\
Now let $ \lambda_1, \lambda_2 \in K^* $ with $ \lambda_1 \neq \lambda_2 $. Then $ \phi_{\lambda_1} - \phi_{\lambda_2} $ is a nonzero multiple of
\[
V \mapsto \det(V,V_1,V_2,V_3).
\]
Hence $ M_{\lambda_1} \cap M_{\lambda_2} \subseteq E_4 $. In particular, the union
\[
\bigcup_{\lambda \in K^*} (G_\lambda \backslash (E_1 \cup E_2 \cup E_3 \cup E_4))
\]
is disjoint and has exactly $ (q-1)^2 $ elements if $ s_2 = + $ and $ q(q-1) $ elements if $ s_2 = - $. The number of elements of $ C_4^\#(n) $ falling in Case 1 is therefore
\begin{align*}
 &(q-1)^2c_4(n-1) + q(q-1)c_4^{+-+}(n-1) \\
=&(q-1)^2c_4(n-1) + q(q-1)c_4^{++-}(n-1)
\end{align*}
using symmetry. \smallskip\\
Case 2: $ s_1 = +, s_3 = - $. \\
Then $ E_3 = E_4 $ and $ E_4 \subseteq M_\lambda $, since both determinant products in $ \phi_\lambda(V) $ will vanish if $ v \in E_3 = E_4 $. If $ E_4 = M_\lambda $, then there are no valid choices for $ v_n $ with this $ \lambda $. Thus we only need to consider values of $ \lambda $ such that $ \phi_\lambda = 0 $. We have $ \phi_\lambda = 0 $ if and only if $ \phi_\lambda(V) = 0 $ for some $ V \notin \langle V_1,V_2,V_3 \rangle $. There is exactly one valid value $ \lambda^* $, which is given by
\[
\lambda^* := \frac{\det(V_{n-1},V,V_1,V_2)}{\det(V,V_1,V_2,V_3)}
\]
for any such a vector $ V $ and is independent of the choice of $ V $. Then $ M_{\lambda^*} = \pr^3(K) $ and thus $ M_{\lambda^*} \cap N_{(\lambda^*)^{-1}} = N_{(\lambda^*)^{-1}} $ and this is a projective hyperplane since $ s_1 = + $. We need to count $ N_{(\lambda^*)^{-1}} \backslash (E_1 \cup E_2 \cup E_3 \cup E_4) $. We still have
\[
E_i \cap N_{(\lambda^*)^{-1}} = E_1 \cap E_2 = v_{n-2} \vee v_{n-1}
\]
for $ i = 1,2 $. If $ s_2 = - $, then $ E_2 = E_3 = E_4 $ and all four planes intersect $ N_{(\lambda^*)^{-1}} $ in the same projective line. Hence there are $ q^2 $ suitable choices for $ v_n $. If $ s_2 = + $, then $ v_{n-2} \notin E_3 $. But we have $ v_{n-2} \in N_{(\lambda^*)^{-1}} $ and therefore $ E_3 \neq N_{(\lambda^*)^{-1}} $. Hence $ E_3 \cap N_{(\lambda^*)^{-1}} $ is a projective line, and it is not the same line as $ E_2 \cap N_{(\lambda^*)^{-1}} $. On the other hand $ E_3 \cap N_{(\lambda^*)^{-1}} = E_4 \cap N_{(\lambda^*)^{-1}}$ since $ E_3 = E_4 $. Accordingly, we need to exclude two different lines on $ N_{(\lambda^*)^{-1}} $ and $ (q-1)q $ options for $ v_n $ are left. The number of elements of $ C_4^\#(n) $ that fall in case 2 is thus
\[
(q-1)qc_4^{++-}(n-1) + q^2c_4^{+--}(n-1).
\]
Case 3: $ s_1 = - , s_3 = + $.\\
This is symmetrical to case 2; instead of a unique choice for $ \lambda $ we have a unique choice for $ \mu $, but the calculations are essentially the same. We again have $ (q-1)q $ or $ q^2 $ choices for $ v_n $, depending on $ s_2 $. The number of elements in case 3 is then
\begin{align*}
 &(q-1)qc_4^{-++}(n-1) + q^2c_4^{--+}(n-1) \\
=&(q-1)qc_4^{++-}(n-1) + q^2c_4^{+--}(n-1),
\end{align*}
again using symmetry. \smallskip\\
Case 4: $ s_1 = s_3 = - $, $ s_2 = + $. \\
Then there are unique values $ \lambda^* $ and $ \mu^* $ such that $ M_{\lambda^*} \neq E_4 $ and $ N_{\mu^*} \neq E_1 $. Since both are necessary conditions for finding a solution, it follows that we can only find a valid $ v_n $ if $ \lambda^* \mu^* = 1 $. In other words, there has to be a $ \lambda \in K^* $ such that $ \phi_\lambda = \psi_{\lambda^{-1}} = 0 $. Since $ E_3 = E_4 \subseteq M_\lambda $ and $ v_{n-2} \notin E_3 $, the condition $ \phi_\lambda = 0 $ is equivalent to $ \phi_\lambda(V_{n-2}) = 0 $. Hence we have
\begin{equation}\label{eq:phi0}
\det(V_{n-2},V_{n-1},V_1,V_2) = -\lambda \det(V_{n-2},V_1,V_2,V_3)
\end{equation}
for a specific $ \lambda $. Similarly, $ \psi_{\lambda^{-1}} = 0 $ is equivalent to $ \psi_{\lambda^{-1}}(V_2) = 0 $ and thus we must have
\begin{equation}\label{eq:psi0}
D_1 \det(V_{n-3},V_{n-2},V_{n-1},V_2) = -\lambda^{-1} D_2 \det(V_{n-2},V_{n-1},V_1,V_2)
\end{equation}
with the same $ \lambda $. Multiplying these equations and canceling out $ \det(V_{n-2},V_{n-1},V_1,V_2) $ yields
\begin{equation}\label{eq:+++to-+-}
D_1 \det(V_{n-3},V_{n-2},V_{n-1},V_2) = D_2 \det(V_{n-2},V_1,V_2,V_3).
\end{equation}
This means that $ (v_1,\dots,v_{n-1}) \in C_4^{-+-\#}(n-1) $. On the other hand, if \eqref{eq:+++to-+-} holds, then rearranging for $ \det(V_{n-2},V_1,V_2,V_3) $, substituting into \eqref{eq:phi0} and rearranging again yields \eqref{eq:psi0}.
Hence any element of $ C_4^{-+-\#}(n-1) $ will fulfill the necessary condition $ \phi_\lambda = \psi_{\lambda^{-1}} = 0 $ for some $ \lambda \in K^* $. Then $ M_\lambda = N_{\lambda^{-1}} = \pr^3(K) $ and the determinant condition becomes trivial. Hence we have the same number of choices for $ v_n $ as in the case $ s_1=s_3=- $, $ s_2 = + $ in the proof of \eqref{eq:k4+++}, namely $ (q-1)q^2 $. \smallskip\\
Case 5: $ s_1 = s_2 = s_3 = - $.\\
This is similar to the last case, except that now $ \phi_\lambda = \psi_{\lambda^{-1}} = 0 $ is equivalent to $ \phi_\lambda(V_{n-4}) = \psi_{\lambda^{-1}}(V_{n-4}) = 0 $, since $ v_{n-4} \notin E_1 = E_2 = E_3 = E_4 $. Multiplying the resulting equations and canceling out $ \det(V_{n-4},V_{n-3},V_{n-2},V_{n-1}) $ now results in
\begin{equation}\label{eq:+++to---}
D_1 \det(V_{n-4},V_{n-1},V_1,V_2) = D_2^\prime \det(V_{n-4},V_{n-2},V_{n-1},V_1) \det(V_{n-4},V_1,V_2,V_3).
\end{equation}
Hence $ (v_1,\dots,v_{n-1}) \in C_4^{---\#}(n-1) $. Just like before, \eqref{eq:+++to---} also implies $ \phi_\lambda = \psi_{\lambda^{-1}} = 0 $ for some $ \lambda \in K^* $, and then we have $ q^3 $ choices for $ v_n $ like in the case $ s_1 = s_2 = s_3 = - $ in the proof of \eqref{eq:k4+++}. \\
Putting together all cases, the desired recursion follows. \medskip\\
\eqref{eq:+(-)[-]2}: Let $ (v_1,\dots,v_n) \in C_4^{+(-)[-]\#}(n) $. Then
\[
(v_1,\dots,v_{n-1}) \in C_4^{s+-}(n-1).
\]
with $ s \in \{+,-\} $. Additionally, consider the map
\begin{align*}
\phi: K^4 \to K; V \mapsto &D_1 \det(V_{n-3},V_{n-2},V_{n-1},V) \det(V_{n-2},V_{n-1},V_1,V_2)\\
                            + &D_2 \det(V_{n-2},V_{n-1},V,V_1) \det(V_{n-2},V_1,V_2,V_3).
\end{align*}
Let $ H := \pr(\Ker \phi) $. Then the determinant condition can be rephrased as $ v_n \in H $.\\
Let $ s = + $, then $ \{v_{n-3},v_{n-2},v_{n-1},v_1\} $ is independent and $ \phi(V_1) \neq 0 $. Hence $ \dim H = 2 $. Moreover, we have $ v_1 \notin H $, but $ v_1 \in v_1 \vee v_2 $. Therefore, $ \{P\} := H \cap (v_1 \vee v_2) $ is a single projective point and since $ \{v_n,v_1,v_2\} $ has to be dependent, $ P $ is the only possible value for $ v_n $. $ P $ only dependents on $ (v_1,\dots,v_{n-1}) $, so there is at most one suitable choice of $ v_n $ for any $ (v_1,\dots,v_{n-1}) \in C_4^{++-}(n-1) $. On the other hand, we have $ E_1 \cap H = E_1 \cap E_2 = E_2 \cap H $ and thus
\begin{align*}
E_1 \cap \{P\} &= E_1 \cap H \cap (v_1 \vee v_2) = E_1 \cap E_2 \cap (v_1 \vee v_2) = E_1 \cap \{v_1\} = \emptyset\\
               &= E_2 \cap \{P\}.
\end{align*}
This shows that $ P $ will always work and that there is exactly one choice of $ v_n $ for all $ (v_1,\dots,v_{n-1}) \in C_4^{++-}(n-1) $. \\
Now let $ s = - $. Then $ E_1 = E_2 \subseteq H $; and since equality would imply $ v_n \in E_1 $ we must have $ H = \pr^3(K) $. This condition is equivalent to $ \phi(V_2) = 0 $, since $ \{v_{n-2},v_{n-1},v_1,v_2\} $ is independent. Hence the condition becomes
\[
D_1 \det(V_{n-3},V_{n-2},V_{n-1},V_2) = D_2 \det(V_{n-2},V_1,V_2,V_3)
\]
after canceling out $ \det(V_{n-2},V_{n-1},V_1,V_2) $. This is in turn equivalent to $ (v_1\dots,v_{n-1}) \in C_4^{-+-\#}(n-1) $. Then the determinant condition becomes trivial and we have $ q $ choices for $ v_n $ as in the proof of \eqref{eq:k4+(-)[-]}, Case $ s = - $. The claimed recursion follows.\medskip\\
\eqref{eq:(--)[+]2}: Let $ (v_1,\dots,v_n) \in C_4^{(--)[+]\#}(n) $. Then we have $ (v_1,\dots,v_{n-1}) \in C_4^{+(-)[+]}(n-1) $ if $ v_{n-1} \neq v_1 $ and $ (v_1,\dots,v_{n-2}) \in C_4(n-2) $ if $ v_{n-1} = v_1 $. We also have the determinant condition
\begin{align}
&D_1 \det(V_{n-3},V_{n-2},V_{n-1},V_n) \det(V_{n-3},V_{n-2},V_1,V_2)\nonumber\\
= &-D_2 \det(V_{n-3},V_{n-2},V_n,V_1) \det(V_{n-2},V_1,V_2,V_3),\label{eq:(--)condition}
\end{align}
which we can rewrite as $ v_n \in H $ if we set
\begin{align*}
\phi: K^4 \to K; V \mapsto &D_1 \det(V_{n-3},V_{n-2},V_{n-1},V) \det(V_{n-3},V_{n-2},V_1,V_2)\\
                         + &D_2 \det(V_{n-3},V_{n-2},V,V_1) \det(V_{n-2},V_1,V_2,V_3)
\end{align*}
and $ H := \pr(\Ker \phi) $. \\
Let $ v_{n-1} \neq v_1 $. In this case $ \{v_{n-3},v_{n-2},v_{n-1},v_1\} $ is independent and $ \phi(V_1) \neq 0 $. In particular, $ \dim H = 2 $. Also, $ v_1 \notin H $ implies that $ \{P\} := H \cap (v_1 \vee v_2) $ is a single projective point, and $ P $ is the only point that can fulfill all conditions on $ v_n $. On the other hand, we have
\[
E_1 \cap \{P\} = H \cap E_1 \cap (v_1 \vee v_2) = H \cap \{v_{n-1}\} = \emptyset.
\]
Moreover, we have $ \{v_{n-1},v_1,v_2\} $ dependent and $ P \in v_1 \vee v_2 $, hence $ \{v_{n-1},P,v_1\} $ is dependent; and $ v_1 \notin H $ implies $ P \neq v_1 $. Therefore, $ (v_1,\dots,v_{n-1},P) \in C_4^{(--)[+]\#}(n) $ for any $ (v_1,\dots,v_{n-1}) \in C_4^{+(-)[+]}(n-1) $. \\
Now let $ v_{n-1} = v_1 $. W.l.o.g. we can choose the same lift $ V_{n-1} = V_1 $ for both $ v_{n-1} $ and $ v_1 $. Then we have $\det(V_{n-3},V_{n-2},V_{n-1},V_n) = - \det(V_{n-3},V_{n-2},V_n,V_1) $. After substituting this into \eqref{eq:(--)condition}, canceling out and using $ V_{n-1} = V_1 $ again, we get the equivalent condition
\[
D_1 \det(V_{n-3},V_{n-2},V_1,V_2) = D_2^\prime \det(V_{n-4},V_{n-3},V_{n-2},V_1) \det(V_{n-2},V_1,V_2,V_3).
\]
This is exactly the condition for $ (v_1,\dots,v_{n-2}) \in C_4^\#(n-2) $. Moreover, we have $ q $ viable choices for $ v_n $ like in the proof of \eqref{eq:k4(--)} (and $ v_{n-1} $ is uniquely determined by $ v_{n-1} = v_1 $). Hence we get the desired recursion.\medskip\\
\eqref{eq:-+-2}: Let $ (v_1,\dots,v_{n-1}) \in C_4^{-+-\#}(n-1) $. Like in the proof of \eqref{eq:k4-+-} we have
\[
(v_1,\dots,v_{n-2}) \in C_4(n-2) \cup C_4^{(-)+[s]}(n-2)
\]
with $ s \in \{+,-\} $. The first case happens if $ \{v_{n-3},v_{n-2},v_1\} $ is independent. The latter two cases happen if this set is dependent. Consider the map
\begin{align*}
\phi: K^4 &\to K\\
V &\mapsto D_1 \det(V_{n-3},V_{n-2},V,V_2) - D_2^\prime \det(V_{n-2},V_1,V_2,V_3) \det(V_{n-4},V_{n-3},V_{n-2},V).
\end{align*}
Let $ H:= \pr(\Ker\phi) $. Then the determinant condition \eqref{eq:+++to-+-} required by $ C_4^{-+-\#}(n-1) $ is equivalent to $ v_{n-1} \in H $.\smallskip\\
Case 1: $ \{v_{n-3},v_{n-2},v_1\} $ is independent.\\
Then the suitable choices for $ v_{n-1} $ are exactly the elements of $ H \cap E_2 \cap E_4 \backslash (E_1 \cup E_3) $. Note that $ \phi(V_{n-3}) = \phi(V_{n-2}) = 0 $. If $ \phi(V_1) \neq 0 $, then $ E_2 \cap H = v_{n-3} \vee v_{n-2} \subseteq E_1 $. But then
\[
H \cap E_2 \cap E_4 \backslash (E_1 \cup E_3) = \emptyset
\]
and there are no values of $ v_{n-1} $ that fulfill all of our conditions. Hence we must have $ \phi(V_1) = 0 $ or equivalently
\[
D_1 \det(V_{n-3},V_{n-2},V_1,V_2) = D_2^\prime \det(V_{n-4},V_{n-3},V_{n-2},V_1) \det(V_{n-2},V_1,V_2,V_3).
\]
This is equivalent to $ (v_1,\dots,v_{n-2}) \in C_4^\#(n-2) $. On the other hand, if we have $ \phi(V_1) = 0 $, then $ E_2 \subseteq H $ and consequently
\[
H \cap E_2 \cap E_4 \backslash (E_1 \cup E_3) = E_2 \cap E_4 \backslash (E_1 \cup E_3).
\]
In the proof of \eqref{eq:k4-+-} we showed that this set has $ q-1 $ elements. \smallskip\\
Case 2: $ \{v_{n-3},v_{n-2},v_1\} $ dependent and $ s = + $.\\
$ s = + $ implies $ \{v_{n-4},v_{n-3},v_{n-2},v_2\} $ independent. Hence, $ \phi(V_2) \neq 0 $ and $ \dim H = 2 $. Moreover, since $ v_2 \notin H $, $ G := H \cap E_4 $ is a projective line and only elements of this line can be viable choices for $ v_{n-1} $. More precisely, the viable choices are exactly the elements of $ G \backslash (E_1 \cup E_3) $ (we do not need to consider $ E_2 $, since $ \{v_{n-3},v_{n-2},v_1\} $ dependent implies $ \{v_{n-3},v_{n-2},v_{n-1},v_1\} $ dependent regardless of $ v_{n-1} $).\\
We have $ E_1 \cap H = E_1 \cap (v_{n-3} \vee v_{n-2} \vee v_2) = v_{n-3} \vee v_{n-2} $ and thus
\[
E_1 \cap G = E_1 \cap H \cap E_4 = (v_{n-3} \vee v_{n-2}) \cap E_4 = \{v_1\}.
\]
Here we use $ v_{n-2} \notin E_4 $ and $ v_1 \in v_{n-3} \vee v_{n-2} $. We also have
\[
E_3 \cap G = E_3 \cap E_4 \cap H = (v_1 \vee v_2) \cap H = \{v_1\}.
\]
It follows that $ \lvert G \backslash (E_1 \cup E_3) \rvert = q $. \smallskip\\
Case 3: $ \{v_{n-3},v_{n-2},v_1\} $ dependent and $ s = - $.\\
$ s = - $ implies $ v_2 \in E_1 $ and thus $ E_1 = v_{n-3} \vee v_{n-2} \vee v_2 $ ($ \{v_{n-3},v_{n-2},v_2\} $ is independent since $ v_{n-3} \vee v_{n-2} = v_{n-2} \vee v_1 $ and $ \{v_{n-2},v_1,v_2,v_3\} $ is independent by definition of $ C_4^{(-)+[-]}(n-2) $). Thus we have $ E_1 \subseteq H $. If we had equality, then there would be no suitable choice for $ v_{n-1} $. Hence, only the the case $ H = \pr^3(K) $ --or equivalently $ \phi = 0 $-- needs to be considered. We have $ v_3 \notin  E_1 = v_{n-2} \vee v_1 \vee v_2 $, therefore $ \phi = 0 $ is equivalent to $ \phi(V_3) = 0 $, which is in turn equivalent to
\[
D_1 \det(V_{n-3},V_{n-2},V_2,V_3) = - D_2^\prime \det(V_{n-2},V_1,V_2,V_3) \det(V_{n-4},V_{n-3},V_{n-2},V_3).
\]
After applying the bijection
\[
C_4^{(-)+[-]}(n-2) \to C_4^{+(-)[-]}(n-2); (v_1,\dots,v_{n-2}) \mapsto (v_{n-2},\dots,v_1)
\]
the equation becomes
\[
D_1 \det(V_{n-4},V_{n-3},V_1,V_2) = - D_2^\prime \det(V_{n-4},V_{n-3},V_{n-2},V_1) \det(V_{n-4},V_1,V_2,V_3).
\]
This is just the determinant condition for $ C_4^{+(-)[-]\#}(n-2) $. Hence the elements in the image of $ C_4^{+(-)[-]\#}(n-2) $ under the bijection are the only elements in $ C_4^{(-)+[-]}(n-2) $ that can be extended to elements of $ C_4^{-+-\#}(n-1) $. On the other hand, $ (v_{n-2},\dots,v_1) \in C_4^{+(-)[-]\#}(n-1) $ implies $ \phi = 0 $ and thus the determinant condition on $ C_4^{-+-\#}(n-1) $ becomes trivial. Then, we have $ q^2 $ choices for $ v_{n-1} $ just like in the proof of \eqref{eq:k4-+-}. Combining all cases now yields \eqref{eq:-+-2}. \medskip\\
\eqref{eq:---2}: Let $ (v_1,\dots,v_{n-1}) \in C_4^{---\#}(n-1) $. Then
\[
(v_1,\dots,v_{n-2}) \in C_4^{+--}(n-2) \cup C_4^{+(-)[-]}(n-2) \cup C_4^{(-)-[+]}(n-2) \cup C_4^{(--)[+]}(n-2).
\]
Let $ M_1 := \{v_{n-3},v_{n-2},v_1\} $ and $ M_2 := \{v_{n-2},v_1,v_2\} $. Also, consider the map
\[
\phi: K^4 \to K; V \mapsto D_1 \det(V_{n-4},V, V_1,V_2) - D_2^\prime \det(V_{n-4},V_{n-2},V,V_1) \det(V_{n-4},V_1,V_2,V_3)
\]
and set $ H := \pr(\Ker \phi) $. Then for $ (v_1,\dots,v_{n-1}) \in C_4^{---}(n-1) $ we have $ (v_1,\dots,v_{n-1}) \in C_4^{---\#}(n-1) $ if and only if $ v_{n-1} \in H $. Additionally, we require $ v_{n-1} \in E_4 $, but $ v_{n-1} $ must not lie in $ E_1 $, $ v_{n-2} \vee v_1 $ or $ v_1 \vee v_2 $.\\
If $ M_2 $ is independent, then $ \{v_{n-4},v_{n-2},v_1,v_2\} $ is independent since $ v_{n-4} \notin E_4 = v_{n-2} \vee v_1 \vee v_2 $. It follows that $ \phi(V_2) \neq 0 $, hence $ \dim H = 2 $.\\
If $ M_2 $ is dependent, then $ v_{n-4} \vee v_{n-2} \vee v_1 = v_{n-4} \vee v_1 \vee v_2 \subseteq H $. If we had equality, then we would have $ v_{n-1} \in (v_{n-4} \vee v_1 \vee v_2) \cap E_4 = v_1 \vee v_2 $, which is not allowed. Hence, $ H = \pr^3(K) $.\smallskip\\
Case 1: $ M_2 $ independent.\\
Then, $ \dim H = 2 $ and $ v_2 \notin H $, but $ v_2 \in E_4 $. Thus $ H \neq E_4 $ and $ G := H \cap E_4 $ is a projective line. Also,
\[
E_1 \cap G = E_1 \cap E_4 \cap H = (v_{n-3} \vee v_{n-2}) \cap H. 
\]
Since $ v_{n-2} \notin H $, this intersection is a single point. If $ M_1 $ is dependent, the intersection is exactly $ \{v_1\} $, otherwise it is a different point. We also have
\[
(v_{n-2} \vee v_1) \cap G = (v_{n-2} \vee v_1) \cap H = \{v_1\}
\]
and
\[
(v_1 \vee v_2) \cap G = (v_1 \vee v_2) \cap H = \{v_1\}.
\]
Hence we have to exclude $ 1 $ or $ 2 $ points on $ G $, depending on $ M_1 $, to fulfill all conditions and we get $ q-1 $ resp. $ q $ values for $ v_{n-1} $. Accordingly, the number of elements of $ C_4^{---\#}(n-1) $ falling into Case 1 is
\[
(q-1)c_4^{+--}(n-2) + qc_4^{-(-)[+]}(n-2),
\]
using the symmetry $ c_4^{(-)-[+]}(n-2) = c_4^{-(-)[+]}(n-2) $. \smallskip\\
Case 2: $ M_2 $ dependent, $ M_1 $ independent. Then $ (v_1,\dots,v_{n-2}) $ has to be such that $ \phi = 0 $. Since $ v_{n-4} \vee v_{n-2} \vee v_1 \subset H $ and $ v_{n-3} \notin v_{n-4} \vee v_{n-2} \vee v_1 $, $ \phi = 0 $ is equivalent to $ \phi(V_{n-3}) = 0 $, which is in turn equivalent to
\[
D_1 \det(V_{n-4},V_{n-3},V_1,V_2) = - D_2^\prime \det(V_{n-4},V_{n-3},V_{n-2},V_1) \det(V_{n-4},V_1,V_2,V_3).
\]
This is exactly the condition for $ (v_1,\dots,v_{n-2}) \in C_4^{+(-)[-]\#}(n-2) $. The determinant condition then becomes trivial and thus we have the same number of viable choices for $ v_{n-1} $ as in the proof of \eqref{eq:k4---} Case 2, namely $ (q-1)q $. \smallskip\\
Case 3: $ M_1 $ and $ M_2 $ are both dependent.\\
This is similar to the previous case, except we use $ V_{n-5} $ to test $ \phi = 0 $ (since $ v_{n-5} \notin E_1 = v_{n-4} \vee v_{n-2} \vee v_1 \subset H $). The resulting condition is then
\[
D_1 \det(V_{n-5},V_{n-4},V_1,V_2) = - D_2^\prime \det(V_{n-5},V_{n-4},V_{n-2},V_1) \det(V_{n-4},V_1,V_2,V_3),
\]
or equivalently $ (v_1,\dots,v_{n-2}) \in C_4^{(--)[+]\#}(n-2) $. Like in the proof of \eqref{eq:k4---} Case 4, we have $ q^2 $ choices for $ v_{n-1} $.\\
The claimed recursion now follows from the three cases.
\end{proof}

Now we can calculate $ c_4^\#(n) $.

\begin{theorem}\label{thm:c4(n)gcd2}
Let $ n \geq 6 $ be even, but not divisible by $ 4 $. Then $ c_4^*(n) = c_4^\#(n) $ is given by
\begin{align*}
(q-1)c_4^*(n) = q^{3n} &+ (q^5-q^4-q^3-4q^2-2q-2)q^{2n}\\
&- (2q^5-2q^4-2q^3-6q^2-4q-4)q^{\frac{3}{2}n+1}\\
&+ (q^5-q^4-q^3-4q^2-2q-2)q^{n+2} + q^6.
\end{align*}
\end{theorem}

\begin{proof}
If $ n \geq 4 $ is divisible by $ 4 $, then
\begin{align*}
(q-1)c_4^\#(n) = q^{3n} &+ (q^5-q^4+3q^3+2q-2)q^{2n} + (2q^5-2q^4+2q^3-6q^2-4)q^{\frac{3}{2}n+1}\\
&+ (q^5-q^4+3q^3+2q-2)q^{n+2} + q^6.
\end{align*}
Together with the formula given in the theorem we now have an explicit formula for $ c_4^\#(n) $ for all even $ n \geq 4 $. An explicit formula for $ c_4^{+(-)[-]\#}(n) $ is given by
\begin{align*}
(q-1)c_4^{+(-)[-]\#}(n) = q^{3n-3} &+ (q^5-q^4-2q^2-1)q^{2n-2}\\
&- (q^6-2q^5+q^4-3q^3+2q^2-q+2)q^{\frac{3}{2}n-1}\\
&- (q^5+q^3-2q^2-q-2)q^{n+1} - q^6
\end{align*}
if $ n $ is divisible by $ 4 $ and
\begin{align*}
(q-1)c_4^{+(-)[-]\#}(n) =q^{3n-3} &+ (q^5+q^3-2q^2-q-2)q^{2n-2}\\
&+ (q^6-2q^5+q^4-3q^3+2q^2-q+2)q^{\frac{3}{2}n-1}\\
&- (q^5-q^4-2q^2-1)q^{n+1} - q^6
\end{align*}
otherwise. Note that the determinant condition in the definition of $ C_4^\#(n) $ is trivial in the case $ n = 4 $. Therefore, we have $ c^\#_4(4) = c_4(4) $. With \eqref{eq:(--)[+]2} and the explicit formula for $ c_4^{+(-)[+]}(n) $ from the previous section we can easily calculate an explicit formula for $ c_4^{(--)[+]\#}(n) $ for $ n \geq 6 $ even. For $ n = 4 $ we also have $ C_4^{+(-)[-]\#}(4) \subseteq C_4^{+(-)[-]}(4) = \emptyset $ and  $ C_4^{(--)[+]\#}(4) \subseteq C_4^{(--)[+]}(4) = \emptyset $ which implies $ c_4^{(--)[+]\#}(4) = c_4^{+(-)[-]\#}(4) = 0 $. Then, \eqref{eq:-+-2} and \eqref{eq:---2} can be used to derive explicit formulas for $ c_4^{-+-\#}(n) $ and $ c_4^{---\#}(n) $ for all odd $ n \geq 5 $. It can then be verified that these formulas and the formulas from the previous section fulfill the recursions in Lemma \ref{lem:recursionsk42}. Hence the given formulas must be the correct solutions.
\end{proof}

\begin{cor}\label{cor:k4gcd2}
Let $ n \geq 6 $ with $ \gcd(4,n) = 2 $ and $ \lvert K \rvert = q $. Then the number of tame $ \SL_4 $-frieze patterns over $ K $ with width $ n - 5 $ is
\begin{align*}
f_q(4,n) = \frac{q^{3n} + f(q)q^{2n} - g(q)q^{\frac{3}{2}n+1} + f(q)q^{n+2} + q^6}{(q^3+q^2+q+1)(q^2+q+1)(q+1)q^6(q-1)^3},
\end{align*}
where
\begin{align*}
f(q) &:= q^5-q^4-q^3-4q^2-2q-2,\\
g(q) &:= 2q^5-2q^4-2q^3-6q^2-4q-4.
\end{align*}
\end{cor}
\begin{proof}
This is the combination of Theorem \ref{thm:c4(n)gcd2} and Theorem \ref{thm:numberOfFriezes} together with
\[
\lvert \PGL(4,K) \rvert = (q^3+q^2+q+1)(q^2+q+1)(q+1)q^6(q-1)^3.
\]
\end{proof}

\subsection{$ \gcd(4,n) = 4 $}
After the previous subsection, it only remains to count tame $ \SL_4 $-frieze patterns with period $ n $ where $ 4 \mid n $. Like before, we need to count the corresponding set $ C_4^*(n) $. Choose a lift $ (V_1,\dots,V_n) $ for an element $ (v_1,\dots,v_n) \in C_4(n) $ and set $ d_i := \det(V_i,V_{i+1},V_{i+2},V_{i+3}) $ (with the indices considered modulo $ n $). Then $ (v_1,\dots,v_n) \in C_4^*(n) $ is equivalent to
\[
d_1 d_5 \dots d_{n-3} = -d_2 d_6 \dots d_{n-2} = d_3 d_7 \dots d_{n-1} = -d_4 d_8 \dots d_n.
\]
Here we have negative signs since $ k $ is even and $ \frac{k}{\gcd(k,n)} = 1 $ is odd. By setting $ D_i := d_i d_{i+4} \dots d_{n-8+i} $ for $ i \in \{1,2,3,4\} $, we can also abbreviate this to
\[
D_1 d_{n-3} = -D_2 d_{n-2} = D_3 d_{n-1} = -D_4 d_n.
\]
Since $ C_4^*(n) $ has $ 3 $ determinant equations, our recursions will lead us to all kind of different sets with $ 3 $, $ 2 $, or $ 1 $ determinant equation. Moreover, some of the resulting sets will have a different sign in their equations. The following definition covers all the variation of $ C_4(n) $ that need to be considered:
\begin{defi}
Let $ n \in 4 \N $. Let $ M \subset \{1,2,3,4\} $ be nonempty. Let $ M_1,M_2 \subset \{1,2,3,4\} $ be nonempty disjoint subsets. For $ (v_1,\dots,v_n) \in C_4(n) $ let $ (V_1,\dots,V_n) $ be an arbitrary lift and define $ d_i $ and $ D_i $ like above. Define 
\begin{enumerate}
    \item[(i)] $ C_4^{M+}(n) := \{(v_1,\dots,v_n) \in C_4(n) \mid D_i d_{n-4+i} = D_j d_{n-4+j} \text{ for } i,j \in M\} $,
    \item[(ii)] $ C_4^{M-}(n) := \{(v_1,\dots,v_n) \in C_4(n) \mid D_i d_{n-4+i} = (-1)^{j-i}D_j d_{n-4+j} \text{ for } i,j \in M\} $,
    \item[(iii)] $ C_4^{M_1|M_2\pm}(n) := C_4^{M_1\pm}(n) \cap C_4^{M_2\pm}(n) $.
\end{enumerate}
Denote the cardinalities of these sets with $ c_4^{M\pm}(n) $ resp. $ c_4^{M_1|M_2\pm}(n) $. If $ M = \{1,3\} $ or $ M = \{2,4\} $ we omit the sign since it does not make a difference then. We will usually abbreviate this notation by omitting the brackets and commas in $ M $ resp. in $ M_1 $ and $ M_2 $. For instance, we write $ C_4^{12+}(n) $ instead of $ c_4^{\{1,2\}+}(n) $.
\end{defi}

\begin{rem}
These definitions are independent of the choice of $ (V_1,\dots,V_n) $. We clearly have $ C_4^{1234-}(n) = C_4^*(n) $. Also note that $ c_4^{12\pm}(n) = c_4^{23\pm}(n) = c_4^{34\pm}(n) = c_4^{14\pm}(n) $ via the map
\[
(v_1,\dots,v_n) \to (v_2,\dots,v_n,v_1).
\]
Similarly, we have $ c_4^{123\pm}(n) = c_4^{234\pm}(n) = c_4^{134\pm}(n) = c_4^{124\pm}(n) $, $ c_4^{12|34\pm}(n) = c_4^{14|23\pm}(n) $ and $ c_4^{13}(n) = c_4^{24}(n) $ via the same map.
\end{rem}

In addition, we will need versions of the other sets from Definition \ref{def:k4} with variations of the conditions $ D_i d_{n-4+1} = D_j d_{n-4+j} $. For now, we restrict ourselves to sets with only one such condition. Our first goal will be to derive a set of recursions for $ c_4^{23\pm}(n) $. The next definition will give the new sets we have to define for that purpose. After we have found a set of recursions for $ c_4^{23\pm}(n) $, we will focus on $ c_4^{13}(n), $ then $ c_4^{234\pm}(n) $, then $ c_4^{12|34\pm}(n) $ and finally $ c_4^{1234\pm}(n) $.

\begin{defi}
Let $ n \in 4\N $. For any tuple $ (v_1,\dots,v_m) \in (\pr^3(K))^m $ with $ m \in \{n,n-1,n-2,n-3\} $ involved in the definition of a set assume that $ (V_1,\dots,V_m) $ is an arbitrary lift to $ (K^4)^m $. We set $ d_i := \det(V_i,V_{i+1},V_{i+2},V_{i+3}) $ for $ i \in [n] $ as well as $ D_j := d_j d_{j+4} \dots d_{n-8+j} $ and $ D_j^\prime := d_j d_{j+4} \dots d_{n-12+j} $ for $ j \in \{1,2,3,4\} $. Let $ s, s^\prime \in \{+,-\} $. Then define
\begin{align*}
C_4^{++-Ms}(n) := \{&(v_1,\dots,v_n) \in C_4^{++-}(n) \mid\\
&D_i d_{n-4+i} = (s)^{j-i}D_j d_{n-4+j} \text{ for } i,j \in M\}\\
\intertext{with $ M \subseteq \{1,2,3\} $. Also define}
C_4^{+--12s}(n) := \{&(v_1,\dots,v_n) \in C_4^{+--}(n) \mid D_1 d_{n-3} = s D_2 d_{n-2}\},\\
C_4^{+(-)[s^\prime]12s}(n) := \{&(v_1,\dots,v_n) \in C_4^{+(-)[s^\prime]}(n) \mid D_1 d_{n-3} = s D_2 d_{n-2}\},\\
C_4^{++-34s}(n-1) := \{&(v_1,\dots,v_{n-1}) \in C_4^{++-}(n-1) \mid\\
&D_3 \det(V_{n-2},V_{n-1},V_1,V_2) = s D_4 \det(V_{n-2},V_1,V_2,V_3)\},\\
C_4^{+--23s}(n-1) := \{&(v_1,\dots,v_{n-1}) \in C_4^{+--}(n-1) \mid\\
&D_2 \det(V_{n-3},V_{n-2},V_{n-1},V_1) = s D_3 \det(V_{n-3},V_{n-1},V_1,V_2)\},\\
C_4^{-+-12s}(n-1) := \{&(v_1,\dots,v_{n-1}) \in C_4^{-+-}(n-1) \mid\\
&D_1 \det(V_{n-3},V_{n-2},V_{n-1},V_2) = s D_2 \det(V_{n-2},V_{n-1},V_1,V_2)\},\\
C_4^{---23s}(n-1) := \{&(v_1,\dots,v_{n-1}) \in C_4^{---}(n-1) \mid\\
&D_2 \det(V_{n-4},V_{n-2},V_{n-1},V_1) = s D_3 \det(V_{n-4},V_{n-1},V_1,V_2)\},\\
C_4^{-(-)[+]12s}(n-1) := \{&(v_1,\dots,v_{n-1}) \in C_4^{-(-)[+]}(n-1) \mid\\
&D_1 \det(V_{n-3},V_{n-2},V_{n-1},V_3) = s D_2 \det(V_{n-2},V_{n-1},V_1,V_3)\},\\
C_4^{+--14s}(n-2) := \{&(v_1,\dots,v_{n-2}) \in C_4^{+--}(n-2) \mid D_1 = s D_4^\prime \det(V_{n-4},V_1,V_2,V_3)\},\\
C_4^{+(-)[s^\prime]23s}(n-2) := \{&(v_1,\dots,v_{n-2}) \in C_4^{+(-)[s^\prime]}(n-2) \mid\\
&D_2 \det(V_{n-4},V_{n-3},V_{n-2},V_1) = s D_3 \det(V_{n-4},V_{n-3},V_1,V_2)\},\\
C_4^{-(-)[+]23s}(n-2) := \{&(v_1,\dots,v_{n-2}) \in C_4^{-(-)[+]}(n-2) \mid\\
&D_2 \det(V_{n-5},V_{n-3},V_{n-2},V_1) = s D_3 \det(V_{n-5},V_{n-3},V_1,V_2)\},\\
C_4^{(--)[+]23s}(n-2) := \{&(v_1,\dots,v_{n-2}) \in C_4^{(--)[+]}(n-2) \mid\\
&D_2 \det(V_{n-5},V_{n-4},V_{n-2},V_1) = s D_3 \det(V_{n-5},V_{n-4},V_1,V_2)\},\\
C_4^{s^\prime+-14s}(n-3) := \{&(v_1,\dots,v_{n-3} \in C_4^{s^\prime+-}(n-3) \mid D_1 = s D_4^\prime \det(V_{n-4},V_1,V_2,V_3)\},\\
C_4^{s^\prime(-)[+]14s}(n-3) := \{&(v_1,\dots,v_{n-3} \in C_4^{s^\prime(-)[+]}(n-3) \mid\\
&D_1 = s D_4^\prime \det(V_{n-4},V_1,V_2,V_3)\},\\
C_4^{+(-)[+]34s}(n-3) := \{&(v_1,\dots,v_{n-3} \in C_4^{+(-)[+]}(n-3) \mid\\
&D_3^\prime \det(V_{n-5},V_{n-4},V_1,V_2) = sD_4^\prime \det(V_{n-4},V_1,V_2,V_3)\}.
\end{align*}
Like before, we denote the cardinalities of these sets with a lower-case $ c $.
\end{defi}

Before we begin deriving recursions, we will verify that two of the sets defined above can be bijectively mapped onto each other and thus have the same cardinalities.

\begin{lem}\label{lem:++-12++-23}
Let $ n \in 4\N $. Then
\[
c_4^{++-12\pm}(n) = c_4^{++-23\pm}(n).
\]
\end{lem}
\begin{proof}
The determinant condition in $ C_4^{++-12\pm}(n) $ is
\[
D_1 d_{n-3} = \pm D_2 d_{n-2}.
\]
Under the bijection
\[
C_4^{++-}(n) \to C_4^{-++}(n); (v_1,\dots,v_n) \mapsto (v_n,\dots,v_1)
\]
this equation transforms to
\[
D_1 d_{n-3} = \pm D_4 d_{n}.
\]
Now applying the bijection
\[
C_4^{-++}(n) \to C_4^{++-}(n); (v_1,\dots,v_n) \mapsto (v_3,\dots,v_n,v_1,v_2)
\]
yields
\[
D_3 d_{n-1} = \pm D_2 d_{n-2}.
\]
This is equivalent to the determinant condition in $ C_4^{++-23\pm}(n) $.
\end{proof}

Now we can develop a system of recursions for $ c_4^{23\pm}(n) $:

\begin{lem}\label{lem:recursionsk423}
Let $ n \geq 8 $ be divisible by $ 4 $.
\begin{align}
c_4^{23\pm}(n) = &(q-1)^2c_4(n-1) + 2q(q-1)c_4^{++-}(n-1) + q^2c_4^{-+-}(n-1)\label{eq:+++23}\\
&+ q(q-1)^2c_4^{++-34\mp}(n-1) + 2q^2(q-1)c_4^{+--23\mp}(n-1)\nonumber\\
&+ q^3c_4^{---23\mp}(n-1),\nonumber\\
c_4^{++-23\pm}(n) = &(q-1)c_4(n-1) + qc_4^{++-}(n-1)\label{eq:++-23}\\
&+ q(q-1)c_4^{++-34\mp}(n-1) + q^2c_4^{+--23\mp}(n-1),\nonumber\\
c_4^{+--12\pm}(n) = &(q-1)c_4^{++-}(n-1) + qc_4^{+(-)[+]}(n-1)\label{eq:+--12}\\
&q(q-1)c_4^{-+-12\mp}(n-1) + q^2c_4^{-(-)[+]12\mp}(n-1),\nonumber\\
c_4^{+(-)[+]12\pm}(n) = &c_4(n-1) + qc_4^{++-34\mp}(n-1),\label{eq:+(-)[+]12}\\
c_4^{+(-)[-]12\pm}(n) = &c_4^{++-}(n-1) + qc_4^{-+-12\mp}(n-1),\label{eq:+(-)[-]12}\\
c_4^{++-34\pm}(n-1) = &(q-1)c_4(n-2) + 2qc_4^{++-}(n-2) + q^2c_4^{+--14\pm}(n-2),\label{eq:++-34}\\
c_4^{+--23\pm}(n-1) = &(q-1)c_4^{++-}(n-2) + qc_4^{-+-}(n-2)\label{eq:+--23}\\
&+ q(q-1)c_4^{+(-)[+]23\mp}(n-2) + q^2c_4^{-(-)[+]23\mp}(n-2)\nonumber,\\
c_4^{-+-12\pm}(n-1) = &c_4(n-2) + q(q-1)c_4^{+(-)[+]23\mp}(n-2)\label{eq:-+-12}\\
&+ q^2c_4^{+(-)[-]23\mp}(n-2),\nonumber\\
c_4^{---23\pm}(n-1) = &(q-1)c_4^{+--}(n-2) + qc_4^{-(-)[+]}(n-2)\label{eq:---23}\\
&+ q(q-1)c_4^{+(-)[-]23\mp}(n-2) + q^2c_4^{(--)[+]23\mp}(n-2),\nonumber\\
c_4^{-(-)[+]12\pm}(n-1) = &c_4^{++-}(n-2) + qc_4^{+(-)[+]23\mp}(n-2),\label{eq:-(-)[+]12}\\
c_4^{+--14\pm}(n-2) = &(q-1)^2c_4^{++-14\pm}(n-3) + q(q-1)c_4^{-+-14\pm}(n-3)\label{eq:+--14}\\
&+ q(q-1)c_4^{+(-)[+]14\pm}(n-3) + q^2c_4^{-(-)[+]14\pm}(n-3),\nonumber\\
c_4^{+(-)[+]23\pm}(n-2) = &c_4(n-3) + qc_4^{++-14\mp}(n-3),\label{eq:+(-)[+]23}\\
c_4^{+(-)[-]23\pm}(n-2) = &c_4^{++-}(n-3) + qc_4^{-+-14\mp}(n-3),\label{eq:+(-)[-]23}\\
c_4^{-(-)[+]23\pm}(n-2) = &c_4^{++-}(n-3) + qc_4^{+(-)[+]14\mp}(n-3),\label{eq:-(-)[+]23}\\
c_4^{(--)[+]23\pm}(n-2) = &c_4^{+(-)[+]}(n-3) + qc_4^{23\mp}(n-4),\label{eq:(--)[+]23}\\
c_4^{++-14\pm}(n-3) = &(q-1)^2c_4^{23\pm}(n-4) + 2q(q-1)c_4^{++-23\pm}(n-4)\label{eq:++-14}\\
&+ q^2c_4^{+--12\pm}(n-4),\nonumber\\
c_4^{-+-14\pm}(n-3) = &(q-1)c_4^{23\pm}(n-4) + q(q-1)c_4^{+(-)[+]12\pm}(n-4)\label{eq:-+-14}\\
&+ q^2c_4^{+(-)[-]12\pm}(n-4),\nonumber\\
c_4^{+(-)[+]14\pm}(n-3) = &(q-1)c_4^{23\pm}(n-4) +qc_4^{++-23\pm}(n-4),\label{eq:+(-)[+]14}\\
c_4^{+(-)[+]34\pm}(n-3) = &c_4^{+(-)[+]14\pm}(n-3),\label{eq:+(-)[+]34}\\
c_4^{-(-)[+]14\pm}(n-3) = &(q-1)c_4^{++-23\pm}(n-4) + qc_4^{+(-)[+]12\pm}(n-4).\label{eq:-(-)[+]14}
\end{align}
\end{lem}
\begin{proof}
For every tuple $ (v_1,\dots,v_m) \in (\pr^3(K))^m $ in this proof with $ m \in \{n,n-1,n-2,n-3\} $, let $ (V_1,\dots,V_m) \in (K^4)^m $ be an arbitrary lift. Let $ d_i := \det(V_i,V_{i+1},V_{i+2},V_{i+3}) $, where $ i $ is considered modulo $ n $ (if $ m < n $ we will only consider values of $ i $ up to $ m-3 $). Let $ D_j := d_j d_{j+4} \dots d_{n-8+j} $ and $ D_j := d_j d_{j+4} \dots d_{n-12+j} $ for $ j \in \{1,2,3,4\} $. Also let $ E_1 := v_{m-3} \vee v_{m-2} \vee v_{m-1} $, $ E_2 := v_{m-2} \vee v_{m-1} \vee v_1 $, $ E_3 := v_{m-1} \vee v_1 \vee v_2 $ and $ E_4 := v_1 \vee v_2 \vee v_3 $.\smallskip\\
\eqref{eq:+++23}: Let $ (v_1,\dots,v_n) \in C_4^{23\pm}(n) $. Then
\[
(v_1,\dots,v_{n-1}) \in C_4^{s_1s_2s_3}(n-1)
\]
with $ s_i \in \{+,-\} $. Consider the map
\[
\phi: K^4 \to K; V \mapsto D_2 \det(V_{n-2},V_{n-1},V,V_1) \mp D_3 \det(V_{n-1},V,V_1,V_2)
\]
and set $ H := \pr(\Ker \phi) $. Then $ v_n \in H $ is equivalent to the determinant condition on $ C_4^{23\pm}(n) $.\\
If $ s_2 = + $, then $ E_2 \neq E_3 $ and $ \det(V_{n-2},V_{n-1},V_2,V_1) \neq 0 $, but $ \det(V_{n-1},V_2,V_1,V_2) = 0 $. Therefore, $ \phi(V_2) \neq 0 $ and $ \dim H = 2 $. \\
If $ s_2 = - $, then $ E_2 = E_3 \subseteq H $. If these sets were equal, we would have $ v_n \in E_2 = E_3 $ and $ \{v_{n-2},v_{n-1},v_n,v_1\} $ would be dependent in contradiction to the definition of $ C_4^{23\pm}(n) $. Hence, $ H = \pr^3(K) $ or equivalently $ \phi = 0 $. \\
We will first focus on the case $ s_2 = + $. For a given $ (v_1,\dots,v_{n-1}) \in C_4^{s_1+s_3}(n-1) $ the choices for $ v_n $ that lead to $ (v_1,\dots,v_n) \in C_4^{23\pm}(n) $ are exactly the elements of $ H \backslash (E_1 \cup E_2 \cup E_3 \cup E_4) $, so we need to count this set. We have
\begin{align*}
E_2 \cap H &= E_2 \cap E_3 = v_{n-1} \vee v_1\\
           &= E_3 \cap H,
\end{align*}
because if one of the determinants with $ V $ in the definition of $ \phi $ vanishes, then $ \phi(V) = 0 $ if and only if the other one also vanishes. We also have $ v_{n-2},v_2 \notin H $ and thus $ E_1 \neq H \neq E_4 $. In particular, $ E_1 \cap H $ and $ E_4 \cap H $ are projective lines, just like $ E_2 \cap H = E_3 \cap H $. \\
If $ s_1 = - $, then $ E_1 = E_2 $ and thus $ E_1 \cap H = E_2 \cap H $. Similarly, $ s_3 = - $ implies $ E_3 \cap H = E_4 \cap H $. \\
If $ s_1 = + $, then $ v_1 \notin E_1 $, but $ v_1 \in E_2 \cap H $. Hence $ E_1 \cap H \neq E_2 \cap H $. An analogous argument with $ v_{n-1} $ instead of $ v_1 $ shows $ E_3 \cap H \neq E_4 \cap H $ if $ s_3 = + $. Also note that
\[
E_1 \cap E_2 \cap E_4 \cap H = E_1 \cap E_2 \cap E_3 \cap E_4 = \emptyset
\]
if $ s_1 = s_3 = + $ (the last equality was shown in the proof of \eqref{eq:k4+++}). Hence the three lines $ E_1 \cap H $, $ E_2 \cap H $ and $ E_4 \cap H $ do not have a common intersection point if they are pairwise distinct. It follows that $ H \backslash (E_1 \cup E_2 \cup E_3 \cup E_4) $ has $ (q-1)^2 $ elements if $ s_1 = s_3 = + $, $ q(q-1) $ elements if $ s_1 = + $ and $ s_3 = - $ or vice versa, and $ q^2 $ points if $ s_1 = s_3 = - $. In conclusion, the number of points of $ C_4^{23\pm}(n) $ with $ s_2 = + $ is
\[
(q-1)^2c_4(n-1) + 2q(q-1)c_4^{++-}(n-1) + q^2c_4^{-+-}(n-1),
\]
where we used the symmetry $ c_4^{-++}(n-1) = c_4^{++-}(n-1) $. \\
Now consider the case $ s_2 = - $. We then have the necessary condition $ \phi = 0 $, which is a condition on $ (v_1,\dots,v_{n-1}) $.\\
If $ s_1 = s_3 = + $, we have $ V_{n-3} \notin E_2 = E_3 \subset H $ and therefore $ \phi = 0 $ is equivalent to $ \phi(V_{n-3}) = 0 $, which is in turn equivalent to
\begin{equation}\label{eq:detCon+++s1+}
D_2 \det(V_{n-3},V_{n-2},V_{n-1},V_1) = \mp D_3 \det(V_{n-3},V_{n-1},V_1,V_2).
\end{equation}
Then we use the bijection
\[
C_4^{+-+}(n-1) \to C_4^{++-}(n-1); (v_1,\dots,v_{n-1}) \mapsto (v_2,\dots,v_{n-1},v_1)
\]
and \eqref{eq:detCon+++s1+} becomes
\[
D_3 \det(V_{n-2},V_{n-1},V_1,V_2) = \mp D_4 \det(V_{n-2},V_1,V_2,V_3),
\]
which is exactly the determinant condition of $ C_4^{++-34\mp}(n-1) $.\\
If $ s_1 = + $ and $ s_3 = -$, $ \phi = 0 $ is still equivalent to \eqref{eq:detCon+++s1+}, and this is immediately the determinant condition on $ C_4^{+--23\mp}(n-1) $.\\
If $ s_1 = - $ and $ s_3 = + $, then we have $ v_3 \notin E_3 = E_2 \subset H $ and we can use $ V_3 $ to test $ \phi = 0 $ instead; the condition is then equivalent to
\begin{equation}\label{eq:detCon+++s1-s3+}
D_2 \det(V_{n-2},V_{n-1},V_1,V_3) = \mp D_3 \det(V_{n-1},V_1,V_2,V_3).
\end{equation}
We use the bijection
\[
C_4^{--+}(n-1) \to C_4^{+--}(n-1); (v_1,\dots,v_{n-1}) \mapsto (v_{n-1},\dots,v_1)
\]
and \eqref{eq:detCon+++s1-s3+} becomes
\[
D_3 \det(V_{n-3},V_{n-1},V_1,V_2) = \mp D_2 \det(V_{n-3},V_{n-2},V_{n-1},V_1).
\]
This is again equivalent to the determinant condition of $ C_4^{+--23\mp}(n-1) $.\\
If $ s_1 = s_3 = - $, then $ v_{n-4} \notin E_1 = E_2 = E_3 \subset H $ and we can test $ \phi = 0 $ with $ V_{n-4} $, which yields
\begin{equation}\label{eq:detCon+++s1-}
D_2 \det(V_{n-4},V_{n-2},V_{n-1},V_1) = \mp D_3 \det(V_{n-4},V_{n-1},V_1,V_2).
\end{equation}
as an equivalent condition, and this is the determinant condition for $ (v_1,\dots,v_{n-1}) \in C_4^{---23\mp}(n-1) $. \\
In all these cases, the determinant condition becomes trivial due to $ \phi = 0 $ and thus the number of viable choices for $ v_n $ is the same as in the corresponding cases in the proof of \eqref{eq:k4+++}, namely $ q(q-1)^2 $ for $ s_1 = s_3 = + $ and $ q^2(q-1) $ for $ s_1 \neq s_3 $ and $ q^3 $ for $ s_1 = s_3 = - $. The claimed recursion follows. \medskip\\
\eqref{eq:++-23}: Let $ (v_1,\dots,v_n) \in C_4^{++-23\pm}(n) $. Then
\[
(v_1,\dots,v_{n-1}) \in C_4^{s_1s_2+}(n-1)
\]
with $ s_i \in \{+,-\} $. Consider again the map
\[
\phi: K^4 \to K; V \mapsto D_2 \det(V_{n-2},V_{n-1},V,V_1) \mp D_3 \det(V_{n-1},V,V_1,V_2)
\]
and let $ H := \pr(\Ker \phi) $. We see that $ v_n \in H $ is equivalent to our determinant condition.\\
Let $ s_2 = + $. Then $ \phi(V_2) \neq 0 $ and $ \dim H = 2 $. Let $ G := H \cap E_4 $. We have $ \dim G = 1 $ since $ v_2 \notin H $. The valid choices for $ v_n $ are the elements of $ G \backslash (E_1 \cup E_2 \cup E_3) $. We have $ E_2 \cap H = E_2 \cap E_3 = v_{n-1} \vee v_1 = E_3 \cap H $ and thus
\begin{align*}
E_2 \cap G &= E_2 \cap H \cap E_4 = (v_{n-1} \vee v_1) \cap E_4 = \{v_1\}\\
&= E_3 \cap G.
\end{align*}
If $ s_1 = + $, then $ v_1 \notin E_1 $. In particular, $ G \not\subset E_1 $ and therefore $ \lvert E_1 \cap G \rvert = 1 $ and $ E_1 \cap G \neq \{v_1\} $. On the other hand, if $ s_1 = - $, then $ E_1 = E_2 $ and thus $ E_1 \cap G = E_2 \cap G = \{v_1\} $. Hence we have $ q-1 $ resp. $ q $ values for $ v_n $ depending on $ s_1 $ and accordingly, the number of elements of $ C_4^{++-23\pm}(n) $ with $ s_2 = + $ is
\[
(q-1)c_4(n-1) + qc_4^{++-}(n-1)
\]
using symmetry.\\
Now let $ s_2 = - $. Then $ E_2 = E_3 \subseteq H$ and $ v_n \in H \backslash E_2 $ implies that $ E_2 = E_3 $ is a proper subset. Then $ H $ must be the whole projective space and $ \phi = 0 $. Therefore, $ \phi = 0 $ is a necessary condition on $ (v_1,\dots,v_{n-1}) $. On the other hand, if $ \phi = 0 $ the determinant condition becomes trivial and the remaining conditions are the same as in the proof of \eqref{eq:k4++-}. We then have the same number of choices for $ v_n $ as we had there. Specifically, there are $ q(q-1) $ choices if $ s_1 = + $ and $ q^2 $ choices if $ s_1 = - $.\\
If $ s_1 = + $, we have $ v_{n-3} \notin E_2 = E_3 $ and thus $ \phi = 0 $ is equivalent to $ \phi(V_{n-3}) $, which is in turn equivalent to
\[
D_2 \det(V_{n-3},V_{n-2},V_{n-1},V_1) = \mp D_3 \det(V_{n-3},V_{n-1},V_1,V_2).
\]
After applying the bijection
\[
C_4^{+-+}(n-1) \to C_4^{++-}(n-1); (v_1,\dots,v_{n-1}) \mapsto (v_2,\dots,v_{n-1},v_1),
\]
the condition becomes
\[
D_3 \det(V_{n-2},V_{n-1},V_1,V_2) = \mp D_4 \det(V_{n-2},V_1,V_2,V_3).
\]
This is exactly the determinant condition on $ C_4^{++-34\mp}(n-1) $.\\
If $ s_1 = - $, we use $ V_3 $ to test $ \phi = 0 $ instead. This results in the condition
\[
D_2 \det(V_{n-2},V_{n-1},V_1,V_3) = \mp D_3 \det(V_{n-1},V_1,V_2,V_3).
\]
This time we use the bijection
\[
C_4^{--+}(n-1) \to C_4^{+--}(n-1); (v_1,\dots,v_{n-1}) \mapsto (v_{n-1},\dots,v_1)
\]
and the condition becomes
\[
D_3 \det(V_{n-3},V_{n-1},V_1,V_2) = \mp D_2 \det(V_{n-3},V_{n-2},V_{n-1},V_1).
\]
This equivalent to the determinant condition of $ C_4^{+--23\mp}(n-1) $. The claimed recursion follows.\medskip\\
\eqref{eq:+--12}: Let $ (v_1,\dots,v_{n-1}) \in C_4^{+--12\pm}(n) $. Then
\[
(v_1,\dots,v_{n-1}) \in C_4^{s+-}(n-1) \cup C_4^{s(-)[+]}(n-1).
\]
Consider the map
\[
\phi: K^4 \to K; V \mapsto D_1 \det(V_{n-3},V_{n-2},V_{n-1},V) \mp D_2 \det(V_{n-2},V_{n-1},V,V_1)
\]
and let $ H := \pr(\Ker \phi) $. The suitable choices for $ v_n $ are then exactly the elements of $ H \cap E_4 \backslash (E_1 \cup E_2 \cup (v_1 \vee v_2)) $.\\
First let $ s = + $. Then $ \phi(V_1) \neq 0 $ and $ \dim H = 2 $. Also, $ G := H \cap E_4 $ is a projective line since $ v_1 \in E_4 \backslash H $. $ v_1 \notin H $ also implies $ \lvert (v_1 \vee v_2) \cap G \rvert = 1 $ (note that the lines $ G $ and $ v_1 \vee v_2 $ are both contained in the projective plane $ E_4 $ and therefore have at least one point in common). Furthermore,
\begin{align*}
E_1 \cap G &= E_1 \cap H \cap E_4 = E_1 \cap E_2 \cap E_4 = (v_{n-2} \vee v_{n-1}) \cap E_4 = \{v_{n-1}\}\\
&= E_2 \cap G.
\end{align*}
If $ \{v_{n-1},v_1,v_2\} $ is independent, then $ v_{n-1} \notin v_1 \vee v_2 $ and in particular, $ E_1 \cap G = E_2 \cap G \neq (v_1 \vee v_2) \cap G $. On the other hand if $ \{v_{n-1},v_1,v_2\} $ is dependent, then $ v_{n-1} \in v_1 \vee v_2 $; and since we also have $ v_{n-1} \in G $, it follows that $ (v_1 \vee v_2) \cap G = \{v_{n-1}\} $. Hence we have $ (q-1) $ resp. $ q $ options for $ v_n $ depending on whether $ \{v_{n-1},v_1,v_2\} $ is independent. Therefore the number of elements of $ C_4^{+--12\pm}(n) $ with $ s = + $ is
\[
(q-1)c_4^{++-}(n-1) + qc_4^{+(-)[+]}(n-1).
\]
Now let $ s = - $. Then $ E_1 = E_2 \subseteq H $ and the subset is proper, because otherwise we would have $ v_n \in E_1 $ or $ v_n \notin H $. Hence we must have $ H = \pr^3(K) $ and the set of choices for $ v_n $ becomes $ E_4 \backslash (E_1 \cup E_2 \cup (v_1 \vee v_2)) $, the same as in the proof of \eqref{eq:k4+--}, case $ s = - $. Accordingly, we have $ q(q-1) $ resp. $ q^2 $ choices for $ v_n $ depending on whether or not $ \{v_{n-1},v_1,v_2\} $ is independent. This is assuming that $ (v_1,\dots,v_{n-1}) $ fulfills the necessary condition $ H = \pr^3(K) $ (or equivalently $ \phi = 0 $), otherwise there are no choices for $ v_n $ that meet all requirements.\\
First we consider $ \{v_{n-1},v_1,v_2\} $ independent. In this case we use $ v_2 \notin E_2 = E_1 $, which means that we can test $ \phi = 0 $ with $ \phi(V_2) = 0 $. This results in the equivalent condition
\[
D_1 \det(V_{n-3},V_{n-2},V_{n-1},V_2) = \mp D_2 \det(V_{n-2},V_{n-1},V_1,V_2).
\]
This is also the condition for $ (v_1,\dots,v_{n-1}) \in C_4^{-+-12\mp}(n-1) $.\\
Now let $ \{v_{n-1},v_1,v_2\} $ be dependent. Then $ v_2 \in E_2 = E_1 $, but $ v_3 \notin E_2 = E_1 $ since $ \{v_{n-2},v_1,v_2,v_3\} $ is independent. Hence we can test $ \phi $ with $ V_3 $. $ \phi = 0 $ is therefore equivalent to
\[
D_1 \det(V_{n-3},V_{n-2},V_{n-1},V_3) = \mp D_2 \det(V_{n-2},V_{n-1},V_1,V_3),
\]
which is also the condition for $ (v_1,\dots,v_{n-1}) \in C_4^{-(-)[+]12\mp}(n-1) $. This concludes the proof of \eqref{eq:+--12}.\medskip\\
\eqref{eq:+(-)[+]12}, \eqref{eq:+(-)[-]12}: Let $ (v_1,\dots,v_n) \in C_4^{+(-)[s]12\pm}(n) $. Then
\[
(v_1,\dots,v_{n-1}) \in C_4^{s^\prime+s}(n-1)
\]
with $ s^\prime \in \{+,-\} $. Consider the map
\[
\phi: K^4 \to K; V \mapsto D_1 \det(V_{n-3},V_{n-2},V_{n-1},V) \mp D_2 \det(V_{n-2},V_{n-1},V,V_1)
\]
and let $ H := \pr(\Ker \phi) $. We have $ v_n \in H $ if and only if the determinant condition in $ C_4^{+(-)[s]12\pm}(n) $ holds.\\
Let $ s^\prime = + $. Then $ \phi(V_1) \neq 0 $ and hence $ \dim H = 2 $. $ v_1 \notin H $ also implies that $ \{P\} := H \cap (v_1 \vee v_2) $ is a single point. Only this point can be a valid choice for $ v_n $, and it is a valid choice if $ P \notin E_1, E_2 $. We have
\begin{align*}
E_1 \cap \{P\} &= E_1 \cap H \cap (v_1 \vee v_2) = (v_{n-2} \vee v_{n-1}) \cap (v_1 \vee v_2) = \emptyset\\
&= E_2 \cap \{P\}
\end{align*}
where we used $ E_1 \cap H = E_1 \cap E_2 = v_{n-2} \vee v_{n-1} = E_2 \cap H $. It follows that the number of elements of $ C_4^{+(-)[s]12\pm}(n) $ with $ s^\prime = + $ is exactly $ c_4^{++s}(n-1) $.\\
Now let $ s^\prime = - $. Then $ E_1 = E_2 \subseteq H $ and we cannot have equality. Hence $ \phi = 0 $ is a necessary condition, and since $ v_2 \notin E_2 $ this condition is equivalent to $ \phi(V_2) = 0 $. Therefore we can write the condition as
\[
D_1 \det(V_{n-3},V_{n-2},V_{n-1},V_2) = \mp D_2 \det(V_{n-2},V_{n-1},V_1,V_2).
\]
If $ s = - $ this is immediately the condition for $ (v_1,\dots,v_{n-1}) \in C_4^{-+-12\mp}(n-1) $. If $ s = + $ we need to apply the bijection
\[
C_4^{-++}(n-1) \to C_4^{++-}(n-1); (v_1,\dots,v_{n-1}) \mapsto (v_{n-1},\dots,v_1)
\]
and the condition becomes
\[
D_4 \det(V_{n-2},V_1,V_2,V_3) = \mp D_3 \det(V_{n-2},V_{n-1},V_1,V_2).
\]
This is equivalent to the determinant condition in $ C_4^{++-34\mp}(n-1) $. In either case we have $ q $ choices for $ v_n $ like in the proof of \eqref{eq:k4+(-)[+]}/\eqref{eq:k4+(-)[-]}. Thus the claimed recursions hold.\medskip\\
\eqref{eq:++-34}: Let $ (v_1,\dots,v_{n-1}) \in C_4^{++-34\pm}(n-1) $. Then
\[
(v_1,\dots,v_{n-2}) \in C_4^{s_1s_2+}(n-2)
\]
with $ s_i \in \{+,-\} $. Consider
\begin{align*}
\phi: K^4 &\to K\\
V &\mapsto D_3 \det(V_{n-2},V,V_1,V_2) \mp D_4^\prime \det(V_{n-4},V_{n-3},V_{n-2},V) \det(V_{n-2},V_1,V_2,V_3)
\end{align*}
and set $ H := \pr(\Ker \phi) $. The determinant condition on $ C_4^{++-34\pm}(n-1) $ is fulfilled if and only if $ v_{n-1} \in H $.\\
First assume that $ s_1 $ and $ s_2 $ are not both $ - $. If $ s_1 = + $, then $ v_1 \notin E_1 $ and $ \phi(V_1) \neq 0 $. If $ s_1 = - $ and $ s_2 = + $, then $ v_2 \notin E_2 = E_1 $. Hence $ \phi(V_2) \neq 0 $. In either case we have $ \dim H = 2 $. We also need to choose $ v_{n-1} $ from $ E_4 $, so only elements of $ G := H \cap E_4 $ need to be considered. Since $ v_1 \notin H $ or $ v_2 \notin H $ it follows that $ H \neq E_4 $ and $ \dim G = 1 $.  Furthermore, the fact that $ H $ does not contain both $ v_1 $ and $ v_2 $ also means that
\[
\lvert E_3 \cap G \rvert = \lvert H \cap E_3 \cap E_4 \rvert = \lvert H \cap (v_1 \vee v_2) \rvert = 1.
\]
Additionally, we have $ E_1 \cap H = E_3 \cap H = E_1 \cap E_3 $ and $ E_1 \cap G = E_3 \cap G $ follows. If $ s_1 = - $ or $ s_2 = - $, then $ E_1 = E_2 $ or $ E_2 = E_3 $ and accordingly, $ E_1 \cap G = E_2 \cap G = E_3 \cap G $. On the other hand, if $ s_1 = s_2 = + $, then $ \dim E_2 \cap E_4 = 1 $, since $ v_2 \notin E_2 $, but $ v_2 \in E_4 $. At the same time we have $ v_1 \in E_2 \cap E_4 $, but $ v_1 \notin H $. Hence we still have $ \lvert E_2 \cap G \rvert = 1 $. But $ E_1 \cap G \neq E_2 \cap G $ since
\[
E_1 \cap E_2 \cap G = E_1 \cap H \cap E_2 \cap E_4 = E_1 \cap E_3 \cap E_2 \cap E_4 = \emptyset.
\]
This means that there are $ q-1 $ choice for $ v_{n-1} $ if $ s_1 = s_2 = + $ and $ q $ choices if $ s_1 \neq s_2 $. So the  number of elements of $ C_4^{++-34\pm}(n-1) $ with $ s_1 = + $ or $ s_2 = + $ is
\[
(q-1)c_4(n-2) + 2qc_4^{++-}(n-2),
\]
where we used $ c_4^{-++}(n-2) = c_4^{+-+}(n-2) = c_4^{++-}(n-2) $.
\\ 
If $ s_1 = s_2 = - $, then $ E_1 = E_2 = E_3 \subseteq H $ and we cannot have equality since $ v_{n-1} \in H \backslash E_1 $. Hence $ \phi = 0 $. $ \{v_{n-2},v_1,v_2,v_3\} $ is independent, therefore we have $ \phi = 0 $ if and only if $ \phi(V_3) = 0 $. After rearranging and canceling out $ \det(V_{n-2},V_1,V_2,V_3) $ this becomes
\[
D_3 = \pm D_4^\prime \det(V_{n-4},V_{n-3},V_{n-2},V_3).
\]
We now apply the bijection
\[
C_4^{--+}(n-2) \to C_4^{+--}(n-2); (v_1,\dots,v_{n-2}) \mapsto (v_{n-2},\dots,v_1)
\]
and the determinant condition becomes
\[
D_1 = \pm D_4^\prime \det(V_{n-4},V_1,V_2,V_3),
\]
exactly as in the definition of $ C_4^{+--14\pm}(n-2) $. We again have the same number of choices as in the corresponding case \eqref{eq:k4++-} of Lemma \ref{lem:recursionsk4}, namely $ q^2 $. The claim follows.\medskip\\
\eqref{eq:+--23}: Let $ (v_1,\dots,v_{n-1}) \in C_4^{+--23\pm}(n-1) $. Then
\[
(v_1,\dots,v_{n-2}) \in C_4^{s+-}(n-2) \cup C_4^{s(-)[+]}(n-2)
\]
with $ s \in \{+,-\} $. Define
\[
\phi: K^4 \to K; V \mapsto D_2 \det(V_{n-3},V_{n-2},V,V_1) \mp D_3 \det(V_{n-3},V,V_1,V_2)
\]
and $ H := \pr(\Ker \phi) $. Then the determinant condition in $ C_4^{+--23\pm}(n-1) $ becomes $ v_{n-1} \in H $, and the valid choices for $v_{n-1} $ are exactly the elements of $ H \cap E_4 \backslash (E_1 \cup E_2 \cup (v_1 \vee v_2)) $. \\
Let $ \{v_{n-2},v_1,v_2\} $ be independent. Then $ \{v_{n-3},v_{n-2},v_1,v_2\} $ is independent and $ \phi(V_2) \neq 0 $  and thus $ \dim H = 2 $ follows. Moreover, $ v_2 \notin H $ implies $ H \neq E_4 $ and $ G := H \cap E_4 $ is a projective line. Additionally,
\[
E_2 \cap G = E_2 \cap E_4 \cap H = E_2 \cap E_3 \cap H = (v_{n-2} \vee v_1) \cap H = \{v_1\}
\]
and
\[
(v_1 \vee v_2) \cap G = (v_1 \vee v_2) \cap H = \{v_1\},
\]
where we used $ E_3 = E_4 $ and $ v_1 \vee v_2 \subset E_4 $. If $ s = - $, then $ E_1 = E_2 $ and $ E_1 \cap G = E_2 \cap G = \{v_1\} $ as well. If $ s = + $, then $ v_1 \notin E_1 $. Combined with $ v_1 \in G $ this implies $ \lvert E_1 \cap G \rvert = 1 $ and $ E_1 \cap G \neq \{v_1\} $. Hence we have $ q-1 $ options for $ v_{n-1} $ if $ s = + $ and $ q $ if $ s = - $. This implies that there are
\[
(q-1)c_4^{++-}(n-2) + qc_4^{-+-}(n-2)
\]
elements with $ \{v_{n-2},v_1,v_2\} $ independent.\\
Now let $ \{v_{n-2},v_1,v_2\} $ be dependent. Then $ E_2 = v_{n-3} \vee v_1 \vee v_2 \subseteq H $ and once again we cannot have equality. We can test $ \phi = 0 $ with $ V_{n-4} $ if $ s = + $ and with $ V_{n-5} $ if $ s = - $, since $ v_{n-4} \notin E_2 $ resp. $ v_{n-5} \notin E_1 = E_2 $ depending on $ s $. In the former case the condition then becomes
\[
D_2 \det(V_{n-4},V_{n-3},V_{n-2},V_1) = \mp D_3 \det(V_{n-4},V_{n-3},V_1,V_2),
\]
in the latter case it becomes
\[
D_2 \det(V_{n-5},V_{n-3},V_{n-2},V_1) = \mp D_3 \det(V_{n-5},V_{n-3},V_1,V_2).
\]
These are exactly the conditions for $ (v_1,\dots,v_{n-2}) \in C_4^{s(-)[+]23\mp}(n-2) $ and since $ \phi = 0 $ implies that the determinant condition becomes trivial, we have $ q(q-1) $ resp. $ q^2 $ choices for $ v_{n-1} $ just like in the proof of \eqref{eq:k4+--}. Hence the stated recursion holds.\medskip\\
\eqref{eq:-+-12}: Let $ (v_1,\dots,v_{n-1}) \in C_4^{-+-12\pm}(n-1) $. Then
\[
(v_1,\dots,v_{n-2}) \in C_4(n-2) \cup C_4^{(-)+[s]}(n-2)
\]
with $ s \in \{+,-\} $. Let
\[
\phi: K^4 \to K; V \mapsto D_1 \det(V_{n-3},V_{n-2},V,V_2) \mp D_2 \det(V_{n-2},V,V_1,V_2)
\]
and $ H := \pr(\Ker \phi) $. Clearly, $ v_{n-1} \in H $ is equivalent to the determinant condition.\\
We start with the case $ \{v_{n-3},v_{n-2},v_1\} $ independent. Then $ (v_1,\dots,v_{n-2}) \in C_4(n-2) $ and in particular, $ \{v_{n-3},v_{n-2},v_1,v_2\} $ is independent. Therefore $ \phi(V_1) \neq 0 $ and $ \dim H = 2 $. Additionally, $ E_2 \neq E_4 $ and $ v_1 \in E_2 \cap E_4 \backslash H $. Hence $ E_2 \cap E_4 \not\subset H $ and $ \{P\} := E_2 \cap E_4 \cap H $ is a single point. We must have $ v_{n-1} = P $ since $ (v_1,\dots,v_{n-1}) \in C_4^{-+-}(n-1) $ requires $ v_{n-1} \in E_2 \cap E_4 $ and our determinant condition requires $ v_{n-1} \in H $. On the other hand we need $ v_{n-1} \notin E_1, E_3 $. We have
\[
E_1 \cap \{P\} = E_1 \cap E_2 \cap E_4 \cap H = (v_{n-3} \vee v_{n-2}) \cap H \cap E_4 = \{v_{n-2}\} \cap E_4 = \emptyset
\]
and likewise
\[
E_3 \cap \{P\} = E_2 \cap E_3 \cap E_4 \cap H = \{v_1\} \cap H = \emptyset.
\]
Hence the choice $ v_{n-1} = P $ will always work and we have a bijection between the elements $ (v_1,\dots,v_{n-1}) \in C_4^{-+-12\pm}(n-1) $ with $ \{v_{n-3},v_{n-2},v_1\} $ independent and the elements of $ C_4(n-2) $.\\
Now let $ \{v_{n-3},v_{n-2},v_1\} $ be dependent. Then $ v_{n-3} \vee v_{n-2} \vee v_2 = v_{n-2} \vee v_1 \vee v_2 = E_3 \subseteq H $. We cannot have equality, thus $ \phi = 0 $, and this is equivalent to the determinant condition. We can test $ \phi = 0 $ with $ V_3 $, since $ v_3 \notin E_3 $. This yields
\[
D_1 \det(V_{n-3},V_{n-2},V_2,V_3) = \mp D_2 \det(V_{n-2},V_1,V_2,V_3).
\]
Now we apply the bijection
\[
C_4^{(-)+[s]}(n-2) \to C_4^{+(-)[s]}; (v_1,\dots,v_{n-2}) \mapsto (v_{n-2},\dots,v_1)
\]
and the condition becomes
\[
D_3 \det(V_{n-4},V_{n-3},V_1,V_2) = \mp D_2 \det(V_{n-4},V_{n-3},V_{n-2},V_1).
\]
This is exactly the determinant condition on $ C_4^{+(-)[s]23\mp}(n-2) $. The number of values for $ v_{n-1} $ is $ q(q-1) $ resp. $ q^2 $ depending on $ s $, since the situation is the same as in the proof of \eqref{eq:k4-+-}. The claimed recursion follows.
\medskip\\
\eqref{eq:---23}: Let $ (v_1,\dots,v_{n-1}) \in C_4^{---23\pm}(n-1) $. Then
\[
(v_1,\dots,v_{n-2}) \in C_4^{+--}(n-2) \cup C_4^{+(-)[-]}(n-2) \cup C_4^{(-)-[+]}(n-2) \cup C_4^{(--)[+]}(n-2).
\]
Let
\[
\phi: K^4 \to K; V \mapsto D_2 \det(V_{n-4},V_{n-2},V,V_1) \mp D_3 \det(V_{n-4},V,V_1,V_2)
\]
and $ H := \pr(\Ker \phi) $. Then $ v_{n-1} \in H $ is equivalent to $ (v_1,\dots,v_{n-1}) $ fulfilling the determinant condition on $ C_4^{---23\pm}(n-1) $. Let $ M_1 := \{v_{n-3},v_{n-2},v_1\} $ and $ M_2 := \{v_{n-2},v_1,v_2\} $.\\
If $ M_2 $ is dependent, then $ v_{n-4} \vee v_{n-2} \vee v_1 = v_{n-4} \vee v_1 \vee v_2 \subseteq H $. If equality holds we must have
\[
v_{n-1} \in H \cap E_4 = (v_{n-4} \vee v_1 \vee v_2) \cap E_4 = v_1 \vee v_2,
\]
which contradicts $ \{v_{n-1},v_1,v_2\} $ independent. Hence $ H = \pr^3(K) $ and $ \phi = 0 $.\\
If $ M_2 $ is independent, then $ v_{n-2} \vee v_1 \vee v_2 = E_4 $, and since $ v_{n-4} \notin v_{n-3} \vee v_{n-2} \vee v_{n-1} = E_2 = E_4 $ we have $ \{v_{n-4},v_{n-2},v_1,v_2\} $ independent. It follows that $ \phi(V_2) \neq 0 $ and $ \dim H = 2 $.\\
Case 1: $ M_2 $ is independent.\\
Let $ G := H \cap E_4 $. Since $ v_2 \in E_4 \backslash H $ it follows that $ \dim G = 1 $. Then we have
\[
(v_{n-2} \vee v_1) \cap G = (v_{n-2} \vee v_1) \cap H = \{v_1\}
\]
and
\[
(v_1 \vee v_2) \cap G = (v_1 \vee v_2) \cap H = \{v_1\},
\]
since $ v_{n-2},v_1,v_2 \in E_4 $. If $ M_1 $ is dependent, then $ E_1 \cap E_4 = v_{n-2} \vee v_1 $ and thus
\[
E_1 \cap G = E_1 \cap E_4 \cap H = (v_{n-2} \vee v_1) \cap H = \{v_1\}.
\]
If $ M_1 $ is independent, then $ v_1 \notin E_1 $ and in particular $ v_1 \notin E_1 \cap G $. This implies $ \lvert E_1 \cap G \rvert = 1 $. The permissible values for $ v_{n-1} $ are exactly the elements of $ G \backslash (E_1 \cup (v_{n-2} \vee v_1) \cup (v_1 \vee v_2)) $, hence there are $ q-1 $ resp. $ q $ choices depending on $ M_1 $. The number of elements of $ C_4^{---23\pm}(n-1) $ with $ M_2 $ independent is therefore
\[
(q-1)c_4^{+--}(n-2) + qc_4^{-(-)[+]}(n-2),
\]
where we used $ c_4^{(-)-[+]}(n-2) = c_4^{-(-)[+]}(n-2) $.\smallskip\\
Case 2: $ M_2 $ dependent and $ M_1 $ independent.\\
We have $ \{v_{n-4},v_{n-3},v_{n-2},v_2\} $ independent and thus the necessary condition $ \phi = 0 $ is equivalent to $ \phi(V_{n-3}) = 0 $. Then the condition becomes
\[
D_2 \det(V_{n-4},V_{n-3},V_{n-2},V_1) = \mp D_3 \det(V_{n-4},V_{n-3},V_1,V_2),
\]
hence we have $ (v_1,\dots,v_{n-2}) \in C_4^{+(-)[-]23\mp}(n-2) $. On the other hand, $ (v_1,\dots,v_{n-2}) \in C_4^{+(-)[-]23\mp}(n-2) $ implies $ \phi = 0 $ and the determinant condition on $ C_4^{---23\pm}(n-1) $ becomes trivial. Hence we have the same number of choices for $ v_{n-1} $ as in Case 2 in the proof of \eqref{eq:k4---}, which is to say $ q(q-1) $.\smallskip\\
Case 3: $ M_1 $ and $ M_2 $ are both dependent.\\
This is almost the same as the previous case, except that we test $ \phi = 0 $ with $ V_{n-5} $ (since $ v_{n-5} \notin v_{n-4} \vee v_{n-3} \vee v_{n-2} = v_{n-4} \vee v_1 \vee v_2 \subset H $) and get
\[
D_2 \det(V_{n-5},V_{n-4},V_{n-2},V_1) = \mp D_3 \det(V_{n-5},V_{n-4},V_1,V_2),
\]
which is the determinant condition on $ C_4^{(--)[+]23\mp}(n-2) $. We now have $ q^2 $ choices for $ v_{n-1} $ as in case 4 of the proof of \eqref{eq:k4---}. The desired formula follows.\medskip\\
\eqref{eq:-(-)[+]12}: Let $ (v_1,\dots,v_{n-1}) \in C_4^{-(-)[+]12\pm}(n-1) $. Then
\[
(v_1,\dots,v_{n-2}) \in C_4^{+-+}(n-2) \cup C_4^{(-)+[+]}(n-1).
\]
Consider the map
\[
\phi: K^4 \to K; V \mapsto D_1 \det(V_{n-3},V_{n-2},V,V_3) \mp D_2 \det(V_{n-2},V,V_1,V_3)
\]
and $ H := \pr(\Ker \phi) $. The determinant condition in $ C_4^{-(-)[+]12\pm}(n-1) $ is fulfilled if and only if $ v_{n-1} \in H $.\\
First let $ \{v_{n-3},v_{n-2},v_1\} $ be independent and thus $ (v_1,\dots,v_{n-2}) \in C_4^{+-+}(n-2) $. Then $ \{v_{n-2},v_1,v_2,v_3\} $ is independent and we also have $ v_{n-3} \vee v_{n-2} \vee v_1 = v_{n-2} \vee v_1 \vee v_2 $. As a result, $ \{v_{n-3},v_{n-2},v_1,v_3\} $ is independent and hence $ \phi(V_1) \neq 0 $. It follows that $ \dim H = 2 $, and since $ v_1 \notin H $, $ \{P\} := H \cap (v_1 \vee v_2) $ must be a single point. On the other hand we have $ v_1 \vee v_2 \subset E_2 = E_3 $ and thus $ P \in E_2 $. It follows that
\[
E_1 \cap \{P\} = E_1 \cap E_2 \cap H \cap (v_1 \vee v_2) = (v_{n-3} \vee v_{n-2}) \cap H \cap (v_1 \vee v_2) = \{v_{n-2}\} \cap (v_1 \vee v_2) = \emptyset.
\]
Since the viable choices for $ v_{n-1} $ are the elements of $ H \cap (v_1 \vee v_2) \backslash (E_1 \cup \{v_1\}) $, we then have exactly one such choice and the number of elements with $ \{v_{n-3},v_{n-2},v_1\} $ independent is $ c_4^{+-+}(n-2) = c_4^{++-}(n-2) $.\\
Now let $ \{v_{n-3},v_{n-2},v_1\} $ be dependent. Then we have $ v_1 \in H $. If $ \phi(V_2) \neq 0 $, then $ H \cap (v_1 \vee v_2) = \{v_1\} $. But this implies $ v_{n-1} = v_1 $, which is not allowed. Hence we must have $ \phi(V_2) = 0 $. On the other hand, $ \phi(V_2) = 0 $ leads to $ v_1 \vee v_2 \subset H $ and therefore any of the $ q $ choices for $ v_{n-1} $ we had in the proof of \eqref{eq:k4-(-)}, Case $ \{v_{n-3},v_{n-2},v_1\} $ dependent will still work here. $ \phi(V_2) = 0 $ is equivalent to
\[
D_1 \det(V_{n-3},V_{n-2},V_2,V_3) = \mp D_2 \det(V_{n-2},V_1,V_2,V_3).
\]
After using the bijection
\[
C_4^{(-)+[+]}(n-2) \to C_4^{+(-)[+]}(n-2); (v_1,\dots,v_{n-2}) \mapsto (v_{n-2},\dots,v_1),
\]
the condition turns into
\[
D_3 \det(V_{n-4},V_{n-3},V_1,V_2) = \mp D_2 \det(V_{n-4},V_{n-3},V_{n-2},V_1),
\]
which is equivalent to the determinant condition on $ C_4^{+(-)[+]23\mp}(n-2) $. This proves the claim.\medskip\\
\eqref{eq:+--14}: The determinant condition in $ C_4^{+--14\pm}(n-2) $ is
\[
D_1 = \pm D_4^\prime \det(V_{n-4},V_1,V_2,V_3),
\]
which does not depend on $ v_{n-2} $. Therefore, $ (v_1,\dots,v_{n-3}) $ will fulfill the same condition. Accordingly, we can use the same recursion as \eqref{eq:k4+--}, except that all the sets whose cardinality is on the right hand side inherit the determinant condition. This means that these sets become $ C_4^{++-14\pm}(n-3) $, $ C_4^{-+-14\pm}(n-3) $, $ C_4^{+(-)[+]14\pm}(n-3) $ and $ C_4^{-(-)[+]14\pm}(n-3) $. The claim follows.\medskip\\
\eqref{eq:+(-)[+]23}, \eqref{eq:+(-)[-]23}: Let $ (v_1,\dots,v_{n-2}) \in C_4^{+(-)[s]23\pm}(n-2) $ with $ s \in \{+,-\} $. Then
\[
(v_1,\dots,v_{n-3}) \in C_4^{s^\prime+s}(n-3)
\]
with $ s^\prime \in \{+,-\} $. Consider the map
\begin{align*}
\phi: K^4 &\to K\\
V &\mapsto D_2 \det(V_{n-4},V_{n-3},V,V_1) \mp D_3^\prime \det(V_{n-5},V_{n-4},V_{n-3},V) \det(V_{n-4},V_{n-3},V_1,V_2)
\end{align*}
and let $ H := \pr(\Ker \phi) $. Then the determinant condition in $ C_4^{+(-)[s]23\pm}(n-2) $ is fulfilled by a tuple $ (v_1,\dots,v_{n-2}) $ if and only if $ v_{n-2} \in H $.\\
If $ s^\prime = + $, then $ \phi(V_1) \neq 0 $ and $ \dim H = 2 $. Additionally, we must have $ v_{n-2} \in v_1 \vee v_2 $, and $ v_1 \notin H $ then implies that $ \{P\} := H \cap (v_1 \vee v_2) $ is a single point. Hence there is only one possible choice $ v_{n-2} = P $ for any $ (v_1,\dots,v_{n-3}) \in C_4^{++s}(n-3) $. On the other hand, we have
\begin{align*}
E_1 \cap H \cap (v_1 \vee v_2) &= E_1 \cap E_2 \cap (v_1 \vee v_2) = E_1 \cap \{v_1\} = \emptyset\\
&= E_2 \cap H \cap (v_1 \vee v_2),
\end{align*}
where we used $ E_1 \cap H = E_1 \cap E_2 = E_2 \cap H $. This guaranties that the one possible choice $ P $ will always work, and we have exactly $ c_4^{++s}(n-3) $ elements with $ s^\prime = + $.\\
Now let $ s^\prime = - $. Then $ E_1 = E_2 \subseteq H $, and once again we cannot have equality. Hence $ (v_1,\dots,v_{n-3})$ must be such that $ \phi = 0 $; and requiring this is equivalent to the determinant condition on $ C_4^{+(-)[s]23\pm}(n-2) $. We have $ v_2 \notin E_2 $, thus we can use $ V_2 $ to test $ \phi = 0 $. This results in
\[
D_2 = \mp D_3^\prime \det(V_{n-5},V_{n-4},V_{n-3},V_2)
\]
after canceling out $ \det(V_{n-4},V_{n-3},V_1,V_2) $. By using the bijection
\[
C_4^{-+s}(n-3) \to C_4^{s+-}(n-3); (v_1,\dots,v_{n-3}) \mapsto (v_{n-3},\dots,v_1),
\]
our determinant condition becomes
\[
D_1 = \mp D_4^\prime \det(V_{n-4},V_1,V_2,V_3).
\]
This is the determinant condition on $ C_4^{s+-14\mp}(n-3) $, and we have $ q $ choices for $ v_{n-2} $ just like in the case $ s^\prime = - $ of the proof of \eqref{eq:k4+(-)[+]}/\eqref{eq:k4+(-)[-]}. The claimed recursions follow.\medskip\\
\eqref{eq:-(-)[+]23}: Let $ (v_1,\dots,v_{n-2}) \in C_4^{-(-)[+]23\pm}(n-2) $. Then
\[
(v_1,\dots,v_{n-3}) \in C_4^{+-+}(n-3) \cup C_4^{(-)+[+]}(n-3).
\]
Let
\begin{align*}
\phi: K^4 &\to K\\
V &\mapsto D_2 \det(V_{n-5},V_{n-3},V,V_1) \mp D_3^\prime \det(V_{n-5},V_{n-4},V_{n-3},V) \det(V_{n-5},V_{n-3},V_1,V_2)
\end{align*}
and $ H := \pr(\Ker \phi) $. Like before, $ v_{n-2} \in H $ is equivalent to the determinant condition in $ C_4^{-(-)[+]23\pm}(n-2) $.\\
Let $ \{v_{n-4},v_{n-3},v_1\} $ be independent. $ \{v_{n-5},v_{n-4},v_{n-3},v_1\} $ is also independent then and thus $ \phi(V_1) \neq 0 $ and $ \dim H = 2 $. Moreover, since $ v_1 \notin H $, $ \{P\} := H \cap (v_1 \vee v_2) $ is a single point, and only this point can be a suitable choice for $ v_{n-2} $. We have $ E_1 \cap H = E_1 \cap(v_{n-5} \vee v_{n-3} \vee v_1) = v_{n-5} \vee v_{n-3} $, hence
\[
E_1 \cap \{P\} = E_1 \cap H \cap (v_1 \vee v_2) = (v_{n-5} \vee v_{n-3}) \cap (v_1 \vee v_2) = \emptyset,
\]
where we used $ v_{n-5} \notin E_2 = E_3 = v_{n-3} \vee v_1 \vee v_2 $. Moreover, since $ v_1 \vee v_2 \subset E_3 = E_2 $ we also have $ P \in E_2 $. Additionally,
\[
(v_{n-3} \vee v_1) \cap \{P\} = (v_{n-3} \vee v_1) \cap H \cap (v_1 \vee v_2) = \{v_{n-3}\} \cap (v_1 \vee v_2) = \emptyset.
\]
It follows that $ (v_1,\dots,v_{n-3},P) \in C_4^{-(-)[+]23\pm}(n-2) $ for any $ (v_1,\dots,v_{n-3}) \in C_4^{+-+}(n-3) $. The number of elements of $ C_4^{-(-)[+]23\pm}(n-2) $ with $ \{v_{n-4},v_{n-3},v_1\} $ independent is therefore $ c_4^{+-+}(n-3) = c_4^{++-}(n-3) $.\\
Now let $ \{v_{n-4},v_{n-3},v_1\} $ be dependent. Then $ E_1 = v_{n-5} \vee v_{n-3} \vee v_1 \subseteq H $ and we cannot have equality. Hence $ \phi = 0 $ is a necessary condition on $ (v_1,\dots.v_{n-3}) $ and by testing with $ V_2 $ (since $ \{v_{n-5},v_{n-4},v_{n-3},v_2\} $ is independent) we see that it is equivalent to
\[
D_2 = \mp D_3^\prime \det(V_{n-5},V_{n-4},V_{n-3},V_2)
\]
after canceling out $ \det(V_{n-5},V_{n-3},V_1,V_2) $. Now apply the bijection
\[
C_4^{(-)+[+]}(n-3) \to C_4^{+(-)[+]}(n-3); (v_1,\dots,v_{n-3}) \mapsto (v_{n-3},\dots,v_1),
\]
and the condition transforms to
\[
D_1 = \mp D_4^\prime \det(V_{n-4},V_1,V_2,V_3).
\]
This is exactly the condition on $ C_4^{+(-)[+]14\mp}(n-3) $. Since $ \phi = 0 $, the determinant condition on $ C_4^{-(-)[+]23\pm}(n-2) $ then becomes trivial. Accordingly, we have $ q $ possibilities for $ v_{n-2} $ like in the proof of \eqref{eq:k4-(-)}. The claim holds.\medskip\\
\eqref{eq:(--)[+]23}: Let $ (v_1,\dots,v_{n-2}) \in C_4^{(--)[+]23\pm}(n-2) $. Then $ (v_1,\dots,v_{n-3}) \in C_4^{+(-)[+]}(n-3) $ if $ v_{n-3} \neq v_1 $ and $ (v_1,\dots,v_{n-4}) \in C_4(n-4) $ if $ v_{n-3} = v_1 $. We have the determinant condition
\begin{align}
&D_2 \det(V_{n-5},V_{n-4},V_{n-2},V_1)\label{eq:detCon(--)23}\\
= &\pm D_3^\prime \det(V_{n-5},V_{n-4},V_{n-3},V_{n-2})\det(V_{n-5},V_{n-4},V_1,V_2).\nonumber
\end{align}
Now consider the case $ v_{n-3} \neq v_1 $. We define the map
\begin{align*}
\phi: K^4 &\to K\\
V &\mapsto D_2 \det(V_{n-5},V_{n-4},V,V_1) \mp D_3^\prime \det(V_{n-5},V_{n-4},V_{n-3},V) \det(V_{n-5},V_{n-4},V_1,V_2)
\end{align*}
and let $ H := \pr(\Ker \phi) $. Like before, $ v_{n-2} \in H $ if and only if $ (v_1,\dots,v_{n-2}) $ fulfills the determinant condition. We have $ \{v_{n-5},v_{n-4},v_{n-3},v_1\} $ independent and thus $ \phi(V_1) \neq 0 $, $ \dim H = 2 $ follows. Then $ \{P\} := H \cap (v_1 \vee v_2) $ is a single point and only this point may satisfy all conditions on $ v_{n-2} $. By definition of $ P $, we have $ P \in v_1 \vee v_2 = v_{n-3} \vee v_1 $. Additionally, $ P $ fulfills the determinant condition, and since $ v_1 \notin H $ we also have $ P \neq v_1 $. Moreover,
\[
E_1 \cap \{P\} = E_1 \cap H \cap (v_1 \vee v_2) = (v_{n-5} \vee v_{n-4}) \cap (v_1 \vee v_2) = (v_{n-5} \vee v_{n-4}) \cap (v_{n-3} \vee v_1) = \emptyset.
\]
Hence, $ (v_1,\dots,v_{n-3},P) \in C_4^{(--)[+]23\pm}(n-2) $ will in fact hold for all $ (v_1,\dots,v_{n-3}) \in C_4^{+(-)[+]}(n-3) $.\\
Now consider the case $ v_{n-3} = v_1 $. We can assume w.l.o.g. that $ V_{n-3} = V_1 $. Then $ \det(V_{n-5},V_{n-4},V_{n-2},V_1) = - \det(V_{n-5},V_{n-4},V_{n-3},V_{n-2}) $, and both sides are nonzero. After canceling this out of \eqref{eq:detCon(--)23}, the condition turns into
\[
D_2^\prime \det(V_{n-6},V_{n-5},V_{n-4},V_1) = \mp D_3^\prime \det(V_{n-5},V_{n-4},V_1,V_2),
\]
where we used $ V_{n-3} = V_1 $ once more. This is the determinant condition in $ C_4^{23\mp}(n-4) $. We have $ q $ choices for $ v_{n-2} $ as in the proof of \eqref{eq:k4(--)} and obviously there is only one choice for $ v_{n-3} $, namely $ v_1 $. This concludes the proof of the recursion.\medskip\\
Now consider the recursions using $ n-3 $ with the determinant condition
\begin{equation}\label{eq:stayingCondition}
D_1 = D_4^\prime \det(V_{n-4},V_1,V_2,V_3),
\end{equation}
namely \eqref{eq:++-14}, \eqref{eq:-+-14}, \eqref{eq:+(-)[+]14} and \eqref{eq:-(-)[+]14}. These are similar to \eqref{eq:+--14} in that the determinant condition does not depend on $ v_{n-3} $. Hence we get the recursions \eqref{eq:k4++-}, \eqref{eq:k4-+-}, \eqref{eq:k4+(-)[+]} and \eqref{eq:k4-(-)} except for the additional determinant requirement in the sets involved. Let $ (v_1,\dots,v_{n-3}) $ be an element from one of the corresponding sets. For \eqref{eq:++-14} we have
\[
(v_1,\dots,v_{n-4}) \in C_4^{s_1s_2+}(n-4)
\]
with $ s_i \in \{+,-\} $. If $ s_1 = s_2 = + $, then we have $ (v_1,\dots,v_{n-4}) \in C_4^{14\pm}(n-4) $ and by symmetry $ c_4^{14\pm}(n-4) = c_4^{23\pm}(n-4) $.\\
If $ s_1 = - $ and $ s_2 = + $ we use the bijection
\begin{equation}\label{eq:-++to++-}
C_4^{-++}(n-4) \to C_4^{++-}(n-4); (v_1,\dots,v_{n-4}) \mapsto (v_3,\dots,v_{n-4},v_1,v_2).
\end{equation}
That way \eqref{eq:stayingCondition} becomes the determinant condition on $ C_4^{++-23\pm}(n-4) $.
If $  s_2 = - $ we instead use the bijection
\begin{equation}\label{eq:s-+to+s-}
C_4^{s_1-+}(n-4) \to C_4^{+s_1-}(n-4); (v_1,\dots,v_{n-4}) \mapsto (v_2,\dots,v_{n-4},v_1)
\end{equation}
and \eqref{eq:stayingCondition} turns into the determinant condition on $ C_4^{+s_1-12\pm}(n-4) $. By Lemma \ref{lem:++-12++-23} we have $ c_4^{++-12\pm}(n-4) = c_4^{++-23\pm}(n-4) $. \eqref{eq:++-14} follows.\\
For \eqref{eq:-+-14} we have
\[
(v_1,\dots,v_{n-4}) \in C_4(n-4) \cup C_4^{(-)+[s]}(n-4)
\]
with $ s \in \{+,-\} $. The bijection
\begin{equation}\label{eq:(-)+[s]to+(-)[s]}
C_4^{(-)+[s]}(n-4) \to C_4^{+(-)[s]}; (v_1,\dots,v_{n-4}) \mapsto (v_2,\dots,v_{n-4},v_1)
\end{equation}
turns \eqref{eq:stayingCondition} into the determinant condition on $ C_4^{+(-)[s]12\pm}(n-4) $. And we again use $ c_4^{14\pm}(n-4) = c_4^{23\pm}(n-4) $. This yields \eqref{eq:-+-14}.\\
For \eqref{eq:+(-)[+]14} we have
\[
(v_1,\dots,v_{n-4}) \in C_4^{s++}(n-4)
\]
with $ s \in \{+,-\} $. If $ s = - $ we use \eqref{eq:-++to++-} again and \eqref{eq:stayingCondition} becomes the determinant condition on $ C_4^{++-23\pm}(n-4) $. If $ s = + $ we use $ c_4^{14\pm}(n-4) = c_4^{23\pm}(n-4) $. \\
Finally, for \eqref{eq:-(-)[+]14} we have
\[
(v_1,\dots,v_{n-4}) \in C_4^{+-+}(n-4) \cup C_4^{(-)+[+]}(n-4).
\]
By using \eqref{eq:s-+to+s-} and \eqref{eq:(-)+[s]to+(-)[s]} once more we get \eqref{eq:-(-)[+]14}.\medskip\\
\eqref{eq:+(-)[+]34}: This is again similar to the previous cases, except that the determinant condition is
\[
D_3^\prime \det(V_{n-5},V_{n-4},V_1,V_2) = \pm D_4^\prime \det(V_{n-4},V_1,V_2,V_3).
\]
This still does not depend on $ v_{n-3} $ and thus the determinant condition can again be passed on to the resulting sets. For $ (v_1,\dots,v_{n-3}) \in C_4^{+(-)[+]34\pm}(n-3) $ we have
\[
(v_1,\dots,v_{n-4}) \in C_4^{s++}(n-4)
\]
with $ s \in \{+,-\} $. In the case $ s = - $ we can again use \eqref{eq:-++to++-} and the determinant condition becomes
\[
D_1 = \pm D_2^\prime \det(V_{n-6},V_{n-5},V_{n-4},V_1).
\]
This is precisely the determinant condition for $ C_4^{++-12\pm}(n-4) $. In the case $ s = + $ we get $ C_4^{34\pm}(n-4) $. Since $ c_4^{++-12\pm}(n-4) = c_4^{++-23\pm}(n-4) $ and $ c_4^{34\pm}(n-4) = c_4^{23\pm}(n-4) $ this yields the recursion
\[
c_4^{+(-)[+]34\pm}(n-3) = (q-1)c_4^{23\pm}(n-4) + qc_4^{++-23\pm}(n-4).
\]
Since $ c_4^{+(-)[+]14\pm}(n-3) $ obeys the same recursion \eqref{eq:+(-)[+]14}, we get $ c_4^{+(-)[+]34\pm}(n-3) = c_4^{+(-)[+]14\pm}(n-3) $ as claimed.
\end{proof}

Next we will consider the set $ C_4^{13}(n) $ with $ n $ divisible by $ 4 $. Like before, counting these sets requires us to define a couple of new sets. Recall that we already defined the sets $ C_4^{13}(n) $ and $ C_4^{++-13}(n) $.

\begin{defi}
Let $ n \geq 4 $ be divisible by $ 4 $. For any tuple $ (v_1,\dots,v_m) \in (\pr^3(K))^m $ with $ m \in \{n,n-1,n-2,n-3\} $ that appears in the definitions of the following sets let $ (V_1,\dots,V_m) \in (K^4)^m $ be an arbitrary lift. Define $ d_i := \det(V_i,V_{i+1},V_{i+2},V_{i+3}) $, where the indices are considered modulo $ n $. Let $ D_j := d_j d_{j+4} \dots d_{n-8+j} $ and $ D_j^\prime := d_j d_{j+4} \dots d_{n-12+j} $ with $ j \in \{1,2,3,4\} $. Then set
\begin{align*}
C_4^{-+-13}(n) := \{&(v_1,\dots,v_n) \in C_4^{-+-}(n) \mid D_1 d_{n-3} = D_3 d_{n-1}\},\\
C_4^{+--24}(n-1) := \{&(v_1,\dots,v_{n-1}) \in C_4^{+--}(n-1) \mid D_2 \det(V_{n-3},V_{n-2},V_{n-1},V_1)\\
&= D_4 \det(V_{n-3},V_1,V_2,V_3)\},\\
C_4^{---13}(n-1) := \{&(v_1,\dots,v_{n-1}) \in C_4^{---}(n-1)\mid D_1 \det(V_{n-4},V_{n-3},V_{n-2},V_{n-1})\\
&= D_3 \det(V_{n-4},V_{n-1},V_1,V_2)\},\\
C_4^{+(-)[-]24}(n-1) := \{&(v_1,\dots,v_{n-1}) \in C_4^{+(-)[-]}(n-1) \mid D_2 \det(V_{n-3},V_{n-2},V_{n-1},V_1)\\
&= D_4 \det(V_{n-3},V_1,V_2,V_3)\},\\
C_4^{-+-24}(n-2) := \{&(v_1,\dots,v_{n-2}) \in C_4^{-+-}(n-2) \mid D_2 \det(V_{n-3},V_{n-2},V_1,V_2)\\
&= - D_4^\prime \det(V_{n-4},V_{n-3},V_{n-2},V_2) \det(V_{n-3},V_1,V_2,V_3)\},\\
C_4^{-(-)[+]24}(n-2) := \{&(v_1,\dots,v_{n-2}) \in C_4^{-(-)[+]}(n-2) \mid D_2 \det(V_{n-3},V_{n-2},V_1,V_3)\\
&=- D_4^\prime \det(V_{n-4},V_{n-3},V_{n-2},V_3) \det(V_{n-3},V_1,V_2,V_3)\},\\
C_4^{(--)[+]13}(n-2) := \{&(v_1,\dots,v_{n-2}) \in C_4^{(--)[+]}(n-2) \mid\\
& D_1 = D_3^\prime \det(V_{n-5},V_{n-4},V_1,V_2)\},\\
C_4^{+(-)[s]13}(n-3) := \{&(v_1,\dots,v_{n-3}) \in C_4^{+(-)[s]}(n-3) \mid\\
&D_1 = D_3^\prime \det(V_{n-5},V_{n-4},V_1,V_2)\}
\end{align*}
with $ s \in \{+,-\} $. Denote the cardinalities of these sets with a lower-case $ c $.
\end{defi}

These definitions are independent of the choice of lift. We can calculate their cardinalities over $ \F_q $ with the following recursions:

\begin{lem}\label{lem:recursionsk413}
Let $ n \geq 8 $ be divisible by $ 4 $. Then
\begin{align}
c_4^{13}(n) = &(q-1)^2c_4(n-1) + 3q(q-1)c_4^{++-}(n-1) + q^2c_4^{+--}(n-1)\label{eq:+++13}\\
&+ q^2c_4^{-+-}(n-1) + q^2(q-1)c_4^{+--24}(n-1)\nonumber\\
&+ q^3c_4^{---13}(n-1),\nonumber\\
c_4^{++-13}(n) = &(q-1)c_4(n-1) + 2qc_4^{++-}(n-1) + q^2c_4^{+--24}(n-1),\label{eq:++-13}\\
c_4^{-+-13}(n) = &c_4(n-1) + qc_4^{+(-)[+]}(n-1) + q^2c_4^{+(-)[-]24}(n-1),\label{eq:-+-13}\\
c_4^{+--24}(n-1) = &(q-1)c_4^{++-}(n-2) +qc_4^{+(-)[+]}(n-2)\label{eq:+--24}\\
&+ q(q-1)c_4^{-+-24}(n-2) + q^2c_4^{-(-)[+]24}(n-2),\nonumber\\
c_4^{---13}(n-1) = &(q-1)c_4^{+--}(n-2) + qc_4^{+(-)[-]}(n-2) + qc_4^{-(-)[+]}(n-2)\label{eq:---13}\\
&+ q^2c_4^{(--)[+]13}(n-2),\nonumber\\
c_4^{+(-)[-]24}(n-1) = &c_4^{++-}(n-2) + qc_4^{-+-24}(n-2),\label{eq:+(-)[-]24}\\
c_4^{-+-24}(n-2) = &c_4(n-3) + q(q-1)c_4^{+(-)[+]13}(n-3)\label{eq:-+-24}\\
&+ q^2c_4^{+(-)[-]13}(n-3),\nonumber\\
c_4^{-(-)[+]24}(n-2) = &c_4^{++-}(n-3) + qc_4^{+(-)[+]13}(n-3),\label{eq:-(-)[+]24}\\
c_4^{(--)[+]13}(n-2) = &(q-1)c_4^{+(-)[+]13}(n-3) + qc_4^{13}(n-4),\label{eq:(--)[+]13}\\
c_4^{+(-)[+]13}(n-3) = &(q-1)c_4^{13}(n-4) + qc_4^{++-13}(n-4),\label{eq:+(-)[+]13}\\
c_4^{+(-)[-]13}(n-3) = &(q-1)c_4^{++-13}(n-4) + qc_4^{-+-13}(n-4).\label{eq:+(-)[-]13}
\end{align}
\end{lem}

\begin{proof}
Like in the last definition, for any tuple $ (v_1,\dots,v_m) \in (\pr^3(K))^m $ with $ m \in \{n,n-1,n-2,n-3\} $ that appears in this proof let $ (V_1,\dots,V_m) \in (K^4)^m $ be an arbitrary lift. Define $ d_i := \det(V_i,V_{i+1},V_{i+2},V_{i+3}) $, where the indices are considered modulo $ n $. Let $ D_j := d_j d_{j+4} \dots d_{n-8+j} $ and $ D_j^\prime := d_j d_{j+4} \dots d_{n-12+j} $ with $ j \in \{1,2,3,4\} $. Also let $ E_1 := v_{m-3} \vee v_{m-2} \vee v_{m-1} $, $ E_2 := v_{m-2} \vee v_{m-1} \vee v_1 $, $ E_3 := v_{m-1} \vee v_1 \vee v_2 $, and $ E_4 := v_1 \vee v_2 \vee v_3 $. For most of the recursions we will consider a linear map $ \phi: K^4 \to K $ with the definition varying based on the recursion. We set $ H := \pr(\Ker \phi) $. This will help us to detect whether the determinant condition of the corresponding set is fulfilled.\smallskip\\
\eqref{eq:+++13}: Let $ (v_1,\dots,v_n) \in C_4^{13}(n) $. Then
\[
(v_1,\dots,v_{n-1}) \in C_4^{s_1s_2s_3}(n-1)
\]
with $ s_i \in \{+,-\} $. Define
\[
\phi: K^4 \to K; V \mapsto D_1 \det(V_{n-3},V_{n-2},V_{n-1},V) - D_3 \det(V_{n-1},V,V_1,V_2).
\]
Then for any $ (v_1,\dots,v_n) \in (\pr^3(K))^n $ the determinant condition $ D_1 d_{n-3} = D_3 d_{n-1} $ of $ C_4^{13}(n) $ is equivalent to $ v_n \in H $.\\
On the other hand, if we choose any $ (v_1,\dots,v_{n-1}) \in C_4^{s_1s_2s_3}(n-1) $, we know from the proof of \eqref{eq:k4+++} that $ (v_1,\dots,v_n) \in C_4(n) $ if and only if $ v_n \in \pr^3(K) \backslash (E_1 \cup E_2 \cup E_3 \cup E_4) $. It follows that we have $ (v_1,\dots,v_n) \in C_4^{13}(n) $ if and only if $ v_n \in H \backslash (E_1 \cup E_2 \cup E_3 \cup E_4) $.\smallskip\\
Case 1: $ s_1 = + $ or $ s_2 = + $.\\
If $ s_1 = + $, then $ \det(V_{n-3},V_{n-2},V_{n-1},V_1) \neq 0 $ and it follows that $ \phi(V_1) \neq 0 $. Hence, $ \dim H = 2 $. If $ s_1 = - $ we must have $ s_2 = + $. Then $ E_1 = E_2 $, but $ v_2 \notin E_2 = E_1 $ and therefore $ \det(V_{n-3},V_{n-2},V_{n-1},V_2) \neq 0 $. Thus, $ \phi(V_2) \neq 0 $ and again $ \dim H = 2 $. So we will always have $ \dim H = 2 $ in Case $ 1 $.\\
If $ v \in E_1 $ and $ V $ is a lift of $ v $, then one of the determinant products in the definition of $ \phi(V) $ vanishes. Accordingly, we have $ v \in H $ if and only if the other determinant product also vanishes, that is if $ v \in E_3 $. This argument also works the other way around. Hence, $ E_1 \cap H = E_1 \cap E_3 = E_3 \cap H $. Moreover, $ s_1 = + $ or $ s_2 = + $ guaranties $ E_1 \neq E_3 $, thus we have $ \dim E_1 \cap E_3 = 1 $. If $ s_1 = - $ or $ s_2 = - $, then $ E_1 = E_2 $ or $ E_2 = E_3 $ and consequently, $ E_1 \cap H = E_2 \cap H = E_3 \cap H $. On the other hand, if $ s_1 = s_2 = + $, then
\[
E_1 \cap E_2 \cap H = (v_{n-2} \vee v_{n-1}) \cap H = \{v_{n-1}\}.
\]
In particular, $ E_2 \cap H $ is then a projective line and $ E_2 \cap H \neq E_1 \cap H = E_3 \cap H $.\\
If $ s_4 = - $, then $ E_3 = E_4 $ and thus $ E_3 \cap H = E_4 \cap H $. If $ s_4 = + $, then
\[
E_4 \cap E_3 \cap H = (v_1 \vee v_2) \cap H,
\]
and since we have $ v_1 \notin H $ or $ v_2 \notin H $, this is just a single point. Hence, $ E_4 \cap H $ is also a projective line and $ E_4 \cap H \neq E_3 \cap H $. If $ s_1 = s_2 = s_3 = + $ also note that
\[
E_2 \cap E_4 \cap E_3 \cap H = E_2 \cap E_4 \cap E_1 \cap E_3 = \emptyset
\]
(as we showed in the proof of \eqref{eq:k4+++}). Hence the three lines $ E_2 \cap H $, $ E_3 \cap H $ and $ E_4 \cap H $ do not have a common intersection in this case.
\\
In conclusion, the hyperplanes $ E_1, E_2, E_3, E_4 $ intersect $ H $ in one, two or three different lines depending on $ s_1, s_2 $ and $ s_3 $: If $ s_1 = s_2 = s_3 = + $ there are three lines, if $ s_1 \neq s_2 $ and $ s_3 = + $, or if $ s_1 = s_2 = + $ and $ s_3 = - $ there are two lines, and if $ s_1 \neq s_2 $ and $ s_3 = - $ there is only one line. Since the three lines do not have a common intersection point in the first case, this corresponds to $ (q-1)^2 $ resp. $ q(q-1) $ resp. $ q^2 $ elements in $ H \backslash (E_1 \cup E_2 \cup E_3 \cup E_4) $. The number of elements of $ C_4^{13}(n) $ with $ s_1 = + $ or $ s_2 = + $ are thus
\[
(q-1)^2c_4(n-1) + 3q(q-1)c_4^{++-}(n-1) + q^2c_4^{+--}(n-1) + q^2c_4^{-+-}(n-1),
\]
where we used $ c_4^{-++}(n-1) = c_4^{+-+}(n-1) = c_4^{++-}(n-1) $.\smallskip\\
Case 2: $ s_1 = s_2 = - $.\\
Then $ E_1 = E_3 \subseteq H $. Since $ v_n \in H \backslash E_1 $, $ E_1 = E_3 $ must be a proper subset and this is only possible if $ H = \pr^3(K) $ or equivalently, if $ \phi = 0 $. Hence we have a necessary condition $ \phi = 0 $ for $ (v_1,\dots,v_{n-1}) \in C_4^{--s_3}(n-1) $; if this condition is not fulfilled, there will be no suitable choices for $ v_n $ at all. Hence we can restrict ourselves to only consider values for $ (v_1,\dots,v_{n-1}) $ that fulfill this condition.\\
We have $ v_{n-4} \notin E_1 $ and it follows that $ \phi = 0 $ if and only if $ \phi(V_{n-4}) = 0 $. We can rewrite this condition as
\[
D_1 \det(V_{n-4},V_{n-3},V_{n-2},V_{n-1}) = D_3 \det(V_{n-4},V_{n-1},V_1,V_2).
\]
If $ s_3 = - $ this is immediately the condition for $ (v_1,\dots,v_{n-1}) \in C_4^{---13}(n-1) $. If $ s_3 = + $ we need to use the bijection
\[
C_4^{--+}(n-1) \to C_4^{+--}(n-1); (v_1,\dots,v_{n-1}) \mapsto (v_2,\dots,v_{n-1},v_1).
\]
The condition then becomes
\[
D_2 \det(V_{n-3},V_{n-2},V_{n-1},V_1) = D_4 \det(V_{n-3},V_1,V_2,V_3),
\]
which is exactly the determinant condition for $ C_4^{+--24}(n-1) $. In either case the determinant condition has become trivial and we thus have the same number of choices for $ v_n $ as in the corresponding cases of the proof of \eqref{eq:k4+++}, which is to say $ q^2(q-1) $ resp. $ q^3 $. The claim follows.\medskip\\
\eqref{eq:++-13}: Let $ (v_1,\dots,v_n) \in C_4^{++-13}(n) $. Then
\[
(v_1,\dots,v_{n-1}) \in C_4^{s_1s_2+}(n-1)
\]
with $ s_i \in \{+,-\} $. Define the linear map
\[
\phi: K^4 \to K; V \mapsto D_1 \det(V_{n-3},V_{n-2},V_{n-1},V) - D_3 \det(V_{n-1},V,V_1,V_2).
\]
The determinant condition $ D_1 d_{n-3} = D_3 d_{n-1} $ of $ C_4^{++-13}(n) $ is fulfilled if and only if $ v_n \in H $.\\
On the other hand, for any $ (v_1,\dots,v_{n-1}) \in C_4^{s_1s_2+}(n-1) $, the suitable choices for $ v_n $ are the elements of $ H \cap E_4 \backslash (E_1 \cup E_2 \cup E_3) $.\smallskip\\
Case 1: $ s_1 = + $ or $ s_2 = + $.\\
If $ s_1 = + $, then $ \phi(V_1) \neq 0 $. Hence, $ \dim H = 2 $. If $ s_1 = - $ but $ s_2 = + $, then $ \phi(V_2) \neq 0 $ and again $ \dim H = 2 $. So we will always have $ \dim H = 2 $ in Case 1. Moreover, $ v_1 \notin H $ or $ v_2 \notin H $ implies $ H \neq E_4 $ and thus $ G := H \cap E_4 $ is a projective line.\\
Like in the proof of the previous recursion, we have $ E_1 \cap H = E_1 \cap E_3 = E_3 \cap H $. Hence,
\begin{align*}
E_1 \cap G &= E_1 \cap H \cap E_4 = E_1 \cap E_3 \cap E_4 = E_1 \cap (v_1 \vee v_2)\\
&= E_3 \cap G,
\end{align*}
and $ \lvert E_1 \cap (v_1 \vee v_2) \rvert = 1 $ since $ v_1 \notin E_1 $ if $ s_1 = + $ and $ v_2 \notin E_1 $ if $ s_1 = - $ and $ s_2 = + $.\\
If $ s_1 \neq s_2 $, then $ E_1 = E_2 $ or $ E_2 = E_3 $ and consequently, $ E_1 \cap G = E_2 \cap G = E_3 \cap G $. On the other hand if $ s_1 = s_2 = + $, then
\[
E_1 \cap E_2 \cap G = (v_{n-2} \vee v_{n-1}) \cap H \cap E_4 = \{v_{n-1}\} \cap E_4 = \emptyset.
\]
In particular, $ \lvert E_2 \cap G \rvert = 1 $ and $ E_2 \cap G \neq E_1 \cap G = E_3 \cap G $. Hence, $ G \backslash (E_1 \cup E_2 \cup E_3) $ has $ q-1 $ possible values for $ v_n $ if $ s_1 = s_2 = + $ and $ q $ otherwise. The number of elements of $ C_4^{++-13}(n) $ with $ s_1 = + $ or $ s_2 = + $ are thus
\[
(q-1)c_4(n-1) + 2qc_4^{++-}(n-1),
\]
using $ c_4^{-++}(n-1) = c_4^{+-+}(n-1) = c_4^{++-}(n-1) $.\smallskip\\
Case 2: $ s_1 = s_2 = - $.\\
Then $ E_1 = E_3 \subseteq H $. We cannot have equality, therefore $ (v_1,\dots,v_{n-1}) \in C_4^{--+}(n-1) $ must fulfill the condition $ \phi = 0 $. We have $ v_{n-4} \notin E_1 $ and thus $ \phi = 0 $ if and only if $ \phi(V_{n-4}) = 0 $. This yields the condition
\[
D_1 \det(V_{n-4},V_{n-3},V_{n-2},V_{n-1}) = D_3 \det(V_{n-4},V_{n-1},V_1,V_2).
\]
We need to use the bijection
\[
C_4^{--+}(n-1) \to C_4^{+--}(n-1); (v_1,\dots,v_{n-1}) \mapsto (v_2,\dots,v_{n-1},v_1).
\]
The condition then turns into
\[
D_2 \det(V_{n-3},V_{n-2},V_{n-1},V_1) = D_4 \det(V_{n-3},V_1,V_2,V_3),
\]
which is exactly the determinant condition for $ C_4^{+--24}(n-1) $. The determinant condition has become trivial and the number of choices for $ v_n $ is the same as in the proof of \eqref{eq:k4+++}, namely $ q^2 $. The claim is proved.\medskip\\
\eqref{eq:-+-13}: Let $ (v_1,\dots,v_n) \in C_4^{-+-13}(n) $. Then
\[
(v_1,\dots,v_{n-1}) \in C_4(n-1) \cup C_4^{(-)+[s]}(n-1)
\]
with $ s \in \{+,-\} $. Let
\[
\phi: K^4 \to K; V \mapsto D_1 \det(V_{n-3},V_{n-2},V_{n-1},V) - D_3 \det(V_{n-1},V,V_1,V_2).
\]
Once again, the determinant condition in $ C_4^{-+-13}(n) $ is fulfilled if and only if $ v_n \in H $.\\
Case 1: $ \{v_{n-2},v_{n-1},v_1\} $ independent.\\
This implies $ (v_1,\dots,v_{n-1}) \in C_4(n-1) $ and in particular, $ \phi(V_1) \neq 0 $. Moreover, $ v_1 \notin H $ and $ E_2 \cap E_4 $ is a projective line containing $ v_1 $. Then $ \{P\} := H \cap E_2 \cap E_4 $ must be a single point. On the other hand, we have $ E_1 \cap H = E_1 \cap E_3 = E_3 \cap H $. Thus
\begin{align*}
E_1 \cap \{P\} &= E_1 \cap H \cap E_2 \cap E_4 = E_1 \cap E_3 \cap E_2 \cap E_4 = \emptyset\\
&= E_3 \cap \{P\}.
\end{align*}
It follows that $ \{P\} = H \cap E_2 \cap E_4 \backslash (E_1 \cup E_3) $ and thus there is exactly one value of $ v_n $ for any $ (v_1,\dots,v_{n-1}) \in C_4(n-1) $ such that $ (v_1,\dots,v_n) \in C_4^{-+-13}(n) $ and it follows that there are $ c_4(n-1) $ elements of $ C_4^{-+-13}(n) $ that fall in Case 1.\\
Case 2: $ \{v_{n-2},v_{n-1},v_1\} $ dependent and $ s = + $.\\
$ s = + $ implies $ \{v_{n-3},v_{n-2},v_{n-1},v_2\} $ independent, hence $ \phi(V_2) \neq 0 $ and $ \dim H = 2 $. $ v_2 \notin H $ means that $ G := H \cap E_4 $ is a projective line. The valid choices of $ v_n $ for a given $ (v_1,\dots,v_{n-1}) \in C_4^{(-)+[+]}(n-1) $ are the elements of $ G \backslash (E_1 \cup E_3) $ (we do not need to consider $ E_2 $ since $ \{v_{n-2},v_{n-1},v_1\} $ dependent makes the condition $ \{v_{n-2},v_{n-1},v_n,v_1\} $ dependent trivial). We still have $ E_1 \cap H = E_1 \cap E_3 = E_3 \cap H $ and thus
\begin{align*}
E_1 \cap G &= E_1 \cap H \cap E_4 = E_1 \cap E_3 \cap E_4 = E_1 \cap (v_1 \vee v_2) = \{v_1\}\\
&= E_3 \cap G.
\end{align*}
Hence there are $ q $ valid choices for $ v_n $ and the number of elements falling in Case 2 is $ qc_4^{+(-)[+]}(n-1) $ (using $ c_4^{(-)+[+]}(n-1) = c_4^{+(-)[+]}(n-1) $).\\
Case 3: $ \{v_{n-2},v_{n-1},v_1\} $ dependent and $ s = - $.\\
$ \{v_{n-2},v_{n-1},v_1\} $ dependent implies $ v_1 \in E_1 $ and $ s = - $ implies $ v_2 \in E_1 $. It follows that $ E_1 = E_3 \subseteq H $. We cannot have equality and therefore we get $ \phi = 0 $ as a necessary condition on $ (v_1,\dots,v_{n-1}) $. We have $ v_3 \notin E_3 $ since $ \{v_{n-1},v_1,v_2,v_3\} $ is independent and thus $ \phi = 0 $ is equivalent to $ \phi(V_3) = 0 $, which we can also rewrite as
\[
D_1 \det(V_{n-3},V_{n-2},V_{n-1},V_3) = D_3 \det(V_{n-1},V_1,V_2,V_3).
\]
Now we use the bijection
\[
C_4^{(-)+[-]}(n-1) \to C_4^{+(-)[-]}; (v_1,\dots,v_{n-1}) \mapsto (v_{n-1},\dots,v_1),
\]
which converts our condition to
\[
D_4 \det(V_{n-3},V_1,V_2,V_3) = D_2 \det(V_{n-3},V_{n-2},V_{n-1},V_1).
\]
As this is the determinant condition in $ C_4^{+(-)[-]24}(n-1) $ and as there are $ q^2 $ options for $ v_n $ like in the proof of \eqref{eq:k4-+-}, the claimed recursion follows.
\medskip\\
\eqref{eq:+--24}: Let $ (v_1,\dots,v_{n-1}) \in C_4^{+--24}(n-1) $. Then
\[
(v_1,\dots,v_{n-2}) \in C_4^{s+-}(n-2) \cup C_4^{s(-)[+]}(n-2)
\]
with $ s \in \{+,-\} $. Consider the map
\begin{align*}
\phi: K^4 &\to K\\
V &\mapsto D_2 \det(V_{n-3},V_{n-2},V,V_1) - D_4^\prime \det(V_{n-4},V_{n-3},V_{n-2},V) \det(V_{n-3},V_1,V_2,V_3).
\end{align*}
The determinant condition in $ C_4^{+--24}(n-1) $ is fulfilled if and only if $ v_{n-1} \in H $.\\
We start with the case $ s = + $. Then $ \phi(V_1) \neq 0 $ and $ \dim H = 2 $. Moreover, $ v_1 \notin H $ implies that $ G := H \cap E_4 $ is a projective line. The suitable choices for $ v_{n-1} $ are the elements of $ G \backslash (E_1 \cup E_2 \cup (v_1 \vee v_2)) $. We also have $ E_1 \cap H = E_1 \cap E_2 = v_{n-3} \vee v_{n-2} = E_2 \cap H $. If $ (v_1,\dots,v_{n-2}) \in C_4^{++-}(n-2) $, then $ v_{n-3} \notin E_3 = E_4 $ and $ v_{n-2} \in E_3 = E_4 $. If $ (v_1,\dots,v_{n-2}) \in C_4^{+(-)[+]}(n-2) $, then the $ [+] $ ensures $ v_{n-3} \notin E_4 $ and the $ (-) $ ensures $ v_{n-2} \in v_1 \vee v_2 \subset E_4 $. In either case, it follows that
\[
E_1 \cap G = E_2 \cap G = (v_{n-3} \vee v_{n-2}) \cap E_4 = \{v_{n-2}\}.
\]
$ v_1 \notin H $ implies $ v_1 \notin G $, hence $ \lvert G \cap (v_1 \vee v_2) \rvert = 1 $ (the intersection cannot be empty since both lines are contained in $ E_4 $). If $ \{v_{n-2},v_1,v_2\} $ is independent, then $ G \cap (v_1 \vee v_2) \neq \{v_{n-2}\} $, otherwise $ G \cap (v_1 \vee v_2) = \{v_{n-2}\} $. Accordingly, we have $ q-1 $ resp. $ q $ choices for $ v_{n-1} $ depending on whether $ \{v_{n-2},v_1,v_2\} $ is independent, and the number of elements of $ C_4^{+--24}(n-1) $ with $ s = + $ is
\[
(q-1)c_4^{++-}(n-2) + qc_4^{+(-)[+]}(n-2).
\]
Now let $ s = - $. Then $ E_1 = E_2 \subseteq H $ and we cannot have equality. Hence, $ \phi = 0 $. If $ \{v_{n-2},v_1,v_2\} $ is independent, $ v_2 \notin E_2 $ and $ \phi = 0 $ is equivalent to $ \phi(V_2) = 0 $. Then the necessary condition $ \phi = 0 $ turns into
\[
D_2 \det(V_{n-3},V_{n-2},V_1,V_2) = - D_4^\prime \det(V_{n-4},V_{n-3},V_{n-2},V_2) \det(V_{n-3},V_1,V_2,V_3),
\]
which is exactly the condition for $ (v_1,\dots,v_{n-2}) \in C_4^{-+-24}(n-2) $.\\
On the other hand, if $ \{v_{n-2},v_1,v_2\} $ is dependent, then $ v_2 \in E_2 $, but $ v_3 \notin E_1 = E_2 $. Hence we can test $ \phi = 0 $ with $ V_3 $ and get the condition
\[
D_2 \det(V_{n-3},V_{n-2},V_1,V_3) = - D_4^\prime \det(V_{n-4},V_{n-3},V_{n-2},V_3) \det(V_{n-3},V_1,V_2,V_3).
\]
This is the condition for $ (v_1,\dots,v_{n-2}) \in C_4^{-(-)[+]24}(n-2) $. In either case, the determinant condition becomes trivial and like in the proof of \eqref{eq:k4+--} we have $ q(q-1) $ resp. $ q^2 $ choices for $ v_{n-1} $. Thus the recursion holds.\medskip\\
\eqref{eq:---13}: Let $ (v_1,\dots,v_{n-1}) \in C_4^{---13}(n-1) $. Then
\[
(v_1,\dots,v_{n-2}) \in C_4^{+--}(n-2) \cup C_4^{+(-)[-]}(n-2) \cup C_4^{(-)-[+]}(n-2) \cup C_4^{(--)[+]}(n-2).
\]
Let $ M_1 := \{v_{n-3},v_{n-2},v_1\} $ and $ M_2 := \{v_{n-2},v_1,v_2\} $. Also let
\[
\phi: K^4 \to K; V \mapsto D_1 \det(V_{n-4},V_{n-3},V_{n-2},V) - D_3 \det(V_{n-4},V,V_1,V_2).
\]
As usual, the determinant condition on $ C_4^{---13}(n-1) $ is equivalent to $ v_{n-1} \in H $.\\
Case 1: $ M_1 $ independent or $ M_2 $ independent.\\
If $ M_1 $ is independent, then $ \{v_{n-4},v_{n-3},v_{n-2},v_1\} $ is independent and thus $ \phi(V_1) \neq 0 $. If $ M_1 $ is dependent but $ M_2 $ is independent, then $ \{v_{n-4},v_{n-3},v_{n-2},v_2\} $ is independent and thus $ \phi(V_2) \neq 0 $. In either case we have $ \dim H = 2 $ and $ H \neq E_4 $. Therefore, $ G := H \cap E_4 $ is a projective line. The viable choices for $ v_{n-1} $ are exactly the elements of $ G \backslash (E_1 \cup (v_{n-2} \vee v_1) \cup (v_1 \vee v_2)) $. We have $ E_1 \cap H = E_1 \cap (v_{n-4} \vee v_1 \vee v_2) $; it follows that
\begin{align*}
E_1 \cap G &= E_1 \cap H \cap E_4 = E_1 \cap (v_{n-4} \vee v_1 \vee v_2) \cap E_4 = E_1 \cap (v_1 \vee v_2)\\
&= E_1 \cap G \cap (v_1 \vee v_2) = (v_1 \vee v_2) \cap G.
\end{align*}
This intersection is a single point since $ v_1 \vee v_2 \not\subset G $.\\
If $ M_2 $ is dependent, then $ v_{n-2} \vee v_1 = v_1 \vee v_2 $ and in particular $ (v_{n-2} \vee v_1) \cap G = (v_1 \vee v_2) \cap G $. Otherwise we have
\[
(v_{n-2} \vee v_1) \cap (v_1 \vee v_2) \cap G = \{v_1\} \cap G.
\]
This is nonempty if and only if $ M_1 $ is dependent. Hence we have $ q-1 $ choices for $ v_{n-1} $ if $ M_1 $ and $ M_2 $ are both independent and $ q $ choices otherwise. This shows that there are
\[
(q-1)c_4^{+--}(n-2) + qc_4^{+(-)[-]}(n-2) + qc_4^{-(-)[+]}(n-2)
\]
elements of $ C_4^{---13}(n-1) $ falling into Case 1. Here we used $ c_4^{(-)-[+]}(n-2) = c_4^{-(-)[+]}(n-2) $.\\
Case 2: $ M_1 $ and $ M_2 $ are both dependent.\\
Then $ E_1 = v_{n-4} \vee v_1 \vee v_2 \subseteq H $, and since $ v_{n-1} \in H \backslash E_1 $ we cannot have equality. Hence $ \phi = 0 $; and this necessary condition is equivalent to $ \phi(V_{n-5}) = 0 $, because $ v_{n-5} \notin E_1 $. After canceling out $ d_{n-5} $, the condition becomes
\[
D_1 = D_3^\prime \det(V_{n-5},V_{n-4},V_1,V_2).
\]
This is also the condition for $ (v_1,\dots,v_{n-2}) \in C_4^{(--)[+]13}(n-2) $. As the determinant condition becomes trivial, there are $ q^2 $ choices for $ v_{n-1} $ like in the proof of \eqref{eq:k4---}. The claimed recursion follows.\medskip\\
\eqref{eq:+(-)[-]24}: Let $ (v_1,\dots,v_{n-1}) \in C_4^{+(-)[-]24}(n-1) $. Then
\[
(v_1,\dots,v_{n-2}) \in C_4^{s+-}(n-2).
\]
with $ s \in \{+,-\} $. Define the map
\begin{align*}
\phi: K^4 &\to K\\
V &\mapsto D_2\det(V_{n-3},V_{n-2},V,V_1) - D_4^\prime \det(V_{n-4},V_{n-3},V_{n-2},V) \det(V_{n-3},V_1,V_2,V_3).
\end{align*}
We have $ v_{n-1} \in H $ if and only if the determinant condition in $ C_4^{+(-)[-]24}(n-1) $ is fulfilled.\\
Let $ s = + $. Consequently, $ \phi(V_1) \neq 0 $ and $ \dim H = 2 $. The intersection $ \{P\} := H \cap (v_1 \vee v_2) $ then contains only one point. We have $ E_1 \cap H = E_1 \cap E_2 = E_2 \cap H $ and thus
\begin{align*}
E_1 \cap \{P\} &= E_1 \cap H \cap (v_1 \vee v_2) = E_1 \cap E_2 \cap (v_1 \vee v_2) = (v_{n-3} \vee v_{n-2}) \cap (v_1 \vee v_2) = \emptyset\\
&= E_2 \cap \{P\},
\end{align*}
since $ \{v_{n-3},v_{n-2},v_1,v_2\} $ is independent. This proves that $ v_{n-1} = P $ meets all requirements for $ (v_1,\dots,v_{n-1}) \in C_4^{+(-)[-]24}(n-1) $ for all $ (v_1,\dots,v_{n-2}) \in C_4^{++-}(n-2) $. Hence there are $ c_4^{++-}(n-2) $ elements with $ s = + $.\\
Now let $ s = - $. Then $ E_1 = E_2 \subseteq H $ and equality is impossible. As usual, we get the necessary condition $ \phi = 0 $ for $ (v_1,\dots,v_{n-2}) $ and using $ v_2 \notin E_2 $ this condition is fulfilled if and only if $ \phi(V_2) = 0 $. This is in turn equivalent to
\[
D_2 \det(V_{n-3},V_{n-2},V_1,V_2) = - D_4^\prime \det(V_{n-4},V_{n-3},V_{n-2},V_2) \det(V_{n-3},V_1,V_2,V_3),
\]
which is the condition for $ (v_1,\dots,v_{n-2}) \in C_4^{-+-24}(n-2) $. We have $ q $ choices for $ v_{n-2} $ as in the proof of \eqref{eq:k4+(-)[-]}, because the determinant condition becomes trivial. The claim follows.
\medskip\\
\eqref{eq:-+-24}: Let $ (v_1,\dots,v_{n-2}) \in C_4^{-+-24}(n-2) $. Then
\[
(v_1,\dots,v_{n-3}) \in C_4(n-3) \cup C_4^{(-)+[s]}(n-3)
\]
with $ s \in \{+,-\} $. Define the map
\[
\phi: K^4 \to K; V \mapsto D_2 \det(V_{n-3},V,V_1,V_2) + D_4^\prime \det(V_{n-4},V_{n-3},V,V_2)\det(V_{n-3},V_1,V_2,V_3).
\]
Clearly we have $ v_{n-2} \in H $ for $ v_{n-2} \in \pr^3(K) $ if and only if the determinant condition in $ C_4^{-+-24}(n-2) $ is fulfilled.\\
First assume $ \{v_{n-4},v_{n-3},v_1\} $ independent. Then $ (v_1,\dots,v_{n-3}) \in C_4(n-3) $. In particular, $ \{v_{n-4},v_{n-3},v_1,v_2\} $ is independent and therefore $ \phi(V_1) \neq 0 $. It follows that $ v_1 \notin H $. Since $ E_2 \cap E_4 $ is a projective line ($ E_2 \neq E_4 $) with $ v_1 \in E_2 \cap E_4 $, the intersection $ \{P\} := H \cap E_2 \cap E_4 $ is only a single point. Accordingly, there is at most one choice of $ v_{n-2} $ for any $ (v_1,\dots,v_{n-3}) \in C_4(n-3) $ with $ (v_1,\dots,v_{n-2}) \in C_4^{-+-24}(n-2) $. On the other hand, $ P $ is a viable choice if $ P \notin E_1, E_3 $. We have
\[
E_3 \cap \{P\} = E_2 \cap H \cap E_3 \cap E_4 = E_2 \cap H \cap (v_1 \vee v_2) = E_2 \cap \{v_2\} = \emptyset
\]
and
\[
E_1 \cap \{P\} = E_1 \cap E_2 \cap H \cap E_4 = (v_{n-4} \vee v_{n-3}) \cap H \cap E_4 = \{v_{n-3}\} \cap E_4 = \emptyset.
\]
Thus there are exactly $ c_4(n-3) $ elements of $ C_4^{-+-24}(n-2) $ with $ \{v_{n-4},v_{n-3},v_1\} $ independent.\\
Now let $ \{v_{n-4},v_{n-3},v_1\} $ be dependent. Then $ E_3 = v_{n-3} \vee v_1 \vee v_2 = v_{n-4} \vee v_{n-3} \vee v_2 \subseteq H $. Equality is prevented by $ v_{n-2} \in H \backslash E_3 $, hence we must have $ \phi = 0 $ as a necessary condition on $ (v_1,\dots,v_{n-3}) $. As $ v_3 \notin E_3 $, this condition is equivalent to $ \phi(V_3) = 0 $. This is in turn equivalent to
\[
D_2 = D_4^\prime \det(V_{n-4},V_{n-3},V_2,V_3)
\]
after canceling out $ \det(V_{n-3},V_1,V_2,V_3) $. Now we apply the bijection
\[
C_4^{(-)+[s]}(n-3) \to C_4^{+(-)[s]}(n-3); (v_1,\dots,v_{n-3}) \mapsto (v_{n-3},\dots,v_1)
\]
and our condition becomes
\[
D_1 = D_3^\prime \det(V_{n-5},V_{n-4},V_1,V_2).
\]
This is also the determinant condition on $ C_4^{+(-)[s]13}(n-3) $. Moreover, the determinant condition on $ C_4^{-+-24}(n-2) $ becomes trivial because of $ \phi = 0 $ and like in the proof of \eqref{eq:k4-+-} we have $ q(q-1) $ resp. $ q^2 $ choices for $ v_{n-2} $ depending on $ s $. This concludes the proof \eqref{eq:-+-24}.\medskip\\
\eqref{eq:-(-)[+]24}: Let $ (v_1,\dots,v_{n-2}) \in C_4^{-(-)[+]24}(n-2) $. Then
\[
(v_1,\dots,v_{n-3}) \in C_4^{+-+}(n-3) \cup C_4^{(-)+[+]}(n-3).
\]
Let
\[
\phi: K^4 \to K; V \mapsto D_2 \det(V_{n-3},V,V_1,V_3) + D_4^\prime \det(V_{n-4},V_{n-3},V,V_3) \det(V_{n-3},V_1,V_2,V_3).
\]
Then $ v_{n-2} \in H $ is equivalent to the determinant condition on $ C_4^{-(-)[+]24}(n-2) $.\\
Let $ \{v_{n-4},v_{n-3},v_1\} $ be independent. In this case we have $ \{v_{n-3},v_1,v_2,v_3\} $ independent and $ v_{n-4} \vee v_{n-3} \vee v_1 = v_{n-3} \vee v_1 \vee v_2 $. This implies $ \{v_{n-4},v_{n-3},v_1,v_3\} $ independent, hence $ \phi(V_1) \neq 0 $ and $ \dim H = 2 $. Moreover, since $ v_1 \notin H $, $ \{P\} := H \cap (v_1 \vee v_2) $ is a single point. $ P $ is the only projective point that could be a suitable choice for $ v_{n-2} $. In order for $ v_{n-2} = P $ to actually work, we need $ P \notin E_1 $, $ P \in E_2 $, and $ P \notin v_{n-3} \vee v_1 $. We have $ P \in v_1 \vee v_2 \subset E_3 = E_2 $, hence $ P \in E_2 $. Moreover, since $ \{v_{n-4},v_{n-3},v_1,v_2\} $ is dependent, $ (v_{n-4} \vee v_{n-3}) \cap (v_1 \vee v_2) \neq \emptyset $. $ E_1 \cap (v_1 \vee v_2) $ is only one point, because $ v_1 \notin E_1 $. Then $ v_{n-4} \vee v_{n-3} \subset E_1 $ implies $ E_1 \cap (v_1 \vee v_2) = (v_{n-4} \vee v_{n-3}) \cap (v_1 \vee v_2) $. Thus
\[
E_1 \cap \{P\} = E_1 \cap (v_1 \vee v_2) \cap H = (v_{n-4} \vee v_{n-3}) \cap (v_1 \vee v_2) \cap H = (v_1 \vee v_2) \cap \{v_{n-3}\} = \emptyset.
\]
And
\[
(v_{n-3} \vee v_1) \cap \{P\} = (v_{n-3} \vee v_1) \cap (v_1 \vee v_2) \cap H = \{v_1\} \cap H = \emptyset.
\]
hence $ P $ is indeed a suitable choice and we have exactly $ c_4^{+-+}(n-3) = c_4^{++-}(n-3) $ elements falling into the case $ \{v_{n-4},v_{n-3},v_1\} $ independent.\\
Now let $ \{v_{n-4},v_{n-3},v_1\} $ be dependent. Then $ v_{n-3} \vee v_1 \vee v_3 = v_{n-4} \vee v_{n-3} \vee v_3 \subseteq H $. We have $ \{v_{n-3},v_1,v_2,v_3\} $ independent (because of the $ [+] $), which implies $ (v_1 \vee v_2) \cap (v_{n-3} \vee v_1 \vee v_3) = \{v_1\} $. Since $ v_{n-2} \in H \cap (v_1 \vee v_2) \backslash \{v_1\} $, we must have $ \phi = 0 $. This condition is equivalent to $ \phi(V_2) = 0 $, which we can rewrite as
\[
D_2 = D_4^\prime \det(V_{n-4},V_{n-3},V_2,V_3)
\]
after canceling out $ \det(V_{n-3},V_1,V_2,V_3) $. Now we apply the bijection
\[
C_4^{(-)+[+]}(n-3) \to C_4^{+(-)[+]}(n-3); (v_1,\dots,v_{n-3}) \mapsto (v_{n-3},\dots,v_1).
\]
The condition becomes
\[
D_1 = D_3^\prime \det(V_{n-5},V_{n-4},V_1,V_2),
\]
which is the determinant condition for $ C_4^{+(-)[+]13}(n-3) $. The number of choices for $ v_{n-2} $ is $ q $ just like in the proof of \eqref{eq:k4-(-)}. The claim follows.\medskip\\
Finally, consider \eqref{eq:(--)[+]13}, \eqref{eq:+(-)[+]13} and \eqref{eq:+(-)[-]13}. The corresponding sets are $ C_4^{(--)[+]13}(n-2) $ and $ C_4^{+(-)[s]}(n-3) $ with $ s \in \{+,-\} $. Let $ (v_1,\dots,v_{n-2}) \in C_4^{(--)[+]13} $ resp. $ (v_1,\dots,v_{n-3}) \in C_4^{+(-)[s]13}(n-3) $. All of these sets have the same determinant condition, namely
\begin{equation}\label{eq:13detCondition}
D_1 = D_3^\prime \det(V_{n-5},V_{n-4},V_1,V_2).
\end{equation}
This equation does not depend on $ v_{n-2} $ or $ v_{n-3} $, hence we have the same recursions as \eqref{eq:k4(--)}/\eqref{eq:k4+(-)[+]}/\eqref{eq:k4+(-)[-]}, except that the sets on the right hand side inherit the condition \eqref{eq:13detCondition}. The only case that requires special attention is \eqref{eq:+(-)[+]13}, where we need to use that the bijection
\[
C_4^{-++}(n-4) \to C_4^{++-}(n-4); (v_1,\dots,v_{n-4}) \mapsto (v_3,\dots,v_{n-4},v_1,v_2)
\]
transforms the condition \eqref{eq:13detCondition} into itself.
\end{proof}

Next we will consider sets sets with two determinant conditions, specifically $ C_4^{234\pm}(n) $ and some related sets.

\begin{defi}
Let $ n \geq 4 $ be divisible by 4. For any tuple $ (v_1,\dots,v_m) \in (\pr^3(K))^m $ with $ m \in \{n,n-1,n-2,n-3\} $ in the definitions of the following sets let $ (V_1,\dots,V_m) \in (K^4)^m $ be an arbitrary lift. Define $ d_i := \det(V_i,V_{i+1},V_{i+2},V_{i+3}) $ with the indices considered modulo $ n $. Let $ D_j^\prime := d_j d_{j+4} \dots d_{n-12+j} $ and $ D_ j := D_j^\prime d_{n-8+j} $ for $ j \in \{1,2,3,4\} $. Let $ s \in \{+,-\} $ Then define
\begin{align*}
C_4^{234s}(n) := \{&(v_1,\dots,v_n) \in C_4(n) \mid D_2d_{n-2} = sD_3 d_{n-1} = D_4 d_n\},\\
C_4^{++-123s}(n) := \{&(v_1,\dots,v_n) \in C_4^{++-}(n) \mid D_1 d_{n-3} = sD_2 d_{n-2} = D_3d_{n-1}\},\\
C_4^{+--234s}(n-1) := \{&(v_1,\dots,v_{n-1}) \in C_4^{+--}(n-1) \mid D_2 \det(V_{n-3},V_{n-2},V_{n-1},V_1)\\
&=sD_3\det(V_{n-3},V_{n-1},V_1,V_2) = D_4 \det(V_{n-3},V_1,V_2,V_3)\},\\
C_4^{---234s}(n-1) := \{&(v_1,\dots,v_{n-1}) \in C_4^{---}(n-1) \mid D_2 \det(V_{n-4},V_{n-2},V_{n-1},V_1)\\
&= sD_3 \det(V_{n-4},V_{n-1},V_1,V_2) = D_4 \det(V_{n-4},V_1,V_2,V_3)\},\\
C_4^{-(-)[+]234s}(n-2) := \{&(v_1,\dots,v_{n-2}) \in C_4^{-(-)[+]}(n-2) \mid D_2 \det(V_{n-5},V_{n-3},V_{n-2},V_1)\\
&= s D_3 \det(V_{n-5},V_{n-3},V_1,V_2) = -D_4^\prime d_{n-5}\det(V_{n-3},V_1,V_2,V_3)\},\\
C_4^{(--)[+]234s}(n-2) := \{&(v_1,\dots,v_{n-2}) \in C_4^{(--)[+]}(n-2) \mid D_2 \det(V_{n-5},V_{n-4},V_{n-2},V_1)\\
&= sD_3 \det(V_{n-5},V_{n-4},V_1,V_2) = - D_4^\prime d_{n-5} \det(V_{n-4},V_1,V_2,V_3)\},\\
C_4^{+(-)[+]134s}(n-3) := \{&(v_1,\dots,v_{n-3}) \in C_4^{+(-)[+]}(n-3) \mid D_1 = D_3^\prime \det(V_{n-5},V_{n-4},V_1,V_2)\\
&= sD_4^\prime \det(V_{n-4},V_1,V_2,V_3)\}.
\end{align*}
\end{defi}

\begin{lem}\label{lem:recursionsk4234}
Let $ n \geq 8 $ be divisible by $ 4 $.
\begin{align}
c_4^{234\pm}(n) = &(q-1)c_4(n-1) + qc_4^{++-}(n-1)\label{eq:+++234}\\
&+ 2q(q-1)c_4^{++-34\mp}(n-1) + q^2c_4^{-+-12\mp}(n-1)\nonumber\\
&+ q^2c_4^{+--23\mp}(n-1) + q^2(q-1)c_4^{+--234\mp}(n-1)\nonumber\\
&+ q^3c_4^{---234\mp}(n-1),\nonumber\\
c_4^{++-123\pm}(n) = &c_4(n-1) + 2qc_4^{++-34\mp}(n-1) + q^2c_4^{+--234\mp}(n-1),\label{eq:++-123}\\
c_4^{+--234}(n-1) = &c_4^{++-}(n-2) + qc_4^{-+-24}(n-2) + qc_4^{+(-)[+]23\mp}(n-2)\label{eq:+--234}\\
&+q^2c_4^{-(-)[+]234\mp}(n-2),\nonumber\\
c_4^{---234\pm}(n-1) = &c_4^{+--}(n-2) + qc_4^{+(-)[-]23\mp}(n-2)\label{eq:---234}\\
&+ qc_4^{-(-)[+]24}(n-2) + q^2c_4^{(--)[+]234\mp}(n-2),\nonumber\\
c_4^{-(-)[+]234\pm}(n-2) = &c_4^{++-14\mp}(n-3) + qc_4^{+(-)[+]134\mp}(n-3),\label{eq:-(-)[+]234}\\
c_4^{(--)[+]234\pm}(n-2) = &c_4^{+(-)[+]14\mp}(n-3) + qc_4^{234\mp}(n-4),\label{eq:(--)[+]234}\\
c_4^{+(-)[+]134\pm}(n-3) = &(q-1)c_4^{234\pm}(n-4) + qc_4^{++-123\pm}(n-4).\label{eq:+(-)[+]134}
\end{align}
\end{lem}

\begin{proof}
For any tuple $ (v_1,\dots,v_m) \in (\pr^3(K))^m $ of projective points with $ m \in \{n,n-1,n-2,m-3\} $ that we consider in this proof let $ (V_1,\dots,V_m) \in (K^4)^m $ be an arbitrary lift. We set $ d_i := \det(V_i,V_{i+1},V_{i+2},V_{i+3}) $ for $ i \leq n $ (the indices are considered modulo $ n $) and $ D_j := d_j d_{j+4} \dots d_{n-8+j} $ and $ D_j^\prime := d_j d_{j+4} \dots d_{n-12+j} $ for $ j \in \{1,2,3,4\} $. Also let $ E_1 := v_{m-3} \vee v_{m-2} \vee v_{m-1} $, $ E_2 := v_{m-2} \vee v_{m-1} \vee v_1 $, $ E_3 := v_{m-1} \vee v_1 \vee v_2 $ and $ E_4 := v_1 \vee v_2 \vee v_3 $. For most of the recursions we consider two linear maps $ \phi_1, \phi_2: K^4 \to K $ and their kernels $ H_i := \pr(\Ker \phi_i) $. The definition of $ \phi_i $ will be dependent of the choice of lift, but $ H_i $ will not be. These maps will have the property that the determinant conditions of the sets we are considering are fulfilled if and only if $ v_m \in H_1 \cap H_2 $. \smallskip\\
\eqref{eq:+++234}: Let $ (v_1,\dots,v_n) \in C_4^{234\pm}(n) $. Then
\[
(v_1,\dots,v_{n-1}) \in C_4^{s_1s_2s_3}(n-1)
\]
with $ s_i \in \{+,-\} $. We define
\[
\phi_1: K^4 \to K; V \mapsto D_2 \det(V_{n-2},V_{n-1},V,V_1) \mp D_3 \det(V_{n-1},V,V_1,V_2)
\]
and
\[
\phi_2: K^4 \to K; V \mapsto D_3 \det(V_{n-1},V,V_1,V_2) \mp D_4 \det(V,V_1,V_2,V_3).
\]
For a given $ (v_1,\dots,v_{n-1}) \in C_4^{s_1s_2s_3}(n-1) $ we have $ (v_1,\dots,v_n) \in C_4^{234\pm}(n) $ if and only if $ v_n \in H_1 \cap H_2 \backslash (E_1 \cup E_2 \cup E_3 \cup E_4) $.\\
Case 1: $ s_2 = s_3 = + $.\\
Then $ \det(V_{n-2},V_{n-1},V_1,V_2) \neq 0 $ and $ \det(V_{n-1},V_1,V_2,V_3) \neq 0 $ and it follows that $ \phi_1(V_2) \neq 0 $ and $ \phi_2(V_3) \neq 0 $. Hence $ \dim H_1 = \dim H_2 = 2 $. Let $ G := H_1 \cap H_2 $. Since $ v_2 \in H_2 \backslash H_1 $ we have $ \dim G = 1 $. Also note that $ E_2 \cap H_1 = E_2 \cap E_3 = E_3 \cap H_1 $ and likewise $ E_3 \cap H_2 = E_3 \cap E_4 = E_4 \cap H_2 $. Then
\begin{align*}
E_2 \cap G &= E_2 \cap H_1 \cap H_2 = E_2 \cap E_3 \cap E_4 = \{v_1\}\\
&= E_3 \cap G = E_4 \cap G.
\end{align*}
If $ s_1 = + $, then $ v_1 \notin E_1 $. Combined with $ v_1 \in G $ this implies that $ E_1 \cap G $ is a single point different from $ v_1 $. If $ s_1 = - $ we have $ E_1 = E_2 $ and thus $ E_1 \cap G = E_2 \cap G = \{v_1\} $. Therefore there are $ q-1 $ resp. $ q $ valid choices for $ v_n $ and the number of elements in Case 1 is
\[
(q-1)c_4(n-1) + qc_4^{++-}(n-1)
\]
using $ c_4^{-++}(n-1) = c_4^{++-}(n-1) $.\\
Case 2: $ s_2 = +, s_3 = - $.\\
We have $ E_3 = E_4 \subseteq H_2 $ in this case and $ v_n \in H_2 \backslash E_3 $ shows that $ H_2 = \pr^3(K) $ or equivalently $ \phi_2 = 0 $. At the same time we still have $ \phi_1(V_2) \neq 0 $ and thus $ \dim H_1 = 2 $. Since $ v_{n-2} \notin E_3 $, the necessary condition $ \phi_2 = 0 $ is equivalent to
\[
D_3 \det(V_{n-2},V_{n-1},V_1,V_2) = \mp D_4 \det(V_{n-2},V_1,V_2,V_3).
\]
If $ s_1 = + $ this is immediately the condition for $ (v_1,\dots,v_{n-1}) \in C_4^{++-34\mp}(n-1) $. If $ s_1 = - $ we use the bijection
\[
C_4^{-+-}(n-1) \to C_4^{-+-}(n-1); (v_1,\dots,v_{n-1}) \mapsto (v_{n-1},\dots,v_1),
\]
and our condition becomes
\[
D_2 \det(V_{n-2},V_{n-1},V_1,V_2) = \mp D_1 \det(V_{n-3},V_{n-2},V_{n-1},V_2).
\]
This is equivalent to the determinant condition for $ C_4^{-+-12\mp}(n-1) $. The definition of $ H_1 $ is the same as the definition of $ H $ in the proof of \eqref{eq:+++23}, hence the number of choices for $ v_n $ is also the same as in the case $ s_2 = + $ and $ s_3 = - $ there: $ q(q-1) $ choices if $ s_1 = + $ and $ q^2 $ if $ s_1 = - $. Accordingly, Case 2 contributes
\[
q(q-1)c_4^{++-34\mp}(n-1) + q^2c_4^{-+-12\mp}(n-1)
\]
elements.\\
Case 3: $ s_2 = -, s_3 = + $.\\
Then $ E_2 = E_3 \subseteq H_1 $ and since $ E_2 = H_1 $ is not possible we must have $ \phi_1 = 0 $. On the other hand, $ \phi_2(V_3) \neq 0 $ and thus $ \dim H_2 = 2 $. We still have
\begin{align*}
E_3 \cap H_2 &= E_3 \cap E_4 = v_1 \vee v_2\\
&= E_4 \cap H_2
\end{align*}
and $ E_2 \cap H_2 = E_3 \cap H_2 $. If $ s_1 = - $, then also $ E_1 = E_2 $ and thus $ E_1 \cap H_2 = E_2 \cap H_2 $. If $ s_1 = + $ then $ v_1 \notin E_1 $. In this case $ E_1 \cap H_2 $ is a projective line in $ H_2 $ different from $ v_1 \vee v_2 $. Hence we get $ q(q-1) $ choices for $ v_n $ if $ s_1 = + $ and $ q^2 $ choices if $ s_1 = - $, provided that $ (v_1,\dots,v_{n-1}) $ fulfills the necessary condition $ \phi_1 = 0 $. This condition is the same as the condition $ \phi = 0 $ in the proof of \eqref{eq:+++23}. Like we showed in that proof, the condition is equivalent to $ (v_2,\dots,v_{n-1},v_1) \in C_4^{++-34\mp}(n-1) $ if $ s_1 = + $ and equivalent to $ (v_{n-1},\dots,v_1) \in C_4^{+--23\mp}(n-1) $ if $ s_1 = - $. Hence there are
\[
q(q-1)c_4^{++-34\mp}(n-1) + q^2c_4^{+--23\mp}(n-1)
\]
elements in Case 3.\\
Case 4: $ s_1 = +, s_2 = s_3 = - $.\\
We then have $ E_2 = E_3 \subseteq H_1 $ and $ E_3 = E_4 \subseteq H_2 $ and since we cannot have equality in either case we must have $ \phi_1 = \phi_2 = 0 $. These conditions are equivalent to $ \phi_1(V_{n-3}) = \phi_2(V_{n-3}) = 0 $ as $ v_{n-3} \notin E_2 = E_3 = E_4 $. This can be rewritten as
\begin{align*}
D_2 \det(V_{n-3},V_{n-2},V_{n-1},V_1) &= \mp D_3 \det(V_{n-3},V_{n-1},V_1,V_2)\\
&= D_4 \det(V_{n-3},V_1,V_2,V_3).
\end{align*}
This is exactly the condition for $ (v_1,\dots,v_{n-1}) \in C_4^{+--234\mp}(n-1) $. The number of choices for $ v_n $ is the same as in the proof of \eqref{eq:k4+++}, namely $ q^2(q-1) $, since both determinant conditions become trivial.\\
Case 5: $ s_1 = s_2 = s_3 = - $.\\
This is similar to Case 4, but we test $ \phi_1 = \phi_2 = 0 $ with $ V_{n-4} $ instead of $ V_{n-3} $, since $ v_{n-4} \notin E_1 = E_2 = E_3 = E_4 \subset H_1, H_2 $. Then our condition $ \phi_1 = \phi_2 = 0 $ is equivalent to
\begin{align*}
D_2 \det(V_{n-4},V_{n-2},V_{n-1},V_1) &= \mp D_3 \det(V_{n-4},V_{n-1},V_1,V_2)\\
&= D_4 \det(V_{n-4},V_1,V_2,V_3).
\end{align*}
This is the condition for $ (v_1,\dots,v_{n-1}) \in C_4^{---234\mp}(n-1) $. The number of values for $ v_n $ we can choose is then $ q^3 $ like in the proof of \eqref{eq:k4+++}. The claimed recursion follows from putting together the cases.\medskip\\
\eqref{eq:++-123}: Let $ (v_1,\dots,v_n) \in C_4^{++-123\pm}(n) $. Then
\[
(v_1,\dots,v_{n-1}) \in C_4^{s_1s_2+}(n-1)
\]
with $ s_i \in \{+,-\} $. Consider the maps
\[
\phi_1: V \mapsto D_1 \det(V_{n-3},V_{n-2},V_{n-1},V) \mp D_2 \det(V_{n-2},V_{n-1},V,V_1)
\]
and
\[
\phi_2: V \mapsto D_2 \det(V_{n-2},V_{n-1},V,V_1) \mp D_3 \det(V_{n-1},V,V_1,V_2).
\]
If $ s_i = + $ for $ i = 1 $ or $ i = 2 $, then $ \phi_i(V_i) \neq 0 $ and thus $ \dim H_i = 2 $. If $ s_i = - $, then $ E_i = E_{i+1} \subseteq H_i $ and $ v_n \in H_i \backslash E_i $ implies $ H_i = \pr^3(K) $ or equivalently $ \phi_i = 0 $. For a given $ (v_1,\dots,v_{n-1}) \in C_4^{s_1s_2+}(n-1) $ we have $ (v_1,\dots,v_n) \in C_4^{++-123\pm}(n) $ if and only if $ v_n \in H_1 \cap H_2 \cap E_4 \backslash (E_1 \cup E_2 \cup E_3) $.\\
Case 1: $ s_1 = s_2 = + $.\\
In this case $ H_1 $ and $ H_2 $ are both projective planes, and $ H_1 \neq H_2 $ since $ v_1 \in H_2 $ and $ v_1 \notin H_1 $. Hence $ H_1 \cap H_2 $ is a projective line. Moreover, $ v_{n-1} \in H_1 \cap H_2 $, but $ v_{n-1} \notin E_4 $. Accordingly, $ \{P\} := H_1 \cap H_2 \cap E_4 $ is just a single point. On the other hand, we have $ E_i \cap H_i = E_i \cap E_{i+1} = E_{i+1} \cap H_i $ for $ i = 1,2 $ and thus
\begin{align*}
E_1 \cap \{P\} &= E_1 \cap H_1 \cap H_2 \cap E_4 = E_1 \cap E_2 \cap E_3 \cap E_4 = \emptyset\\
&= E_2 \cap \{P\}\\
&= E_3 \cap \{P\}.
\end{align*}
Therefore, the choice $ v_n = P $ will work for every $ (v_1,\dots,v_{n-1}) \in C_4(n-1) $ and the number of elements of $ C_4^{++-123\pm}(n) $ with $ s_1 = s_2 = + $ is exactly $ c_4(n-1) $.\\
Case 2: $ s_1 = + $, $ s_2 = - $.\\
Let $ G := H_1 \cap E_4 $. $ v_{n-1} \in H_1 \backslash E_4 $ implies $ \dim G = 1 $. We also have
\begin{align*}
E_1 \cap G = E_1 \cap H_1 \cap E_4 &= E_1 \cap E_2 \cap E_4\\
&=E_2 \cap G\\
&=E_3 \cap G.
\end{align*}
These intersections are all a single point since $ \lvert E_1 \cap E_2 \cap E_4 \rvert = 1 $ if $ s_1 = + $ ( $ E_1 \neq E_2 $ and $ v_{n-1} \in E_1 \cap E_2 \backslash E_4 $). Hence there are $ q $ valid choices for $ v_n $. The map $ \phi_2 $ is the same as the map $ \phi $ in the proof of \eqref{eq:++-23}, and like in the case $ s_1 = + $, $ s_2 = - $ there $ \phi_2 = 0 $ is equivalent to $ (v_2,\dots,v_{n-1},v_1) \in C_4^{++-34\mp}(n-1) $. Hence we have $ qc_4^{++-34\mp}(n-1) $ elements falling in Case 2.\\
Case 3: $ s_1 = - $, $ s_2 = + $.\\
Let $ G := H_2 \cap E_4 $. We have $ \dim G = 1 $, because $ v_{n-1} \in H_2 \backslash E_4 $. We also have
\begin{align*}
E_3 \cap G &= E_2 \cap E_3 \cap E_4 = \{v_1\}\\
&=E_2 \cap G\\
&=E_1 \cap G.
\end{align*}
Accordingly, there are again $ q $ valid choices for $ v_n $. The map $ \phi_1 $ is the same as the map $ \phi $ in the proof of \eqref{eq:+(-)[+]12}/\eqref{eq:+(-)[-]12}, and like in the case $ s^\prime = - $ and $ s = + $ in that proof the condition $ \phi_1 = 0 $ is equivalent to $ (v_{n-1},\dots,v_1) \in C_4^{++-34\mp}(n-1) $. Therefore, the number of elements in Case 3 is also $ q c_4^{++-34\mp}(n-1) $.\\
Case 4: $ s_1 = s_2 = - $.\\
In this case we have the necessary condition $ \phi_1 = \phi_2 = 0 $. This is equivalent to $ \phi_1(V_{n-4}) = \phi_2(V_{n-4}) = 0 $, because $ v_{n-4} \notin E_1 = E_2 = E_3 \subset H_1, H_2 $. We can then rephrase our condition as
\begin{align*}
D_1 \det(V_{n-4},V_{n-3},V_{n-2},V_{n-1}) &= \mp D_2 \det(V_{n-4},V_{n-2},V_{n-1},V_1)\\
&= D_3 \det(V_{n-4},V_{n-1},V_1,V_2).
\end{align*}
After applying the bijection
\[
C_4^{--+}(n-1) \to C_4^{+--}(n-1); (v_1,\dots,v_{n-1}) \mapsto (v_2,\dots,v_{n-1},v_1),
\]
the condition turn into
\begin{align*}
D_2 \det(V_{n-3},V_{n-2},V_{n-1},V_1) &= \mp D_3 \det(V_{n-3},V_{n-1},V_1,V_2)\\
&= D_4 \det(V_{n-3},V_1,V_2,V_3).
\end{align*}
This is the determinant condition in $ C_4^{+--234\mp}(n-1) $. In this situation the determinant conditions in $ C_4^{++-123\pm}(n) $ have become trivial due to $ H_1 = H_2 = \pr^3(K) $ and like in the proof of \eqref{eq:k4++-}, Case $ s_1 = s_2 = - $ there are $ q^2 $ valid choices for $ v_n $. The claim follows.
\medskip\\
\eqref{eq:+--234}: Let $ (v_1,\dots,v_{n-1}) \in C_4^{+--234\pm}(n-1) $. Then
\[
(v_1,\dots,v_{n-2}) \in C_4^{s+-}(n-2) \cup C_4^{s(-)[+]}(n-2)
\]
with $ s \in \{+,-\} $. Define
\[
\phi_1: V \mapsto D_2 \det(V_{n-3},V_{n-2},V,V_1) \mp D_3 \det(V_{n-3},V,V_1,V_2)
\]
and
\[
\phi_2: V \mapsto D_2 \det(V_{n-3},V_{n-2},V,V_1) - D_4^\prime \det(V_{n-4},V_{n-3},V_{n-2},V)\det(V_{n-3},V_1,V_2,V_3)
\]
Then the suitable values for $ v_{n-1} $ are the elements of $ H_1 \cap H_2 \cap E_4 \backslash (E_1 \cup E_2 \cup (v_1 \vee v_2)) $. If $ s = + $, then $ \phi_2(V_1) \neq 0 $ and $ \dim H_2 = 2 $. If $ s = - $ then $ E_1 = E_2 \subseteq H_2 $ and it follows that $ \phi_2 = 0 $. If $ \{v_{n-2},v_1,v_2\} $ is independent, then $ \{v_{n-3},v_{n-2},v_1,v_2\} $ is independent. Thus $ \phi_1(V_2) \neq 0 $ and $ \dim H_2 = 2 $. If $ \{v_{n-2},v_1,v_2\} $ is dependent, then $ E_2 = v_{n-3} \vee v_{n-2} \vee v_1 = v_{n-3} \vee v_1 \vee v_2 \subseteq H_1 $ and $ \phi_1 = 0 .$\\
Case 1: $ \{v_{n-2},v_1,v_2\} $ independent and $ s = + $.\\
We have $ v_1 \in H_1 \backslash H_2 $ and therefore $ H_1 \cap H_2 $ is a projective line. We also have $ v_{n-3} \in H_1 \cap H_2 $ and $ v_{n-3} \notin E_3 = E_4 $, hence $ \{P\} := H_1 \cap H_2 \cap E_4 $ is a single point. Furthermore, $ E_1 \cap H_2 = E_1 \cap E_2 = v_{n-3} \vee v_{n-2} = E_2 \cap H_2 $ and thus
\begin{align*}
E_1 \cap \{P\} &= E_1 \cap H_2 \cap H_1 \cap E_4 = (v_{n-3} \vee v_{n-2}) \cap H_1 \cap E_4 = \{v_{n-3}\} \cap E_4 = \emptyset\\
&= E_2 \cap \{P\}.
\end{align*}
Additionally,
\[
(v_1 \vee v_2) \cap \{P\} = (v_1 \vee v_2) \cap H_1 \cap H_2 \cap E_4 = \{v_1\} \cap H_2 = \emptyset.
\]
Hence there is exactly one viable choice for $ v_{n-1} $ and we have $ c_4^{++-}(n-2) $ elements in Case 1.\\
Case 2: $ \{v_{n-2},v_1,v_2\} $ independent and $ s = - $.\\
In this case we have $ \dim H_1 = 2 $ and $ \phi_2 = 0 $. $ H_1 $ is identical to $ H $ in the proof of \eqref{eq:+--23} and just like there we have $ q $ choices for $ v_{n-1} $ (as long as $ \phi_2 = 0 $ is fulfilled). The map $ \phi_2 $ is the same as the map $ \phi $ in the proof of \eqref{eq:+--24}, and as we showed there $ \phi_2 = 0 $ is equivalent to $ (v_1,\dots,v_{n-2}) \in C_4^{-+-24}(n-2) $. Hence we have $ q c_4^{-+-24}(n-2) $ elements of $ C_4^{+--234\pm}(n-1) $ that fall in Case 2.\\
Case 3: $ \{v_{n-2},v_1,v_2\} $ dependent and $ s = + $.\\
Accordingly, we have $ \phi_1 = 0 $ and $ \dim H_2 = 2 $. As we already used, the map $ \phi_1 $ is the same as the map $ \phi $ in the proof of \eqref{eq:+--23}, and as we proved there $ \phi_1 = 0 $ is equivalent to $ (v_1,\dots,v_{n-2}) \in C_4^{+(-)[+]23\mp}(n-2) $. $ H_2 $ is also still the same as $ H $ in the proof of \eqref{eq:+--24} and once again we have $ q $ choices for $ v_{n-1} $. Correspondingly, we have $ qc_4^{+(-)[+]23\mp}(n-2) $ elements in Case 3.\\
Case 4: $ \{v_{n-2},v_1,v_2\} $ dependent and $ s = - $.\\
Then $ \phi_1 = \phi_2 = 0 $. Since $ v_{n-5} \notin E_1 = E_2 \subset H_1,H_2 $, this is equivalent to $ \phi_1(V_{n-5}) = \phi_2(V_{n-5}) = 0 $, which is in turn equivalent to
\begin{align*}
D_2 \det(V_{n-5},V_{n-3},V_{n-2},V_1) &= \mp D_3 \det(V_{n-5},V_{n-3},V_1,V_2)\\
&= - D_4^\prime d_{n-5} \det(V_{n-3},V_1,V_2,V_3).
\end{align*}
This is equivalent to $ (v_1,\dots,v_{n-2}) \in C_4^{-(-)[+]234\mp}(n-2) $. Finally note that there are $ q^2 $ choices for $ v_{n-1} $ like in the proof of \eqref{eq:k4+--}. The claim is proved.\medskip\\
\eqref{eq:---234}: Let $ (v_1,\dots,v_{n-1}) \in C_4^{---234\pm}(n-1) $. Then
\[
(v_1,\dots,v_{n-2}) \in C_4^{+--}(n-2) \cup C_4^{+(-)[-]}(n-2) \cup C_4^{(-)-[+]}(n-2) \cup C_4^{(--)[+]}(n-2).
\]
Define
\[
\phi_1: V \mapsto D_2 \det(V_{n-4},V_{n-2},V,V_1) - D_4^\prime \det(V_{n-4},V_{n-3},V_{n-2},V) \det(V_{n-4},V_1,V_2,V_3)
\]
and
\[
\phi_2: V \mapsto D_2 \det(V_{n-4},V_{n-2},V,V_1) \mp D_3 \det(V_{n-4},V,V_1,V_2).
\]
The valid choices for $ v_{n-1} $ are exactly the elements of $ H_1 \cap H_2 \cap E_4 \backslash (E_1 \cup (v_{n-2} \vee v_1) \cup (v_1 \vee v_2)) $. Let $ M_1 := \{v_{n-3},v_{n-2},v_1\} $ and $ M_2 := \{v_{n-2},v_1,v_2\} $. If $ M_1 $ is independent, then $ \{v_{n-4},v_{n-3},v_{n-2},v_1\} $ is independent and $ \phi_1(V_1) \neq 0 $ and $ \dim H_1 = 2 $. If $ M_1 $ is dependent, then $ E_1 = v_{n-4} \vee v_{n-3} \vee v_{n-2} = v_{n-4} \vee v_{n-2} \vee v_1 \subseteq H_1 $ and since we cannot have equality $ \phi_1 = 0 $ follows. If $ M_2 $ is independent, then $ \{v_{n-4},v_{n-2},v_1,v_2\} $ is independent, $ \phi_2(V_2) \neq 0 $ and $ \dim H_2 = 2 $. If $ M_2 $ is dependent, then $ v_{n-4} \vee v_{n-2} \vee v_1 = v_{n-4} \vee v_1 \vee v_2 \subseteq H_2 $. If we had equality, $ H_2 \cap E_4 = v_1 \vee v_2 $ and $ v_{n-1} \in v_1 \vee v_2 $ would follow, which cannot be the case. Hence $ \phi_2 = 0 $.\\
Case 1: $ M_1 $ and $ M_2 $ are both independent.
Then $ v_1 \in H_2 \backslash H_1 $ and thus $ H_1 \cap H_2 $ is a projective line. Moreover, $ v_{n-4} \in H_1 \cap H_2 $, but $ v_{n-4} \notin E_4 $. Therefore, $ \{P\} := H_1 \cap H_2 \cap E_4 $ is only a single point. We have $ E_1 \cap H_1 = E_1 \cap (v_{n-4} \vee v_{n-2} \vee v_1) = v_{n-4} \vee v_{n-2} $ and thus
\[
E_1 \cap \{P\} = E_1 \cap H_1 \cap H_2 \cap E_4 = (v_{n-4} \vee v_{n-2}) \cap H_2 \cap E_4 = \{v_{n-4}\} \cap E_4 = \emptyset.
\]
We also have
\[
(v_{n-2} \vee v_1) \cap \{P\} = (v_{n-2} \vee v_1) \cap H_1 \cap H_2 = \{v_{n-2}\} \cap H_2 = \emptyset
\]
and
\[
(v_1 \vee v_2) \cap \{P\} = (v_1 \vee v_2) \cap H_2 \cap H_1 = \{v_1\} \cap H_1 = \emptyset.
\]
This implies that $ v_{n-1} = P $ will work for every $ (v_1,\dots,v_{n-2}) \in C_4^{+--}(n-2) $ and thus there are exactly $ c_4^{+--}(n-2) $ elements falling in Case 1.\\
Case 2: $ M_1 $ independent, $ M_2 $ dependent.\\
The map $ \phi_2 $ is the same as the map $ \phi $ in the proof of \eqref{eq:---23} and accordingly the condition $ \phi_2 = 0 $ is equivalent to $ (v_1,\dots,v_{n-2}) \in C_4^{+(-)[-]23\mp}(n-2) $ as we showed in that proof. $ G := H_1 \cap E_4 $ is a projective line since $ v_1 \in E_4 \backslash H_1 $. We also have
\[
E_1 \cap G = (v_{n-4} \vee v_{n-2}) \cap E_4 = \{v_{n-2}\}
\]
and
\[
(v_{n-2} \vee v_1) \cap G = \{v_{n-2}\} \cap E_4 = \{v_{n-2}\}
\]
as well as
\[
(v_1 \vee v_2) \cap G = \{v_{n-2}\} \cap E_4 = \{v_{n-2}\}.
\]
It follows that $ \lvert H_1 \cap E_4 \backslash (E_1 \cup (v_{n-2} \vee v_1) \cup (v_1 \vee v_2)) \rvert = q $. Hence there are $ qc_4^{+(-)[-]23\mp}(n-2) $ elements in Case 2.\\
Case 3: $ M_1 $ dependent, $ M_2 $ independent.\\
$ H_2 $ is the same as $ H $ in the proof of \eqref{eq:---23} and just like there we have $ q $ choices for $ v_{n-1} $ (as long as the necessary condition $ \phi_1 = 0 $ is met). We have $ v_2 \notin E_1 \subseteq H_1 $ because of the $ [+] $, hence $ \phi_1 = 0 $ is equivalent to $ \phi_1(V_2) = 0 $. We can rewrite this as
\[
D_2 \det(V_{n-4},V_{n-2},V_1,V_2) = - D_4^\prime \det(V_{n-4},V_{n-3},V_{n-2},V_2) \det(V_{n-4},V_1,V_2,V_3).
\]
After using the bijection
\[
C_4^{(-)-[+]}(n-2) \to C_4^{-(-)[+]}(n-2); (v_1,\dots,v_{n-2}) \mapsto (v_{n-2},\dots,v_1)
\]
this turns into
\[
D_2 \det(V_{n-3},V_{n-2},V_1,V_3) = - D_4^\prime \det(V_{n-3},V_1,V_2,V_3) \det(V_{n-4},V_{n-3},V_{n-2},V_3),
\]
which is the determinant condition in $ C_4^{-(-)[+]24}(n-2) $. Consequently, the number of elements in Case 3 is $ qc_4^{-(-)[+]24}(n-2) $.\\
Case 4: $ M_1 $ and $ M_2 $ are both dependent.\\
Then $ \phi_1 = \phi_2 = 0 $ and $ E_1 = v_{n-4} \vee v_{n-2} \vee v_1 = v_{n-4} \vee v_1 \vee v_2 \subseteq H_1, H_2 $. Since $ v_{n-5} \notin E_1 $ the necessary condition $ \phi_1 = \phi_2 = 0 $ is equivalent to $ \phi_1(V_{n-5}) = \phi_2(V_{n-5}) = 0 $. Therefore, we get the equivalent condition
\begin{align*}
D_2 \det(V_{n-5},V_{n-4},V_{n-2},V_1) &= \mp D_3 \det(V_{n-5},V_{n-4},V_1,V_2)\\
&= - D_4^\prime d_{n-5} \det(V_{n-4},V_1,V_2,V_3).
\end{align*}
This is the condition for $ (v_1,\dots,v_{n-2}) \in C_4^{(--)[+]234\mp}(n-2) $. The number of options for $ v_{n-1} $ is the same as in the proof of \eqref{eq:k4---}, which is to say $ q^2 $. The claim follows.\medskip\\
\eqref{eq:-(-)[+]234}: Let $ (v_1,\dots,v_{n-2}) \in C_4^{-(-)[+]234\pm}(n-2) $. Then
\[
(v_1,\dots,v_{n-3}) \in C_4^{+-+}(n-3) \cup C_4^{(-)+[+]}(n-3).
\]
Let
\[
\phi: V \mapsto D_2 \det(V_{n-5},V_{n-3},V,V_1) \mp D_3^\prime \det(V_{n-5},V_{n-4},V_{n-3},V) \det(V_{n-5},V_{n-3},V_1,V_2).
\]
Let $ H := \pr(\Ker \phi) $. The first determinant equation is then equivalent to $ v_{n-2} \in H $. After canceling out $ d_{n-5} $, the other determinant equation becomes
\begin{equation}\label{eq:otherCondition-(-)[+]234}
D_3^\prime \det(V_{n-5},V_{n-3},V_1,V_2) = \mp D_4^\prime \det(V_{n-3},V_1,V_2,V_3),
\end{equation}
which does not depend on $ v_{n-2} $ and will still be fulfilled by $ (v_1,\dots,v_{n-3}) $.\\
First let $ \{v_{n-4},v_{n-3},v_1\} $ be independent. Then $ \{v_{n-5},v_{n-4},v_{n-3},v_1\} $ is independent and $ \phi(V_1) \neq 0 $, hence $ \dim H = 2 $. $ v_{n-2} $ can only be chosen from $ H \cap (v_1 \vee v_2) $, and since $ v_1 \notin H $ this is a single point $ \{P\} := H \cap (v_1 \vee v_2) $. On the other hand, we have $ E_1 \cap H = (v_{n-5} \vee v_{n-4} \vee v_{n-3}) \cap (v_{n-5} \vee v_{n-3} \vee v_1) = v_{n-5} \vee v_{n-3} $, Hence
\[
E_1 \cap \{P\} = E_1 \cap H \cap (v_1 \vee v_2) = (v_{n-5} \vee v_{n-3}) \cap (v_1 \vee v_2) = \emptyset,
\]
where we use that $ \{v_{n-5},v_{n-3},v_1,v_2\} $ is independent (since $ v_{n-4} \vee v_{n-3} \vee v_1 = v_{n-3} \vee v_1 \vee v_2 $ and $ \{v_{n-5},v_{n-4},v_{n-3},v_1\} $ is independent). We also have
\[
(v_{n-3} \vee v_1) \cap \{P\} = (v_{n-3} \vee v_1) \cap (v_1 \vee v_2) \cap H = \{v_1\} \cap H = \emptyset
\]
and $ v_1 \vee v_2 \subset E_3 = E_2 $ implies $ P \in E_2 $. Hence the choice $ v_{n-2} = P $ will always work. Additionally, the bijection
\begin{equation}\label{eq:bijection-(-)[+]234}
C_4^{+-+}(n-3) \to C_4^{++-}(n-3); (v_1,\dots,v_{n-3}) \mapsto (v_2,\dots,v_{n-3},v_1)
\end{equation}
will map the condition \eqref{eq:otherCondition-(-)[+]234} to
\begin{equation}\label{eq:otherConditionTransformed-(-)[+]234}
D_4^\prime \det(V_{n-4},V_1,V_2,V_3) = \mp D_1.
\end{equation}
This is equivalent to the determinant condition in $ C_4^{++-14\mp}(n-3) $. Hence, there are $ c_4^{++-14\mp}(n-3) $ elements of $ C_4^{-(-)[+]234\pm}(n-2) $ with $ \{v_{n-4},v_{n-3},v_1\} $ independent.\\
Now let $ \{v_{n-4},v_{n-3},v_1\} $ be dependent. Then $ E_1 = v_{n-5} \vee v_{n-3} \vee v_1 \subseteq H $ and we must not have equality. $ \phi = 0 $ follows and this is equivalent to $ \phi(V_{n-6}) = 0 $ since $ v_{n-6} \notin E_1 $. After canceling out $ d_{n-6} $, this condition becomes
\[
D_2^\prime \det(V_{n-6},V_{n-5},V_{n-3},V_1) = \mp D_3^\prime \det(V_{n-5},V_{n-3},V_1,V_2).
\]
Note that \eqref{eq:bijection-(-)[+]234} is also a bijection $ C_4^{(-)+[+]}(n-3) \to C_4^{+(-)[+]}(n-3) $. Applying this turns our condition into
\[
D_3^\prime \det(V_{n-5},V_{n-4},V_1,V_2) = \mp D_4^\prime \det(V_{n-4},V_1,V_2,V_3).
\]
Meanwhile \eqref{eq:otherCondition-(-)[+]234} still becomes \eqref{eq:otherConditionTransformed-(-)[+]234}. These two conditions are precisely the determinant conditions of $ C_4^{+(-)[+]134\mp}(n-3) $. We have $ q $ choices for $ v_{n-2} $ just like in the proof of \eqref{eq:k4-(-)}. This proves the recursion.\medskip\\
\eqref{eq:(--)[+]234}: After canceling out $ d_{n-5} $, one of the determinant conditions turns into
\begin{equation}\label{eq:otherCondition(--)[+]234}
D_3^\prime \det(V_{n-5},V_{n-4},V_1,V_2) = \mp D_4^\prime \det(V_{n-4},V_1,V_2,V_3).
\end{equation}
This does not depend on $ v_{n-2} $ and $ v_{n-3} $ and will simply be inherited by $ (v_1,\dots,v_{n-3}) $ and $ (v_1,\dots,v_{n-4}) $. The other determinant condition
\[
D_2 \det(V_{n-5},V_{n-4},V_{n-2},V_1) = \pm D_3 \det(V_{n-5},V_{n-4},V_1,V_2)
\]
is also the determinant condition in $ C_4^{(--)[+]23\pm}(n-2) $. Therefore, we can use the recursion \eqref{eq:(--)[+]23}, except that the sets on the right hand side additionally have the condition \eqref{eq:otherCondition(--)[+]234}. The sets on the right hand side of \eqref{eq:(--)[+]23} are $ C_4^{+(-)[+]}(n-3) $ and $ C_4^{23\mp}(n-4) $. For the former, adding the condition \eqref{eq:otherCondition(--)[+]234} yields the set $ C_4^{+(-)[+]34\mp}(n-3) $. For the latter it yields the set $ C_4^{234\mp}(n-4) $. Combined with $ c_4^{+(-)[+]34\mp}(n-3) = c_4^{+(-)[+]14\mp}(n-3) $ \eqref{eq:+(-)[+]34} the claimed recursion follows.\medskip\\
\eqref{eq:+(-)[+]134}: This is the same as \eqref{eq:k4+(-)[+]}, except that all sets have the additional condition
\[
D_1 = D_3^\prime \det(V_{n-5},V_{n-4},V_1,V_2) = \pm D_4^\prime \det(V_{n-4},V_1,V_2,V_3),
\]
that does not depend on $ v_{n-3} $. In the proof of \eqref{eq:k4+(-)[+]} we first got the set $ C_4(n-4) $. Adding the determinant condition turns this into $ C_4^{134\pm}(n-4) $. Recall that $ c_4^{134\pm}(n-4) = c_4^{234\pm}(n-4)$. The second set we got in the proof of \eqref{eq:k4+(-)[+]} was $ C_4^{-++}(n-4) $. We can use the map $ (v_1,\dots,v_{n-4}) \mapsto (v_3,\dots,v_{n-4},v_1,v_2) $ to get to $ C_4^{++-}(n-4) $. This turns the determinant condition into
\[
D_3^\prime \det(V_{n-5},V_{n-4},V_1,V_2) = D_1 = \pm D_2^\prime \det(V_{n-6},V_{n-5},V_{n-4},V_1),
\]
which is the determinant condition for $ C_4^{++-123\pm}(n-4) $. This concludes the proof.
\end{proof}

Now we will consider $ C_4^{12|34\pm}(n) $ and define the following related sets:

\begin{defi}
Let $ n \geq 4 $ be divisible by 4. For any tuple $ (v_1,\dots,v_m) \in (\pr^3(K))^m $ with $ m \in \{n,n-1,n-2,n-3\} $ in the definitions of the following sets let $ (V_1,\dots,V_m) \in (K^4)^m $ be an arbitrary lift. Define $ d_i := \det(V_i,V_{i+1},V_{i+2},V_{i+3}) $ with the indices considered modulo $ n $. Let $ D_j^\prime := d_j d_{j+4} \dots d_{n-12+j} $ and $ D_ j := D_j^\prime d_{n-8+j} $ for $ j \in \{1,2,3,4\} $. Let $ s \in \{+,-\} $. Then define
\begin{align*}
C_4^{12|34s}(n) := \{&(v_1,\dots,v_n) \in C_4(n) \mid D_1d_{n-3} = sD_2d_{n-2} \text{ and } D_3 d_{n-1} = sD_4 d_n\},\\
C_4^{+(-)[-]12|34s}(n) := \{&(v_1,\dots,v_n) \in C_4^{+(-)[-]}(n) \mid D_1 d_{n-3} = sD_2 d_{n-2} \text{ and}\\
&D_3 \det(V_{n-2},V_{n-1},V_1,V_2) = -sD_4\det(V_{n-2},V_1,V_2,V_3)\},\\
C_4^{-+-12|34s}(n-1) := \{&(v_1,\dots,v_{n-1}) \in C_4^{-+-}(n-1) \mid\\
&D_1 \det(V_{n-3},V_{n-2},V_{n-1},V_2) = sD_2\det(V_{n-2},V_{n-1},V_1,V_2) \text{ and }\\
&D_3 \det(V_{n-2},V_{n-1},V_1,V_2) = sD_4 \det(V_{n-2},V_1,V_2,V_3)\},\\
C_4^{---12|34s}(n-1) := \{&(v_1,\dots,v_{n-1}) \in C_4^{---}(n-1) \mid\\
&D_1 d_{n-4} = sD_2 \det(V_{n-4},V_{n-2},V_{n-1},V_1) \text{ and}\\
&D_3 \det(V_{n-4},V_{n-1},V_1,V_2) = sD_4\det(V_{n-4},V_1,V_2,V_3)\},\\
C_4^{+(-)[-]14|23s}(n-2) := \{&(v_1,\dots,v_{n-2}) \in C_4^{+(-)[-]}(n-2) \mid\\
&D_1 = -s D_4^\prime \det(V_{n-4},V_1,V_2,V_3) \text{ and}\\
&D_2\det(V_{n-4},V_{n-3},V_{n-2},V_1) = sD_3 \det(V_{n-4},V_{n-3},V_1,V_2)\},\\
C_4^{(--)[+]12|34s}(n-2) := \{&(v_1,\dots,v_{n-2}) \in C_4^{(--)[+]}(n-2) \mid\\
&D_1 d_{n-5} = sD_2 \det(V_{n-5},V_{n-4},V_{n-2},V_1) \text{ and}\\
&D_3^\prime \det(V_{n-5},V_{n-4},V_1,V_2) = -sD_4^\prime \det(V_{n-4},V_1,V_2,V_3)\},\\
C_4^{-+-14|23s}(n-3) := \{&(v_1,\dots,v_{n-3}) \in C_4^{-+-}(n-3) \mid\\
&D_1 = s D_4^\prime \det(V_{n-4},V_1,V_2,V_3),\\
&D_2 = sD_3^\prime\det(V_{n-5},V_{n-4},V_{n-3},V_2)\}.
\end{align*}
\end{defi}

Once again, we derive a system of recursions:

\begin{lem}\label{lem:recursionsk412|34}
Let $ n \geq 8 $ be divisible by $ 4 $. Then
\begin{align}
    c_4^{12|34\pm}(n) = &(q-1)c_4(n-1) + qc_4^{++-}(n-1)\label{eq:+++12|34}\\
    &+2q(q-1)c_4^{++-34\mp}(n-1) +2q^2c_4^{+--23\mp}(n-1)\nonumber\\
    &+q^2(q-1)c_4^{-+-12|34\mp}(n-1)+q^3c_4^{---12|34\mp}(n-1),\nonumber\\
    c_4^{+(-)[-]12|34\pm}(n) = &c_4^{++-34\mp}(n-1) + qc_4^{-+-12|34\mp}(n-1),\label{eq:+(-)[-]12|34}\\
    c_4^{-+-12|34\pm}(n-1) = &c_4^\#(n-2) + qc_4^{+(-)[+]23\mp}(n-2)+ q^2c_4^{+(-)[-]14|23\mp}(n-2),\label{eq:-+-12|34}\\
    c_4^{---12|34\pm}(n-1) = &c_4^{+--}(n-2) + qc_4^{+(-)[-]\#}(n-2)\label{eq:---12|34}\\
    &+qc_4^{-(-)[+]23\mp}(n-2) + q^2c_4^{(--)12|34\mp}(n-2),\nonumber\\
    c_4^{+(-)[-]14|23\pm}(n-2) = &c_4^{++-14\mp}(n-3) + qc_4^{-+-14|23\mp}(n-3),\label{eq:+(-)[-]14|23}\\
    c_4^{(--)[+]12|34\pm}(n-2) = &c_4^{+(-)[+]14\mp}(n-3) + qc_4^{12|34\mp}(n-4),\label{eq:(--)[+]12|34}\\
    c_4^{-+-14|23\pm}(n-3) = &(q-1)c_4^{12|34\pm}(n-4) + qc_4^{+(-)[+]12\pm}(n-4)\label{eq:-+-14|23}\\
    &+ q^2c_4^{+(-)[-]12|34\pm}(n-4).\nonumber
\end{align}
\end{lem}
\begin{proof}
For any tuple $ (v_1,\dots,v_m) \in (\pr^3(K))^m $ of projective points with $ m \in \{n,n-1,n-2,m-3\} $ that we consider in this proof let $ (V_1,\dots,V_m) \in (K^4)^m $ be an arbitrary lift. We set $ d_i := \det(V_i,V_{i+1},V_{i+2},V_{i+3}) $ for $ i \leq n $ (the indices are considered modulo $ n $) and $ D_j := d_j d_{j+4} \dots d_{n-8+j} $ and $ D_j^\prime := d_j d_{j+4} \dots d_{n-12+j} $ for $ j \in \{1,2,3,4\} $. Also let $ E_1 := v_{m-3} \vee v_{m-2} \vee v_{m-1} $, $ E_2 := v_{m-2} \vee v_{m-1} \vee v_1 $, $ E_3 := v_{m-1} \vee v_1 \vee v_2 $ and $ E_4 := v_1 \vee v_2 \vee v_3 $. For most of the recursions we consider a linear map $ \phi: K^4 \to K $ or two linear maps $ \phi_i: K^4 \to K $ and their kernels $ H := \Ker(\phi) $ resp. $ H_i := \pr(\Ker \phi_i) $. The definition of $ H $ resp. $ H_i $ will be independent of the choice of lift. \smallskip\\
\eqref{eq:+++12|34}: Let $ (v_1,\dots,v_n) \in C_4^{12|34\pm}(n) $. Then
\[
(v_1,\dots,v_{n-1}) \in C_4^{s_1s_2s_3}(n-1) 
\]
with $ s_i \in \{+,-\} $. Let
\[
\phi_1: K^4 \to K; V \mapsto D_1 \det(V_{n-3},V_{n-2},V_{n-1},V) \mp D_2 \det(V_{n-2},V_{n-1},V,V_1)
\]
and
\[
\phi_3: K^4 \to K; V \mapsto D_3 \det(V_{n-1},V,V_1,V_2) \mp D_4 \det(V,V_1,V_2,V_3).
\]
If $ s_1 = + $, then $ \phi_1(V_1) \neq 0 $ and thus $ \dim H_1 = 2 $. Likewise, $ s_3 = + $ implies $ \phi_3(V_{n-1}) \neq 0 $ and $ \dim H_3 = 2 $. If $ s_i = - $, then $ E_i = E_{i+1} \subseteq H_i $ and equality is impossible since $ v_n \in H_i \backslash E_i $. Hence $ H_i = \pr^3(K) $ and $ \phi_i = 0 $ in this case.\smallskip\\
Case 1: $ s_1 = s_3 = + $.\\
Let $ G := H_1 \cap H_3 $. We have $ E_1 \cap H_1 = E_1 \cap E_2 = v_{n-2} \vee v_{n-1} = E_2 \cap H_1 $. Similarly, we have $ E_3 \cap H_3 = E_4 \cap H_3 = E_3 \cap E_4 = v_1 \vee v_2 $. Hence
\[
E_1 \cap G = E_2 \cap G = (v_{n-2} \vee v_{n-1}) \cap H_3
\]
and
\[
E_3 \cap G = E_4 \cap G = (v_1 \vee v_2) \cap H_1.
\]
In particular, both intersections are single points since $ v_{n-1} \notin H_3 $ and $ v_1 \notin H_1 $. This also implies $ \dim G = 1 $. If $ s_2 = - $, then $ E_2 = E_3 $ and both intersection points are the same. If $ s_2 = + $, then
\[
E_1 \cap E_3 \cap H_1 \cap H_3 = E_1 \cap E_2 \cap E_3 \cap E_4 = \emptyset
\]
and the points are different. Hence $ G \backslash (E_1 \cup E_2 \cup E_3 \cup E_4) $ has $ q-1 $ elements if $ s_2 = + $ and $ q $ elements if $ s_2 = - $. For an arbitrary $ (v_1,\dots,v_{n-1}) \in C_4^{+s_2+}(n-1) $ these elements are exactly the valid choices for $ v_n $. Hence the number of elements of $ C_4^{12|34\pm}(n) $ in Case 1 is
\[
(q-1)c_4(n-1) + qc_4^{++-}(n-1),
\]
where we use $ c_4^{+-+}(n-1) = c_4^{++-}(n-1) $.\medskip\\
Case 2: $ s_1 = + $ and $ s_3 = - $.\\
Then we still have $ E_1 \cap H_1 = E_2 \cap H_1 = E_1 \cap E_2 = v_{n-2} \vee v_{n-1} $. If $ s_2 = - $ we have $ E_2 = E_3 = E_4 $ and thus $ E_1 $, $ E_2 $, $ E_3 $ and $ E_4 $ all intersect $ H_1 $ in the same projective line. If $ s_2 = + $, then $ v_{n-1} \notin E_3 = E_4 $. In particular, $ E_3 \cap H_1 = E_4 \cap H_1 $ is a projective line different from $ v_{n-2} \vee v_{n-1} $. Hence $ H_1 \backslash (E_1 \cup E_2 \cup E_3 \cup E_4) $ has $ q(q-1) $ elements if $ s_2 = + $ and $ q^2 $ elements if $ s_2 = - $. These are exactly the valid choices for $ v_n $, as long as $ (v_1,\dots,v_{n-1}) $ meets the necessary condition $ \phi_3 = 0 $.\\
If $ s_2 = + $ this is the case if and only if $ \phi_3(V_{n-2}) = 0 $ since $ v_{n-2} \notin E_3 = E_4 $ then. We can rewrite this condition as
\[
D_3 \det(V_{n-2},V_{n-1},V_1,V_2) = \mp D_4 \det(V_{n-2},V_1,V_2,V_3).
\]
This is equivalent to $ (v_1,\dots,v_{n-1}) \in C_4^{34\mp}(n-1) $.\\
If $ s_2 = - $, then we instead use the fact that $ v_4 \notin E_3 = E_4 $ and that accordingly, $ \phi_3 = 0 $ is equivalent to $ \phi_3(V_4) = 0 $. This is in turn equivalent to
\[
D_3 \det(V_{n-1},V_1,V_2,V_4) = \mp D_4 d_1.
\]
Now we first use the bijection
\[
C_4^{+--}(n-1) \to C_4^{--+}(n-1); (v_1,\dots,v_{n-1}) \mapsto (v_{n-1},\dots,v_1)
\]
and then the bijection
\[
C_4^{--+}(n-1) \to C_4^{+--}(n-1); (v_1,\dots,v_{n-1}) \mapsto (v_2,\dots,v_{n-1},v_1).
\]
The determinant condition first becomes
\[
D_2 \det(V_{n-4},V_{n-2},V_{n-1},V_1) = \mp D_1 d_{n-4}
\]
and then
\[
D_3 \det(V_{n-3},V_{n-1},V_1,V_2) = \mp D_2 \det(V_{n-3},V_{n-2},V_{n-1},V_1).
\]
This is equivalent to the determinant condition in $ C_4^{+--23\mp}(n-1) $. Hence the number of elements falling in Case 2 is
\[
q(q-1)c_4^{++-34\mp}(n-1) + q^2c_4^{+--23\mp}(n-1).
\]
Case 3: $ s_1 = - $ and $ s_3 = + $.\\
The bijection $ C_4(n) \to C_4(n); (v_1,\dots,v_{n-1},v_n) \mapsto (v_{n-1},\dots,v_1,v_n) $ restricts to a bijection between the set of elements of $ C_4^{12|34\pm}(n) $ that fall under Case 2 and the set of elements that fall under Case 3. Hence we have the same number of elements in both cases.\medskip\\
Case 4: $ s_1 = s_3 = - $.\\
Let $ s_2 = + $. The condition $ \phi_1 = 0 $ is equivalent to $ \phi_1(V_2) = 0 $ since $ v_2 \notin E_1 = E_2 $. This can be rewritten as
\[
D_1 \det(V_{n-3},V_{n-2},V_{n-1},V_2) = \mp D_2 \det(V_{n-2},V_{n-1},V_1,V_2).
\]
The condition $ \phi_3 = 0 $ is equivalent to $ \phi_3(V_{n-2}) = 0 $, which is in turn equivalent to
\[
D_3 \det(V_{n-2},V_{n-1},V_1,V_2) = \mp D_4 \det(V_{n-2},V_1,V_2,V_3).
\]
Now let $ s_2 = - $. We can use $ V_{n-4} $ to test both $ \phi_1 = 0 $ and $ \phi_3 = 0 $ and this results in the conditions
\[
D_1 d_{n-4} = \mp D_2 \det(V_{n-4},V_{n-2},V_{n-1},V_1).
\]
and
\[
D_3 \det(V_{n-4},V_{n-1},V_1,V_2) = \mp D_4 \det(V_{n-4},V_1,V_2,V_3).
\]
Instead, we could also use $ V_4 $ to test these conditions. In particular, the first determinant equation is also equivalent to
\[
D_1 \det(V_{n-3},V_{n-2},V_{n-1},V_4) = \mp D_2 \det(V_{n-2},V_{n-1},V_1,V_4).
\]
We will use this fact later.\\
In either case, the two respective determinant equations are equivalent to $ (v_1,\dots,v_{n-1}) \in C_4^{-s_2-12|34\mp}(n-1) $. The suitable choices for $ v_n $ are exactly the elements of $ \pr^3(K) \backslash (E_1 \cup E_2 \cup E_3 \cup E_4) $. In the proof of \eqref{eq:k4+++} we showed that this set has $ q^2(q-1) $ elements if $ s_2 = + $ and $ q^3 $ elements if $ s_2 = - $ (as long as $ s_1 = s_3 = - $). Hence the number of elements that fall under Case 4 is
\[
q^2(q-1)c_4^{-+-12|34\mp}(n-1) + q^3c_4^{---12|34\mp}(n-1).
\]
All cases put together prove \eqref{eq:+++12|34}. \medskip\\
\eqref{eq:+(-)[-]12|34}: Let $ (v_1,\dots,v_n) \in C_4^{+(-)[-]12|34\pm}(n) $. Then
\[
(v_1,\dots,v_{n-1}) \in C_4^{s+-}(n-1)
\]
with $ s \in \{+,-\} $. One of the determinant conditions is 
\begin{equation}\label{eq:+(-)[-]12|34detCon}
D_3 \det(V_{n-2},V_{n-1},V_1,V_2) = \mp D_4 \det(V_{n-2},V_1,V_2,V_3),
\end{equation}
which does not depend on $ v_n $. For the other determinant equation, consider
\[
\phi: K^4 \to K; V \mapsto D_1 \det(V_{n-3},V_{n-2},V_{n-1},V) \mp D_2 \det(V_{n-2},V_{n-1},V,V_1)
\]
For a given $ (v_1,\dots,v_{n-1}) \in C_4^{s+-}(n-1) $, the other determinant equation is fulfilled if and only if $ v_n \in H $.\\
Case 1: $ s = + $.
We have $ \phi(V_1) \neq 0 $. Accordingly, $ \{P\} := H \cap (v_1 \vee v_2) $ is a single point. Additionally,
\begin{align*}
E_1 \cap \{P\} &= E_1 \cap H \cap (v_1 \vee v_2) = E_1 \cap E_2 \cap (v_1 \vee v_2) = (v_{n-2} \vee v_{n-1}) \cap (v_1 \vee v_2) = \emptyset\\
&= E_2 \cap \{P\}.
\end{align*}
Hence $ v_n = P $ is a viable choice as long as \eqref{eq:+(-)[-]12|34detCon} is fulfilled and it is the only such choice. \eqref{eq:+(-)[-]12|34detCon} is equivalent to $ (v_1,\dots,v_{n-1}) \in C_4^{++-34\mp}(n-1) $. Accordingly, there are $ c_4^{++-34\mp}(n-1) $ elements in Case 1.\smallskip\\
Case 2: $ s = - $.\\
In this case $ E_1 = E_2 \subseteq H $. Then $ v_n \in H \backslash E_1 $ implies $ H = \pr^3(K) $ and $ \phi = 0 $. The latter is equivalent to $ \phi(V_2) = 0 $ since $ v_2 \notin E_2 = E_1 $. We can now write our condition as
\[
D_1 \det(V_{n-3},V_{n-2},V_{n-1},V_2) = \mp D_2 \det(V_{n-2},V_{n-1},V_1,V_2).
\]
Together with \eqref{eq:+(-)[-]12|34detCon} this is the condition for $ (v_1,\dots,v_{n-1}) \in C_4^{-+-12|34\mp}(n-1) $. The number of choices for $ v_n $ is $ q $ like in the proof of \eqref{eq:k4+(-)[-]}. The claimed recursion follows.\medskip\\
\eqref{eq:-+-12|34}: Let $ (v_1,\dots,v_{n-1}) \in C_4^{-+-12|34\pm}(n-1) $. Then
\[
(v_1,\dots,v_{n-2}) \in C_4(n-2) \cup C_4^{(-)+[s]}(n-2)
\]
with $ s \in \{+,-\} $. The determinant conditions that $ (v_1,\dots,v_{n-1}) $ fulfills are
\begin{equation}\label{eq:-+-12|34detCon1}
D_1 \det(V_{n-3},V_{n-2},V_{n-1},V_2) = \pm D_2 \det(V_{n-2},V_{n-1},V_1,V_2).
\end{equation}
and
\begin{equation}\label{eq:-+-12|34detCon2}
D_3 \det(V_{n-2},V_{n-1},V_1,V_2) = \pm D_4 \det(V_{n-2},V_1,V_2,V_3).
\end{equation}
Multiplying \eqref{eq:-+-12|34detCon1} and \eqref{eq:-+-12|34detCon2} and canceling out $ \det(V_{n-2},V_{n-1},V_1,V_2) $ yields
\[
D_1 D_3 \det(V_{n-3},V_{n-2},V_{n-1},V_2) = D_2 D_4 \det(V_{n-2},V_1,V_2,V_3).
\]
Consider the maps
\[
\phi_1: K^4 \to K; V \mapsto D_1 \det(V_{n-3},V_{n-2},V,V_2) \mp D_2 \det(V_{n-2},V,V_1,V_2)
\]
and
\begin{align*}
\phi_2 : K^4 \to K; V \mapsto &D_1 D_3 \det(V_{n-3},V_{n-2},V,V_2)\\
- &D_2 D_4^\prime \det(V_{n-4},V_{n-3},V_{n-2},V) \det(V_{n-2},V_1,V_2,V_3).
\end{align*}
We have $ v_{n-1} \in H_1 \cap H_2 $ if and only if \eqref{eq:-+-12|34detCon1} and \eqref{eq:-+-12|34detCon2} are both fulfilled.\\
Case 1: $ \{v_{n-3},v_{n-2},v_1\} $ is independent.\\
For a given $ (v_1,\dots,v_{n-2}) \in C_4(n-2) $ we have $ (v_1,\dots,v_{n-1}) \in C_4^{-+-12|34\pm}(n-1) $ if and only if $ v_{n-1} \in H_1 \cap H_2 \cap E_2 \cap E_4 \backslash (E_1 \cup E_3) $. We have $ \phi_1(V_1) \neq 0 $ and thus $ \dim H_1 = 2 $. Clearly, $ v_{n-3} \vee v_{n-2} \subset H_2 $. If $ v_1 \notin H_2 $, then $ H_2 \cap E_2 = v_{n-3} \vee v_{n-2} $. In this case we have
\[
H_1 \cap H_2 \cap E_2 \cap E_4 = H_1 \cap (v_{n-3} \vee v_{n-2}) \cap E_4 = \{v_{n-2}\} \cap E_4 = \emptyset,
\]
contradicting $ v_{n-1} \in H_1 \cap H_2 \cap E_2 \cap E_4 $. Therefore, $ v_1 \in H_2 $ is a necessary condition. This condition is equivalent to $ \phi_2(V_1) = 0 $, which we can also write as
\[
D_1 D_3 \det(V_{n-3},V_{n-2},V_1,V_2) = D_2 D_4^\prime \det(V_{n-4},V_{n-3},V_{n-2},V_1) \det(V_{n-2},V_1,V_2,V_3).
\]
This equation is equivalent to $ (v_1,\dots,v_{n-2}) \in C_4^{\#}(n-2) $.\\
On the other hand, if this condition is fulfilled we have $ v_1 \in H_2 $ and thus $ E_2 \subseteq H_2 $. It follows that $ H_1 \cap H_2 \cap E_2 \cap E_4 = H_1 \cap E_2 \cap E_4 $. Note that $ E_2 \cap E_4 $ is a projective line with $ v_1 \in E_2 \cap E_4 $, while $ v_1 \notin H_1 $. Accordingly, $ \lvert H_1 \cap E_2 \cap E_4 \rvert = 1 $. We also have
\[
E_1 \cap H_1 \cap E_2 \cap E_4 = (v_{n-3} \vee v_{n-2}) \cap H_1 \cap E_4 = \{v_{n-2}\} \cap E_4 = \emptyset.
\]
Additionally,
\[
E_3 \cap H_1 \cap E_2 \cap E_4 = H_1 \cap E_2 \cap E_3 \cap E_4 = H_1 \cap \{v_1\} = \emptyset.
\]
Hence, $ \lvert H_1 \cap H_2 \cap E_2 \cap E_4 \backslash (E_1 \cup E_3) \rvert = 1 $ as long as the necessary condition is fulfilled. Consequently, there is exactly one valid choice for $ v_{n-1} $ and the number of elements in Case 1 is $ c_4^{\#}(n-2) $. \smallskip\\
Now let $ \{v_{n-3},v_{n-2},v_1\} $ be dependent. Then $ v_{n-3} \vee v_{n-2} \vee v_2 = E_3 \subseteq H_1 $ and we must have $ \phi_1 = 0 $. Consider the map
\begin{align*}
\phi_3: K^4 &\to K\\
V &\mapsto D_3 \det(V_{n-2},V,V_1,V_2) \mp D_4^\prime \det(V_{n-4},V_{n-3},V_{n-2},V) \det(V_{n-2},V_1,V_2,V_3)
\end{align*}
instead of $ \phi_2 $. \eqref{eq:-+-12|34detCon2} is now equivalent to $ v_{n-1} \in H_3 $ while $ \eqref{eq:-+-12|34detCon1} $ is trivial thanks to $ \phi_1 = 0 $.\\
Case 2: $ s = + $.\\
Then $ \phi_3(V_2) \neq 0 $ and $ H_3 = \pr(\Ker \phi_3) $ is a projective plane. $ G := H_3 \cap E_4 $ is a projective line, because $ v_2 \notin H_3 $ and thus $ H_3 \neq E_4 $. We have $ E_1 = v_{n-4} \vee v_{n-3} \vee v_{n-2} = v_{n-4} \vee v_{n-2} \vee v_1 $ and thus
\[
E_1 \cap H_3 \cap E_4 = (v_{n-2} \vee v_1) \cap E_4 = \{v_1\}.
\]
Additionally,
\[
E_3 \cap H_3 \cap E_4 = H_3 \cap (v_1 \vee v_2) = \{v_1\}.
\]
Hence $ H_3 \cap E_4 \backslash (E_1 \cup E_3) $ has $ q $ elements, and these are exactly the viable options for $ v_{n-1} $, as long as $ \phi_1 = 0 $. The latter condition is equivalent to $ \phi_1(V_3) = 0 $ since $ v_3 \notin E_3 $. We get the condition
\[
D_1 \det(V_{n-3},V_{n-2},V_2,V_3) = \mp D_2 \det(V_{n-2},V_1,V_2,V_3).
\]
After applying the bijection
\begin{equation}\label{eq:-+-12|34bijection}
C_4^{(-)+[s]}(n-2) \to C_4^{+(-)[s]}(n-2); (v_1,\dots,v_{n-2}) \mapsto (v_{n-2},\dots,v_1)
\end{equation}
the condition becomes
\begin{equation}\label{eq:-+-12|34detCon3}
D_3 \det(V_{n-4},V_{n-3},V_1,V_2) = \mp D_2 \det(V_{n-4},V_{n-3},V_{n-2},V_1).
\end{equation}
This is equivalent to the determinant condition in $ C_4^{+(-)[+]23\mp}(n-2) $.\smallskip\\
Case 3: $ s = - $.\\
We still have the necessary condition $ \phi_1 = 0 $ and after using the bijection it is still equivalent to \eqref{eq:-+-12|34detCon3}. In addition, we now have $ E_1 = E_3 \subseteq H_3 $ and equality is once again not possible. Hence $ \phi_3 = 0 $ and we can also test this with $ V_3 $, obtaining the condition
\[
D_3 = \pm D_4^\prime \det(V_{n-4},V_{n-3},V_{n-2},V_3),
\]
where we canceled out $ \det(V_{n-2},V_1,V_2,V_3) $. The bijection \eqref{eq:-+-12|34bijection} transforms this equation into
\[
D_1 = \pm D_4^\prime \det(V_{n-4},V_1,V_2,V_3).
\]
This plus \eqref{eq:-+-12|34detCon3} are the determinant conditions in $ C_4^{+(-)[-]14|23\mp}(n-2) $. The number of choices for $ v_{n-1} $ is $ q^2 $ like in the proof of \eqref{eq:k4-+-}. The claim follows.\medskip\\
\eqref{eq:---12|34}: Let $ (v_1,\dots,v_{n-1}) \in C_4^{---12|34\pm}(n-1) $. Then
\[
(v_1,\dots,v_{n-2}) \in C_4^{+--}(n-2) \cup C_4^{+(-)[-]}(n-2) \cup C_4^{(-)-[+]}(n-2) \cup C_4^{(--)[+]}(n-2).
\]
The determinant conditions are
\begin{equation}\label{eq:---12|34detCon1}
D_1 d_{n-4} = \pm D_2 \det(V_{n-4},V_{n-2},V_{n-1},V_1)
\end{equation}
and
\[
D_3 \det(V_{n-4},V_{n-1},V_1,V_2) = \pm D_4 \det(V_{n-4},V_1,V_2,V_3).
\]
Multiplying both equations and canceling out $ d_{n-4} $ yields
\[
D_1 D_3 \det(V_{n-4},V_{n-1},V_1,V_2) = D_2 D_4^\prime \det(V_{n-4},V_{n-2},V_{n-1},V_1) \det(V_{n-4},V_1,V_2,V_3).
\]
Now consider the maps
\[
\phi_1: K^4 \to K; V \mapsto D_1 \det(V_{n-4},V_{n-3},V_{n-2},V) \mp D_2 \det(V_{n-4},V_{n-2},V,V_1)
\]
and
\begin{align*}
\phi_2: K^4 &\to K\\
V &\mapsto D_1 D_3 \det(V_{n-4},V,V_1,V_2) - D_2 D_4^\prime \det(V_{n-4},V_{n-2},V,V_1) \det(V_{n-4},V_1,V_2,V_3).
\end{align*}
The two determinant equations are then equivalent to $ v_{n-1} \in H_1 \cap H_2 $. Let $ M_1 := \{v_{n-3},v_{n-2},v_1\} $ and $ M_2 := \{v_{n-2},v_1,v_2\} $. If $ M_1 $ is independent, then $ \{v_{n-4},v_{n-3},v_{n-2},v_1\} $ is independent and $ \phi_1(V_1) \neq 0 $. In particular, $ \dim H_1 = 2 $. If $ M_1 $ is dependent, then $ E_1 = v_{n-4} \vee v_{n-2} \vee v_1 \subseteq H_1 $, which implies $ \phi_1 = 0 $. $ M_2 $ independent implies $ \phi_2(V_2) \neq 0 $ and thus $ \dim H_2 = 2 $. If $ M_2 $ is dependent, we have $ v_{n-4} \vee v_{n-2} \vee v_1 = v_{n-4} \vee v_1 \vee v_2 \subseteq H_2 $ and equality would imply $ v_{n-1} \in H_2 \cap E_4 = v_1 \vee v_2 $. Hence $ \phi_2 = 0 $.\smallskip\\
Case 1: $ M_1 $ and $ M_2 $ are both independent.\\
In this case we have
\[
E_1 \cap H_1 \cap H_2 \cap E_4 = (v_{n-4} \vee v_{n-2}) \cap H_2 \cap E_4 = \{v_{n-4}\} \cap E_4 = \emptyset
\]
as well as
\[
(v_{n-2} \vee v_1) \cap H_1 \cap H_2 \cap E_4 = \{v_{n-2}\} \cap H_2 = \emptyset
\]
and
\[
(v_1 \vee v_2) \cap H_1 \cap H_2 \cap E_4 = \{v_1\} \cap H_1 = \emptyset.
\]
These equations also imply $ \lvert H_1 \cap H_2 \cap E_4 \rvert = 1 $. Hence there is exactly $ 1 $ suitable choice of $ v_{n-1} $ for any given $ (v_1,\dots,v_{n-2}) \in C_4^{+--}(n-2) $. The number of elements in Case 1 is thus $ c_4^{+--}(n-2) $.\smallskip\\
Case 2: $ M_1 $ is independent and $ M_2 $ dependent.\\
Let $ G := H_1 \cap E_4 $. We have
\[
E_1 \cap G = (v_{n-4} \vee v_{n-2}) \cap E_4 = \{v_{n-2}\}
\]
and
\begin{align*}
(v_{n-2} \vee v_1) \cap G &= \{v_{n-2}\} \cap E_4 = \{v_{n-2}\}\\
&= (v_1 \vee v_2) \cap G.
\end{align*}
Hence there are $ q $ choices for $ v_{n-1} $ in the projective line $ G $, as long as $ \phi_2 = 0 $ is fulfilled. The latter condition is equivalent to $ \phi_2(V_{n-3}) = 0 $. Thus the necessary condition is
\[
D_1 D_3 \det(V_{n-4},V_{n-3},V_1,V_2) = -D_2 D_4^\prime \det(V_{n-4},V_{n-3},V_{n-2},V_1) \det(V_{n-4},V_1,V_2,V_3),
\]
or equivalently, $ (v_1,\dots,v_{n-2}) \in C_4^{+(-)[-]\#}(n-2) $. Accordingly, there are $ q c_4^{+(-)[-]\#}(n-2) $ elements in this case.\smallskip\\
Case 3: $ M_1 $ is dependent and $ M_2 $ independent.\\
As we showed during the proof of \eqref{eq:+++12|34} Case $ s_1 = s_2 = s_3 = - $, the first determinant condition in $ C_4^{---12|34\pm}(n-1) $ is equivalent to
\[
D_1 \det(V_{n-3},V_{n-2},V_{n-1},V_4) = \pm D_2 \det(V_{n-2},V_{n-1},V_1,V_4).
\]
Hence the map $ \phi_1 $ can be replaced with
\[
\phi_1^\prime: K^4 \to K; V \mapsto D_1 \det(V_{n-3},V_{n-2},V,V_4) \mp D_2 \det(V_{n-2},V,V_1,V_4).
\]
Since $ M_1 $ is dependent we have $ v_{n-3} \vee v_{n-2} \vee v_4 = v_{n-2} \vee v_1 \vee v_4 \subseteq H_1^\prime := \pr(\Ker \phi_1^\prime) $. If we had equality, then $ v_{n-1} \in H_1^\prime \cap E_4 = v_1 \vee v_2 $. This cannot be the case and so we must have $ \phi_1^\prime = 0 $ as a necessary condition for the first determinant equation. On the other hand $ \phi_1^\prime = 0 $ causes the first determinant equation to become trivial. We also have $ v_2 \notin v_{n-2} \vee v_1 \vee v_4 $, hence the first determinant equation is equivalent to $ \phi_1^\prime(V_2) = 0 $, which can also be written as
\[
D_1 \det(V_{n-3},V_{n-2},V_2,V_4) = \mp D_2 \det(V_{n-2},V_1,V_2,V_4).
\]
After applying the bijection
\[
C_4^{(-)-[+]}(n-2) \to C_4^{-(-)[+]}(n-2); (v_1,\dots,v_{n-2}) \mapsto (v_{n-2},\dots,v_1)
\]
the condition becomes
\[
D_3 \det(V_{n-5},V_{n-3},V_1,V_2) = \mp D_2 \det(V_{n-5},V_{n-3},V_{n-2},V_1).
\]
This is equivalent to the determinant condition for $ C_4^{-(-)[+]23\mp}(n-2) $.\\
Let $ G := H_2 \cap E_4 $. We need to count $ G \backslash (E_1 \cup (v_{n-2} \vee v_1) \cup (v_1 \vee v_2)) $. We have $ \phi_2(V_{n-2}) \neq 0 $ and thus $ v_{n-2} \notin H_2 $. It follows that
\[
E_1 \cap G = (v_{n-4} \vee v_1) \cap E_4 = \{v_1\}.
\]
Moreover,
\begin{align*}
    (v_{n-2} \vee v_1) \cap G &= \{v_1\}\\
    &= (v_1 \vee v_2) \cap G.
\end{align*}
In particular, $ \dim G = 1 $ and the number of choices for $ v_{n-1} $ is $ q $ if $ \phi_1^\prime = 0 $. In conclusion, the number of elements in Case 3 is $ q c_4^{-(-)[+]23\mp}(n-2) $.\smallskip\\
Case 4: $ M_1 $ and $ M_2 $ are both dependent.\\
There are now two necessary conditions, namely $ \phi_1 = 0 $ and $ \phi_2 = 0 $. We can test both with $ V_{n-5} $, which results in the conditions
\begin{equation}\label{eq:---12|34detCon2}
D_1 d_{n-5} = \mp D_2 \det(V_{n-5},V_{n-4},V_{n-2},V_1)
\end{equation}
and
\[
D_1 D_3 \det(V_{n-5},V_{n-4},V_1,V_2) = - D_2 D_4^\prime \det(V_{n-5},V_{n-4},V_{n-2},V_1) \det(V_{n-4},V_1,V_2,V_3).
\]
By substituting \eqref{eq:---12|34detCon2} into the second equation and then canceling out $ D_1 $ and $ d_{n-5} $, it becomes
\[
D_3^\prime \det(V_{n-5},V_{n-4},V_1,V_2) = \pm D_4^\prime \det(V_{n-4},V_1,V_2,V_3).
\]
This together with \eqref{eq:---12|34detCon2} are exactly the determinant conditions for $ (v_1,\dots,v_{n-2}) \in C_4^{(--)[+]12|34\mp}(n-2) $. As long as these necessary conditions are fulfilled there are $ q^2 $ choices for $ v_{n-1} $ just like in the proof of \eqref{eq:k4---}.\medskip\\
\eqref{eq:+(-)[-]14|23}: This recursion is similar to \eqref{eq:+(-)[-]23}, except that there is the additional determinant condition
\begin{equation}\label{eq:+(-)[-]14|23detCon}
D_1 = \mp D_4^\prime \det(V_{n-4},V_1,V_2,V_3).
\end{equation}
Since this condition does not depend on $ v_{n-3} $, the situation is essentially the same as in the proof of \eqref{eq:+(-)[-]23}. Let $ (v_1,\dots,v_{n-2}) \in C_4^{+(-)[-]14|23\pm}(n-2) $. Then
\[
(v_1,\dots,v_{n-3}) \in C_4^{s+-}(n-3)
\]
with $ s \in \{+,-\} $. If $ s = + $, the argument is exactly the same as in the proof of \eqref{eq:+(-)[-]23}, except that we get $ c_4^{++-14\mp}(n-3) $ such elements instead of $ c_4^{++-}(n-3) $. In the case $ s = - $ we do not use the bijection that was utilized in that case of the proof of \eqref{eq:+(-)[-]23}. Rather, we directly get the two determinant conditions
\[
D_2 = \mp D_3^\prime \det(V_{n-5},V_{n-4},V_{n-3},V_2)
\]
and \eqref{eq:+(-)[-]14|23detCon}. This is equivalent to $ (v_1,\dots,v_{n-3}) \in C_4^{-+-14|23\mp}(n-3) $. The number of values for $ v_{n-2} $ is $ q $ like in the proof of \eqref{eq:k4+(-)[-]}. The claim follows.\medskip\\
\eqref{eq:(--)[+]12|34}: Let $ (v_1,\dots,v_{n-2}) \in C_4^{(--)[+]12|34\pm}(n-2) $. Then $ (v_1,\dots,v_{n-3}) \in C_4^{+(-)[+]}(n-3) $ if $ v_{n-3} \neq v_1 $ and $ (v_1,\dots,v_{n-4}) \in C_4(n-4) $ if $ v_{n-3} = v_1 $. Define
\[
\phi: K^4 \to K; V \mapsto D_1 \det(V_{n-5},V_{n-4},V_{n-3},V) \mp D_2 \det(V_{n-5},V_{n-4},V,V_1).
\]
One of the determinant equations in $ C_4^{(--)[+]12|34\pm}(n-2) $ is equivalent to $ v_{n-2} \in H $. The other one,
\begin{equation}\label{eq:(--)[+]12|34detCon}
D_3^\prime \det(V_{n-5},V_{n-4},V_1,V_2) = \mp D_4^\prime \det(V_{n-4},V_1,V_2,V_3)
\end{equation}
does not depend on $ v_{n-2} $ or $ v_{n-3} $. If $ v_{n-3} \neq v_1 $, then $ \{v_{n-5},v_{n-4},v_{n-3},v_1\} $ is independent and $ \phi(V_1) \neq 0 $. In particular, $ \dim H = 2 $. Moreover, $ \{P\} := H \cap (v_1 \vee v_2) $ is then a single point. We have
\[
E_1 \cap \{P\} = E_1 \cap H \cap (v_1 \vee v_2) = (v_{n-5} \vee v_{n-4}) \cap (v_1 \vee v_2) = (v_{n-5} \vee v_{n-4}) \cap (v_{n-3} \vee v_1) = \emptyset.
\]
$ P \in v_{n-3} \vee v_1 = v_1 \vee v_2 $ and $ P \neq v_1 $ are clearly fulfilled. Hence there is exactly one suitable value of $ v_{n-2} $ for each $ (v_1,\dots,v_{n-3}) \in C_4^{+(-)[+]34\mp}(n-3) $, namely $ v_{n-2} = P $. By \eqref{eq:+(-)[+]34}, we have $ c_4^{+(-)[+]34\mp}(n-3) = c_4^{+(-)[+]14\mp}(n-3) $. Hence the number of elements in the case $ v_{n-3} \neq v_1 $ is exactly $ c_4^{+(-)[+]14\mp}(n-3) $.\\
If $ v_{n-3} = v_1 $, then $ E_1 = v_{n-5} \vee v_{n-4} \vee v_1 \subseteq H $ and equality is prevented by the existence of $ v_{n-2} \in H \backslash E_1 $. Hence $ H = \pr^3(K) $ and $ \phi = 0 $. This condition is equivalent to $ \phi(V_{n-6}) = 0 $, which we can rephrase as
\[
D_1 = \mp D_2^\prime \det(V_{n-6},V_{n-5},V_{n-4},V_1).
\]
Here, we canceled out $ d_{n-6} $. This and \eqref{eq:(--)[+]12|34detCon} are the conditions for $ (v_1,\dots,v_{n-4}) \in C_4^{12|34\mp}(n-4) $. The number of viable choice for $ v_{n-2} $ is $ q $ just like in the proof of \eqref{eq:k4(--)}, while $ v_{n-3} $ is uniquely determined by $ v_{n-3} = v_1 $. The recursion follows.\medskip\\
\eqref{eq:-+-14|23}: Let $ (v_1,\dots,v_{n-3}) \in C_4^{-+-14|23\pm}(n-3) $. Then
\[
(v_1,\dots,v_{n-4}) \in C_4(n-4) \cup C_4^{(-)+[s]}(n-4)
\]
with $ s \in \{+,-\} $. One determinant condition is
\begin{equation}\label{eq:-+-14|23detCon}
D_1 = \pm D_4^\prime \det(V_{n-4},V_1,V_2,V_3),
\end{equation}
which does not depend on $ v_{n-3} $. For the other determinant condition consider
\[
\phi: K^4 \to K; V \mapsto D_2^\prime \det(V_{n-6},V_{n-5},V_{n-4},V) \mp D_3^\prime \det(V_{n-5},V_{n-4},V,V_2).
\]
The second condition is equivalent to $ v_{n-3} \in H $.\\
Case 1: $ \{v_{n-5},v_{n-4},v_1\} $ is independent.\\
Then $ \dim E_2 = 2 $ and $ v_{n-3} \in E_2 $. We also clearly have $ v_{n-5}, v_{n-4} \in H $. Hence $ v_{n-3} \in H $ is equivalent to $ E_2 \subseteq H $. At the same time $ E_2 = v_{n-5} \vee v_{n-4} \vee v_1 $. Therefore, $ E_2 \subseteq H $ is also equivalent to $ v_1 \in H $. Thus the second determinant condition can be equivalently rewritten as
\[
D_2^\prime \det(V_{n-6},V_{n-5},V_{n-4},V_1) = \pm D_3^\prime \det(V_{n-5},V_{n-4},V_1,V_2).
\]
It follows that $ (v_1,\dots,v_{n-4}) \in C_4^{14|23\pm}(n-4) $. On the other hand, $ (v_1,\dots,v_{n-4}) \in C_4^{14|23\pm}(n-4) $ implies that both determinant conditions in $ C_4^{-+-14|23\pm}(n-3) $ are fulfilled for any $ v_{n-3} \in E_2 $. Hence the valid choices for $ v_{n-3} $ are exactly the elements of $ E_2 \cap E_4 \backslash (E_1 \cup E_3) $. We showed in the proof of \eqref{eq:k4-+-} that there are $ q-1 $ such choices. Since $ c_4^{14|23\pm}(n-4) = c_4^{12|34\pm}(n-4) $ the number of elements in Case 1 is $ (q-1)c_4^{12|34\pm}(n-4) $.\smallskip\\
Case 2: $ \{v_{n-5},v_{n-4},v_1\} $ is dependent and $ s = + $.\\
Let $ G := H \cap E_4 $. We have $ \{v_{n-6},v_{n-5},v_{n-4},v_2\} $ independent and $ \phi(V_2) \neq 0 $. In particular, $ \dim H = 2 $ and $ \dim G = 1 $. The valid choices for $ v_{n-3} $ are the elements of $ G \backslash (E_1 \cup E_3) $. We have
\[
E_1 \cap G = (v_{n-5} \vee v_{n-4}) \cap E_4 = \{v_1\}. 
\]
Here we used the fact that $ v_1 \in v_{n-5} \vee v_{n-4} $ and that $ v_{n-4} \notin E_4 $. We also have
\[
E_3 \cap G = (v_1 \vee v_2) \cap H = \{v_1\}.
\]
Accordingly, there are $ q $ suitable options for $ v_{n-3} $. After using the bijection
\[
C_4^{(-)+[+]}(n-4) \to C_4^{+(-)[+]}(n-4); (v_1,\dots,v_{n-4}) \mapsto (v_2,\dots,v_{n-4},v_1)
\]
\eqref{eq:-+-14|23detCon} becomes
\[
D_2^\prime \det(V_{n-6},V_{n-5},V_{n-4},V_1) = \pm D_1.
\]
This is the determinant condition for $ C_4^{+(-)[+]12\pm}(n-4) $. Hence there are $ qc_4^{+(-)[+]12\pm}(n-4) $ elements in Case 2.\smallskip\\
Case 3: $ \{v_{n-5},v_{n-4},v_1\} $ is dependent and $ s = - $.\\
$ s = - $ implies $ E_1 = v_{n-5} \vee v_{n-4} \vee v_2 \subseteq H $ and thus we must have $ \phi = 0 $. This is equivalent to $ \phi(V_3) = 0 $, which we can rewrite as
\[
D_2^\prime \det(V_{n-6},V_{n-5},V_{n-4},V_3) = \mp D_3^\prime(V_{n-5},V_{n-4},V_2,V_3).
\]
After applying the bijection
\[
C_4^{(-)+[-]}(n-4) \to C_4^{+(-)[-]}(n-4); (v_1,\dots,v_{n-4}) \mapsto (v_{n-4},\dots,v_1)
\]
this equation turns into
\[
D_4^\prime \det(V_{n-6},V_1,V_2,V_3) = \mp D_3^\prime \det(V_{n-6},V_{n-5},V_1,V_2).
\]
At the same time, \eqref{eq:-+-14|23detCon} turns into
\[
D_1 = \pm D_2^\prime \det(V_{n-6},V_{n-5},V_{n-4},V_1).
\]
These are equivalent to the determinant conditions in $ C_4^{+(-)[-]12|34\pm}(n-4) $. We have $ q^2 $ choices for $ v_{n-3} $ just like in the proof of \eqref{eq:k4-+-}. The three cases put together prove \eqref{eq:-+-14|23}.
  
\end{proof}

Finally, we can consider sets with $ 3 $ determinant conditions.

\begin{defi}
Let $ n \geq 4 $ be divisible by 4. For any tuple $ (v_1,\dots,v_m) \in (\pr^3(K))^m $ with $ m \in \{n,n-1,n-2,n-3\} $ in the definitions of the following sets let $ (V_1,\dots,V_m) \in (K^4)^m $ be an arbitrary lift. Define $ d_i := \det(V_i,V_{i+1},V_{i+2},V_{i+3}) $ with the indices considered modulo $ n $. Let $ D_j^\prime := d_j d_{j+4} \dots d_{n-12+j} $ and $ D_ j := D_j^\prime d_{n-8+j} $ for $ j \in \{1,2,3,4\} $. Let $ s \in \{+,-\} $ Then define
\begin{align*}
C_4^{1234s}(n) := \{&(v_1,\dots,v_n) \in C_4(n) \mid D_1 d_{n-3} = sD_2d_{n-2} = D_3 d_{n-1} = sD_4 d_n\},\\
C_4^{---1234s}(n-1) := \{&(v_1,\dots,v_{n-1}) \in C_4^{---}(n-1) \mid\\
&D_1 d_{n-4} = sD_2 \det(V_{n-4},V_{n-2},V_{n-1},V_1)\\
&= D_3 \det(V_{n-4},V_{n-1},V_1,V_2) = sD_4 \det(V_{n-4},V_1,V_2,V_3)\},\\
C_4^{(--)[+]1234s}(n-2) := \{&(v_1,\dots,v_{n-2}) \in C_4^{(--)[+]}(n-2) \mid\\
&D_1 d_{n-5} = sD_2 \det(V_{n-5},V_{n-4},V_{n-2},V_1)\\
&= D_3 \det(V_{n-5},V_{n-4},V_1,V_2) = -s D_4^\prime d_{n-5} \det(V_{n-4},V_1,V_2,V_3)\}.\\
\end{align*}
\end{defi}

\begin{lem}\label{lem:recursionsk41234}
Let $ n \geq 8 $ be divisible by $ 4 $. Then the following recursions hold:
\begin{align}
c_4^{1234\pm}(n) = &c_4(n-1) + 3qc_4^{++-34\mp}(n-1) + 2q^2c_4^{+--234\mp}(n-1)\label{eq:+++1234}\\
&+ q^2c_4^{-+-12|34\mp}(n-1) + q^3c_4^{---1234\mp}(n-1),\nonumber\\
c_4^{---1234\pm}(n-1) = &c_4^{+--14\pm}(n-2) + qc_4^{+(-)[-]14|23\mp}(n-2)\label{eq:---1234}\\
&+ qc_4^{-(-)[+]234\mp}(n-2) + q^2c_4^{(--)[+]1234\mp}(n-2),\nonumber\\
c_4^{(--)[+]1234\pm}(n-2) = &c_4^{+(-)[+]134\mp}(n-3) + qc_4^{1234\mp}(n-4).\label{eq:(--)[+]1234}
\end{align}
\end{lem}
\begin{proof}
For any tuple $ (v_1,\dots,v_m) \in (\pr^3(K))^m $ of projective points with $ m \in \{n,n-1,n-2\} $ that we consider in this proof let $ (V_1,\dots,V_m) \in (K^4)^m $ be an arbitrary lift. We set $ d_i := \det(V_i,V_{i+1},V_{i+2},V_{i+3}) $ for $ i \leq n $ (the indices are considered modulo $ n $) and $ D_j := d_j d_{j+4} \dots d_{n-8+j} $ and $ D_j^\prime := d_j d_{j+4} \dots d_{n-12+j} $ for $ j \in \{1,2,3,4\} $. Also let $ E_1 := v_{m-3} \vee v_{m-2} \vee v_{m-1} $, $ E_2 := v_{m-2} \vee v_{m-1} \vee v_1 $, $ E_3 := v_{m-1} \vee v_1 \vee v_2 $ and $ E_4 := v_1 \vee v_2 \vee v_3 $. For two of the recursions we consider linear maps $ \phi_i: K^4 \to K $ and their kernels $ H_i := \pr(\Ker \phi_i) $. The definition of $ H_i $ will not depend on the choice of lift. \smallskip\\
\eqref{eq:+++1234}: Let $ (v_1,\dots,v_n) \in C_4^{1234\pm}(n) $. We have
\[
(v_1,\dots,v_{n-1}) \in C_4^{s_1s_2s_3}(n-1)
\]
with $ s_i \in \{+,-\} $. Consider the maps
\begin{align*}
&\phi_1: K^4 \to K; V \mapsto D_1 \det(V_{n-3},V_{n-2},V_{n-1},V) \mp D_2 \det(V_{n-2},V_{n-1},V,V_1),\\
&\phi_2: K^4 \to K; V \mapsto D_2 \det(V_{n-2},V_{n-1},V,V_1) \mp D_3 \det(V_{n-1},V,V_1,V_2),\\
\intertext{and}
&\phi_3: K^4 \to K; V \mapsto D_3 \det(V_{n-1},V,V_1,V_2) \mp D_4 \det(V,V_1,V_2,V_3).
\end{align*}
For any $ (v_1,\dots,v_{n-1}) \in C_4^{s_1s_2s_3}(n-1) $ and any $ v_n \in \pr^3(K) $ the three determinant equations in the definition of $ C_4^{1234\pm}(n) $ are fulfilled if and only if $ v_n \in H_1 \cap H_2 \cap H_3 $. It follows that $ (v_1,\dots,v_n) \in C_4^{1234\pm}(n) $ is equivalent to $ v_n \in H_1 \cap H_2 \cap H_3 \backslash (E_1 \cup E_2 \cup E_3 \cup E_4) $. \\
If $ s_i = + $ for $ i = 1,2,3 $, then $ \phi_i(V_i) \neq 0 $ and in particular, $ \dim H_i = 2 $. If $ s_i = - $, then $ E_i = E_{i+1} \subseteq H_i $. In this case, $ v_n \in H_i \backslash E_i $ implies $ H_i = \pr^3(K) $ or equivalently $ \phi_i = 0 $. Also note that $ E_i \cap H_i = E_i \cap E_{i+1} = E_{i+1} \cap H_i $ for $ i = 1,2,3 $.\smallskip\\
Case 1: $ s_1 = s_2 = s_3 = +$.\\
$ H_1 \cap H_2 $ is a projective line since $ v_{n-2} \in H_1 \backslash H_2 $. Moreover, $ v_{n-1} \in H_1 \cap H_2 $, but $ v_{n-1} \notin H_3 $. Hence $ \{P\} := H_1 \cap H_2 \cap H_3 $ is a single point, and only this point can be a suitable choice for $ v_n $. On the other hand, we have
\[
E_j \cap \{P\} = E_j \cap H_1 \cap H_2 \cap H_3 = E_1 \cap E_2 \cap E_3 \cap E_4 = \emptyset
\]
for $ j = 1,2,3,4 $ (see the proof of \eqref{eq:k4+++} for the last equality). Consequently, there are exactly $ c_4(n-1) $ elements of $ C_4^{1234\pm}(n) $ with $ s_1 = s_2 = s_3 = + $.\smallskip\\
Case 2: $ s_1 = s_2 = + $, $ s_3 = - $.\\
With the same argument as in Case 1 we see that $ G := H_1 \cap H_2 $ is a projective line. We have
\[
E_j \cap G = E_1 \cap E_2 \cap E_3 = \{v_{n-1}\}
\]
for $ j = 1,2,3,4 $ (for $ j = 4 $ we use $ E_3 = E_4) $. Hence there are $ q $ viable choices for $ v_n $ as long as the necessary condition $ \phi_3 = 0 $ is fulfilled. The latter condition is equivalent to $ \phi_3(V_{n-2}) = 0 $ since $ v_{n-2} \notin E_3 = E_4 $. Therefore, the condition can be rephrased as
\begin{equation}\label{eq:+++1234detCon}
D_3 \det(V_{n-2}, V_{n-1},V_1,V_2) = \mp D_4 \det(V_{n-2},V_1,V_2,V_3).
\end{equation}
This is equivalent to $ (v_1,\dots,v_{n-1}) \in C_4^{++-34\mp}(n-1) $. It follows that there are exactly $ qc_4^{++-34\mp}(n-1) $ elements in Case 2.\smallskip\\
Case 3: $ s_1 = s_3 = + $, $ s_2 = - $.\\
Let $ G:= H_1 \cap H_3 $. $ v_1 \in H_3 \backslash H_1 $, hence $ \dim G = 1 $. Since $ E_2 = E_3 $ we have
\[
E_j \cap H_1 \cap H_3 = E_1 \cap E_2 \cap E_3 \cap E_4 = E_1 \cap E_2 \cap E_4
\]
for $ j = 1,2,3,4 $. We showed in the proof of \eqref{eq:k4+++} that $ E_1 $, $ E_2 $ and $ E_4 $ intersect in a single point. Hence there are $ q $ points on $ G $ not in $ E_1 \cup E_2 \cup E_3 \cup E_4 $ that work for our purposes. The final condition $ \phi_2 = 0 $ can be tested with $ V_{n-3} $, resulting in the equation
\[
D_2 \det(V_{n-3},V_{n-2},V_{n-1},V_1) = \mp D_3 \det(V_{n-3},V_{n-1},V_1,V_2).
\]
After applying the bijection
\begin{equation}\label{eq:+++1234bijection}
C_4^{s_1s_2+}(n-1) \to C_4^{+s_1s_2}(n-1); (v_1,\dots,v_{n-1}) \mapsto (v_2,\dots,v_{n-1},v_1)
\end{equation}
our condition turns into \eqref{eq:+++1234detCon}. Hence we also have $ qc_4^{++-34\mp}(n-1) $ elements in the current case.\smallskip\\
Case 4: $ s_1 = - $, $ s_2 = s_3 = + $.\\
Let $ G := H_2 \cap H_3 $. Like we showed in Case 1, $ \dim G = 1 $. Additionally,
\[
E_j \cap G = E_2 \cap E_3 \cap E_4 = \{v_1\}
\]
for $ j = 1,2,3,4 $. The necessary condition $ \phi_1 = 0 $ can be tested with $ V_{n-4} $, resulting in
\[
D_1 d_{n-4} = \mp D_2 \det(V_{n-4},V_{n-2},V_{n-1},V_1).
\]
After using \eqref{eq:+++1234bijection} twice, this equation becomes \eqref{eq:+++1234detCon}. In conclusion, Cases 2, 3 and 4 combined yield $ 3qc_4^{++-34\mp}(n-1) $ elements.\smallskip\\
Case 5: $ s_1 = + $, $ s_2 = s_3 = - $:\\
We have
\begin{align*}
E_1 \cap H_1 &= E_1 \cap E_2 = v_{n-2} \vee v_{n-1}\\
&= E_2 \cap H_1
\end{align*}
and $ E_2 = E_3 = E_4 $. Hence all points on $ H_1 $ except for one projective line are suitable picks for $ v_n $ and hence we have $ q^2 $ options. The necessary conditions $ \phi_2 = \phi_3 = 0 $ can be tested with $ V_{n-3} $, resulting in the equivalent equations
\begin{align}
D_2 \det(V_{n-3},V_{n-2},V_{n-1},V_1) &= \mp D_3 \det(V_{n-3},V_{n-1},V_1,V_2)\label{eq:+++1234detCon2}\\
&=D_4 \det(V_{n-3},V_1,V_2,V_3).\nonumber
\end{align}
This is true if and only if $ (v_1,\dots,v_{n-1}) \in C_4^{+--234\mp}(n-1) $. Hence the number of elements in Case 5 is $ q^2c_4^{+--234\mp}(n-1) $.\smallskip\\
Case 6: $ s_1 = s_2 = - $, $ s_3 = + $.\\
Similar to the last case we have
\[
E_j \cap H_3 = E_3 \cap E_4 = v_1 \vee v_2
\]
for $ j = 1,2,3,4 $ and there are again $ q^2 $ valid choices for $ v_n $. We can test $ \phi_1 = \phi_2 = 0 $ with $ V_{n-4} $, resulting in
\begin{align*}
D_1 d_{n-4} &= \mp D_2 \det(V_{n-4},V_{n-2},V_{n-1},V_1)\\
&= D_3 \det(V_{n-4},V_{n-1},V_1,V_2).
\end{align*}
After applying the bijection \eqref{eq:+++1234bijection}, this becomes \eqref{eq:+++1234detCon2}. Hence Case 6 also yields $ q^2 c_4^{+--234\mp}(n-1) $ elements.\smallskip\\
Case 7: $ s_1 = s_3 = - $, $ s_2 = + $.\\
Now we have
\[
E_j \cap H_2 = E_2 \cap E_3 = v_{n-1} \vee v_1
\]
for $ j = 1,2,3,4 $ since $ E_1 = E_2 $ and $ E_3 = E_4 $. Consequently there are $ q^2 $ suitable values for $ v_n $. We test $ \phi_1 = 0 $ with $ V_2 $ and $ \phi_3 = 0 $ with $ V_{n-2} $. This results in the equations
\[
D_1 \det(V_{n-3},V_{n-2},V_{n-1},V_2) = \mp D_2 \det(V_{n-2},V_{n-1},V_1,V_2)
\]
and
\[
D_3 \det(V_{n-2},V_{n-1},V_1,V_2) = \mp D_4 \det(V_{n-2},V_1,V_2,V_3).
\]
These are exactly the conditions for $ (v_1,\dots,v_{n-1}) \in C_4^{-+-12|34\mp}(n-1) $. Hence Case 7 accounts for $ q^2 c_4^{-+-12|34\mp}(n-1) $ elements.\smallskip\\
Case 8: $ s_1 = s_2 = s_3 = - $.\\
The determinant conditions now only depend on $ (v_1,\dots,v_{n-1}) $. As long as they are fulfilled, we will have $ q^3 $ choices for $ v_n $ just like in the proof of \eqref{eq:k4+++}. We can test $ \phi_1 = \phi_2 = \phi_3 = 0 $ with $ V_{n-4} $ and get
\begin{align*}
D_1 d_{n-4} &= \mp D_2 \det(V_{n-4},V_{n-2},V_{n-1},V_1)\\
&= D_3 \det(V_{n-4},V_{n-1},V_1,V_2)\\
&= \mp D_4 \det(V_{n-4},V_1,V_2,V_3).
\end{align*}
This is exactly the condition for $ (v_1,\dots,v_{n-1}) \in C_4^{---1234\mp}(n-1) $ and the claimed recursion is proved. Note that we could have tested with $ V_4 $ instead. In particular, the first equation is equivalent to
\[
D_1 \det(V_{n-3},V_{n-2},V_{n-1},V_4) = \mp D_2 \det(V_{n-2},V_{n-1},V_1,V_4).
\]
We will use this in the proof of the next recursion.\medskip\\
\eqref{eq:---1234}: Let $ (v_1,\dots,v_{n-1}) \in C_4^{---1234\pm}(n-1) $. Then
\[
(v_1,\dots,v_{n-2}) \in C_4^{+--}(n-2) \cup C_4^{+(-)[-]}(n-2) \cup C_4^{(-)-[+]}(n-2) \cup C_4^{(--)[+]}(n-2).
\]
Note that one of the required determinant equations is
\[
D_1 d_{n-4} = \pm D_4 \det(V_{n-4},V_1,V_2,V_3),
\]
which can be simplified to
\begin{equation}\label{eq:---1234detCon}
D_1 = \pm D_4^\prime \det(V_{n-4},V_1,V_2,V_3).
\end{equation}
This condition no longer depends on $ v_{n-1} $. For the other two required equations, consider the maps
\begin{align*}
&\phi_1: K^4 \to K; V \mapsto D_1 \det(V_{n-4},V_{n-3},V_{n-2},V) \mp D_2 \det(V_{n-4},V_{n-2},V,V_1)\\
\intertext{and}
&\phi_2: K^4 \to K; V \mapsto D_2 \det(V_{n-4},V_{n-2},V,V_1) \mp D_3 \det(V_{n-4},V,V_1,V_2).
\end{align*}
The determinant equations in $ C_4^{---1234\pm}(n-1) $ are fulfilled if and only if $ v_{n-1} \in H_1 \cap H_2 $ and \eqref{eq:---1234detCon} holds.\\
Let $ M_1 := \{v_{n-3},v_{n-2},v_1\} $ and $ M_2 := \{v_{n-2},v_1,v_2\} $. If $ M_i $ is independent for $ i \in \{1,2\} $, then $ E_4 = \pr(\langle M_i \rangle) $ and $ \phi_i(V_i) \neq 0 $. In particular, $ \dim H_i = 2 $. If $ M_i $ is dependent, then $ v_{n-4} \vee v_{n-2} \vee v_1 \subseteq H_i $. If we have equality, then $ H_i \cap E_4 = v_{n-2} \vee v_1 $ and since we must have $ v_{n-1} \in H_1 \cap H_2 \cap E_4 $ it follows that $ \{v_{n-2},v_{n-1},v_1\} $ is dependent. This contradicts our assumptions (recall that the definition of $ C_4^{---}(n-1) $ requires $ \{v_{n-2},v_{n-1},v_1\} $ and $ \{v_{n-1},v_1,v_2\} $ to be independent). Hence we must have $ H_i = \pr^3(K) $ or equivalently $ \phi_i = 0 $ if $ M_i $ is dependent. The viable choices for $ v_{n-1} $ for a given $ (v_1,\dots,v_{n-2}) $ are exactly the elements of $ H_1 \cap H_2 \cap E_4 \backslash (E_1 \cup (v_{n-2} \vee v_1) \cup (v_1 \vee v_2)) $.\smallskip\\
Case 1: $ M_1 $ and $ M_2 $ are both independent.\\
We have $ v_1 \in H_2 \backslash H_1 $, hence $ \dim H_1 \cap H_2 = 1 $. Moreover, $ v_{n-4} \in H_1 \cap H_2 \backslash E_4 $ implies that $ \{P\} := H_1 \cap H_2 \cap E_4 $ is a single point. On the other hand,
\[
E_1 \cap \{P\} = (v_{n-4} \vee v_{n-2}) \cap H_2 \cap E_4 = \{v_{n-4}\} \cap E_4 = \emptyset.
\]
Additionally,
\[
(v_{n-2} \vee v_1) \cap \{P\} = \{v_{n-2}\} \cap H_2 \cap E_4 = \emptyset
\]
and
\[
(v_1 \vee v_2) \cap \{P\} = \{v_1\} \cap H_1 \cap E_4 = \emptyset.
\]
Therefore the one value $ v_{n-1} = P $ actually works. \eqref{eq:---1234detCon} is equivalent to $ (v_1,\dots,v_{n-2}) \in C_4^{+--14\pm}(n-2) $. The number of elements in Case 1 is accordingly $ c_4^{+--14\pm}(n-2) $.\smallskip\\
Case 2: $ M_1 $ is independent, $ M_2 $ is dependent.\\
Let $ G := H_1 \cap E_4 $. We still have $ v_{n-4} \in H_1 \backslash E_4 $ and hence $ \dim G = 1 $. We also have
\[
E_1 \cap G = (v_{n-4} \vee v_{n-2}) \cap E_4 = \{v_{n-2}\}
\]
and
\begin{align*}
(v_{n-2} \vee v_1) \cap G &= \{v_{n-2}\}\\
&= (v_1 \vee v_2) \cap G.
\end{align*}
Thus there are $ q $ suitable choices for $ v_{n-1} $. The condition $ \phi_2 = 0 $ is equivalent to $ \phi_2(V_{n-3}) = 0 $, which is in turn equivalent to
\[
D_2 \det(V_{n-4},V_{n-3},V_{n-2},V_1) = \mp D_3 \det(V_{n-4},V_{n-3},V_1,V_2).
\]
Together with \eqref{eq:---1234detCon} these are the condition for $ (v_1,\dots,v_{n-2}) \in C_4^{+(-)[-]14|23\mp}(n-2) $. Hence we get $ q c_4^{+(-)[-]14|23\mp}(n-2) $ elements in the current case.\smallskip\\
Case 3: $ M_1 $ is dependent, $ M_2 $ is independent.\\
Let $ G:= H_2 \cap E_4$. $ v_{n-4} \in H_2 \backslash E_4 $ implies $ \dim G = 1 $. Additionally,
\begin{align*}
E_1 \cap G &= (v_{n-2} \vee v_1) \cap H_2 = \{v_1\}\\
&=  (v_{n-2} \vee v_1) \cap G
\end{align*}
and
\[
(v_1 \vee v_2) \cap G = \{v_1\}.
\]
Hence there are $ q $ choices for $ v_{n-1} $ if the other determinant conditions are fulfilled. Note that the first determinant condition
\[
D_1 d_{n-4} = \pm D_2 \det(V_{n-4},V_{n-2},V_{n-1},V_1)
\]
is equivalent to
\[
D_1 \det(V_{n-3},V_{n-2},V_{n-1},V_4) = \pm D_2 \det(V_{n-2},V_{n-1},V_1,V_4)
\]
as we showed in the proof of \eqref{eq:+++1234}. Hence we can replace $ \phi_1 $ by
\[
\phi_1^\prime: K^4 \to K; V \mapsto D_1 \det(V_{n-3},V_{n-2},V,V_4) \mp D_2 \det(V_{n-2},V,V_1,V_4).
\]
Let $ H_1^\prime := \pr(\Ker \phi_1^\prime) $. Then $ v_{n-3} \vee v_{n-2} \vee v_4 = v_{n-2} \vee v_1 \vee v_4 \subseteq H_1^\prime $, and in the case of equality $ H_1^\prime \cap E_4 = v_{n-3} \vee v_{n-2} $ follows, which is impossible. Hence $ \phi_1^\prime = 0 $ is again a necessary condition and it is equivalent to $ \phi_1^\prime(V_2) = 0 $. Accordingly, we can rewrite the condition as
\[
D_1 \det(V_{n-3},V_{n-2},V_2,V_4) = \mp D_2 \det(V_{n-2},V_1,V_2,V_4).
\]
Through the bijection
\begin{equation}\label{eq:---1234bijection}
C_4^{(-)-[+]}(n-2) \to C_4^{-(-)[+]}(n-2); (v_1,\dots,v_{n-2}) \mapsto (v_{n-2},\dots,v_1)
\end{equation}
our condition becomes
\begin{equation}\label{eq:---1234detCon2}
D_3 \det(V_{n-5},V_{n-3},V_1,V_2) = \mp D_2 \det(V_{n-5},V_{n-3},V_{n-2},V_1),
\end{equation}
which is equivalent to the first determinant condition in $ C_4^{-(-)[+]234\mp}(n-2) $.\\
Now observe our other determinant condition \eqref{eq:---1234detCon} and define the map
\[
\psi_1: K^4 \to K; V \mapsto \det(V_1,V_2,V,V_4) D_1^* \mp D_4^\prime \det(V_{n-4},V_1,V_2,V),
\]
where $ D_1^* := d_5 \dots d_{n-7} $. It is clear that \eqref{eq:---1234detCon} is equivalent to $ \psi_1(V_3) = 0 $. It is also clear that $ V_1,V_2 \in \Ker(\psi_1) $. Moreover, $ v_1 \vee v_2 \vee v_3 = v_{n-2} \vee v_1 \vee v_2 $. Hence \eqref{eq:---1234detCon} is also equivalent to $ \psi_1(V_{n-2}) = 0 $, which can be rewritten as
\begin{equation}\label{eq:---1234detCon3}
\det(V_{n-2},V_1,V_2,V_4) D_1^* = \pm D_4^\prime \det(V_{n-4},V_{n-2},V_1,V_2).
\end{equation}
Now define
\[
\psi_2: K^4 \to K; V \mapsto \det(V_{n-2},V,V_2,V_4) D_1^* \mp D_4^\prime \det(V_{n-4},V_{n-2},V,V_2).
\]
\eqref{eq:---1234detCon3} is equivalent to $ \psi_2(V_1) = 0 $. Since $ v_{n-2} \vee v_1 \vee v_2 = v_{n-3} \vee v_{n-2} \vee v_2 $ and $ V_{n-2},V_2 \in \Ker \psi_2 $, \eqref{eq:---1234detCon3} is also equivalent to $ \psi_2(V_{n-3}) = 0 $, which is in turn equivalent to
\[
D_1^* \det(V_{n-3},V_{n-2},V_2,V_4) = \pm D_4^\prime \det(V_{n-4},V_{n-3},V_{n-2},V_2).
\]
Under \eqref{eq:---1234bijection} this transforms into
\[
D_3^\prime \det(V_{n-5},V_{n-3},V_1,V_2) = \pm D_4^\prime \det(V_{n-3},V_1,V_2,V_3).
\]
After multiplying with $ \mp d_{n-5} $, this becomes the second determinant equation in the definition of $ C_4^{-(-)[+]234\mp}(n-2) $. Hence the number of elements in Case 3 is $ q c_4^{-(-)[+]234\mp}(n-2) $.\smallskip\\
Case 4: $ M_1 $ and $ M_2 $ dependent.\\
We can test $ \phi_1 = \phi_2 = 0 $ with $ V_{n-5} $, resulting in the condition
\begin{align*}
D_1 d_{n-5} &= \mp D_2 \det(V_{n-5},V_{n-4},V_{n-2},V_1)\\
&= D_3 \det(V_{n-5},V_{n-4},V_1,V_2).
\end{align*}
Together with \eqref{eq:---1234detCon}, this is equivalent to $ (v_1,\dots,v_{n-2}) \in C_4^{(--)[+]1234\mp}(n-2) $. The number of choices for $ v_{n-2} $ is $ q^2 $ just like in the proof of \eqref{eq:k4---}. Hence the recursion is proved.\medskip\\
\eqref{eq:(--)[+]1234}: This recursion is similar to \eqref{eq:(--)[+]23}, except that we have the additional determinant equations
\begin{align*}
D_1 &= D_3^\prime \det(V_{n-5},V_{n-4},V_1,V_2)\\
&= \mp D_4^\prime \det(V_{n-4},V_1,V_2,V_3)
\end{align*}
(after canceling out $ d_{n-5} $), which do not depend on $ v_{n-2} $ or $ v_{n-3} $. Recall that \eqref{eq:(--)[+]23} is the recursion
\[
c_4^{(--)[+]23\pm}(n-2) = c_4^{+(-)[+]}(n-3) + qc_4^{23\mp}(n-4).
\]
After adding the additional determinant equations to the sets on the right hand side, the summands turn into $ c_4^{+(-)[+]134\mp}(n-3) $ and $ qc_4^{1234\mp}(n-4) $. This yields the claimed recursion.
\end{proof}

We will now calculate the cardinalities of the sets we defined. Since the computation is very long and tedious, we will not cover all the steps of the calculation and instead give the results. One can easily check that the solutions do in fact satisfy the recursions and the starting values, for instance with a computer. We will also not give solutions for all the sets, but we will give sufficient solutions that the cardinalities of all the remaining sets can easily be derived through the recursions.

\begin{theorem}\label{thm:c4(n)gcd4}
Let $ n \in \N $ be divisible by $ 4 $. Let $ \lvert K \rvert = q $.
Set
\[
s := \begin{cases}
    + &\text{if }\charac(K) \neq 2 \text{ and } 8 \mid n,\\
    - &\text{otherwise}.
\end{cases}
\]
Then
\begin{align*}
(q-1)^3 c_4^*(n) = q^{3n} &+ 4f_1^s(q)q^{\frac{9}{4}n} + f_2(q)q^{2n} + 6f_3^s(q)q^{\frac{7}{4}n} + f_4^s(q)q^{\frac{3}{2}n}\\
&+ 6f_3^s(q)q^{\frac{5}{4}n+1} + f_2(q)q^{n+2} + 4f_1^s(q)q^{\frac{3}{4}n+3} + q^6
\end{align*}
with
\begin{align*}
f_1^+(q) &:= -q^3-q^2-q-1,\\
f_2(q) &:= 3q^5-3q^4+3q^3-6q^2-6,\\
f_3^+(q) &:= -q^6+q^5+3q^4+6q^3+7q^2+5q+6,\\
f_4^+(q) &:= 3q^7+6q^6-26q^5-41q^4-73q^3-56q^2-44q-9,
\intertext{and}
f_1^-(q) &:= q^4-q^3-q^2-q-2,\\
f_3^-(q) &:= q^7-3q^6-q^5+5q^3+9q^2+7q+6,\\
f_4^-(q) &:= q^9-6q^8+4q^7+29q^6+3q^5-20q^4-74q^3-79q^2-74q-24.
\end{align*}
\end{theorem}

\begin{proof}
First we solve the system of recursions in Lemma \ref{lem:recursionsk423}. Let $ s \in \{+,-\} $ and set
\[
\varphi(s,n) := \begin{cases}
    + &\text{if } \charac(K) \neq 2 \text{ and } s = - \text{ and } 8 \mid n,\\
    + &\text{if } \charac(K) \neq 2 \text{ and } s = + \text{ and } 8 \nmid n,\\
    - &\text{otherwise}.
\end{cases}
\]
For any $ n \geq 4 $ divisible by $ 4 $ the following equations hold:
\begin{align*}
(q-1)c_4^{23s}(n) = q^{3n} &+ f_1^{\varphi(s,n)}(q)q^{\frac{9}{4}n} + 3g_1(q)q^{2n+1} + 2g_2^{\varphi(s,n)}(q)q^{\frac{3}{2}n+1}\\
&+ 3g_1(q)q^{n+3} + f_1^{\varphi(s,n)}(q)q^{\frac{3}{4}n+3} + q^6,\\
(q-1)c_4^{++-23s}(n) = q^{3n-1} &+ f_1^{\varphi(s,n)}(q)q^{\frac{9}{4}n-1} + g_3(q)q^{2n} + g_4^{\varphi(s,n)}(q)q^{\frac{3}{2}n}\\
&+ g_5(q)q^{n+2} -f_1^{\varphi(s,n)}(q)q^{\frac{3}{4}n+3} - q^6,\\
(q-1)c_4^{+--12s}(n) = q^{3n-2} &+ f_1^{\varphi(s,n)}(q)q^{\frac{9}{4}n-2} - g_6(q)q^{2n-1} -g_2^{\varphi(s,n)}(q)q^{\frac{3}{2}n}\\
&+ g_6(q)q^{n+2} + f_1^{\varphi(s,n)}(q)q^{\frac{3}{4}n+3} + q^6,\\
(q-1)c_4^{+(-)[-]12s}(n) = q^{3n-3} &+ f_1^{\varphi(s,n)}(q)q^{\frac{9}{4}n-3} + g_1(q)q^{2n-2}\\
&- g_1(q)q^{n+3} - f_1^{\varphi(s,n)}q^{\frac{3}{4}n+3} - q^6,
\end{align*}
and $ c_4^{+(-)[+]12s}(n) = c_4^{+--12s}(n) $. Here we define
\begin{align*}
g_1(q) &:= q^2 + q + 1,\\
g^+_2(q) &:= -q^4 - q^3 - 2q^2 - q - 1,\\
g^-_2(q) &:= q^5 - q^3 -2q^2 -2q -2,\\
g_3(q) &:= -q^3+q^2+q+2,\\
g^+_4(q) &:= q^5+q^3-q^2-1,\\
g^-_4(q) &:= -q^6+q^5+q^4+q^3-2,\\
g_5(q) &:= -2q^3-q^2-q+1,\\
g_6(q) &:= q^3-1,
\end{align*}
With these explicit formulas, the other explicit formulas we already know and the recursions from Lemma \ref{lem:recursionsk423}, we can derive explicit formulas for all the other cardinalities mentioned in that lemma. Then it can be checked that these formulas do indeed satisfy the recursions in the Lemma. \\
Next we consider the recursions in Lemma \ref{lem:recursionsk413}. Some of the solutions for $ n \geq 4 $ divisible by $ 4 $ are
\begin{align*}
c_4^{13}(n) &= c_4^\#(n),\\
(q-1)c_4^{++-13}(n) &= q^{3n-1} + h_1(q)q^{2n-1} + h_2(q)q^{\frac{3}{2}n} + h_3(q)q^{n+2} - q^6,\\
(q-1)c_4^{-+-13}(n) &= q^{3n-2} + h_4q^{2n-2} + h_5(q)q^{\frac{3}{2}n} + h_6(q)q^{n+2} + q^6,
\end{align*}
with
\begin{align*}
h_1(q) &:= q^5-2q^4+q^3-2q^2+q-2,\\
h_2(q) &:= -q^6+2q^5-2q^4+4q^3-3q^2+2q-2,\\
h_3(q) &:= -q^5+q^4-2q^3+2q^2+3,\\
h_4(q) &:= q^5-3q^4-q^3-4q^2-2,\\
h_5(q) &:= -2q^5+3q^4+8q^2+2q+5,\\
h_6(q) &:= q^5-q^4+q^3-4q^2-2q-4.
\end{align*}
Like before, we can derive explicit formulas for the other sets in the lemma with these recursions, and then check that the recursions hold.\\
Next up are Lemma \ref{lem:recursionsk4234} and Lemma \ref{lem:recursionsk412|34}. Let $ 4 \in 4 \N $. We have
\begin{align*}
(q-1)^2c_4^{234s}(n) = q^{3n} &+ 2f_1^{\varphi(s,n)}(q)q^{\frac{9}{4}n} + i_1(q)q^{2n} + f_3^{\varphi(s,n)}(q)q^{\frac{7}{4}n} + i_2^{\varphi(s,n)}(q)q^{\frac{3}{2}n+1}\\
&+ f_3^{\varphi(s,n)}(q)q^{\frac{5}{4}n+1} + i_1(q)q^{n+2} + 2f_1^{\varphi(s,n)}(q)q^{\frac{3}{4}n+3} + q^6
\end{align*}
and $ c_4^{12|34s}(n) = c_4^{234s}(n) $ with
\begin{align*}
i_1(q) &:= q^5-q^4+3q^3+2q-2\\
i_2^+(q) &:= 2q^5-6q^4-6q^3-14q^2-8q-8\\
i_2^-(q) &:= 6q^5-2q^4-6q^3-14q^2-12q-12.
\end{align*}
This is enough to calculate the cardinalities of the other sets in Lemma \ref{lem:recursionsk4234}: First we calculate $ c_4^{(--)[+]234\pm}(n) $, then $ c_4^{---234\pm}(n) $, then $ c_4^{+--234\pm}(n) $ (using \eqref{eq:+++234}) and then the other sets follow (here we no longer assume $ n $ divisible by $ 4 $, instead we assume that $ n $ has the appropriate residue modulo $ 4 $). Similarly, all cardinalities in Lemma \ref{lem:recursionsk412|34} can be calculated from $ c_4^{12|34\pm}(n) $. Then it can be verified that the given formulas fulfill the recursions.\\
The explicit formula for $ c_4^{1234s}(n) $ with $ n $ divisible by $ 4 $ is
\begin{align*}
(q-1)^3 c_4^{1234s}(n) = q^{3n} &+ 4f_1^{\varphi(s,n)}(q)q^{\frac{9}{4}n} + f_2(q)q^{2n} + 6f_3^{\varphi(s,n)}(q)q^{\frac{7}{4}n} + f_4^{\varphi(s,n)}(q)q^{\frac{3}{2}n}\\
&+ 6f_3^{\varphi(s,n)}(q)q^{\frac{5}{4}n+1} + f_2(q)q^{n+2} + 4f_1^{\varphi(s,n)}(q)q^{\frac{3}{4}n+3} + q^6.
\end{align*}
Again, this is enough to calculate explicit formulas for the sets in Lemma \ref{lem:recursionsk41234} and these fulfill the recursions in that lemma. The formula in the theorem arises as a special case since $ c_4^*(n) = c_4^{1234-}(n) $.\\
The only thing that remains is to check the starting values. If $ s = - $, then all determinant conditions that appear in the relevant sets are trivial in the case $ n = 4 $, hence we have
\[
c_4^{1234-}(4) = c_4^{12|34-}(4) = c_4^{234-}(4) = c_4^{23-}(4) = c_4^{13}(4) = c_4(4).
\]
If $ s = + $, then the corresponding determinant conditions are impossible to fulfill and the corresponding sets are empty for $ n = 4 $, except if $ \charac(K) = 2 $. In the latter case, the determinant conditions are the same for $ s = + $ and $ s = - $. All other sets in this proof are empty in the case $ n = 4 $, since even the corresponding sets without determinant conditions are empty for $ n = 4 $. One can check that these starting values agree with the given formulas.
\end{proof}

\begin{rem}
Our calculations show $ c_4^{13}(n) = c_4^\#(n) $ and $ c_4^{234s}(n) = c_4^{12|34s}(n) $. This implies that bijections exist between the corresponding sets. However, it is unclear if there are nice, canonical bijections. If such bijections were found, parts of the proof could be simplified.
\end{rem}

\begin{cor}\label{cor:k4gcd4}
Let $ n \in \N $ be divisible by $ 4 $. Let $ \lvert K \rvert = q $.
Set
\[
s := \begin{cases}
    + &\text{if }\charac(K) \neq 2 \text{ and } 8 \mid n,\\
    - &\text{otherwise}.
\end{cases}
\]
Then
\begin{align*}
f_q(4,n) = \frac{f(q)}{(q^3+q^2+q+1)(q^2+q+1)(q+1)q^6(q-1)^3}
\end{align*}
is the number of tame $ \SL_4 $-frieze patterns of width $ n - k - 1 $ over $ K $, where
\begin{align*}
f(q) = q^{3n} &+ 4f_1^s(q)q^{\frac{9}{4}n} + f_2(q)q^{2n} + 6f_3^s(q)q^{\frac{7}{4}n} + f_4^s(q)q^{\frac{3}{2}n}\\
&+ 6f_3^s(q)q^{\frac{5}{4}n+1} + f_2(q)q^{n+2} + 4f_1^s(q)q^{\frac{3}{4}n+3} + q^6
\end{align*}
and
\begin{align*}
f_1^+(q) &:= -q^3-q^2-q-1,\\
f_2(q) &:= 3q^5-3q^4+3q^3-6q^2-6,\\
f_3^+(q) &:= -q^6+q^5+3q^4+6q^3+7q^2+5q+3,\\
f_4^+(q) &:= 3q^7+6q^6-26q^5-41q^4-73q^3-56q^2-44q-9
\intertext{and}
f_1^-(q) &:= q^4-q^3-q^2-q-2,\\
f_3^-(q) &:= q^7-3q^6-q^5+5q^3+9q^2+7q+6,\\
f_4^-(q) &:= q^9-6q^8+4q^7+29q^6+3q^5-20q^4-74q^3-79q^2-74q-24.
\end{align*}
\end{cor}
\begin{proof}
This follows from Theorem \ref{thm:c4(n)gcd4}, Theorem \ref{thm:numberOfFriezes} and
\[
\lvert \PGL(4,K) \rvert = (q^3+q^2+q+1)(q^2+q+1)(q+1)q^6(q-1)^3.
\]
\end{proof}

In theory, it is likely possible to tackle the case $ k = 5 $ or even higher $ k $ with the same method. However, the last section indicates that this would likely require an extremely large amount of time and effort. Therefore, future efforts should probably focus on simplifying the method or finding a completely new approach before attempting to solve the case $ k = 5 $ or higher.

\newpage

\bibliographystyle{alphaurl}

\newcommand{\etalchar}[1]{$^{#1}$}

\end{document}